\documentclass[10pt]{amsart}
\DeclareRobustCommand{\VAN}[3]{#2}
\usepackage[utf8]{inputenc}
\usepackage{tikz-cd}
\usepackage{stmaryrd}
\usepackage[OT2,T1]{fontenc}
\DeclareSymbolFont{cyrletters}{OT2}{wncyr}{m}{n}
\DeclareMathSymbol{\Sha}{\mathalpha}{cyrletters}{"58}

\usepackage{comment}
\usepackage{fancyhdr}
\usepackage[shortlabels]{enumitem}
\usepackage{amsmath}
\usepackage{mathtools} 
\usepackage{relsize}
\usepackage[bbgreekl]{mathbbol}
\usepackage{amsfonts}
\usepackage{amssymb} 
\usepackage{amsthm}
\usepackage{mathrsfs}
\usepackage{calligra}

\DeclareSymbolFontAlphabet{\mathbb}{AMSb}
\DeclareSymbolFontAlphabet{\mathbbl}{bbold}
\newcommand{\prism}{{\mathlarger{\mathbbl{\Delta}}}}

\usepackage[linktocpage=true, hyperindex,pagebackref]{hyperref}

\usepackage[margin=1.5in]{geometry}

\newcommand\restr[2]{{
  \left.\kern-\nulldelimiterspace 
  #1 
  \vphantom{\big|} 
  \right|_{#2} 
  }}

\mathtoolsset{showonlyrefs}
\numberwithin{equation}{subsection}
\newtheorem{Thm}[subsubsection]{Theorem}
\newtheorem{Lem}[subsubsection]{Lemma}
\newtheorem{Prop}[subsubsection]{Proposition}
\newtheorem{Cor}[subsubsection]{Corollary}
\newtheorem{Conj}[subsubsection]{Conjecture}
\newtheorem{mainThm}{Theorem}
\newtheorem{Claim}[subsubsection]{Claim}
\theoremstyle{definition}
\newtheorem{Def}[subsubsection]{Definition}
\newtheorem{Construction}[subsubsection]{Construction}
\newtheorem{defn}[subsubsection]{Definition}
\newtheorem{Hyp}[subsubsection]{Hypothesis}

\newtheorem{Question}[subsubsection]{Question}
\newtheorem{eg}[subsubsection]{Example}
\newtheorem{Rem}[subsubsection]{Remark}
\newtheorem{Assump}[subsubsection]{Assumption}
\newtheorem{Setup}[subsubsection]{Setup}

\newtheorem{IntroConj}{Conjecture}

\newcommand{\isom}{\xrightarrow{\sim}}

\newcommand{\ebar}{\overline{E}}

\newcommand{\qpbr}{\breve{\mathbb{Q}}_p}

\newcommand{\zpbr}{\breve{\mathbb{Z}}_p}

\newcommand{\qp}{{\mathbb{Q}_{p}}}
\newcommand{\zp}{{\mathbb{Z}_{p}}}
\newcommand{\zlocp}{{\mathbb{Z}_{(p)}}}
\newcommand{\ql}{\mathbb{Q}_{\ell}}
\newcommand{\qlbar}{\overline{\mathbb{Q}}_{\ell}}
\newcommand{\flbar}{\overline{\mathbb{F}}_{\ell}}
\newcommand{\zlbar}{\overline{\mathbb{Z}}_{\ell}}

\newcommand{\zl}{\mathbb{Z}_{\ell}}
\newcommand{\gal}{\operatorname{Gal}}
\newcommand{\afp}{\mathbb{A}_f^p}
\newcommand{\af}{\mathbb{A}_f}

\newcommand{\mad}{\mathrm{ad}}
\newcommand{\ab}{\mathrm{ab}}
\newcommand{\der}{\mathrm{der}}
\newcommand{\qbar}{\overline{\mathbb{Q}}}
\newcommand{\qpbar}{\overline{\mathbb{Q}}_p}
\newcommand{\spec}{\operatorname{Spec}}
\newcommand{\spf}{\operatorname{Spf}}
\newcommand{\spa}{\operatorname{Spa}}
\newcommand{\spd}{\operatorname{Spd}}
\newcommand{\perf}{\operatorname{Perfd}}
\newcommand{\affperfk}{\mathbf{Aff}_k^{\text{perf}}}
\newcommand{\affperf}{\mathbf{Aff}^{\text{perf}}}

\newcommand{\ovfp}{\overline{\mathbb{F}}_{p}}
\newcommand{\fpbar}{\overline{\mathbb{F}}_{p}}

\newcommand{\grgloc}{\mathrm{Gr}_{\mathcal{G}}^{\mathrm{W}}}
\newcommand{\grgmu}{\operatorname{Gr}_{G, -\mu}}

\newcommand{\xspec}{X_{\widehat{G}}^{\mathrm{spec}}}

\newcommand{\shtg}{\operatorname{Sht}_{\mathcal{G}}}

\newcommand{\gad}{G^{\mathrm{ad}}}
\newcommand{\Gad}{\g^{\mathrm{ad}}}
\newcommand{\gadqone}{\g^{\mathrm{ad}}(\mathbb{Q})^{1}}

\newcommand{\calgad}{\mathcal{G}^{\mathrm{ad}}}

\newcommand{\shtgmu}{\operatorname{Sht}_{\mathcal{G},\mu}}
\newcommand{\shtgmuone}{\mathrm{Sht}_{\mathcal{G},\mu, \delta=1}}

\newcommand{\shtlocg}{\operatorname{Sht}^{\mathrm{W}}_{\mathcal{G}}}

\newcommand{\shtlocgmu}{\operatorname{Sht}^{\mathrm{W}}_{\mathcal{G},\mu}}
\newcommand{\shtlocgmuone}{\mathrm{Sht}^{\mathrm{W}}_{\mathcal{G},\mu, \delta=1}}
\newcommand{\shtlocgpmup}{\operatorname{Sht}^{\mathrm{W}}_{\mathcal{G}',\mu'}}
\newcommand{\shtlocgmup}{\operatorname{Sht}^{\mathrm{W}}_{\mathcal{G},\mu'}}

\newcommand{\gisoc}{G_{}\mathrm{-Isoc}}
\newcommand{\gtwoisoc}{G_{2}\mathrm{-Isoc}}

\newcommand{\gvpisoc}{G_{V'}\mathrm{-Isoc}}
\newcommand{\gisocmu}{G\mathrm{-Isoc}_{\le -\mu}}
\newcommand{\gpisocmup}{G'\mathrm{-Isoc}_{\le -\mu'}}

\newcommand{\gisocmumup}{G\mathrm{-Isoc}_{\le -\mu,-\mu'}}
\newcommand{\gpisocmumup}{G'\mathrm{-Isoc}_{\le -\mu,-\mu'}}

\newcommand{\hisoc}{H_{}\mathrm{-Isoc}}

\newcommand{\gpisoc}{G'\mathrm{-Isoc}}

\newcommand{\Q}{\mathbb{Q}}
\newcommand{\R}{\mathbb{R}}
\newcommand{\mC}{\mathbb{C}}
\newcommand{\fp}{\mathbb{F}_p}

\newcommand{\bdtimes}{\buildrel{\boldsymbol{.}}\over\times}

\newcommand{\bun}{\operatorname{Bun}}
\newcommand{\bungmu}{\operatorname{Bun}_{G, -\mu}}
\newcommand{\bung}{\operatorname{Bun}_{G}}

\newcommand{\bungpmup}{\operatorname{Bun}_{G', -\mu'}}
\newcommand{\bungp}{\operatorname{Bun}_{G'}}
\newcommand{\bungmup}{\operatorname{Bun}_{G, -\mu'}}

\newcommand{\bungpmumup}{\operatorname{Bun}_{G',P, -\mu', -\mu}}
\newcommand{\bungmumup}{\operatorname{Bun}_{G,P, -\mu, -\mu'}}
\newcommand{\bungmucmup}{\operatorname{Bun}_{G, -\mu \cup -\mu'}}
\newcommand{\bungmucmupk}{\operatorname{Bun}_{G, -\mu \cup -\mu',k}}

\newcommand{\igs}{\operatorname{Igs}^{\circ}}
\newcommand{\ig}{\operatorname{Ig}^{\circ}}
\newcommand{\scrs}{\mathscr{S}}
\newcommand{\sh}{\operatorname{Sh}}
\newcommand{\msh}{\mathbf{Sh}}
\newcommand{\scrshat}{\widehat{\mathscr{S}}}

\newcommand{\bgmu}{\rB(G,-\mu)}
\newcommand{\bgmup}{\rB(G,-\mu')}
\newcommand{\bgmumup}{\rB(G,P,-\mu,-\mu')}
\newcommand{\bgpmumup}{\rB(G',P,-\mu,-\mu')}

\newcommand{\g}{\mathsf{G}}

\newcommand{\ul}[1]{\underline{#1}}
\newcommand{\mb}{\mathsf{B}}
\newcommand{\zg}{Z_{\mathsf{G}}}

\newcommand{\zglocp}{Z_{\mathsf{G},\zlocp}(\zlocp)}
\newcommand{\zgq}{Z_{\mathsf{G}}(\mathbb{Q})}

\newcommand{\x}{\mathsf{X}}
\newcommand{\mL}{\mathsf{L}}
\newcommand{\ml}{\mathsf{L}}

\newcommand{\mlp}{\mathsf{L}^{+}}
\newcommand{\mkp}{\mathsf{K}^{+}}

\newcommand{\tor}{\mathsf{T}}
\newcommand{\h}{\mathsf{H}}
\newcommand{\f}{\mathsf{F}}
\newcommand{\ho}{\mathsf{H}_{1}}

\newcommand{\gp}{\mathsf{G}'}
\newcommand{\xp}{\mathsf{X}'}

\newcommand{\hy}{{(\mathsf{H}, \mathsf{Y})}}
\newcommand{\hytwo}{{(\mathsf{H}_2, \mathsf{Y}_2)}}

\newcommand{\hytwop}{{(\mathsf{H}_2', \mathsf{Y}_2')}}

\newcommand{\hyo}{{(\mathsf{H}_1, \mathsf{Y}_1)}}

\newcommand{\hyop}{{(\mathsf{H}_1', \mathsf{Y}_1')}}

\newcommand{\gtwo}{\g_{2}}

\newcommand{\gafp}{\mathsf{G}(\afp)}

\newcommand{\gaf}{\mathsf{G}(\af)}
\newcommand{\gx}{{(\mathsf{G}, \mathsf{X})}}
\newcommand{\gxkappa}{{(\mathsf{G}, \mathsf{X},\mathsf{P}, \kappa^p)}}

\newcommand{\gxone}{{(\mathsf{G}_1, \mathsf{X}_1)}}
\newcommand{\gxtwo}{{(\mathsf{G}_2, \mathsf{X}_2)}}
\newcommand{\gxthree}{{(\mathsf{G}_3, \mathsf{X}_3)}}
\newcommand{\gxfour}{{(\mathsf{G}_4, \mathsf{X}_4)}}

\newcommand{\gxonep}{{(\mathsf{G}_1', \mathsf{X}_1')}}
\newcommand{\gxtwop}{{(\mathsf{G}_2', \mathsf{X}_2')}}
\newcommand{\gxthreep}{{(\mathsf{G}_3', \mathsf{X}_3')}}

\newcommand{\gxad}{{(\mathsf{G}^{\mathrm{ad}}, \mathsf{X}^{\mathrm{ad}})}}

\newcommand{\gv}{\mathsf{G}_{V}}

\newcommand{\gwx}{(\mathsf{G}_{W},\mathsf{H}_{W})}

\newcommand{\gvx}{(\mathsf{G}_{V},\mathsf{H}_{V})}
\newcommand{\gvlx}{(\mathsf{G}_{V_{\ml}},\mathsf{H}_{V_{\ml}})}
\newcommand{\gvlxp}{(\mathsf{G}_{V_{\ml}'},\mathsf{H}_{V_{\ml}}')}
\newcommand{\gvxone}{(\mathsf{G}_{V_1},\mathsf{H}_{V_1})}
\newcommand{\gvxtwo}{(\mathsf{G}_{V_2},\mathsf{H}_{V_2})}
\newcommand{\gvxthree}{(\mathsf{G}_{V_3},\mathsf{H}_{V_3})}

\newcommand{\gxp}{(\mathsf{G}', \mathsf{X}')}

\newcommand{\gvxp}{(\mathsf{G}_{V'},\mathsf{H}_{V'})}

\newcommand{\dcirc}{\diamond/\circ}

\newcommand{\scrshg}{\scrs_K\gx}

\newcommand{\shginf}{\operatorname{Sh}_{K_p}\gx}

\newcommand{\shgvinf}{\operatorname{Sh}_{M_p}\gvx}

\newcommand{\shgpinf}{\operatorname{Sh}_{K_p'}\gxp}

\newcommand{\igspre}{\operatorname{Igs}^{\circ,\mathrm{pre}}}

\newcommand{\igspree}{\operatorname{Igs}^{\mathrm{spre}}}
\newcommand{\igse}{\operatorname{Igs}^{\mathrm{s}}}

\newcommand{\igsprecc}{\operatorname{Igs}^{\mathrm{pre},\dcirc}}
\newcommand{\igscc}{\operatorname{Igs}^{\dcirc}}
\newcommand{\igsperf}{\operatorname{Igs}^{\mathrm{perf}}}
\newcommand{\igspreperf}{\operatorname{Igs}^{\mathrm{pre}, \mathrm{perf}}}

\newcommand{\red}{\mathrm{red}}
\newcommand{\pre}{\mathrm{pre}}
\newcommand{\crys}{\mathrm{crys}}

\newcommand{\shcorrpre}{\operatorname{ShCorr}^{\mathrm{pre}}_{\Omega, \mathsf{P}}}
\newcommand{\shcorr}{\operatorname{ShCorr}^{}_{\Omega, \mathsf{P}}}

\newcommand{\igscorrpre}{\operatorname{IgsCorr}^{\mathrm{pre}}_{\Omega, \mathsf{P}}}
\newcommand{\igscorrpreperf}{\operatorname{IgsCorr}^{\mathrm{pre}, \mathrm{perf}}_{\Omega, \mathsf{P}}}
\newcommand{\igscorrperfpre}{\operatorname{IgsCorr}^{\mathrm{pre}, \mathrm{perf}}_{\Omega, \mathsf{P}}}
\newcommand{\igscorrperf}{\operatorname{IgsCorr}^{ \mathrm{perf}}_{\Omega, \mathsf{P}}}

\newcommand{\igscorr}{\operatorname{IgsCorr_{\Omega, \mathsf{P}}}}
\newcommand{\igscorriotakappa}{\operatorname{IgsCorr_{\mathsf{P},\iota,\kappa^p}}}
\newcommand{\igscorrperfiotakappa}{\operatorname{IgsCorr^{\mathrm{perf}}_{\mathsf{P},\iota,\kappa^p}}}
\newcommand{\igscorrperfnaiveiotakappa}{\operatorname{IgsCorr^{\mathrm{perf,naive}}_{\mathsf{P},\iota,\kappa^p}}}

\newcommand{\igscorrpreviotakappa}{\operatorname{IgsCorr^{\mathrm{pre}}_{\mathsf{P}_V,\kappa^p_V}}}
\newcommand{\igscorrvperfpreiotakappa}{\operatorname{IgsCorr^{\mathrm{pre,perf}}_{\mathsf{P}_V,\kappa^p_V}}}
\newcommand{\igscorrvperfiotakappa}{\operatorname{IgsCorr^{\mathrm{perf}}_{\mathsf{P}_V,\kappa^p_V}}}
\newcommand{\igscorrviotakappa}{\operatorname{IgsCorr_{\mathsf{P}_V,\kappa^p_V}}}
\newcommand{\igscorrvoneiotakappa}{\operatorname{IgsCorr_{\mathsf{P}_{V_1},\kappa^p_{V_1}}}}
\newcommand{\igscorrvtwoiotakappa}{\operatorname{IgsCorr_{\mathsf{P}_{V_2},\kappa^p_{V_2}}}}
\newcommand{\igscorrvonetwoiotakappa}{\operatorname{IgsCorr_{\mathsf{P}_{V_1,V_2},\kappa^p_{V_1,V_2}}}}
\newcommand{\igscorrvthreeiotakappa}{\operatorname{IgsCorr_{\mathsf{P}_{V_3},\kappa^p_{V_3}}}}
\newcommand{\igscorrad}
{\operatorname{IgsCorr_{\Theta_2, \Sigma}}}
\newcommand{\igscorradone}{\operatorname{IgsCorr_{\Theta_{3}, \Sigma_{1}}}}
\newcommand{\igscorradfour}{\operatorname{IgsCorr_{\Theta_{2}, \Sigma_{4}}}}

\newcommand{\igscorrone}{\operatorname{IgsCorr_{\Omega_1, \mathsf{P}_1}}}

\newcommand{\igscorrthree}{\operatorname{IgsCorr_{\Omega_3, \mathsf{P}_1}}}

\newcommand{\shtlocgmumup}{\operatorname{Sht}^{\mathrm{W}}_{\mathcal{G},\mu,\mathcal{G}',\mu',P}}

\newcommand{\rH}{\mathrm{H}}
\newcommand{\rZ}{\mathrm{Z}}
\newcommand{\rB}{\mathrm{B}}

\newcommand{\agext}{\mathcal{A}^{\mathrm{Ext}}(\g)}
\newcommand{\agpext}{\mathcal{A}^{\mathrm{Ext}}(\g')}
\newcommand{\ag}{\mathcal{A}(\g)}
\newcommand{\agone}{\mathcal{A}(\g_1)}

\newcommand{\agtwop}{\mathcal{A}^p(\g_2)}

\newcommand{\apg}{\mathcal{A}^p(\mathsf{G})}
\newcommand{\apgtwo}{\mathcal{A}^p(\mathsf{G}_2)}
\newcommand{\agp}{\mathcal{A}(\mathsf{G}')}
\newcommand{\apgp}{\mathcal{A}^p(\gp)}

\newcommand{\agintext}{\mathcal{A}^{\mathrm{Ext}}(\mathcal{G}_{\zlocp})}
\newcommand{\agint}{\mathcal{A}(\mathcal{G}_{\zlocp})}
\newcommand{\agpint}{\mathcal{A}(\mathcal{G}'_{\zlocp})}
\newcommand{\tildepi}{\tilde{\pi}}
\newcommand{\tildef}{\tilde{\mathcal{F}}}
\newcommand{\Div}{\mathrm{Div}^{1}_{E}}

\newcommand{\Divp}{\mathrm{Div}^{1}_{E'}}

\newcommand{\loc}{\operatorname{Par}_{G}}
\newcommand{\cocycle}{Z^1(W_\qp,{\widehat{G}})}
\newcommand{\dualG}{\widehat{G}}

\newcommand{\gm}{\mathbb{G}_{m}}
\newcommand{\HT}{\mathrm{HT}}
\newcommand{\pibar}{\overline{\pi}}
\newcommand{\dr}{\operatorname{dR}}
\newcommand{\indperf}{\operatorname{Ind} \operatorname{Perf}}
\newcommand{\lock}{\operatorname{Par}_{G,k}}
\newcommand{\lochat}{\operatorname{Par}_{G,\phi}^{\wedge}}

\newcommand{\twopartdef}[3]{%
   \left\{
      \begin{array}{ll}
          #1 & \mbox{if } #2 \\
          #3 & \mbox{otherwise}
      \end{array}
   \right.
}

\newcommand{\threepartdef}[6]
{
	\left\{
		\begin{array}{lll}
			#1 & \mbox{if } #2 \\
			#3 & \mbox{if } #4 \\
			#5 & \mbox{} #6
		\end{array}
	\right.
}
\iftrue
\newcommand{\commentPol}[1]{\textcolor{red}{Pol: #1}}
\newcommand{\commentJack}[1]{\textcolor{blue}{Jack: #1}}
\else
\newcommand{\commentPol}[1]{}
\newcommand{\commentJack}[1]{}
\fi

\begin{document}
\title[A geometric Jacquet--Langlands correspondence for Shimura varieties]{A geometric Jacquet--Langlands correspondence for Shimura varieties}

\author{Pol van Hoften} 
\address{School of Mathematical Sciences, Zhejiang University, 866 Yuhangtang Rd, Hangzhou, 310058, P. R. China}
\email{pvhoften@zju.edu.cn}
\thanks{PvH was (partly) funded by the Dutch Research Council (NWO) under the grant VI.Veni.232.127.}
\author[Jack Sempliner]{Jack Sempliner}\email{jsempliner@math.ucla.edu}\address{Mathematical Sciences Building, 520 Portola Plaza, Los Angeles, CA 90095, United States of America}
\thanks{J.S. received funding from the European Research Council (ERC) under the European Union's Horizon 2020 research and innovation program (grant agreement No. 884596).}

\begin{abstract}
We formulate a conjecture predicting the existence of exotic isomorphisms between Igusa stacks associated to different Shimura data whose underlying groups are pure inner forms of each other, and prove it in many cases of interest. In combination with the spectral action of Fargues--Scholze, this allows us to relate the cohomology groups of the two Shimura varieties, giving a geometric incarnation of the Jacquet--Langlands correspondence. An important ingredient in our proofs is the work of Xiao--Zhu on exotic Hecke correspondences between the perfect special fibers of different Shimura varieties, which we reinterpret in terms of exotic isomorphisms of perfect Igusa stacks.
\end{abstract}

\maketitle  
\setcounter{tocdepth}{2}

\tableofcontents

\section{Introduction}

\subsection{Overview} It is an empirical fact, observed by Serre, Ribet, Helm, and Tian--Xiao, that the mod $p$ geometry of \emph{different} Shimura varieties is often related, see \cite{Serre}, \cite{Ribet}, \cite{HelmI}, \cite{HelmII}, \cite{TianXiao}, \cite{TianXiaoII}, \cite{HelmTianXiao}. For example, the mod $p$ fibers of Hilbert modular surfaces associated to a real quadratic field $\mlp$ will naturally contain the mod $p$ fibers of Shimura curves associated to certain quaternion algebras over $\mlp$, if $p$ is split in $\mlp$. These geometric observations have found several applications to the cohomology of Shimura varieties. For example the work of Tian--Xiao has been used in work of Diamond--Kassaei--Sasaki \cite{DiamondKasseaiSasaki} and of Caraiani--Tamiozzo \cite{CaraianiTamiozzo}. It was observed by Caraiani--Tamiozzo that the work of Tian--Xiao gives rise to exotic isomorphisms between Igusa varieties associated to different Shimura varieties; they apply this observation to prove new torsion-vanishing results for Hilbert modular varieties. \smallskip 

More generally, there are sometimes correspondences between the mod $p$ fibers of different Shimura varieties which were named \emph{exotic Hecke correspondences} by Xiao and Zhu \cite{XiaoZhu}. Xiao and Zhu conjectured the existence and shape of quite general exotic Hecke correspondences. Their conjecture predicts the existence of exotic Hecke correspondences between the (perfect) mod $p$ fibers of the Shimura varieties associated to Shimura data $\gx$ and $\gxp$ such that $\g \otimes \af \simeq \g' \otimes \af$. \smallskip

We will state a conjecture, see Conjecture \ref{Conj:IntroMain}, on the existence of exotic isomorphisms between Igusa stacks associated to Shimura data $\gx$ and $\gxp$ such that $\g \otimes \afp \simeq \g' \otimes \afp$. Conjecture \ref{Conj:IntroMain} implies Conjecture \ref{Conj:MainIII} on the existence of exotic Hecke correspondences between the perfect mod $p$ fibers of the corresponding Shimura varieties, generalizing the conjecture of Xiao--Zhu. Our main result, Theorem \ref{Thm:IntroMainIgusaStack}, proves Conjecture \ref{Conj:IntroMain} for many abelian type Shimura data. A crucial ingredient in the proof is a rigidity theorem for Igusa stacks, Theorem \ref{Thm:IntroIgusaDCirc}, which should be of independent interest. Conjecture \ref{Conj:IntroMain} allows us to relate the cohomology groups of the Shimura varieties for $\gx$ and $\gxp$, and we begin by stating our main cohomological results.

\subsection{Main cohomological results} Fix a prime $p$ and an isomorphism $\alpha:\mathbb{C} \xrightarrow{} \qpbar$. Let $\gx$ be a Shimura datum of abelian type satisfying Milne's axiom SV5 and let $\mathsf{P}$ be a $\g$-torsor over $\spec \mathbb{Q}$ that is trivial over $\spec \mathbb{Q}_{\ell}$ for all $\ell \not=p$, and let $\g'=\operatorname{Aut}_{\g}(\mathsf{P})$. We let $\mathsf{X}'$ be a Shimura datum for $\g'$ such that Assumption \ref{Assump:Infinity} holds (this is automatic e.g. if $\g$ is adjoint). There exists a canonical identification ${}^{\mathrm{L}}\g_{\qp} = {}^{\mathrm{L}}\gp_{\qp}$ of the associated Langlands dual groups. We moreover fix an isomorphism $\mathsf{P} \otimes \afp \isom \g \otimes \afp$ inducing $\g \otimes \afp \xrightarrow{\sim} \g' \otimes \afp$, choose $K^p \subset \gafp = \g'(\afp)$ a sufficiently small compact open subgroup, and write $\mathbb{T}_{K^p}$ for the Hecke algebra of level $K^p$ with values in $\Lambda$, and $\mathcal{Z}_{K^p}$ for its center. Write $P=\mathsf{P}_{\qp}$. \smallskip 

Fix a prime $\ell \not=p$, let $\Lambda \in \{\qlbar, \flbar\}$ and fix $\sqrt{p} \in \Lambda$. Let $\loc$ be the moduli stack of local Langlands parameters for ${}^{\mathrm{L}}\g_{\qp}$ over $\Lambda$. Associated to $\x, \xp,\alpha$ and $\sqrt{p}$ there are representations $V_{\mu}, V_{\mu'}$ of $\dualG$ which induce vector bundles $\mathcal{V}_\mu, \mathcal{V}_{\mu'}$ over $\loc$ of dimensions $n,n'$. One of the most striking upshots of the work \cite{FarguesScholze} of Fargues--Scholze is the construction of the so-called \emph{spectral action} of the category of perfect complexes on $\loc$ on the derived category $\mathcal{D}_{\mathrm{lis}}(\bung, \Lambda)$. Using this action, we show the following. Let $e=|\pi_0(Z(G))|$.

\begin{mainThm} \label{Thm:IntroCohomology}
Suppose Conjecture \ref{Conj:IntroMain} holds for $\gx, \mathsf{P}$, that $\kappa(\mu)+\kappa([P])=\kappa(\mu')$ (see Remark \ref{Rem:VacuousConjectureRemark}) and that $e \in \Lambda^{\times}$. Let $\phi:W_{\qp} \to {}^{\mathrm{L}}\g_{\qp}(\Lambda)$ be a semisimple $L$-parameter with centralizer $S_{\phi}$. \\

\noindent \textbf{(1)}: Let $\chi:\mathcal{Z}_{K^p} \to \Lambda$ be a character. If $n,n' \in \Lambda^{\times}$ and $\phi$ is supercuspidal, then
\begin{align}
    R\Gamma(\mathbf{Sh}_{K^p}\gx, \Lambda)_{\phi,\chi} \not = 0 \Leftrightarrow R\Gamma(\mathbf{Sh}_{K^p}\gxp, \Lambda)_{\phi,\chi} \not=0.
\end{align}\\
\textbf{(2):} If $[P] \in H^1(\qp, \g)$ is trivial, then there is a map
    \begin{align}
            \operatorname{RHom}(V_{\mu}, V_{\mu'}) \to \operatorname{RHom}_{\g(\qp) \times \mathbb{T}_{K^p}}\left(R\Gamma(\mathbf{Sh}_{K^p}\gx, \Lambda)[d], R\Gamma(\mathbf{Sh}_{K^p}\gxp, \Lambda)[d'] \right)
    \end{align}
induced by the spectral action. Here $d$ (resp. $d'$) is the dimension of $\mathbf{Sh}_K\gx$ (resp. $\mathbf{Sh}_{K'}\gxp$).\\

\noindent\textbf{(3)}: Suppose that $S_{\phi}=Z(\dualG)^{\gal_{\qp}}$. If $[P] \in H^1(\qp, \g)$ is trivial, then there exists a complex of smooth $\g(\qp) \times \mathbb{T}_{K^p}$-representations $N$ together with quasi-isomorphisms
    \begin{align}
        N^{\oplus n} & \xrightarrow{\sim} R\Gamma(\mathbf{Sh}_{K^p}\gx, \Lambda)_{\phi}[d] \\
        N^{\oplus n'} & \xrightarrow{\sim} R\Gamma(\mathbf{Sh}_{K^p}\gxp, \Lambda)_{\phi}[d']
    \end{align}
that are $G(\qp) \times \mathbb{T}_{K^p}$-equivariant.

\noindent \textbf{(4)} Suppose that $S_{\phi}$ is linearly reductive and that $\phi$ is generous. If $[P] \in H^1(\qp, \g)$ is trivial and $V_{\mu}$ is a direct summand of $V_{\mu'}$ in the category of $S_{\phi}$-representations, then $R\Gamma(\mathbf{Sh}_{K^p}\gx, \Lambda)_{\phi}[d]$ is a $G(\qp) \times \mathbb{T}_{K^p}$-equivariant direct summand of $R\Gamma(\mathbf{Sh}_{K^p}\gxp, \Lambda)_{\phi}[d']$.
\end{mainThm}
Part \textbf{(2)} of Theorem \ref{Thm:IntroCohomology} is closely related to \cite[Conjecture 4.7.16]{ZhuCoherent}.
\begin{Rem}
    A similar looking map to the one in \textbf{(2)} is obtained in the work of Xiao--Zhu in \cite[\S 7.3.11]{XiaoZhu} for the cohomology at \emph{hyperspecial level $K_p$} under the assumption that $[P] \in H^1(\qp,\g)$ is trivial, implying that $\g_{\qp} \simeq \gp_{\qp}$. The advantage of our approach is threefold: First that we are able to study the ramified part of the cohomology of Shimura varieties, second we are able to handle the case that $\g_{\qp} \not\cong \gp_{\qp}$, and third we are able to prove refined results like \textbf{(3)} and \textbf{(4)} using the spectral action. 
\end{Rem}
\begin{Rem}
Part \textbf{(1)} simply asserts the transfer of certain mod $\ell$ Hecke eigensystems from $\g$ to $\g'$ and back, and can be thought of as a qualitative mod $\ell$ Jacquet--Langlands transfer. The isomorphism in Part \textbf{(3)} is a quantitative or "multiplicity preserving" mod $\ell$ Jacquet--Langlands transfer. Part \textbf{(4)} is a less precise version of \textbf{(3)} under much weaker assumptions on $S_{\phi}$. We also prove a version of \textbf{(3)} when $[P]$ is nontrivial, see Theorem \ref{Thm:arii2j3asd}. In Section \ref{Sec:Cohomology} our results are moreover all stated for cohomology with coefficients in automorphic local systems. 
\end{Rem}

\subsection{A conjecture} Our main geometric results are stated in terms of the Igusa stacks introduced by Zhang in \cite{ZhangThesis} and constructed in greater generality by Daniels--van Hoften--Kim--Zhang in \cite{DvHKZIgusaStacks},\cite{DvHKZIgusaStacksII} and by Kim in \cite{KimFunctoriality}. We first state a precise conjecture, and describe our main geometric results afterwards.

\subsubsection{} Fix a prime $p$ and an identification $\mathbb{C} \isom \qpbar$. Let $\gx$ be a Shimura datum of abelian type satisfying Milne's axiom SV5 and let $\mathsf{P}$ be a $\g$-torsor over $\spec \mathbb{Q}$ that is trivial over $\spec \mathbb{Q}_{\ell}$ for all $\ell \not=p$, and let $\g'=\operatorname{Aut}_{\g}(\mathsf{P})$. We let $\mathsf{X}'$ be a Shimura datum for $\g'$ such that Assumption \ref{Assump:Infinity} holds (this is automatic e.g. if $\g$ is adjoint). The isomorphism $\mathbb{C} \isom \qpbar$ defines $p$-adic places $v$ and $v'$ of the reflex fields $\mathsf{E}$ and $\mathsf{E}'$ of $\gx$ and $\gxp$, and we let $E$ and $E'$ be the respective $p$-adic completions. We write $G=\g \otimes \qp$ and $G' = \g' \otimes \qp$. Let $\mu$ be the $G(\qpbar)$-conjugacy class of Hodge cocharacters induced by $\mathsf{X}$ and $v$, and let $\mu'$ be the $G'(\qpbar)$-conjugacy class of Hodge cocharacters induced by $\mathsf{X}'$ and $v'$.

\subsubsection{} The $G$-torsor $P=\mathsf{P} \otimes \qp$ over $\spec \qp$ induces an isomorphism $\beta_{P}:\bun_{G} \to \bun_{G'}$ by work of Fargues--Scholze, see \cite[Section III.4.1]{FarguesScholze}. We consider the open substacks $\bungmu \subset \bun_G$ and $\bungpmup \subset \bun_{G'}$ and consider their intersection
\begin{align}
    \bungpmumup:=\bungpmup \times_{\bun_{G'}} \beta_P(\bungmu).
\end{align}
\subsubsection{} Let $\pi:\igs \gx \to \bungmu$ (resp. $\pi':\igs \gxp \to \bungpmup$) be the Igusa stacks\footnote{The superscript $\circ$ indicates that we are using the Igusa stacks that are quotients of the good reduction locus of the generic fiber of the Shimura variety. We expect a version of Conjecture \ref{Conj:IntroMain} to hold for the larger Igusa stacks $\operatorname{Igs}\gx \supset \igs \gx$ constructed in \cite{KimFunctoriality},\cite{DvHKZIgusaStacksII}.} constructed in \cite{DvHKZIgusaStacksII}. We define the open substack $\igs \gx_{-\mu'} \subset \igs \gx$ to be the inverse image of $\bungpmumup$ and similarly we define $\igs \gxp_{-\mu} \subset \igs \gxp$. Choose an isomorphism $\kappa^p:\mathsf{P} \otimes \afp \isom \g \otimes \afp$ which induces an isomorphism $\gp \otimes \afp \isom \g \otimes \afp$ which we will use to identify $\gp(\afp)$ with $\gafp$. 
\begin{IntroConj}[Conjecture \ref{Conj:Main}] \label{Conj:IntroMain}
Let $\gx$ and $\mathsf{P}$ be as above. If $\Sha^1(\mathbb{Q},\g)=0$, then there is a $\gafp$-equivariant isomorphism
    \begin{align}
        \igs \gx_{-\mu'} \to \igs \gxp_{-\mu}
    \end{align}
    of v-stacks over $\bungp$. 
\end{IntroConj}
Conjecture \ref{Conj:IntroMain} can be used to show the existence of exotic Hecke correspondences between the special fiber of the Shimura varieties for $\gx$ and the Shimura varieties for $\gxp$, see Section \ref{Sub:Conjectures} and Conjecture \ref{Conj:MainIII}. This conjecture on exotic Hecke correspondences is a generalization of a conjecture of Xiao--Zhu, see \cite[Hypothesis 7.3.2]{XiaoZhu}, who deal with the case that $[P] \in H^1(\qp,G)$ is trivial and that $G$ is unramified. Note that Xiao--Zhu have announced a proof of their conjecture.
\begin{Rem} \label{Rem:VacuousConjectureRemark}
Let $\pi_1(G)$ be the algebraic fundamental group of $G$ equipped with its action of the Galois group $\gal_{\qp}$ of $\qp$, and note that $P$ induces a $\gal_{\qp}$-equivariant isomorphism $\pi_1(G) \isom \pi_1(G')$. The open substack $\bungmumup \subset \bun_{G}$ is nonempty precisely when $\kappa(\mu)+\kappa([P])=\kappa(\mu')$ in $\pi_1(G)_{\gal_{\qp}}=\pi_1(G')_{\gal_{\qp}}$. Thus Conjecture \ref{Conj:IntroMain} is vacuously true if $\bungmumup$ is empty. This also explains the occurrence of the assumption $\kappa(\mu)+\kappa([P])=\kappa(\mu')$ in Theorem \ref{Thm:IntroCohomology}; we should not be able to deduce cohomological consequences from a vacuous conjecture! 
\end{Rem}
\begin{Rem} \label{Rem:Sha1Trivial}
    The assumption that $\Sha^1(\mathbb{Q},\g)=0$ can be removed from the conjecture if one replaces the Shimura varieties (and thus the Igusa stacks) for $\gx$ with the rational Shimura varieties for $\gx$ of Sempliner--Taylor \cite{SemplinerTaylorShimura} or the extended Shimura varieties of Xiao--Zhu \cite{XiaoZhu2}. The conjecture should moreover be generalized to allow $\mathsf{P}$ to be a basic Kottwitz cocycle in the sense of \cite{sempliner2024cocycleskottwitzcohomology}. In this generality, we only expect an isomorphism after basechange along $\spd \fpbar \to \spd \fp$ because $\beta_{P}$ will generally only be defined over $\fpbar$.

    A special case of this generalized conjecture is the case that $\g'(\mathbb{R})$ is compact modulo center. In this case, the generalized conjecture gives a description of the basic Igusa variety. We prove this in some cases by relating it to basic uniformization, see Corollary \ref{Cor:BasicIgusaVariety}. 
\end{Rem}
\subsection{Main geometric results}
We can now state our main geometric results.
\begin{mainThm} \label{Thm:IntroMainIgusaStack}
Let $\gx$ be a Shimura datum of abelian type. Assume that $G$ splits over an unramified extension of $\qp$. Conjecture \ref{Conj:IntroMain} holds in the following situations\footnote{Note that $\Sha^1(\mathbb{Q},\g)=0$ for the groups considered in Theorem \ref{Thm:IntroMainIgusaStack} so that the statement is not vacuous.}:
\begin{itemize}
    \item The Shimura datum $\gx$ is of PEL type $A_{2n+1}$ and $p>2$.

    \item Conjecture \ref{Conj:XiaoZhu} holds and $\g=\operatorname{Res}_{\mlp/\mathbb{Q}} \operatorname{SO}(V)$ for a quadratic space $V$ of odd dimension over a totally real field $\mlp$. 

    \item Conjecture \ref{Conj:XiaoZhu} holds and $\g=\operatorname{Res}_{\mlp/\mathbb{Q}} \mathsf{J}$ for a totally real field $\mlp$ and $\mathsf{J}$ an inner form of $\operatorname{PGSp}_{2g}$ over $\mlp$.
\end{itemize}
\end{mainThm}
We note that a proof of Conjecture \ref{Conj:XiaoZhu} has been announced by Xiao--Zhu. Conjecture \ref{Conj:XiaoZhu} predicts the existence of exotic Hecke correspondence between the special fibers of certain Shimura varieties of Hodge type of hyperspecial level. More precisely, we give a definition of a space of exotic Hecke correspondences and Conjecture \ref{Conj:XiaoZhu} asserts that it is of the expected shape (e.g. that it is nonempty). 

\begin{mainThm} \label{Thm:IntroMainIgusa}
If Conjecture \ref{Conj:IntroMain} holds, then for $b: \spd \ovfp \to \bungmumup$, there is a $\gafp \times \tilde{G}_b$-equivariant isomorphism of v-sheaf Igusa varieties
\begin{align}
    \operatorname{Ig}^{b,\mathrm{v}}\gx \isom \operatorname{Ig}^{b,\mathrm{v}}\gxp.
\end{align}
\end{mainThm}
A version of this result was proved for Hilbert modular varieties by Caraiani--Tamiozzo, see \cite[Theorem 4.2.4]{CaraianiTamiozzo}. During the process of writing up this manuscript, a version of Conjecture \ref{Conj:IntroMain} was proved for Hilbert modular varieties in \cite{jiang2026classicalityhilbertmodularforms}.

\subsection{The proof of Theorem \ref{Thm:IntroMainIgusaStack}} The proof starts with a reduction to the case of a Hodge type Shimura datum $\hyo$ with $\h_{1,\qp}$ unramified and a torsor $\mathsf{P}_1$ which is trivial over $\af$, we will explain this reduction step in Section \ref{subsub:IntroReduction}. We now first explain the proof for $\hyo$.

\subsubsection{} The Igusa stack $\igs \hyo$ is a sheaf for the v-topology on the category $\perf$ of affinoid perfectoid spaces $\spa(R,R^+)$ in characteristic $p$. A presheaf $\mathcal{F}$ on $\perf$ can be turned into a presheaf on the category $\affperf$ of affine perfect schemes $\spec A$ by sending $\mathcal{F}$ to the presheaf $\spec A \mapsto \mathcal{F}((\spec A)^{\diamond})$; this operation turns v-sheaves into v-sheaves. The reduction of the Igusa stack $\igs \hyo^{\red}$ has been identified in \cite[Theorem 6.5.1]{DvHKZIgusaStacks} as a v-sheaf quotient of the perfect scheme $\operatorname{Sh}_{K_p}\hyo$, the perfect special fiber of a canonical integral model of the Shimura variety for some choice of parahoric model $\mathcal{G}$ of $G$ (with $\mathcal{G}(\zp)=K_p$). Since $G$ is a quasi-split and splits over an unramified extension, we may take $\mathcal{G}$ to be reductive; this is crucial later. We will denote the reduction of the Igusa stack by $\igsperf \hyo$, the so-called \emph{perfect} Igusa stack. \smallskip 

There is a variant of Conjecture \ref{Conj:IntroMain} for perfect Igusa stacks, see Conjecture \ref{Conj:MainII}. This asserts the existence of an isomorphism
\begin{align}
    \igsperf \hyo_{\mu'} \xrightarrow{\sim} \igsperf \hyop_{\mu},
\end{align}
where now $\igsperf \hyo_{\mu'} \subset \igsperf \hyo$ is a \emph{closed} union of Newton strata. Given 
\cite[Theorem 6.5.1]{DvHKZIgusaStacks}, it is fairly straightforward to show that this conjecture is equivalent to a conjecture about exotic Hecke correspondences between the perfect special fibers $\shginf$ and $\shgpinf$, see Conjecture \ref{Conj:XiaoZhu}. 

\subsubsection{} The operation $\mathcal{F} \mapsto \mathcal{F}^{\mathrm{red}}$ typically loses a lot of information. For example it sends the diamond of a formal scheme over $\spf \zp$ to the perfect special fiber of that formal scheme, and the reduction of a locally spatial diamond (e.g. a perfectoid space) is empty by a result of Gleason, see \cite[Proposition 4.8.(4)]{GleasonSpecialization}. Quite surprisingly, it turns out that applying the reduction functor to the Igusa stack does not lose any information. \smallskip 

More precisely, given a presheaf $\mathcal{G}$ on $\affperf$ we can turn it into a presheaf $\mathcal{G}^{\diamond/\circ, \pre}$ on $\perf$ which sends $(R,R^+)$ to $\mathcal{G}(R^+/R^{\circ\circ})$, where $R^{\circ \circ} \subset R^+$ is the ideal of topologically nilpotent elements (this operation was introduced by Heuer \cite{HeuerPicard}, and expanded upon by Gleason \cite{GleasonSpecialization}). This does not generally turn v-sheaves into v-sheaves, and so we need to distinguish $\mathcal{G}^{\diamond/\circ, \pre}$ from its v-sheafification $\mathcal{G}^{\diamond/\circ}$. 
\begin{mainThm} \label{Thm:IntroIgusaDCirc}
Let $\hyo$ be a Shimura datum of Hodge type. If $H_1=\h_{1,\qp}$ is unramified, then there is a natural isomorphism
    \begin{align}
        (\igsperf\hyo)^{\diamond/\circ} \to \igs \hyo.
    \end{align}
\end{mainThm}
It follows from results of \cite{GleasonIvanovZillinger} that under this isomorphism the subfunctor $(\igsperf\hyo_{\mu'})^{\diamond/\circ}$ is identified with $\igs \hyo_{\mu'}$. Thus an isomorphism $\igsperf\hyo_{\mu'} \to \igsperf\hyop_{\mu}$ induces an isomorphism $\igs \hyo_{\mu'} \to \igs \hyop_{\mu}$. Therefore, in this special case, we can deduce Conjecture \ref{Conj:IntroMain} from a statement in perfect algebraic geometry.
\begin{Rem}
If $\mathcal{G}=\mathbb{A}^1_{\fp}$, then $\mathcal{G}^{\diamond/\circ}$ is the (v-sheafification of) the functor
\begin{align}
    (R,R^+) \mapsto (R^+/R^{\circ \circ}).
\end{align}
Geometrically, the functor $(R,R^+) \mapsto R^+$ can be thought of as the closed unit disk while $(R,R^+) \mapsto R^{\circ \circ}$ can be thought of as the open unit disk. Thus the sheaf $\mathcal{G}^{\diamond/\circ}$ is a rather exotic object, although it is a smooth Artin v-stack of $\ell$-cohomological dimension zero for all $\ell \not=p$, like the Igusa stack.
\end{Rem}

\subsubsection{} Our proof of Theorem \ref{Thm:IntroIgusaDCirc} is where a significant amount of the technical work of this paper happens. One ingredient in the proof is the following generalization of an observation of Heuer \cite{HeuerProEtaleUniformisation}.
\begin{Lem}[Lemma \ref{Lem:SpreadingOut}]
    For a perfectoid Huber pair $(R,R^+)$, the category of abelian schemes over $\spf R^+$ up to formal\footnote{See Definition \ref{defn:formalQIsog}.}quasi-isogeny is equivalent to the category of abelian schemes over $R^+/R^{\circ \circ}$ up to quasi-isogeny. 
\end{Lem}
Generalizing this lemma to Shimura varieties of Hodge type requires us to tackle some questions in (rational) $p$-adic Hodge theory. We do this by making use of the prismatic $F$-crystals with $\mathcal{G}$-structure (with $\mathcal{G}$ reductive) constructed by Imai--Kato--Youcis and Madapusi--Youcis, \cite{ImaiKatoYoucis}, \cite{MadapusiYoucis}; this is where the assumption that $G$ is unramified comes in. We also crucially use the smoothness of the canonical integral models for $\mathcal{G}$ reductive. To finish the proof of Theorem \ref{Thm:IntroIgusaDCirc}, we need to deal with subtleties regarding the interaction of $\mathcal{G} \mapsto \mathcal{G}^{\diamond/\circ}$ with v-sheafification, see Question \ref{Question:VtopologyLemmaReduction}.

\subsubsection{The reduction step} \label{subsub:IntroReduction} We now explain how to reduce Theorem \ref{Thm:IntroMainIgusaStack} to the Hodge type case discussed above. Let us call our abelian type Shimura datum $\gxtwo$ with torsor $\mathsf{P}_2$ trivial over $\afp$. First we choose a Galois totally real field $\f$ in which $p$ is unramified and write $\h_2 = \operatorname{Res}_{\f/\mathbb{Q}} \g_{2,\f}$ and $\mathsf{P}_{2,\f}=\mathsf{P}_{2} \times^{\g_2} \h_{2}$ and $\Gamma=\gal_{\f/\mathbb{Q}}$. We may choose $\f$ such that $H_{2} = \h_{2} \otimes \qp$ is quasi-split and $\mathsf{P}_{2,\f}$ is trivial at $p$. We prove using the methods of \cite{vHSemplinerFixedPoints} that there is a Cartesian diagram
\begin{equation}
    \begin{tikzcd}
        \igs\gxtwo \arrow{r} \arrow{d} & \igs \hytwo^{h \Gamma} \arrow{d} \\
        \bun_{G_{2},-\mu_2} \arrow{r} & \bun_{H_{2}}^{h \Gamma}, 
    \end{tikzcd}
\end{equation}
where $h \Gamma$ denotes homotopy fixed points; we prove the same statement for $\gxtwop$. It follows from this that a $\Gamma$-equivariant version of Conjecture \ref{Conj:IntroMain} for $\hytwo$ and $\mathsf{P}_{2,\f}$ implies Conjecture \ref{Conj:IntroMain} for $\gxtwo$ and $\mathsf{P}_2$. \smallskip 

Next, we choose a Hodge type Shimura variety $\gx$ together with an ad-isomorphism $\gx \to \gxtwo$ and a lift $\mathsf{P}$ of $\mathsf{P}_2$ which is trivial over $\afp$. Using the totally real field $\f$, we construct the auxiliary Hodge type Shimura datum $\hyo \subset  \operatorname{Res}_{\f/\mathbb{Q}} \g_{\f}$ with an action of $\Gamma$ such that $H_1=\h_{1} \otimes \qp$ is quasi-split and $\mathsf{P}_1=\mathsf{P} \times^{\g} \h_1$ is trivial at $p$. We then prove a $\Gamma$-equivariant version of Conjecture \ref{Conj:IntroMain} for $\hyo$ and $\mathsf{P}_1$ as explained above. Finally, we deduce from this a $\Gamma$-equivariant version of Conjecture \ref{Conj:IntroMain} for $\hytwo$ and $\mathsf{P}_{2,\f}$, by descending along the morphism $\hyo \to \hytwo$.

Choosing $\gx$ is far from routine: First of all we need to lift the torsor $\mathsf{P}_2$ to a torsor $\mathsf{P}$ which remains trivial over $\afp$. Second of all, we need $\h_1$ to satisfy rather many properties, see Definition \ref{Def:ExcellentHodgeTypeLifting}. For instance, we need $\h_{1,\qp}$ to be quasi-split and split over an unramified extension, and $\Sha^1(\mathbb{Q},\h_1)=0$. We check that such $\gx$ exists by an extensive case-by-case analysis, see Section \ref{Sec:Examples} and Appendix \ref{Appendix:HodgeEmbeddings}.

\begin{Rem}
It would be interesting to prove Theorem \ref{Thm:IntroMainIgusaStack} in the case that $\g_2$ is an inner form of $\operatorname{Res}_{\mlp/\mathbb{Q}} \operatorname{GSp}_{2g}$ rather than its adjoint quotient. There are two difficulties to extend our methods to this case, which we hope to be surmountable with additional effort: The first is that the methods of \cite{vHSemplinerFixedPoints} require the assumption that the center of $\g_2$ has $\mathbb{R}$-split rank equal to its $\mathbb{Q}$-split rank, which is false here. The second is that inner forms of $\operatorname{Res}_{\mlp/\mathbb{Q}} \operatorname{GSp}_{2g}$ are typically not quotients of Hodge type Shimura data. It would similarly be interesting to prove Theorem \ref{Thm:IntroMainIgusaStack} in the case that $\g_2=\operatorname{Res}_{\mlp/\mathbb{Q}} \operatorname{SO}(V)$ for a quadratic space $V$ of even dimension over a totally real field $\mlp$. The difficulty here is that $\gxtwo$ is not a quotient of a Shimura datum of Hodge type with connected center. 
\end{Rem}

\subsection{Proofs of the main cohomological results} Our main cohomological results follow by studying the cohomology of Shimura varieties via Igusa stacks, as in \cite{DvHKZIgusaStacks}, \cite{DvHKZIgusaStacksII}, and using the spectral action of Fargues--Scholze \cite{FarguesScholze}. 

\subsubsection{} We consider the Igusa sheaves $\mathcal{F}=R \pi_{\ast} \Lambda$ and $\mathcal{F}'=R \pi'_{\ast} \Lambda$ which we think of as living on $\bung$ via $\beta_{P}^{-1}$.\footnote{When $\ell$ is not torsion in $\Lambda$, then we have to define $\mathcal{F}$ in an indirect way following \cite{CaraianiHamannZhang}, see Section \ref{subsub:Rational}.} By Conjecture \ref{Conj:IntroMain} and smooth base change, these sheaves agree on the open substack
\begin{align}
    \bungmumup \subset \bung
\end{align}
which we have assumed to be nonempty in Theorem \ref{Thm:IntroCohomology}. Thus we can glue them together to a sheaf $\tildef$ on $\bung$; alternatively we can glue together the Igusa stacks and pushforward $\Lambda$ along its structure map to $\bung$. We show that the sheaf $\tildef$ controls the cohomology of both Shimura varieties, in that
\begin{align}
    i_1^{\ast} T_{\mu} \tildef &= R\Gamma(\mathbf{Sh}_{K^p}\gx, \Lambda)[d] \\
    i_{1'}^{\ast} T_{\mu'} \tildef&=R\Gamma(\mathbf{Sh}_{K^p}\gxp, \Lambda)[d'].
\end{align}
Here $i_1$ is the inclusion of the neutral stratum of $\bung$, and $i_{1'}$ is the inclusion of the neutral stratum of $\bungp \xrightarrow{\beta_{P}^{-1}} \bung$. We do this by studying how the Hecke operators interact with the isomorphism $\beta_{[P]}$, and using results of \cite[Section 8]{DvHKZIgusaStacks} and \cite[Section 5]{DvHKZIgusaStacksII}. If $[P]$ is trivial then $i_1=i_{1'}$, and part \textbf{(2)} of Theorem \ref{Thm:IntroCohomology} is now a direct consequence of the functoriality of the spectral action. \smallskip 

We prove part \textbf{(1)} of Theorem \ref{Thm:IntroCohomology} using the conservativity of Hecke operators, together with the fact that the sheaves localized at supercuspidal parameters $\phi$ are supported on the basic locus; this argument was suggested to us by David Hansen. For part \textbf{(3)} and \textbf{(4)} of Theorem \ref{Thm:IntroCohomology}, write $\lochat$ for the formal completion of $\loc$ in the inverse image of $\phi$ under $\loc \to \xspec$. Then since $S_{\phi}$ is linearly reductive, it follows from Luna's \'etale slice theorem (of \cite{AlperHallRydh}) that  
\begin{align}
    \lochat \simeq \left[ \spf A / S_{\phi} \right].
\end{align}
We show that $\operatorname{Perf}(\lochat)$ acts on $\mathcal{D}(\bun_{G}, \Lambda)_{\phi} \ni \tilde{\mathcal{F}}_{\phi}$, see Corollary \ref{Cor:LocalizationCompletion}. Part \textbf{(4)} now follows from the additivity of the spectral action. If $S_{\phi}=Z(\dualG)^{\gal_{\qp}}$, then $\restr{V_{\mu}}{\lochat}$ is isomorphic to $\mathcal{L}^{\oplus n}$, where $\mathcal{L}$ is the character of $Z(\hat{G})^{\gal_{\qp}}$ corresponding to the element $\kappa(\mu)$ in $\pi_1(G)_{\gal_{\qp}}=X^{\ast}(Z(\hat{G})^{\gal_{\qp}})$. Since $[P]$ is trivial it follows that $\kappa(\mu)=\kappa(\mu')$ and thus $\restr{V_{\mu'}}{\lochat} \simeq  \mathcal{L}^{\oplus n'}$. The theorem now follows from the additivity of the spectral action. 

\subsection{Tool and computational resource disclosure} We used ChatGPT Pro to proofread this article. Its main contribution was to discover that Lemma \ref{Lem:ConnectedCenter} is false in type $D$.\footnote{This error was discovered at a late stage, which is why a large number of our arguments do include groups of type $D$, while our main theorem does not include groups of type $D$.} The mathematics and writing were done by the authors. 

\subsection{Acknowledgments} We would like to thank George Boxer, Ana Caraiani, Marco D'Addezio, Patrick Daniels, Toby Gee, David Hansen, Dongryul Kim, Zhiyou Wu, Mingjia Zhang and Konrad Zou for helpful conversations. The material in Sections \ref{Sub:ReductionAnalytification} and \ref{Sub:ReductionIgusaStack} benefited greatly from discussions with Ian Gleason. The material in Section \ref{Sec:Cohomology} owes a special debt to the work of Nguyen \cite{nguyen}, which was a source of great inspiration. The second named author would like to thank Hamann for some useful early conversations about the categorical conjecture. Finally, we would like to thank Liang Xiao and Xinwen Zhu for many stimulating conversations about their inspiring works \cite{XiaoZhu},\cite{XiaoZhu2}. \smallskip 

This project started while both authors were attending the 2022 IHES summer school on the Langlands Program, and we would like to thank the IHES for providing excellent working conditions. The first-named author gives special thanks to Yau Mathematical Sciences Center and the Morningside Center for Mathematics for inviting him to give a series of talks about this project. He also thanks Stanford University and VU Amsterdam where a significant portion of this work was carried out, and Imperial College London for hosting him during the preparation of this work. The second named author thanks VU Amsterdam, Imperial College London, and Stanford for providing excellent working conditions during the preparation of this manuscript.

\section{Preliminaries} In this section we collect some preliminaries, starting with Archimedean preliminaries in Section \ref{Sec:Archimedes}. In Section \ref{Sub:PerfectoidPrelim} we collect some preliminaries in perfectoid geometry, and in Section \ref{Sub:PerfectPrelim} we recall some perfect algebraic geometry and prove a few lemmas. In Section \ref{Sub:ReductionAnalytification}, we discuss various functors intermediating between perfectoid geometry and perfect algebraic geometry. In Section \ref{Sub:PrismaticComparison} we discuss the prismatic theory of crystalline local systems, and in Section \ref{Sub:IgusaStacks} we recall the theory of Igusa stacks. Finally in \ref{Sub:Conjectures} we state our main conjectures on exotic isomorphisms of (perfect) Igusa stacks and exotic Hecke correspondences, and discuss their dependencies.

\subsection{Archimedean preliminaries} \label{Sec:Archimedes}
The goal of this subsection is to introduce the real Kottwitz set $B(\R,G)$ for a connected reductive group $G$ over $\mathbb{R}$, and to discuss its relation with Shimura data. 

\subsubsection{} Let us write $\mathbb{S}=\operatorname{Res}_{\mC/\R} \gm$ for the Deligne torus, and $S^1 \subset \mathbb{S}$ for the kernel of the norm homomorphism to $\gm$. 

\subsubsection{} Let $G/\mathbb{R}$ be a real reductive algebraic group and let $T_0 \subset G$ be a maximal compact torus in $G$ (a torus isomorphic to $(S^1)^d$ for some $d$, which is maximal with that property). Let $T=Z(T_0)$ be the centraliser of $T_0$, which is a maximal torus of $G$ by \cite[second paragraph of Section 3]{BorovoiReal}. Set $N_0 = N_G(T_0)$, $W_0 = N_0/T$, and we let $T_1 \subset T$ denote the maximal split subtorus of $T$. It follows from \cite[Lemma 2.4]{BorovoiReal} that $T_0$ and $T$ are unique up to $G(\R)$-conjugacy, and we call $T$ a \emph{fundamental torus} of $G$.

\subsubsection{} By a result of Borovoi \cite[Theorem 3.1]{BorovoiReal} there exists a well-defined action of $W_0$ on $\rH^1(\mathbb{R}, T)$ via the following formula: If $w \in W_0(\mathbb{C}), c \in \rZ^1(\mathbb{R}, T)$ then $c$ is uniquely determined by its value on complex conjugation $\sigma$, and we define $w \cdot c(\sigma) = n^{-1} c(\sigma) \bar{n}$ where $n \in N_0(\mathbb{C})$ lifts the element $w$. This induces a right action on $\rH^1(\mathbb{R}, T)$ which can easily be shown to be independent of the choice of lift $n$, and the choice of cocycle $c$. 

\begin{Prop}\cite[Lemma 2.1, Theorem 3.1]{BorovoiReal}\label{Prop:BorovoiRealCohomology}
    The natural inclusion 
    \[
        \rH^1(\mathbb{R}, T) \to \rH^1(\mathbb{R},G)
    \]
    factors through the quotient
    \[
        \rH^1(\mathbb{R}, T) \to \rH^1(\mathbb{R}, T)/W_0 
    \]
    and the induced map 
    \[
        \psi_T: \rH^1(\mathbb{R}, T)/W_0 \to \rH^1(\mathbb{R}, G)
    \]
    is an isomorphism of pointed sets.
    
    Further one has the following description of the Galois cohomology group $\rH^1(\mathbb{R}, T)$,
    \[
    \rH^1(\mathbb{R}, T) \simeq T_0(\mathbb{R})[2]/(T_0 \cap T_1)(\mathbb{R})[2]
    \]where the map $T_0(\mathbb{R})[2] \to \rH^1(\mathbb{R}, T)$ is obtained by taking $t$ to the cocycle $c(1) = 1, c(\sigma) = t$ where $\sigma$ is complex conjugation. 
\end{Prop}
\begin{Lem}\label{Lem:RealSquare}
    Let $\g$ be a connected reductive algebraic group over $\mathbb{R}$, and let $[P] \in \rH^1(\mathbb{R}, \g)$ be such that $\g' = \operatorname{Aut}(P)$. Then there exists a natural cartesian diagram of pointed sets
    \begin{equation}
        \begin{tikzcd}
            \rH^1(\mathbb{R}, \g) \arrow{r}{\beta_P}\arrow{d}{\kappa} & \rH^1(\mathbb{R}, \g')\arrow{d}{\kappa} \\
            \pi_1(\g)_{\gal_{\mathbb{R}}} \arrow{r}{+ \kappa([P])} & \pi_1(\g')_{\gal_{\mathbb{R}}}
        \end{tikzcd}
    \end{equation}
    and the morphism $\beta_P$ is a bijection. 
\end{Lem}
\begin{proof}
    Taking a complex point of $P$ we obtain an isomorphism 
    \[
        \psi: \g_{\mathbb{C}} \to \g'_{\mathbb{C}} 
    \]
    such that if $\sigma_\g$ is complex conjugation on $\g_{\mathbb{C}}$ and $\sigma_{\g'}$ is complex conjugation on $\g'_{\mathbb{C}}$, we have that 
    \[
        \psi^{-1} \circ \sigma_{\g'} \psi = \operatorname{Ad}_{k} \circ \sigma_\g,
    \]
    where $\tau(1) = 1, \tau(c) = k$ is a cocycle for $\gal_\R = \{1, c\}$. The usual pushout procedure $[Q] \to [Q \times^{\g} P]$ defines an isomorphism between the set of isomorphism classes of right $\g$-torsors to the set of left $\g'$-torsors, which is the map whence the isomorphism $\beta_P$ comes from. To calculate Kottwitz invariants however, we need to be slightly more explicit. We instead define a map on cocycles, let $\phi: \gal_\R \to \g$ be a Galois cocycle, then we define $\beta_P(\phi)(1) = 1, \beta_P(\phi)(c) = \psi(\phi(c) \cdot k^{-1})$. If $\phi'(c) = x\phi(c)(x)^{c, -1}$ we check that
    \begin{align*}
        \beta_P(\phi')(c) &= \psi(x) \cdot \psi(\phi(c)) \cdot \psi(\sigma_{\g}(x)^{-1}) \cdot \psi(k^{-1}) \\
        &= \psi(x) \cdot \psi(\phi(c) \cdot k^{-1}) \cdot \sigma_{\g'}(\psi(x^{-1}))
    \end{align*}
    thus $\beta_P(\phi')$ is cohomologous to $\beta_P(\phi)$. On the other hand 
    \begin{align*}
        \psi(\phi(c) \cdot k^{-1}) \cdot \sigma_{\g'}(\psi(\phi(c) \cdot k^{-1})) &= \psi(\phi(c)) \cdot \psi(k)^{-1} \cdot (\psi \circ \operatorname{Ad}_{k} \circ \sigma_\g)(\phi(c) \cdot k^{-1}) \\ 
        &= \psi(\phi(c) \cdot \sigma_{\g}(\phi(c))) \cdot \sigma_{\g}(k)^{-1}k^{-1} \\
        &= 1,
    \end{align*}
    where the final line follows because $\g'$ is a pure inner form. Thus $\beta_P$ induces a bijection between normalized cocycles of $\g$ and those of $\g'$ which respects the equivalence relation defined by cohomology, thus a fortiori $\beta_P$ induces a bijection on cohomology.

    It remains to verify that this twisting isomorphism is compatible with the Kottwitz morphism. For this we first assume that $\g_{\mathrm{der}}$ is simply connected, in this case the Kottwitz morphism $\kappa_{\mathbb{R}}: \rH^1(\mathbb{R}, \g) \to \pi_1(\g)_{\gal_\R}$ is simply the natural map 
    \[
        \rH^1(\mathbb{R}, \g) \to \rH^1(\mathbb{R}, \g^{\mathrm{ab}})
    \]
    composed with the Tate--Nakayama morphism
    \[
        \rH^1(\mathbb{R}, \g^{\mathrm{ab}}) \xrightarrow \pi_1(\g)_{\gal_\R}
    \]
    where here $\g^{\mathrm{ab}}$ denotes the cocenter of $\g$. Thus the statement is clear. In the general case we can pass to a z-extension of $\g$, and thus reduce to the case that $\g_{\mathrm{der}}$ is simply connected. 
\end{proof}

\subsubsection{} Kottwitz has defined a functorial cohomology set $\rB(\mathbb{R}, \g)$, we refer to \cite[Section 10]{kottwitzglobal} for additional details. This comes equipped with a functorial Kottwitz map (\cite[Section 11]{kottwitzglobal})
\begin{align}
    \kappa_{\mathbb{R}, \g}: \rB(\mathbb{R}, \g) \to \pi_1(\g)_{\sigma},
\end{align}
where $\pi_1(\g)$ is the algebraic fundamental group of $\g$ equipped with the natural action of $\gal(\mC/\R)$; we use $\sigma \in \gal(\mC/\R)$ to denote the nontrivial element. In the case that $\g=\tor$ is a torus we have that $\pi_1(\tor) \simeq X_{\ast}(\tor)$ and in this case the map $\kappa_{\mathbb{R}, \tor}$ defines an isomorphism $\rB(\mathbb{R}, \tor) \simeq X_{\ast}(\tor)_{\sigma}$. 

\subsubsection{} We define $\rB(\mathbb{R}, \g)_{\mathrm{bsc}} \subset \rB(\mathbb{R}, \g)$ to the subset of basic elements, that is those elements $[b]$ with $\nu_{[b]}$ central. Then there exists a functorial inclusion $\rH^1(\mathbb{R}, \g) \subset \rB(\mathbb{R}, \g)_{\mathrm{bsc}}$ (see \cite[Section 1.3.4]{kottwitzglobal}. For $\tor \subset \g$ a fundamental maximal torus, we define $\rB(\mathbb{R}, \tor)_{{\g-\mathrm{bsc}}}$ to be the preimage of $\rB(\mathbb{R}, \g)_{\mathrm{bsc}}$ under the natural map $\rB(\mathbb{R}, \tor) \to \rB(\mathbb{R}, \g)$. It is shown in \cite[Section 13.5]{kottwitzglobal} that the right action of $W_0$ given above extends along the natural embedding $\rH^1(\mathbb{R}, \tor) \hookrightarrow \rB(\mathbb{R}, \tor)$, and we have the following extension of the above theorem of Borovoi. 

\begin{Prop}\cite[Lemma 13.2]{kottwitzglobal}
The natural map induced by the inclusion $\tor \to \g$ induces a bijection of pointed sets $\rB(\mathbb{R}, \tor)_{{\g-\mathrm{bsc}}}/W_0 \simeq \rB(\mathbb{R}, G)_{\mathrm{bsc}}$.
\end{Prop}

\begin{Prop}\cite[Theorem 1.2]{KottwitzSTFE}
    The map $\kappa_{\mathbb{R}, \g}$ restricts to a map $\kappa_{\mathbb{R}, \g}: \rH^1(\mathbb{R}, \g) \to \pi_1(\g)_{\sigma}$ with kernel given by the image of $\rH^1(\mathbb{R}, \g^{sc}) \to \rH^1(\mathbb{R}, \g)$ and image the kernel of the natural norm map $\pi_1(\g)_{\sigma} \to \pi_1(\g)$. 
\end{Prop}

Finally the natural map $\rH^1(\mathbb{R}, \g^{\mathrm{ad}}) \to \rB(\mathbb{R}, \g^{\mathrm{ad}})_{\mathrm{bsc}}$ is an isomorphism when $\g$ is adjoint. Thus given a class $[b] \in \rB(\mathbb{R}, \g)_{\mathrm{bsc}}$ we may associate an inner form $\g_{[b]}$ of $\g$ over $\mathbb{R}$ to the class $[b]$ just as in the setting of non-archimedean local fields.

\subsubsection{} A \emph{weak real Shimura datum} is a pair $\gx$, where $\g$ is a connected reductive group over $\R$ and $\x$ is a $\g(\mathbb{R})$-conjugacy class of homomorphism $h:\mathbb{S} \to \g$ satisfying Milne's axioms SV1 and SV2, see \cite[Definition 5.5]{Milne}. A \emph{morphism of weak real Shimura data} is a morphism of pairs. A \emph{real Shimura datum} is a weak real Shimur datum $\gx$ satisfying Milne's axiom SV3, see \cite[Definition 5.5]{Milne}

For a symplectic space $(V,\psi)$ over $\mathbb{R}$ we denote by $\gvx$ the standard weak real Shimura datum, where $\x_{V}$ is the $\gv(\mathbb{R})$-conjugacy class of complex structures $J$ on $V$ such that the bilinear form $\psi(x,Jx)$ is positive definite or negative definite on $V$.

\subsubsection{} \label{subsub:RealBGShimura} Attached to a weak real Shimura datum $\gx$ is a class $[X] \in B(\R, \g)$, see \cite[\S 1.2-1.3]{SemplinerTaylorShimura} where this class is referred to as $\widehat{\boldsymbol{\lambda}}_{\g}(Y)$, where $Y$ is the $\g(\R)$-conjugacy class of Hodge cocharacters associated to $X$. Its image under $B(\R,\g) \to B(\R,\g^{\mathrm{ad}})=H^1(\R, \gad)$ is the class of the (unique) compact inner form of $\g^{\mathrm{ad}}$, and its image under the Kottwitz map is the class of $\mu_X$ inside of $\pi_1(\g)_{\sigma}$. 

\subsubsection{} \label{subsub:RealTwistKottwitz} If $P$ is a $G$-torsor over $\spec \mathbb{R}$ with automorphism group $G'=\operatorname{Aut}_{G}(P)$, then there is a canonical bijection
\begin{align}
    \beta_{P}:B(\R,\g) \to B(\R,\g'),
\end{align}
extending $H^1(\R,\g) \xrightarrow{\sim} H^1(\R, \g')$, and compatible with the real Kottwitz maps. This can be proved in the same way as \ref{Lem:RealSquare} by working at the level of real Kottwitz cocycles, rather than Galois cocycles. We omit the proof. 

\subsubsection{} We will sometimes find ourselves in the following rather baroque situation. Let $\mathsf{P}$ be a $\g$-torsor over $\spec \mathbb{R}$ with automorphism group $\g'=\operatorname{Aut}_{\g}(\mathsf{P})$. Choose $x$ a $\mathbb{C}$-point of $\mathsf{P}$ defining an isomorphism $\psi: \g_{\mathbb{C}} \to \g'|_{\mathbb{C}}$ be an isomorphism satisfying $\psi^{-1} \circ \sigma_{\g'} \psi \circ \sigma_{\g}^{-1} = \operatorname{Ad}_k$ for some $k \in \g(\mathbb{C})$, where $\sigma_{\g'}$ and $\sigma_{\g}$ denotes the complex conjugations on $\g'$ and $\g$. 

Let $h: \mathbb{C}^* \to \g(\mathbb{R})$ be a (pointed) weak Shimura datum, and let $h'$ be a (pointed) weak Shimura datum for $\g'$. Let $\rho: \g \to \operatorname{GSp}(V)$ be a Hodge embedding. We can twist $\g$ and $\rho$ by $\mathsf{P}$ to get an inner form $\operatorname{GSp}'(V)/\mathbb{R}$ together with $\rho': \g' \to \operatorname{GSp}'(V)$. Writing $\sigma_0$ for the complex conjugation on $\operatorname{GSp}(V)$, we get an isomorphism $\psi_0: \operatorname{GSp}(V_{\mathbb{C}}) \to \operatorname{GSp}'(V_{\mathbb{C}})$ which satisfies $\psi_0^{-1} \circ \sigma_0' \circ \psi_0 = \mathrm{Ad}_{\rho(k)} \circ \sigma_0$, and which restricts to $\psi$ along the morphisms $\rho: \g \to \operatorname{GSp}(V), \rho': \g' \to \operatorname{GSp}'(V)$. \noindent

By the triviality of $\rH^1(\mathbb{R}, \operatorname{GSp}(V))$, there exists $d \in \operatorname{GSp}(V)(\mathbb{C})$ such that $d^{-1}\rho(k)\sigma_0(d) = 1$, from this one obtains an isomorphism $\psi_0' = \psi_0 \circ \mathrm{Ad}_d$ which is defined over $\mathbb{R}$, and a homomorphism of algebraic groups
    \[
        \rho'' = (\psi_0')^{-1} \circ \rho': \g' \to \operatorname{GSp}(V) 
    \]
    also defined over $\mathbb{R}$.
\begin{Lem}\label{Lem:TwistingHodgeEmbedding}
Let the notation be as above. Suppose that $\beta_P([X]) = [X']$. Then $\rho''$ and thus $\rho'$ is an Hodge embedding if and only if $\rho(\mu_h)$ and $\rho''(\mu_{h'})$ are $\operatorname{GSp}(V)(\mathbb{C})$-conjugate. 
\end{Lem}
\begin{proof}
We first check that $\rho'' \circ h$ gives a Shimura datum for $\operatorname{GSp}(V)$. The condition on the Hodge cocharacters forces $\rho'(h')$ to satisfy SV1. Since $\beta_P([X]) = [X']$, we know that $[X]$ and $[X']$ have the same Newton point and thus that $[X]$ and $[X']$ have the same (central) weight cocharacter under the canonical isomorphisms of their centers. Because Since $\beta_P([X])=[X']$ and the pushforward of $\mathsf{P}$ along $\rho$ is trivial, it follows that $\rho''_{\ast}[X']=\rho_{\ast}[X]$ and thus that $\rho_{\ast}[X]$ is basic. 

We see that the image of $\rho''_{\ast}[X']$ in $B(\mathbb{R},\operatorname{GSp}(V)^{\mathrm{ad}})$ lies in $H^1(\mathbb{R}, \operatorname{GSp}(V)^{\mathrm{ad}})$, and SV2 is precisely the assertion that this class corresponds to the compact inner form. 
This is true for $\rho_{\ast}[X]$ and thus for $\rho''_{\ast}[X']=\rho_{\ast}[X]$. \smallskip 

To conclude the proof, it suffices to note that $\operatorname{GSp}(V)$ only admits a single $\g(\mathbb{R})$-conjugacy class of Shimura data with weight equal to the weight of the standard Shimura datum and with Hodge cocharacter conjugate to the Hodge cocharacter of the standard Shimura datum. 
\end{proof}

\subsubsection{An important calculation} 
\newcommand{\y}{\mathsf{Y}}
\label{subsub:ComputationHodgeEmbedding} In this section we record a calculation which will be a key archimedean input in treating Shimura varieties of type $B, C,$ and $D$. Let $\gx$ be a weak real Shimura datum and let $\iota:\g \to \gv$ be an injective morphism. Choose an nondegenerate symmetric $\mathbb{R}$-linear bilinear form $\psi'$ on $\mathbb{C}$ and consider $W=V \otimes_{\mathbb{R}} \mathbb{C}$ equipped with the $\mathbb{R}$-linear symplectic form given by $\psi_{w}(v_1 \otimes z_1, v_2 \otimes z_2) =\psi(v_1,v_2) \cdot \psi'(z_1, z_2)$. We let $\h = (\g \times \mathbb{S})/\mathbb{G}_m$ with the weak real Shimura datum $\mathsf{Y}$ induced from $\x \times 1$. There is a map $\iota:\hy \to \gwx$ by letting $(g,t)$ acts on $v \otimes z$ by $gv \otimes tz$. \smallskip

Let $h \in \x$, and observe that $k=[(h(i), -i)] \in \h(\R)$ has the property that $k^2=[(-1,-1)]$ is the identity element of $\h(\mathbb{R})$. Thus the assignment sending $\sigma \mapsto k$ is a cocycle for $\gal(\mathbb{C}/\R)$. If $\g^{\mathrm{ad}}(\mathbb{R})$ is not compact, then the class of this cocycle is nontrivial in $H^1(\R,\h)$ because it maps to the (nontrivial) class of the compact inner form under the map $H^1(\R,\h) \to H^1(\R,\g^{\mathrm{ad}})$, by axiom SV2. The following lemma is a special case of a result of Xiao--Zhu, and we give a simplified version of their proof. 

\begin{Lem} \label{Lem:HodgeEmbeddingComputation}
Let $\h'$ be the pure inner twist of $\h$ by $k$ and consider the morphism $\iota':\h' \to \operatorname{GSp}(W', \psi_{W'})$, where $(W', \psi_{W'})$ is the twist of $(W, \psi_W)$ by $k$. Define a morphism $h':\mathbb{S} \to \h'$ by $h'(z)=[(1,z)]$. If $\psi_{W}(v, h(i)v)$ is a positive (resp. negative) definite pairing on $W$, then $(w,w') \mapsto \psi_{W'}(w, h'(i)w')$ is a positive (resp. negative) definite pairing on $W'$. Furthermore, if $\iota \circ h$ lands in $\mathsf{H}_{W}^{\pm}$, then $\iota' \circ h'$ lands in $\mathsf{H}_{W'}^{\pm}$. In particular, if we equip $\h'$ with the weak real Shimura datum given by the $\h'(\mathbb{R})$-conjugacy class of $h'$ and if $\iota$ is a morphism of weak real Shimura data, then $\iota'$ is a morphism of weak real Shimura data. 
\end{Lem}
\begin{proof}
Morally this Lemma is an application of Lemma \ref{Lem:TwistingHodgeEmbedding}, and as in that lemma, the essence of the matter is to calculate that $(\iota' \circ h')(i)$ is a Cartan involution and that $\iota'_{\mathbb{C}} \circ \mu_{h'}$ is a minuscule cocharacter of $\g_{W'}$. However, we give a direct proof of these claims in this special case. 

Recall that $\h'(\mathbb{R})=\{h \in \h(\mathbb{C}) \; : \; k \sigma(h) k^{-1}=h\}$ which acts on $W'=\{w \in W_{\mathbb{C}} : k \sigma(w)=w \}$, and that $\psi_{W'}$ is simply the restriction of $\psi_{W,\mathbb{C}}$. Note that $h'(z) \in \h'(\mathbb{R})$ and that $h'(z)$ defines a complex structure on $W'$ because $h'(i)^2=[(1,-1)]=[(-1,1)]$. We now compute
\begin{align}
    \psi_{W,\mathbb{C}}(w, h'(i) w) &= \psi_{W,\mathbb{C}}(w, h(i) k^{-1} w) \\
    &=\psi_{W,\mathbb{C}}(w, h(i) \sigma(w)).
\end{align}
Since $w \mapsto \psi_{W}(w,h(i) w)$ is positive (or negative) definite on $W$, it follows that $w \mapsto \psi_{W,\mathbb{C}}(w, h(i) \sigma(w))$ is positive (or negative) definite on $W_{\mathbb{C}}$. We deduce that $w \mapsto \psi_{W'}(w, h'(i) w)$ is positive (or negative) definite on $W'$, and so $(\iota' \circ h')(i)$ is a Cartan involution of $\g_W$ as desired. That $\iota' \circ \mu_{h'}$ lies in the minuscule conjugacy class of cocharacters of $\g_{W'}$ is obvious from the definition of $h'$.  
\end{proof} 

\subsection{Perfectoid geometry} \label{Sub:PerfectoidPrelim}

\subsubsection{} \label{Sec:DiamondFunctors} We will write $\perf$ for the category of affinoid perfectoid spaces over $\fp$. If $X$ is a pre-adic space over $\spa(\zp)$ in the sense of \cite[Definition 3.4.1]{ScholzeWeinsteinBerkeley}, let $X^\lozenge$ denote the set-valued functor on $\perf$ given by
\begin{align}
    X^\lozenge(S) = \{(S^\sharp, f)\} / \sim
\end{align}
for any $S$ in $\perf$, where $S^\sharp$ is an untilt of $S$ and $f:S^\sharp \to X$ is a morphism of adic spaces. This determines a v-sheaf on $\perf$ by \cite[Lemma 18.1.1]{ScholzeWeinsteinBerkeley}. For a Huber pair $(A,A^+)$ over $\zp$, we write $\spd(A,A^+)$ in place of $\spa(A,A^+)^\lozenge$, and when $A^+ = A^\circ$ we write $\spd(A)$ instead of $\spd(A,A^+)$. In particular, we see that $\spd(\zp)$ parametrizes isomorphism classes of untilts, cf. \cite[Definition 10.1.3]{ScholzeWeinsteinBerkeley}. 

\subsubsection{Formal schemes} For us a formal scheme over $\spf \zp$ is a functor on the category of $\zp$-algebras in which $p$ is nilpotent, which is locally isomorphic to $\spf A$ where $A$ is a topological ring which is complete and
separated for the $I$-adic topology for some finitely generated ideal $I \subset A$. For a formal scheme $\mathfrak{X}$ over $\spf \zp$, we write $\mathfrak{X}^\mathrm{ad}$ for the pre-adic space associated to $\mathfrak{X}$ as in \cite[Proposition 2.2.1]{ScholzeWeinsteinModuli}, and we often abbreviate $(\mathfrak{X}^\mathrm{ad})^\lozenge$ to $\mathfrak{X}^\lozenge$. If $X$ is a scheme over $\spec \zp$, then we write $X^{\diamond}$ for $(\widehat{X})^{\lozenge}$, where $\widehat{X} = X \times_{\spec \zp} \spf \zp$ is the $p$-adic completion of $X$. \smallskip

For a formal scheme over $\spf \zp$, we define $\mathfrak{X}^{\lozenge,\mathrm{pre}}$ to be the presheaf on $\perf$ sending $(R,R^+)$ to the equivalence class of triples $(R^\sharp, R^{\sharp+}, f)$, where
 $(R^\sharp, R^{\sharp+})$ is an untilt of $(R, R^+)$ and
 $f \colon \spf R^{\sharp+} \to \mathfrak{X}$ is a morphism of formal schemes, see \cite[Definition 2.1.3]{DvHKZIgusaStacks}. The following lemma is \cite[Lemma 2.1.4]{DvHKZIgusaStacks}. 
\begin{Lem} \label{Lem:DiamondOfFormalScheme}
Let $\mathfrak{X}$ be a formal scheme over $\spf \zp$, which locally admits a finitely generated ideal of definition. Then $\mathfrak{X}^{\lozenge} / \spd \zp$ is the analytic sheafification of the presheaf on $\perf$ given by $\mathfrak{X}^{\lozenge,\mathrm{pre}}$.
\end{Lem}

\subsubsection{The Fargues--Fontaine curve} For any perfectoid space $S$ over $\fp$, we write $S \bdtimes \zp$ for the analytic adic space whose associated v-sheaf is $S^\lozenge \times \spd \zp$. This comes equipped with a Frobenius $\varphi$, and any untilt $S^\sharp$ of $S$ determines a closed Cartier divisor $S^\sharp \hookrightarrow S \bdtimes \spa\zp$, see \cite[Proposition 11.3.1]{ScholzeWeinsteinBerkeley}. 

For $S$ in $\perf$ we define $\mathcal{Y}(S) = S \bdtimes \zp \setminus \{p = 0\}$, and for any interval $I \subset [0,\infty)$ we have an open subset $\mathcal{Y}_I(S)$ of $S\bdtimes \spa\zp$, see \cite[Section 12.2]{ScholzeWeinsteinBerkeley}. In particular $\mathcal{Y}_{[0,\infty)}(S) = S \bdtimes \spa\zp$ and we also denote $\mathcal{Y}_{(0,\infty)}(S)$ by $\mathcal{Y}(S)$. 

\subsubsection{} For any $S$ in $\perf$, the relative adic Fargues--Fontaine curve over $S$ is the quotient
\begin{align}
    X_S = \mathcal{Y}(S) / \varphi^\mathbb{Z},
\end{align}
see \cite[beginning of Chapter II]{FarguesScholze}. Let $G$ be a reductive group over $\qp$. Following \cite{FarguesScholze}, we denote by $\bun_G(S)$ the groupoid of $G$-torsors on $X_S$. By \cite[Theorem~III.0.2]{FarguesScholze}, the presheaf of groupoids $\bun_G$ on $\perf$ sending $S$ to $\bun_G(S)$ is a small v-stack. The following lemma is \cite[Proposition 6.3.11]{DvHKZIgusaStacks}. 
\begin{Lem} \label{Lem:GStructurePreservingClosedPerfectoid}
Let $G \to H$ be a closed immersion of connected reductive groups over $\qp$. Then the natural map
\begin{align}
    \bung \times_{\bung} \bung \to \bung \times_{\bun_{H}} \bung
\end{align}
is representable in closed immersions.
\end{Lem}
\begin{Lem} \label{Lem:RepresentableClosedImmersionPerfectoid}
    Let $G \to H$ be a closed immersion of connected reductive groups over $\qp$ and let $\mu$ be a $G(\qpbar)$-conjugacy class of cocharacters with reflex field $E$; Let $\mu_{H}$ be its image in $H$. If $X \to \bun_{G}$ is a v-sheaf, then the natural map
    \begin{align}
        X \times_{\bung} \grgmu \to X \times_{\bun_{H}} \operatorname{Gr}_{H, -\mu_H,E}
    \end{align}
    is representable in closed immersions.
\end{Lem}
\begin{proof}
We factor the map of the lemma as
\begin{align}
    X \times_{\bung} \grgmu \to X \times_{\bun_{H}} \operatorname{Gr}_{G, -\mu} \to X \times_{\bun_{H}} \operatorname{Gr}_{H, -\mu_{H},E}.
\end{align}
The second map is a closed immersion because $\grgmu \to \operatorname{Gr}_{H, -\mu_H,E}$ is a closed immersion (being induced from a closed immersion of classical flag varieties). The first map sits in a $2$-Cartesian diagram
\begin{equation}
    \begin{tikzcd}
        X \times_{\bung} \grgmu \arrow{r} \arrow{d} & X \times_{\bun_{H}} \grgmu \arrow{d} \\
        \bung \times_{\bung} \bung \arrow{r} & \bung \times_{\bun_{H}} \bung,
    \end{tikzcd}
\end{equation}
and is thus a closed immersion by Lemma \ref{Lem:GStructurePreservingClosedPerfectoid}; this completes the proof.
\end{proof}

\subsubsection{} Let $B(G)$ be the set of $\sigma$-conjugacy classes in $G(\qpbr)$, equipped with the topology coming from the \emph{opposite} of the partial order defined in \cite[Section 2.3]{RapoportRichartz}. By \cite[Theorem 1]{Viehmann}, there is a homeomorphism
\begin{align}
    |\bun_G| \to B(G). 
\end{align}
If $\mu$ is a $G(\qpbar)$-conjugacy class of minuscule cocharacters, we let $\bgmu \subset B(G)$ be the set of $\mu^{-1}$-admissible elements, as defined in \cite[Section 1.1.5]{KMPS}; note that this set is closed in the partial order and thus defines an open substack
$$\bungmu \subset \operatorname{Bun}_G,$$ via \cite[Proposition 12.9]{EtCohDiam}. We use the additive notation $-\mu$ instead of the multiplicative notation $\mu^{-1}$ because $-\mu'$ looks better than $(\mu')^{-1}$.

\subsubsection{} Let $G$ be as above and let $\mathcal{G}$ be a (quasi-)parahoric model of $G$ over $\zp$ in the sense of \cite[Section 2.2]{PappasRapoportRZ}. We recall the stack
\begin{align}
    \shtg \to \spd \zp
\end{align}
defined in \cite[Section 3.1.2]{DanielsVHKimZhangCompanion}. Now let $\mu$ be as above and let $E \subset \qpbar$ be the reflex field of $\mu$. Then there is a closed substack
\begin{align}
    \shtgmu \subset \shtg \times_{\spd \zp} \spd \mathcal{O}_E,
\end{align}
see \cite[Section 3.1.2]{DanielsVHKimZhangCompanion}. We further recall the open and closed substack
\begin{align}
    \shtgmuone \subset \shtgmu
\end{align}
of \cite[Section 3.3.8]{DanielsVHKimZhangCompanion}. We further recall the integral Beauville--Laszlo map $\operatorname{BL}^{\circ}:\shtg \to \bung$, under which $ \shtgmuone$ maps to $\bungmu$ by \cite[Proposition 3.1.10]{DanielsVHKimZhangCompanion}.

\subsection{Perfect geometry} \label{Sub:PerfectPrelim} Let $k$ be a perfect field and let $\affperf_{k}$ denote the category of affine perfect schemes over $k$. We will equip this with the Grothendieck (pre-)topology where covers are given by v-covers in the sense of \cite[Definition 2.1]{BhattScholze}. By \cite[Theorem 4.1]{BhattScholze}, perfect $\fp$-schemes define v-sheaves on $\affperf$, and for $X$ a scheme over $\fp$ we will write $X^{\mathrm{perf}}$ for the perfection of $X$ (the inverse limit over the Frobenius of $X$). When $k=\fp$, we will write $\affperf$ for $\affperf_{\fp}$.

In this section we will recall the theory of Witt-vector shtukas introduced in \cite{Zhu1},\cite{XiaoZhu},\cite{ShenYuZhang}. In order to work in maximal generality, we follow the presentation of \cite[Section 3]{DvHKZIgusaStacks}.

\subsubsection{} Let $G$ be a connected reductive group over $\qp$ with (quasi-)parahoric model $\mathcal{G}$ and let $\mu$ be a $G(\qpbar)$-conjugacy class of cocharacters with reflex field $E \subset \qpbar$. We recall the Witt vector affine flag variety $\grgloc$ from \cite[Section 3.1.2]{DvHKZIgusaStacks}; it is representable by an ind-(perfectly projective scheme) by \cite[Theorem 1.1]{BhattScholze}. 

Let $\gisoc$ denote the stack on $\affperf$ sending $R$ to the groupoid of (\'etale)  $G$-torsors $P$ over $\spec W(R)[1/p]$ together with an isomorphism
\begin{align}
    \beta:\sigma^{\ast}P \to P;
\end{align}
this is a v-stack, see \cite[Proposition 6.3]{GleasonIvanovZillinger}. There is a homeomorphism $\lvert\gisoc\rvert \simeq B(G)$, where now $B(G)$ is equipped with the order topology for the partial order defined in \cite[Section 2.3]{RapoportRichartz}. We define the closed substack $\gisocmu \subset \gisoc$ to be the closed substack corresponding to $\bgmu \subset B(G)$.

\begin{Lem} \label{Lem:GStructurePreservingClosedPerfect}
    Let $G \to H$ be a closed immersion of connected reductive groups over $\qp$. The natural map
    \begin{align}
        \gisoc \times_{\gisoc} \gisoc \to \gisoc \times_{\hisoc} \gisoc
    \end{align}
    is representable in closed immersions.
\end{Lem}
\begin{proof}
We can identify the right hand side with the stack of triples $(P,Q,\alpha)$ where $P$ and $Q$ are $G$-isocrystals and $\alpha:P \times^{G} H \to Q \times^{G} H$ is an isomorphism of $H$-isocrystals. The left hand side is the stack of triples $(E,F,\beta)$ where $E$ and $F$ are $G$-isocrystals and $\beta:E \to F$ is an isomorphism of $G$-isocrystals. Thus we are asking that for a triple $(P,Q,\alpha)$ over a perfect scheme $R$, there is a closed subscheme $\spec B \subset \spec R$ where $\alpha$ uniquely comes from an isomorphism of $G$-isocrystals $\beta:P \to Q$. \smallskip 

The v-topology is subcanonical on perfect schemes, and v-covers are topological quotient maps (see \cite[Tag 0ETP]{stacks-project}); thus the property of being representable in closed immersions is v-local on the target. Thus by \cite[Lemma 11.2,Proposition 11.5]{Anschutz} we may assume that the underlying $G$-bundles of $P$ and $Q$ over $\spec W(R)[1/p]$ are trivial and identify them with $G$. Write $b_{P}, b_{Q} \in G(W(R)[1/p])$ for the elements corresponding to $\beta_{P}$ and $\beta_{Q}$, respectively. Recall the sheaf $LH$ on $\affperf$ which sends $R$ to $H(W(R)[1/p])$; it is represented by an ind-(perfect scheme), see \cite[Proposition 1.1]{Zhu1}. The set of isomorphisms $\alpha$ is given by the subfunctor
\begin{align}
    \{h \in LH : h^{-1} b_{P} \sigma(h) = b_{Q} \} \subset LH
\end{align}
while the set of isomorphisms $\beta$ is given by the subfunctor
\begin{align}
    \{g \in LG : g^{-1} b_{P} \sigma(g) = b_{Q}\} \subset LG.
\end{align}
The lemma now follows from the fact that $LG \to LH$ is representable in closed immersions, see \cite[Lemma 1.2.(ii)]{Zhu1}.
\end{proof}

\subsubsection{} Recall the stack of $\mathcal{G}$-shtukas $\shtlocg$. We further recall
\begin{align}
    \shtlocgmu \subset \shtlocg \times_{\spec \fp} \spec k_{E},
\end{align}
see \cite[Section 3.1.6]{DvHKZIgusaStacks}. Recall moreover the natural morphism $\shtlocg \to \gisoc$. If $\mathcal{G}$ is parahoric, then under this map $\shtlocgmu$ is sent to $\gisocmu$ (as explained in \cite[Section 3.2.2]{DvHKZIgusaStacks}). 

\begin{Lem} \label{Lem:IndRepresentableI}
Let $R$ be a perfect $\mathbb{F}_p$-algebra and let $\spec R \to \gisoc$ be a morphism corresponding to a $G$-isocrystal $(P,\Phi)$. If there is an \'etale cover $R \to R'$ such that $P \times_{\spec W(R)} \spec W(R')[\tfrac{1}{p}]$ is trivial, then the fiber product
\begin{align}
    \shtlocg \times_{\gisoc} \spec R
\end{align}
is an ind-(perfectly proper and perfectly finitely presented algebraic space) over $\spec R$.
\end{Lem}
\begin{proof}
By \'etale descent, it suffices to prove that the basechange via $\spec R' \to \spec R$ is an ind-perfectly proper ind-scheme. But over $\spec R'$ there is an isomorphism
\begin{align}
    \shtlocg \times_{\gisoc} \spec R' \simeq \spec R' \times \operatorname{Gr}_{\mathcal{G}},
\end{align}
and the affine flag variety $\operatorname{Gr}_{\mathcal{G}}$ is an ind-(perfectly projective perfectly finitely presented scheme). 
\end{proof}
The following result is a direct consequence of Lemma \ref{Lem:IndRepresentableI}.
\begin{Cor} \label{Cor:IndRepresentableII}
Let $R$ be a perfect $\mathbb{F}_p$-algebra and let $\spec R \to \gisoc$ be a morphism corresponding to a $G$-isocrystal $(P,\Phi)$ over $\spec R$. If there is an \'etale cover $R \to R'$ such that $P \times_{\spec W(R)} \spec W(R')[\tfrac{1}{p}]$ is trivial, then the fiber product
\begin{align}
    \shtlocgmu \times_{\gisoc} \spec R
\end{align}
is an ind-(perfectly proper and perfectly finitely presented algebraic space) over $\spec R$.
\end{Cor}
\subsubsection{} We call a perfect scheme $X=\spec R$ a \emph{strict comb} (cf. \cite[Definition 2.12]{GleasonIvanovZillinger}) if the set of closed points $X^c$ is closed within $X$, for every $y \in \pi_0(X)$ the corresponding connected component $X_{y} \subset X$ is isomorphic to $\spec V_y$ where $V_y$ is a valuation ring with algebraically closed fraction field, and the topological space $\pi_0(X)$ is extremally disconnected in the sense of \cite[Definition 2.4.4]{BhattScholzeProEtale}. By \cite[Theorem 1.8]{BhattScholzeProEtale}, strict combs are in particular w-contractible. 

\begin{Prop} \label{Prop:VSurjectiveNewtonMap}
    If $X=\spec R$ is a strict comb and $\mathcal{G}$ is parahoric, then the natural map $\shtlocgmu(\spec R) \to \gisocmu(\spec R)$ is essentially surjective. 
\end{Prop}

First we recall the following lemma. 

\begin{Lem} \label{Lem:combwContractible}
    Let $\pi:Y \to X=\spec R$ be a pro-\'etale cover. If $X$ is a strict comb, then $\pi$ admits a section.
\end{Lem}
\begin{proof}
This is a simple consequence of \cite[Theorem 1.8]{BhattScholzeProEtale}.
\end{proof}

\begin{Lem} \label{Lem:combhlocal}
    Let $\pi:Y \to X=\spec R$ be a perfectly finitely presented morphism which is a $v$-cover of perfect schemes. If $X$ is a strict comb, then $\pi$ admits a section.
\end{Lem}
\begin{proof}
By \cite[Proposition 3.13]{BhattScholze} we fix a finitely presented model $\pi_0: Y_0 \to X$ for the morphism $\pi$. Because the morphism $Y = Y_0^{\mathrm{perf}} \to Y_0$ is a universal homeomorphism and by \cite[Proposition 2.1]{RydhSub} in tandem \cite[remark after Definition 2.1]{BhattScholze}, the morphism $\pi_0$ is also a v-covering. By the universal property of perfection, it suffices to show that $\pi_0$ admits a section. As the scheme $X$ is affine, by the structure theorem of \cite[Theorem 3.12]{RydhSub} the morphism $\pi_0$ refines to $\pi_0': Y_0' \to X$ where the morphism $\pi_0'$ factors as  
\[
    Y_0' \xrightarrow{g} Z_0 \xrightarrow{f} X
\]
with $g$ a quasicompact open cover, and $f$ a finitely presented proper surjective morphism. By Lemma \ref{Lem:combwContractible} it suffices to show that the morphism $f$ admits a section. Indeed, given a section $s$ of $f$ we reduce to producing a section of $Y_0' \times_{Z_0} s(X) \to X$.

Let $x \in \pi_0(X)$ with fiber $X_x \subset X$. Then $X_x \simeq \spec V$ with $V$ an absolutely integrally closed valuation ring, and we denote by $K$ the fraction field of $V$. Then $Z_0 \times_{X} \spec K$ is nonempty because $f$ is surjective, and admits a section because $K$ is algebraically closed and $f$ is of finite presentation. By the valuative criterion of properness, this extends to a section $s$ of $Z_0 \times_{X} \spec V$. Now, since the morphism $f$ is finitely presented, we can spread out this section over a Zariski open neighborhood of $X_x$, writing $X_x$ as a limit of opens containing $x$. As $x$ was arbitrary, we get sections over an open cover of $X$, and we conclude by another application of Lemma \ref{Lem:combwContractible}.
\end{proof}

\begin{proof}[Proof of Proposition \ref{Prop:VSurjectiveNewtonMap}]
We follow the proof of \cite[Proposition 3.2.3]{DvHKZIgusaStacks}. We start with a morphism 
\begin{align}
    \spec R \to \gisocmu
\end{align}
corresponding to a $G$-isocrystal $(P, \beta)$ over $\spec R$. By \cite[Lemma 6.4]{GleasonIvanovZillinger}, the torsor $P$ is trivial over $W(R)[1/p]$, and so $\beta$ corresponds to an element $b \in G(W(R)[1/p])$. By Corollary \ref{Cor:IndRepresentableII} and its proof, the fiber product (this agrees with the scheme $Z(b)_{\mu}$ as defined in \cite[discussion before the proof of Lemma 3.2.4]{DvHKZIgusaStacks} by inspection)
\begin{align}
    Z(b)_{\mu}=\shtlocgmu \times_{\gisoc} \spec R \to \spec R
\end{align}
is an ind-(perfectly projective perfectly finitely presented perfect scheme) over $\spec R$. By \cite[Lemma 3.2.4]{DvHKZIgusaStacks}, there is a closed quasicompact subscheme $Z \subset Z(b)_{\mu}$ such that $Z \to \spec R$ is surjective. Since $ Z(b)_{\mu}$ is an (ind-perfectly projective perfectly finitely presented perfect scheme), it follows that $Z \to \spec R$ is a perfectly finitely presented perfectly proper surjection of perfect schemes, and thus admits a section by Lemma \ref{Lem:combhlocal}; this concludes the proof.
\end{proof}

\subsubsection{} We collect the following lemma for later use (cf. \cite[Section 6.5]{Kottwitz2}).
\begin{Lem} \label{Lem:Intersection}
Let $G_1 \to G_2$ be an ad-isomorphism of connected reductive groups over $\qp$, and let $\mu,\mu'$ be two $G_1(\qpbar)$-conjugacy classes of cocharacters of $G_1$. If $B(G_1,-\mu) \cap B(G_1,-\mu')$ is nonempty, then the natural map
\begin{align}
    B(G_1,-\mu) \cap B(G_1,-\mu') \to B(G_2,-\mu_{2}) \cap B(G_2,-\mu'_{2})
\end{align}
is a bijection.
\end{Lem}
\begin{proof}
    By \cite[Proposition 4.10]{Kottwitz2}, the following diagram is Cartesian
    \begin{equation}
        \begin{tikzcd}
            B(G_1) \arrow{r}{\kappa_{G_1}} \arrow{d} & \pi_1(G_1)_{\gal_{\qp}} \arrow{d} \\
            B(G_2) \arrow{r}{\kappa_{G_2}} & \pi_1(G_2)_{\gal_{\qp}}.
        \end{tikzcd}
    \end{equation}
By definition, the sets $B(G_1,-\mu),B(G_1,-\mu')$ both lie in a single fiber of $\kappa_{G_1}$, and by assumption this is the same fiber. Similarly, it follows that the sets $B(G_2,-\mu_2),B(G_2,-\mu'_2)$ both lie in the same fiber of $\kappa_{G_2}$. But by the Cartesian diagram above, the map $B(G_1) \to B(G_2)$ induces a bijection on fibers of the Kottwitz map, and so we are done.
\end{proof}

\subsection{Reduction and analytification} \label{Sub:ReductionAnalytification} In this section, we explain various ways of passing between sheaves on $\perf$ and sheaves on $\affperf$.

\subsubsection{} \label{subsub:SmallDiamondPerfectScheme} Given a perfect $\fp$-scheme $X$, there is a v-sheaf $X^\diamond$ on $\perf$. When $\spec R$ is an affine perfect scheme, its functor of points is given by
\begin{align}
    (\spec R)^{\diamond}(A,A^+)= (\spec R)(A^+),
\end{align}
and the general case is handled by gluing, see \cite[Proposition 3.2]{GleasonSpecialization}. If $\spec B \to \spec A$ is a v-cover of perfect $\fp$-algebras, then the induced map $\spd B \to \spd A$ is v-surjective, see \cite[Proposition 3.7]{GleasonSpecialization}. 

\subsubsection{} For a presheaf in groupoids $\mathcal{F} \colon \perf^\mathrm{op} \to \mathsf{Grpd}$ we recall from \cite[Section 3.2]{GleasonSpecialization} the reduction $\mathcal{F}^{\red}$ of $\mathcal{F}$, defined by
\begin{align*}
    \mathcal{F}^{\mathrm{red}}(A) = \operatorname{Hom}(\spd A, \mathcal{F}).
\end{align*}
If $\mathcal{F}$ is a v-stack on $\perf$, then $\mathcal{F}^{\mathrm{red}}$ is a v-stack on $\affperf$ (by the last sentence of Section \ref{subsub:SmallDiamondPerfectScheme}). We will often use the following lemma implicitly (see \cite[Lemma 3.0.6]{DvHKZIgusaStacks}).
\begin{Lem} \label{Lem:ReductionOfScheme}
  Let $X$ be a scheme over $\zp$. Then the natural map
  \[
    X_{\fp}^\mathrm{perf} \to (X^\diamond)^\mathrm{red}
  \]
  is an isomorphism.
\end{Lem}

\subsubsection{} There is a functor $\perf \to \affperf$ sending $(R,R^+)$ to $R^+_{\red}=R^+/R^{\circ\circ}$ with $R^{\circ \circ} \subset R^+$ the ideal of topologically nilpotent elements. For a presheaf of groupoids $\mathcal{F}$ on $\affperf$, we will write $\mathcal{F}^{\dcirc, \pre}$ for the presheaf on $\perf$ sending $(R,R^+)$ to $\mathcal{F}(R^+_{\red})$. We use $\mathcal{F}^{\dcirc}$ to denote its v-sheafification (our notation follows that of \cite[Section 4.4]{GleasonSpecialization}). We have the following question.
\begin{Question} \label{Question:VtopologyLemmaReduction}
If $\mathcal{F}$ is a presheaf on $\affperf$ with v-sheafification $\mathcal{F} \to \mathcal{G}$, is the natural map
    \begin{align}
        \mathcal{F}^{\dcirc} \to \mathcal{G}^{\dcirc}
    \end{align}
    an isomorphism?
\end{Question}
In the proof of Theorem \ref{Thm:IgusaDCircII}, we will use the following proposition to circumvent answering Question \ref{Question:VtopologyLemmaReduction}.
\begin{Prop} \label{Prop:VSurjectiveNewtonDiamondCirc}
If $\mathcal{G}$ is parahoric, then the natural map $(\shtlocgmu)^{\dcirc} \to (\gisocmu)^{\dcirc}$ is a v-cover.
\end{Prop}
\subsubsection{} Recall from \cite[proof of Proposition VII.1.6]{FarguesScholze} that a perfectoid space $X$ is \emph{w-contractible} if $X$ is strictly totally disconnected, the set of closed points $Z \subset |X|$ is closed, and $\pi_0(X)$ is extremally disconnected. 
\begin{Lem} \label{Lem:ReductionETDStrictComb}
Let $(R,R^+) \in \perf$. If $\spa(R,R^+)$ is $w$-contractible, then $\spec R^+_{\red}$ is a strict comb.
\end{Lem}
\begin{proof}
The natural specialization map $|\spa(R,R^+)| \to |\spec R^+_{\red}|$ of \cite[Proposition 4.2]{GleasonSpecialization} is a homeomorphism by \cite[Proposition 4.3]{GleasonSpecialization}. We deduce that $\pi_0(\spec R^+_{\red})$ is extremally disconnected and that the set of closed points in $\spec R^+_{\red}$ is closed. Now by \cite[Proposition 1.15]{EtCohDiam}, every connected component of $\spa(R,R^+)$ is isomorphic to $\spa(C,C^+)$ with $C$ an algebraically closed field. It follows from this that the corresponding connected component of $\spec R^+_{\red}$ is isomorphic to $\spec C^+_{\red}$. Now $C^+_{\red}$ will itself be a valuation ring, and its fraction field is algebraically closed since $C$ is algebraically closed. We conclude that $\spec R^+_{\red}$ is a strict comb.
\end{proof}
We take the time here to state an obvious corollary of this. 
\begin{Cor}
    Let $(R, R^+)$ be a product of points, then $(R, R^+)$ is $w$-contractible, in particular $\spec R^+_{\red}$ is a strict comb. 
\end{Cor}
\begin{proof}
    That $X = \spa(R, R^+)$ is strictly totally disconnected is \cite[Proposition 1.6]{GleasonSpecialization}, from this and \cite[Proposition 4.3]{GleasonSpecialization} it follows that the map $|X| \to |\spec R^+/\varpi |$ is a homeomorphism for any choice of pseudouniformizer $\varpi$. On the other hand the map $R^+ \to R^+/\varpi$ induces a closed immersion $\spec R^+/\varpi \to \spec R^+$. Since the latter has $\pi_0$ extremally disconnected and the set of closed points is closed, the same is true for the closed subset $\spec R^+/\varpi$, and thus finally for $X$. 
\end{proof}

\begin{Lem} \label{Lem:BasisTopology}
The set of w-contractible affinoid perfectoid spaces forms a basis of the v-topology.
\end{Lem}
\begin{proof}
This is stated in \cite[proof of Proposition VII.1.6]{FarguesScholze}.
\end{proof}
\begin{proof}[Proof of Proposition \ref{Prop:VSurjectiveNewtonDiamondCirc}]
By Lemma \ref{Lem:BasisTopology}, it suffices to prove that for $(R,R^+) \in \perf$ with $\spa(R,R^+)$ w-contractible, the functor
\begin{align}
    (\shtlocgmu)^{\dcirc, \pre}(R,R^+) \to (\gisocmu)^{\dcirc, \pre}(R,R^+)
\end{align}
is essentially surjective. This functor identifies with
\begin{align}
\shtlocgmu(R^+_{\red}) \to \gisocmu(R^+_{\red}),
\end{align}
which is essentially surjective by Proposition \ref{Prop:VSurjectiveNewtonMap} and Lemma \ref{Lem:ReductionETDStrictComb}.
\end{proof}

\begin{Lem} \label{Lem:ClosedToOpen}
    Let $N \to M$ be a functor of v-stacks on $\affperf$ that is representable in closed immersions that are perfectly finitely presented. Then $N^{\dcirc} \to M^{\dcirc}$ is representable in open immersions.
\end{Lem}
\begin{proof}
    This follows from \cite[Proposition 4.22]{GleasonSpecialization}, as explained in \cite[discussion before Lemma 4.25]{GleasonSpecialization}.
\end{proof}

\subsubsection{} By \cite[Lemma 5.2]{HeuerPicard}, for a perfect $\fp$-algebra $R$ we have that $(\spec R)^{\dcirc, \pre}=(\spec R)^{\dcirc}$. This implies as before that for a perfect $\fp$-scheme $X$ the analytic sheafification of $(R,R^+) \mapsto X(R^+_{\red})$ is a v-sheaf. 

\subsubsection{} We note as in \cite[Section 4.2]{GleasonSpecialization} that for a perfect $\fp$-scheme $X$ there is a natural map $X^{\diamond} \to X^{\dcirc}$. By \cite[Proposition 4.24]{GleasonSpecialization}, this induces an isomorphism
\begin{align}
    (X^{\diamond})^{\red} \to (X^{\dcirc})^{\red},
\end{align}
where we note that the left hand side is isomorphic to $X$ by Lemma \ref{Lem:ReductionOfScheme}. 

\subsubsection{} \label{subsub:ReductionShtukaBung} The reduction of the morphism $\shtgmu \to \bung$ can be identified with $\shtlocgmu \to \gisoc$, see \cite[Lemma 3.1.7]{DvHKZIgusaStacks}, \cite[Theorem 1.11]{GleasonIvanovZillinger} and \cite[Section 3.2.1]{DvHKZIgusaStacks}. We will write $\shtlocgmuone \subset \shtlocgmu$ for the reduction of $\shtgmuone \subset \shtgmu$, which maps to $\gisocmu$ even if $\mathcal{G}$ is not parahoric. 

\subsubsection{} Let $R \in \affperf$ together with a map $\spec R \to \gisoc$, and let $\spd R \to \bun_{G}$ be the induced map. Let $[b] \in B(G)$ and let $Z \subset \gisoc$ be the (finitely presented, see \cite[Theorem 6.8]{GleasonIvanovZillinger} or \cite{RapoportRichartz}) closed substack corresponding to $[b'] \le [b]$, let $U \subset \bun_{G}$ be the corresponding open substack. This defines a closed subscheme $\spec B \subset \spec R$ and an open substack $(\spd R)_U \subset \spd R$. We learned the following result from Ian Gleason.
\begin{Lem} \label{Lem:NewtonStrata}
   The natural map $\spd R \to (\spec R)^{\dcirc}$ identifies $(\spd R)_U$ with the inverse image of $(\spec B)^{\dcirc}$.
\end{Lem}
\begin{proof}
This is a direct consequence of \cite[Proposition 7.19]{GleasonIvanov}, together with \cite[Discussion before Lemma 4.25]{GleasonSpecialization}.
\end{proof}

\subsection{Prismatic comparison theorems } \label{Sub:PrismaticComparison} Let $K$ be a finite unramified extension of $\qp$ with ring of integers $\mathcal{O}_K=V_0$ and residue field $k$, and let $X \to \spf \mathcal{O}_K$ be a smooth $p$-adic formal scheme. Let $\operatorname{Vect}^{\mathrm{an}, \varphi}(X_{\prism})$ denote the category of analytic prismatic $F$-crystals on $X$ (see \cite[Definition 3.5]{GuoReinecke}). 

\subsubsection{} We will write $\operatorname{Loc}_{\zp}^{\mathrm{crys}}(X_{\eta})$ for the category of crystalline $\zp$-local systems on the rigid generic fiber $X_{\eta}$ of $X$, see \cite[Definition 2.31]{GuoReinecke}. We will instead use the superscript $\dr$ when we want to discuss de-Rham local systems, see e.g. \cite[Definition 2.5]{GuoReinecke}. 

\subsubsection{} Fix a reductive group scheme $\mathcal{G}$ over $\zp$ with generic fiber $G$. For a $\zp$-linear tensor category $\mathcal{C}$, we will denote by $\mathcal{G}-\mathcal{C}$ the groupoid of exact $\zp$-linear tensor functors $\operatorname{Rep}_{\zp} \mathcal{G} \to \mathcal{C}$. Similarly, for a $\qp$-linear tensor category $\mathcal{D}$, we will denote by $G-\mathcal{D}$ the groupoid of exact $\qp$-linear tensor functors $\operatorname{Rep}_{\qp} G \to \mathcal{D}$. Let $\shtg \to \spd \zp$ denote the stack of $\mathcal{G}$-shtukas and recall the morphism $\operatorname{BL}^{\circ}:\shtg \to \bung$. For a v-sheaf $Y \to \spd \zp$ we denote by $\shtg(Y)$ the groupoid of morphisms $Y \to \shtg$ over $\spd \zp$.  

\subsubsection{} \label{subsub:FullyfaithfulShtukasPrisms} There is a shtuka-realisation functor
\begin{align}
    T_{\mathrm{sht}}:\mathcal{G}-\operatorname{Vect}^{\mathrm{an}, \varphi}(X_{\prism}) \to \shtg(X^{\lozenge}),
\end{align}
see \cite[Construction 3.16]{ImaiKatoYoucisTannakian}. Recall moreover from \cite[Section 3.2.2]{ImaiKatoYoucisCrystalline} the functor
\begin{align}
    U_{\mathrm{sht}}:\mathcal{G}-\operatorname{Loc}_{\zp}^{\dr}(X_{\eta}) \to \shtg(X_{\eta}^\lozenge)
\end{align}
and from \cite[Section 2.3.4]{ImaiKatoYoucis}, the \'etale realization functor
\begin{align}
    T_{\text{\'et}}:\mathcal{G}-\operatorname{Vect}^{\mathrm{an}, \varphi}(X_{\prism}) \to \mathcal{G}-\operatorname{Loc}_{\zp}^{\mathrm{crys}}(X_{\eta}).
\end{align}
Then by \cite[Theorem 3.17]{ImaiKatoYoucisTannakian}, the following diagram is $2$-commutative
\begin{equation}
    \begin{tikzcd}
        \mathcal{G}-\operatorname{Vect}^{\mathrm{an}, \varphi}(X_{\prism}) \arrow{r}{T_{\text{\'et}}} \arrow{d}{T_{\mathrm{sht}}}& \mathcal{G}-\operatorname{Loc}_{\zp}^{\mathrm{crys}}(X_{\eta}) \arrow{d}{U_{\mathrm{sht}}} \\
        \shtg(X^{\lozenge}) \arrow{r} & \shtg(X^{\lozenge}_{\eta}).
    \end{tikzcd}
\end{equation}
Moreover, the top horizontal arrow is an equivalence by \cite[Proposition 2.22]{ImaiKatoYoucisTannakian}, see also \cite{GuoReinecke}, and the bottom horizontal arrow is fully faithful by \cite[Theorem 2.7.7]{PappasRapoportShtukas}. 

\subsubsection{} Let $V_0=W(k)$ and let $\operatorname{Isoc}(X_s/V_0)$ be the category of $F$-isocrystals on the special fiber $X_s$.\footnote{Since we do not consider isocrystals without $F$-structure, we omit the Frobenius from the notation.} Let us write $\bung(X^{\lozenge})$ for the groupoid of morphisms $X^{\lozenge} \to \bun_{G}$. 
\begin{Construction} \label{Constr:AnalyticIsocrystal}
We now construct a functor
\begin{align}
    \delta_{\ast}:G-\operatorname{Isoc}(X_s/V_0) \to \bung(X^{\lozenge}).
\end{align}
Given $(R,R^+) \in \perf$ together with an untilt $(R^{\sharp}, R^{\sharp+})$ and a morphism $x:\spf R^{\sharp+} \to X$, we will construct a $G$-bundle on $X_{(R,R^+)}$. For a pseudouniformizer $\varpi \in R^{+}$ with image $\varpi^{\sharp} \in R^{\sharp+}$ let $\operatorname{Mod}^{\varphi}(\mathbf{B}^+_{\crys}(R^+/\varpi))$ be the category of $\varphi$-modules over $\mathbf{B}^+_{\crys}(R^+/\varpi)$, which we note is canonically equivalent to $\operatorname{F}\mathrm{-Isoc}(\spec R^{\sharp+}/\varpi^{\sharp})$ (see \cite[Proposition 4.1.3]{ScholzeWeinsteinModuli}). Let $\operatorname{Vect}(X_{(R,R^+)})$ be the $\qp$-linear tensor category of vector bundles on the relative Fargues--Fontaine curve over $\spa(R,R^+)$, then there is an exact $\qp$-linear tensor functor 
\begin{align}
    \delta:\operatorname{Mod}^{\varphi}(\mathbf{B}_{\crys}^{+}(R^+/\varpi)) \to \operatorname{Vect}(X_{(R,R^+)})
\end{align}
described in \cite[Section 2.5.3]{DvHKZIgusaStacks}.\footnote{The functor is constructed by pushing out via a $\varphi$-equivariant map $\mathbf{B}_{\crys}^{+}(R^+/\varpi) \to \mathcal{O}_{\mathcal{Y}_{[r,\infty)}}(R,R^+)$ for some $r \gg 0$ and then descending to $X_{(R,R^+)}$, see \cite[Section 2.5.3]{DvHKZIgusaStacks}; this will be relevant later.} The functor $\delta_{\ast}$ is now given by taking a $\qp$-linear exact tensor functor
\begin{align}
    \omega:\operatorname{Rep}_{\qp} G \to \operatorname{Isoc}(X_s/V_0)
\end{align}
to the morphism $X^{\lozenge} \to \bung$ which over $x:\spd(R,R^+) \to X^{\lozenge}$ evaluates to
\begin{align}
    \delta \circ x^{\ast} \omega.
\end{align}
\end{Construction}
\subsubsection{} \label{subsub:IsocrystalConstruction} There is a realization functor (see \cite[Remark 3.14]{GuoReinecke})
\begin{align}
    D_{\mathrm{crys}}:\operatorname{Vect}^{\mathrm{an}, \varphi}(X_{\prism}) &\to \operatorname{Vect}^{\mathrm{an}, \varphi}(X_{s,\prism}) \simeq \operatorname{Isoc}(X_s/V_0).
\end{align}
\begin{Lem} \label{Lem:PrismaticCompatibilityI}
    The following diagram is $2$-commutative. 
    \begin{equation}
        \begin{tikzcd}
            \mathcal{G}-\operatorname{Vect}^{\mathrm{an}, \varphi}(X_{\prism}) \arrow{r}{T_{\mathrm{sht}}} \arrow{d}{D_{\crys}} & \shtg(X^{\lozenge}) \arrow{d}{\operatorname{BL}^{\circ}} \\
            G-\operatorname{Isoc}(X_s/V_0) \arrow{r}{\delta_{\ast}} & \bung(X^{\lozenge}).
        \end{tikzcd}
    \end{equation}
\end{Lem}
\begin{proof}
Let $(R,R^+) \in \perf$ together with an untilt $(R^{\sharp}, R^{\sharp+})$ and a morphism $x:\spf R^{\sharp+} \to X$. In this case
\begin{align}
    \mathcal{G}-\operatorname{Vect}^{\mathrm{an}, \varphi}\left((\spf R^{\sharp+})_{\prism}\right) \simeq \mathcal{G}-\operatorname{BKF}(R^{\sharp+}),
\end{align}
where $\operatorname{BKF}(R^{\sharp+})$ is the category of Breuil--Kisin--Fargues modules over $R^{\sharp+}$ as in \cite[Definition 2.2.4]{PappasRapoportShtukas}. Moreover, the functor $T_{\mathrm{sht}}$ factors, by construction, through this equivalence via composition with an $\zp$-linear exact tensor functor (see \cite[Definition 2.2.6]{PappasRapoportShtukas})
\begin{align}
    \operatorname{BKF}(R^{\sharp+}) \to \operatorname{Sht}(\spd(R,R^+) \to \spd \zp).
\end{align}
The lemma comes down to showing that the following diagram is $2$-commutative. 
\begin{equation} \label{Eq:DiagramOfRealizations}
    \begin{tikzcd}
        \operatorname{BKF}(R^{\sharp+}) \arrow{r} \arrow{d}&  \operatorname{Sht}(\spd(R,R^+) \to \spd(\zp)) \arrow{d}{\operatorname{BL}^{\circ}} \\
        \operatorname{Mod}^{\varphi}(\mathbf{B}_{\crys}^{+}(R^+/\varpi)) \arrow{r}{\delta} & \operatorname{Vect}(X_{(R,R^+)}).
    \end{tikzcd}
\end{equation}
This is well known and moreover a straightforward consequence of the constructions, see \cite[Section 2.5]{DvHKZIgusaStacks} for a very similar argument.\footnote{
The key point is that $W(R^+)[1/p] \to \Gamma(\mathcal{Y}_{[r, \infty)}(R, R^+), \mathcal{O}_{\mathcal{Y}_{[r,\infty)}(R,R^+)})$ factors through $\mathbf{B}_{\crys}^{+}(R^+/\varpi)$ for $r \gg 0$, see \cite[Section 2.5.3]{DvHKZIgusaStacks}.} 
\end{proof}

\subsubsection{} Let $X_s^{\mathrm{perf}}$ be the perfection of $X_s$. There is a functor
\begin{align}
    \epsilon:\bun_G(X^{\lozenge}) \to G-\operatorname{Isoc}(X_s^{\mathrm{perf}})=\gisoc(X_s^{\mathrm{perf}}/V_0)
\end{align}
coming from the isomorphisms $(X^{\diamond})^{\mathrm{red}}=X_s^{\mathrm{perf}}$ (see Lemma \ref{Lem:ReductionOfScheme}) and $(\bung)^{\mathrm{red}} \simeq \gisoc$, see Section \ref{subsub:ReductionShtukaBung}.
\begin{Lem} \label{Lem:ObviousHopefully}
    The composition $\epsilon \circ \delta_{\ast}$ is isomorphic to the functor obtained from pullback along $X_s^{\mathrm{perf}} \to X_s$.
\end{Lem}
\begin{proof}
Morphisms $X_s^{\mathrm{perf}} \to \gisoc$ correspond uniquely to morphisms $X_s^{\diamond} \to \bung$. Thus it suffices to show that for all affinoid perfectoid spaces $(R,R^+)$ and for all maps $\spd(R,R^+) \to X_s^{\diamond}$, the corresponding $G$-bundles on $X_{(R,R^+)}$ agree. We may moreover replace $X_s^{\diamond}$ with $X_s^{\diamond,\mathrm{pre}}$, whose $(R,R^+)$-points are morphisms $\spec R^+ \to X_s$. \smallskip 

Unwinding the constructions above, we are trying to show that the following two $G$-bundles on $X_{(R,R^+)}$ are isomorphic:
\begin{itemize}
    \item Pull back the $G$-bundle on $X$ along $\spec R^+/\varpi \to X$ for some $\varpi$, and then apply $\delta_{\ast}$.

    \item Pull back the $G$-bundle on $X$ to a $G$-bundle on $\spec R^+$, and then pushforward via $W(R^+)[1/p] \to \Gamma(\mathcal{Y}_{[r, \infty)}(R, R^+),
  \mathcal{O}_{\mathcal{Y}_{[r,\infty)}(R,R^+)})$.
\end{itemize}
This follows from the $2$-commutative diagram at the end of the proof of Lemma \ref{Lem:PrismaticCompatibilityI} above, with $R^{\sharp}=R$. 
\end{proof}

\subsection{Recollections on Igusa stacks} \label{Sub:IgusaStacks} Let $\gx$ be a Shimura datum of abelian type with reflex field $\mathsf{E}$, fix a prime number $p$ and write $G=\mathsf{G} \otimes \qp$. Let $v$ be a prime of $\mathsf{E}$ above $p$ and let $E$ be the $v$-adic completion of $\mathsf{E}$ with ring of integers $\mathcal{O}_E$. For $K \subset \gaf$ we consider the adic space $ \msh_{K}\gx^{\mathrm{an}}$ over $\spa E$ given by the analytification of the Shimura variety over $\spec E$. This contains a quasicompact open subspace 
\begin{align}
    \msh_{K}\gx^{\circ, \mathrm{an}} \subset \msh_{K}\gx^{\mathrm{an}}
\end{align}
called the potentially crystalline locus, see \cite[Theorem 5.17]{ImaiMieda}. We will write
\begin{align}
    \msh_{K}\gx^{\circ, \lozenge} \subset \msh_{K}\gx^{\lozenge}
\end{align}
for the associated diamonds over $\spd E$. The formation of $\msh_K\gx^{\circ, \lozenge}$ is compatible with changing $K$, see \cite[Corollary 5.29]{ImaiMieda}, and we will also consider
\begin{align}
    \msh\gx^{\circ, \lozenge} \coloneqq\varprojlim_{K \subset \gaf} \msh_{K}\gx^{\circ, \lozenge}.
\end{align}
Let $\mu$ be the $G(\qpbar)$-conjugacy class of Hodge cocharacters of $G$ induced by $\x,v$. Consider the following result of \cite{ZhangThesis}, \cite{DvHKZIgusaStacks}, \cite{KimFunctoriality} and \cite{DanielsVHKimZhangCompanion}. 
\begin{Thm} \label{Thm:IgusaMain}
  There is a small v-stack $\igs\gx$ over $\spd \fp$ equipped with an
  action of $\ul{\gafp}$ and a $\ul{\gafp}$-invariant map 
  \begin{align*}
      \pibar_{\HT} \colon \igs\gx \to \bungmu
  \end{align*}
  fitting in a $2$-Cartesian diagram 
  \begin{equation} \label{Eq:ConjectureCartesianDiagramGen} \begin{tikzcd} 
    \msh\gx^{\circ, \lozenge} \arrow{r}{\pi_{\mathrm{HT}}} \arrow{d} & \operatorname{Gr}_{G, -\mu} \arrow{d}{\mathrm{BL}} \\
    \igs\gx \arrow{r}{\overline{\pi}_\mathrm{HT}} & \bungmu.
  \end{tikzcd} \end{equation}
\end{Thm}
If $\gx$ is of PEL type, then Theorem \ref{Thm:IgusaMain} is \cite[Theorem 1.3]{ZhangThesis}. If $\gx$ is of Hodge type, then Theorem \ref{Thm:IgusaMain} is \cite[Theorem I]{DvHKZIgusaStacks} or \cite[Theorem D]{KimFunctoriality}. If $\gx$ is of abelian type, then Theorem \ref{Thm:IgusaMain} is \cite[Theorem I]{DvHKZIgusaStacksII}.

\subsubsection{} \label{subsub:IntegralFiberProduct} We consider the following integral refinement of Theorem \ref{Thm:IgusaMain}. Let $\mathcal{G}$ be a quasi-parahoric model of $G$. Recall the quotient $\g \to \g^c$, which induces a parahoric model $\mathcal{G}^c$ of $G^c=\g^c \otimes \qp$ (see \cite[Section 4.1.1]{DanielsYoucis}). Then there is a (conjectural) canonical integral model $\scrs_{K_p}\gx$ of $\msh_{K_p}\gx$ over $\mathcal{O}_{E}$, characterized by a morphism
\begin{align}
    \scrs_{K_p}\gx^{\lozenge/} \to \operatorname{Sht}_{\mathcal{G}^c, \mu^c},
\end{align}
see \cite[Conjecture 4.5]{Daniels} and \cite[Conjecture 4.2.2]{PappasRapoportShtukas}. This conjecture is known in almost all cases (e.g. if $p \ge 5$), see \cite{DanielsYoucis}. We have the following conjectural refinement of Theorem \ref{Thm:IgusaMain}. 
\begin{Conj} \label{Conj:IgusaMainInt}
Given a parahoric model $\mathcal{G}$ of $G$ with $K_p=\mathcal{G}(\zp)$, we consider the fiber product
  \[ \begin{tikzcd} \label{Eq:ConjectureCartesianDiagramInt}
    S \arrow{r}{\pi_{\mathrm{crys}}} \arrow{d} & \shtgmuone \arrow{d}{\mathrm{BL}^{\circ}} \\
    \igs\gx \arrow{r}{\overline{\pi}_\mathrm{HT}} & \bungmu.
  \end{tikzcd} \]
Then $S \to \shtgmu \to \operatorname{Sht}_{\mathcal{G}^c, \mu^c}$ is $\gafp$-equivariantly isomorphic to $\scrs_{K_p}\gx^{\diamond} \to \operatorname{Sht}_{\mathcal{G}^c, \mu^c}$.
\end{Conj}
If $\gx$ is of Hodge type, then $\g=\g^c$ and this is \cite[Theorem VII]{DvHKZIgusaStacks}. The unramified PEL type case is \cite[Theorem 1.3]{ZhangThesis}.

\subsubsection{} Now assume that Conjecture \ref{Conj:IgusaMainInt} holds. Let $\mathcal{G}$ be a quasi-parahoric integral model of $G$ and write $K_p=\mathcal{G}(\zp)$. Then we can take the reduction of \eqref{Eq:ConjectureCartesianDiagramInt} as in \cite[Section 4.3.1]{DvHKZIgusaStacks} (see Section \ref{subsub:ReductionShtukaBung})
\begin{equation}
    \begin{tikzcd}
        \shginf \arrow{r} \arrow{d} & \shtlocgmuone \arrow{d} \\
        \igs \gx^{\mathrm{red}} \arrow{r} & \gisocmu.
    \end{tikzcd}
\end{equation}
It follows from \cite[Proposition 3.0.21, Lemma 3.0.2]{DvHKZIgusaStacks} that the map $\shtlocgmuone \to \gisocmu$ is v-surjective, and thus that $\operatorname{Sh}_{K_p}\gx \to \igs \gx^{\mathrm{red}}$ is v-surjective. 

\subsection{Conjectural exotic Hecke correspondences} \label{Sub:Conjectures} In this section we state a conjecture on the existence of exotic isomorphisms between Igusa stacks, and explain that it generalizes a conjecture of Xiao--Zhu on the existence of exotic Hecke correspondences. 

\subsubsection{} \label{Sec:StandardSetup} Let us start with a Shimura datum $\gx$ of abelian type, a prime $p$ and an isomorphism $\mathbb{C} \isom \qpbar$. Let $\mathsf{P}$ be a $\g$-torsor over $\spec \mathbb{Q}$ that is trivial over $\spec \mathbb{Q}_{\ell}$ for all $\ell \not=p$, and let $\g'=\operatorname{Aut}_{\g}(\mathsf{P})$. We let $\x'$ be a Shimura datum for $\g'$. We consider the following condition on $\x'$ (see Sections \ref{subsub:RealBGShimura}, \ref{subsub:RealTwistKottwitz} for the definitions of the involved objects).
\begin{Assump} \label{Assump:Infinity}
Under the natural map $\beta_{\mathsf{P}_{\R}}:B(\R, \g) \to B(\R, \g')$, the class $[\x]$ is mapped to $[\x']$.
\end{Assump}

\subsubsection{} Our isomorphism $\mathbb{C} \isom \qpbar$ defines $p$-adic places $v$ and $v'$ of the reflex fields $\mathsf{E}$ and $\mathsf{E}'$ of $\gx$ and $\gxp$, and we let $E$ and $E'$ be the respective $p$-adic completions. We write $G=\g \otimes \qp$ and $G' = \g' \otimes \qp$. Let $\mu$ be the $G(\qpbar)$-conjugacy class of Hodge cocharacters induced by $\mathsf{X}$ and $v$, and let $\mu'$ be the $G'(\qpbar)$-conjugacy class of Hodge cocharacters induced by $\mathsf{X}'$ and $v'$.

\subsubsection{} The $G$-torsor $P=\mathsf{P} \otimes \qp$ over $\spec \qp$ defines a $\spd \fp$-point $P:\spd \fp \to \bun_G$ by pullback. Then as explained in \cite[Section III.4.1]{FarguesScholze} the map
\begin{align}
    \beta_{P}:\bun_G &\to \bun_{G'} \\    
    \mathcal{E} &\to \operatorname{Isom}_{G}(\mathcal{E}, P)
\end{align}
is an isomorphism, which takes $P$ to the trivial $G'$-torsor on the Fargues--Fontaine curve. We consider the open substacks $\bungmu \subset \bun_G$ and $\bungpmup \subset \bun_{G'}$ and consider
\begin{align}
    \bungpmumup:=\bungpmup \times_{\bun_{G'}} \beta_P(\bungmu).
\end{align}
\subsubsection{} We now define two open substacks of our Igusa stacks via pullback diagrams 
\begin{equation}
    \begin{tikzcd}
        \igs \gx_{-\mu'} \arrow{r} \arrow{d} & \igs \gx \arrow{d}{\beta_P} \\
        \bungpmumup \arrow{r} & \bun_{G'}
    \end{tikzcd}
        \begin{tikzcd}
        \igs \gxp_{-\mu} \arrow{r} \arrow{d} & \igs \gxp \arrow{d} \\
        \bungpmumup \arrow{r} & \bun_{G'}.
    \end{tikzcd}
\end{equation}
We have the following quite general conjecture. Choose an isomorphism $\kappa^p:\mathsf{P} \otimes \afp \isom \g \otimes \afp$ which induces an isomorphism $\gp \otimes \afp \isom \g \otimes \afp$ which we will use to identify $\gp(\afp)$ with $\gafp$.
\begin{Conj} \label{Conj:Main}
If Assumption \ref{Assump:Infinity} holds and $\Sha^1(\mathbb{Q},\g)=0$, then there is a $\gafp$-equivariant isomorphism
    \begin{align}
        \igs \gx_{-\mu'} \to \igs \gxp_{-\mu}
    \end{align}
    of v-stacks over $\bungpmumup$. 
\end{Conj}
\begin{Rem}
    The assumption that $\Sha^1(\mathbb{Q},\g)=0$ can be removed if one replaces the Shimura varieties (and thus the Igusa stacks) for $\gx$ with the rational Shimura varieties for $\gx$ of Sempliner--Taylor \cite{SemplinerTaylorShimura} or the extended Shimura varieties of Xiao--Zhu \cite{XiaoZhu2}. The conjecture should moreover be generalized to allow $\mathsf{P}$ to be a basic Kottwitz cocycle in the sense of \cite{SemplinerTaylorShimura}, although then $P$ generally only defines an $\fpbar$-point of $\bung$ and so the generalized conjecture would be for Igusa stacks over $\fpbar$.
\end{Rem}
\begin{Rem}
The stack $\bungmumup$ is nonempty precisely when the set $\bgpmumup:=\beta_{P}(\bgmu) \cap \bgmup$ is nonempty. If it is empty, then Conjecture \ref{Conj:Main} is vacuous. Assumption \ref{Assump:Infinity} can be seen as an Archimedean analogue of the nonemptiness of $\bgpmumup$, although it is not clear to us that the conjecture is vacuously true if Assumption \ref{Assump:Infinity} does not hold.
\end{Rem}
\begin{Rem}
The conjecture extends verbatim to the larger Igusa stacks $\operatorname{Igs} \gx$ constructed in \cite{KimFunctoriality},\cite{DvHKZIgusaStacksII}. 
\end{Rem}

\subsubsection{} Taking the reduction of the stacks and morphisms in the previous conjecture, we obtain an isomorphism (see Section \ref{subsub:ReductionShtukaBung} and Lemma \ref{Lem:NewtonStrata})
\begin{align}
    \beta_P:\gisoc \to \gpisoc, 
\end{align}
a closed substack $\gpisocmumup \subset \gpisoc$ and Cartesian diagrams
\begin{equation}
    \begin{tikzcd}
        \igs \gx_{-\mu'}^{\mathrm{red}} \arrow{r} \arrow{d} & \igs \gx^{\mathrm{red}} \arrow{d}{\beta_P} \\
        \gpisocmumup \arrow{r} & \gpisoc
    \end{tikzcd}
        \begin{tikzcd}
        \igs \gxp_{-\mu}^{\mathrm{red}} \arrow{r} \arrow{d} & \igs \gxp^{\mathrm{red}} \arrow{d} \\
        \gpisocmumup \arrow{r} & \gpisoc.
    \end{tikzcd}
\end{equation}
We have the following direct consequence of Conjecture \ref{Conj:Main} for reductions of the Igusa stacks.
\begin{Conj} \label{Conj:MainII}
If Assumption \ref{Assump:Infinity} holds and $\Sha^1(\mathbb{Q},\g)=0$, then there is a $\gafp$-equivariant isomorphism
    \begin{align}
        \igs \gx^{\mathrm{red}}_{-\mu'} \to \igs \gxp_{-\mu}^{\mathrm{red}}
    \end{align}
    of v-stacks over $\gpisocmumup$. 
\end{Conj}
As a corollary of Conjecture \ref{Conj:MainII}, we get the following consequence for Igusa varieties. Recall that for $b:\spec \ovfp \to \gisocmu$, respectively, $b':\spec \ovfp \to \gpisocmup$ we define perfect Igusa varieties as the following fiber products (see \cite[Lemma 4.3.3]{DvHKZIgusaStacks})
\begin{equation}
    \begin{tikzcd}
        \operatorname{Ig}^{ b}\gx \arrow{r} \arrow{d} & \spec \ovfp \arrow{d}{b} \\
        \igs \gx^{\mathrm{red}} \arrow{r} & \gisocmu
    \end{tikzcd}
     \begin{tikzcd}
        \operatorname{Ig}^{ b'}\gxp \arrow{r} \arrow{d} & \spec \ovfp \arrow{d}{b'} \\
        \igs \gxp^{\mathrm{red}} \arrow{r} & \gpisocmup.
    \end{tikzcd}
\end{equation}
Note that $\operatorname{Ig}^{ b}\gx$ receives an action of the $\sigma$-centralizer $G_b(\qp)$ of $b$ and that $\operatorname{Ig}^{ b'}\gxp$ receives an action of the $\sigma$-centralizer $G'_{b'}(\qp)$ of $b'$.

\begin{Conj} \label{Conj:IgusaVarieties}
If Assumption \ref{Assump:Infinity} holds and $\Sha^1(\mathbb{Q},\g)=0$, then for $b':\spec \ovfp \to \gpisocmumup$ with corresponding morphism $b:\spec \ovfp \to \gisocmu$ (under $\beta_P^{-1}$), there is a $\g(\afp) \times G_b(\qp)$-equivariant isomorphism
    \begin{align}
        \operatorname{Ig}^{b}\gx \isom \operatorname{Ig}^{ b'}\gxp,
    \end{align}
    where we identify $G'_{b'}(\qp) = G_b(\qp)$ using $\beta_{P}$.
\end{Conj}
\begin{Rem}
One can similarly formulate a conjecture about $\g(\afp) \times \tilde{G}_b$-equivariant isomorphisms of the v-sheaf Igusa varieties of \cite[Section 4.4]{DvHKZIgusaStacks}, which similarly follows from Conjecture \ref{Conj:Main}.
\end{Rem}

\subsubsection{} We now assume that Conjecture \ref{Conj:IgusaMainInt} holds.\footnote{If $\g=\g^c$, this is known by \cite{DanielsYoucis} (for $p>2$) and \cite{MaoWu} (in general).} Choose parahoric models $\mathcal{G}$ and $\mathcal{G}'$ of $G$ and $G'$ respectively and consider the fiber product
\begin{equation}
    \begin{tikzcd}
        \shtlocgmumup \arrow{r} \arrow{d} & \shtlocgmu \arrow{d}{\beta_P} \\
    \shtlocgpmup \arrow{r} & \gpisoc.
    \end{tikzcd}
\end{equation}
Note that the image of $\shtlocgmumup \to \gpisoc$ is precisely $\gpisocmumup$. Moreover, the stack $\shtlocgmumup$ fits in a fiber product diagram
\begin{equation}
    \begin{tikzcd}
        \shtlocgmumup \arrow{r} \arrow{d} & \shtlocgmu \times \shtlocgpmup \arrow{d} \\
     \gpisocmumup \arrow{r}{\beta_P^{-1} \times 1} & \gisoc \times \gpisoc.
    \end{tikzcd}
\end{equation}
Finally, by Corollary \ref{Cor:IndRepresentableII}, the morphisms $\shtlocgmumup \to \shtlocgmu, \shtlocgpmup$ are representable in ind-(perfectly proper and perfectly finitely presented perfect algebraic spaces). 

\subsubsection{} \label{subsub:IgusaConjToCorrConj} We continue to assume that Conjecture \ref{Conj:IgusaMainInt} holds. Let $K_p=\mathcal{G}(\zp)$ and $K_p'=\mathcal{G}'(\zp)$ and let $\shginf$ (resp. $\shgpinf$) be the special fibers of the integral models of the Shimura varieties for $\gx$ (resp. $\gxp$) of level $K_p$ (resp. $K_p'$). Assume Conjecture \ref{Conj:MainII} and basechange the graph of the isomorphism $\alpha$ of Conjecture \ref{Conj:MainII} along the natural map 
\begin{align}
    \shginf \times \shgpinf \to \igs \gx^{\mathrm{red}} \times \igs \gxp^{\mathrm{red}}.
\end{align}
We obtain using Lemma \ref{Lem:IndRepresentableI} an ind-(perfect scheme)
\begin{align}
    \operatorname{Sh}_{\mathsf{P}, \mu,\mu'} \to \shginf \times \shgpinf
\end{align}
which fits in a $2$-commutative diagram 
\begin{equation}
\begin{tikzcd}
   \operatorname{Sh}_{\mathsf{P}, \mu,\mu'} \arrow{r} \arrow{d} & \shginf\times \shgpinf \arrow{d} \\
    (\igs \gx^{\mathrm{red}})_{-\mu'} \arrow{r}{1 \times \alpha} & (\igs \gx^{\mathrm{red}})_{-\mu'} \times (\igs \gxp^{\mathrm{red}})_{-\mu}.
\end{tikzcd}
\end{equation}
Since $\alpha$ is an isomorphism, we see that $\operatorname{Sh}_{\mathsf{P}, \mu,\mu'} \simeq\shginf \times_{(\igs \gx^{\mathrm{red}})_{-\mu'}} \shgpinf$ and fits in a $2$-commutative cube
\begin{equation} \label{Eq:TheCube}
  \begin{tikzcd}[column sep=tiny, row sep=tiny]
    & \shginf \arrow[rr] \arrow[dd] & &  (\igs \gx^{\mathrm{red}})_{-\mu'} \arrow{dd} \\
     \msh_{P, \mu,\mu'}\arrow[rr, crossing over] \arrow{ur} \arrow{dd} & &  \shgpinf  \arrow{ur}\\
    &  \shtlocgmu \arrow{rr} & & \gisocmumup \\
      \shtlocgmumup \arrow{rr}{\overline{\pi}_{\mathrm{HT}}} \arrow{ur} & & \arrow{ur} \arrow[from=uu, crossing over] \shtlocgpmup.
  \end{tikzcd}
\end{equation}
The top and bottom faces as well as the left and right faces of the cube are $2$-Cartesian, which implies that the front and back faces are also Cartesian. To conclude, we find that $\operatorname{Sh}_{\mathsf{P}, \mu,\mu'}$ satisfies Conjecture \ref{Conj:MainIII} below.
\begin{Conj} \label{Conj:MainIII}
If Assumption \ref{Assump:Infinity} holds and $\Sha^1(\mathbb{Q},\g)=0$, then there is an ind-(perfect scheme) $\operatorname{Sh}_{\mathsf{P}, \mu,\mu'}$ and a $2$-commutative diagram
    \begin{equation}
        \begin{tikzcd}
            \operatorname{Sh}_{K_p}\gx \arrow{d} & \operatorname{Sh}_{\mathsf{P}, \mu,\mu'} \arrow{r} \arrow{l} \arrow{d} & \operatorname{Sh}_{K_p'}\gxp \arrow{d} \\
            \shtlocgmu & \shtlocgmumup \arrow{r} \arrow{l} & \shtlocgpmup
        \end{tikzcd}
    \end{equation}
    such that both squares are $2$-Cartesian.
\end{Conj}
If $P$ is the trivial torsor and $\mathcal{G}$ and $\mathcal{G}'$ are reductive group schemes over $\zp$, then this recovers a conjecture of Xiao--Zhu, see \cite[Hypothesis 7.3.2]{XiaoZhu}. 

\begin{Rem}
    We have seen that Conjecture \ref{Conj:Main} implies Conjecture \ref{Conj:MainII}. In this article, we will show that Conjecture \ref{Conj:MainII} implies Conjecture \ref{Conj:Main} under some hypotheses. 
\end{Rem}

\section{Isogenies and Igusa stacks} \label{Sec:IgusaStacks} 

We start this section by recalling the definition of Igusa stacks for Hodge type Shimura varieties and their relation to canonical integral models at parahoric level of these Shimura varieties; here we follow \cite[Section 5]{DvHKZIgusaStacks}. In Section \ref{Sub:ReductionIgusaStack} we prove Theorem \ref{Thm:IntroIgusaDCirc}.

\subsection{Recollection on integral models of Shimura varieties of Hodge type} 
Let $\gx$ be a Shimura datum of Hodge type with reflex field $\mathsf{E}$, let $p$ be a prime and write $G=\mathsf{G}_{\qp}$. Fix a prime $v$ above $p$ of $\mathsf{E}$ and let $E$ be the completion of $\mathsf{E}$ at $v$. 
\subsubsection{} \label{subsub:XiChoice} Let $\mathcal{G}$ be a stabilizer Bruhat--Tits model of $G$ over $\zp$ and let $K_p=\mathcal{G}(\zp)$. Because $\mathcal{G}$ is the stabilizer of a point in the extended building $\mathcal{B}^e(G,\qp)$, it follows from the discussion in \cite[Section~1.3.2]{KMPS} that there exists a Hodge embedding $\iota:\gx \to \gvx$ and a $\zp$-lattice $V_{\zp} \subset V_{\qp}$ on which $\psi$ is $\zp$-valued, such that $\mathcal{G}(\zpbr)$ is the stabilizer in $G(\qpbr)$ of $V_{\zp} \otimes_{\zp} \zpbr$. By Zarhin's trick, see \cite[Remark 2.2.4]{ShenYuZhang}, we may moreover assume (after possibly changing $\iota$ and the symplectic space) that $V_{\zp}$ is a self-dual lattice. The fact that $\mathcal{G}(\zpbr)$ stabilizes $V_{\zpbr}$ implies, by \cite[Corollary 2.10.10]{KalethaPrasad}, that $G \to G_V$ extends to a morphism $\mathcal{G} \to \mathcal{G}_{V}=\mathrm{GSp}(V_{\zp})$. We denote these data by $\Xi = (\mathcal{G}, \iota, V_{\zp})$.

\subsubsection{} Let $\Xi$ be as above. By \cite[Lemma 2.1.2]{KisinModels}, for $K^p \subset \gafp$ we can find $M^p \subset \gv(\afp)$ containing $K^p$ such that the natural map
\begin{align}
    \msh_{K}\gx \to \msh_{M} \gvx \otimes_{\mathbb{Q}_p} E
\end{align}
is a closed immersion, where $M_p=\mathrm{GSp}(V_{\zp})(\zp)$ and $M=M^pM_p$ and $K=K^pK_p$. We then define $\scrs_{K}\gx$ to be the normalization of the Zariski closure of $\msh_{K}\gx$ in $\scrs_{M}\gvx \otimes_{\zp} \mathcal{O}_E$. By \cite[Theorem~4.1.3]{DanielsVHKimZhangCompanion}, the integral model $\scrs_{K}\gx$ is an integral canonical model which admits an essentially unique morphism
\begin{align}
    \scrs_{K}\gx^{\diamond} \xrightarrow{\pi_{\mathrm{crys}, \mathcal{G}}} \shtgmuone. 
\end{align}
We will also consider
\begin{align}
    \scrs_{K_p}\gx:=\varprojlim_{K^p \subset \mathsf{G}(\afp)} \scrs_{K^pK_p}\gx,
\end{align}
which is equipped with an action of the locally profinite group scheme $\ul{\mathsf{G}(\afp)}$.  

\subsubsection{Etale Local Systems} \label{Sec:EtaleTensors} We now summarize \cite[Section 5.1.8, 5.1.9]{DvHKZIgusaStacks}. We will write $\pi:A_{\Xi} \to \scrshg$ for the pullback along $\scrs_{K}\gx \to \scrs_M\gvx$ of the universal abelian variety up to prime-to-$p$ isogeny over $\scrs_M\gvx$. We will write $\mathcal{V}^p$ for the dual of the pro-\'etale sheaf
\begin{align}
    R^1 \pi_{\ast} \underline{\afp} 
\end{align}
on $\scrs_{K}\gx$. It is explained in \cite[Section 5.1.4]{KisinShinZhu} that there is an exact tensor functor
\begin{align}
    \mathbb{L}^p:\operatorname{Rep}_{\mathbb{Q}}(\mathsf{G}) \to \{\afp \text{ local systems on } \scrs_{K}\gx\}
\end{align}
such that $\mathbb{L}^p(V)=\mathcal{V}^p$, where $V$ is the representation induced by $\mathsf{G} \to \gv \to \operatorname{GL}_V$.\footnote{As in \cite[Section 5.1.8, 5.1.9]{DvHKZIgusaStacks}, we note that what \cite[Section 5.1.4]{KisinShinZhu} calls $\mathcal{V}^p$ is $\mathbb{L}^p(V^{\ast})$ in our notation.}Over $\scrs_{K_p}\gx$, there is a canonical isomorphism
\begin{align}
    \eta: \mathcal{V}^p\isom V \otimes \underline{\mathbb{A}}_{f}^{p},
\end{align}
of $\gafp$-torsors, see \cite[Lemma 5.1.9]{KisinShinZhu}.
\subsubsection{Crystals}\label{subsub:CrystallineTensors} Let $S=\spa(R,R^+)$ be an object in $\perf$ with an untilt $S^\sharp=\spa(R^\sharp, R^{\sharp+})$, and let $x$ be a morphism $x:\spf R^{\sharp+} \to \scrshat_K\gx$. By pulling back $A_{\Xi}$ along
\begin{align}
    \spf R^{\sharp+} \to \scrshat_{K_p}\gx \to \scrshat_{M}\gvx \otimes_{\zp} \mathcal{O}_E,
\end{align}
we get a formal abelian scheme up to prime-to-$p$ isogeny $A_x \to \spf R^{\sharp+}$. As explained in \cite[Section 2.5.3]{DvHKZIgusaStacks}, this gives rise to a vector bundle $\mathcal{E}(A_x)$ on the relative Fargues--Fontaine curve $X_S$. \smallskip 

The untilt $S^\sharp$ along with the composition
\begin{align}
    \spa(R^{\sharp},R^{\sharp+}) \to \spa R^{\sharp+} \to (\scrshat_{K_p}\gx)^\mathrm{ad}
\end{align}
determines a point $\tilde{x}$ of $\scrs_{K_p}\gx^\diamond(S)$. In turn, via the composition
\begin{equation}\label{Eq:MapToBunG}
    \scrs_{K_p}\gx^\diamond \xrightarrow{\pi_\mathrm{crys}} \shtgmu \xrightarrow{\mathrm{BL}^{\circ}} \bun_G,
\end{equation}
$\tilde{x}$ determines a $G$-torsor $Q$ over $X_S$. By the Tannakian description of $G$-torsors, see \cite[Appendix to Lecture 19]{ScholzeWeinsteinBerkeley}, this gives rise to an exact tensor functor
\begin{align}
    \mathbb{L}_{\mathrm{crys},x}:\operatorname{Rep}_{\qp}(G) \to
    \{ \text{vector bundles on } X_S\}.
\end{align}
Let $Y_x := A_x[p^\infty]$ be the $p$-divisible group over $R^{\sharp +}$ associated with $A_x$. By construction, the composition
\begin{align*}
\spa(R^{\sharp},R^{\sharp+}) \to \scrs_{K_p}\gx^\diamond \xrightarrow{\pi_\mathrm{crys}} \shtgmu \to \mathrm{Sht}_{\mathcal{GL}(V)}
\end{align*}
corresponds to the shtuka $\mathscr{V}(Y_x)$ defined in \cite[Section 2.4]{DvHKZIgusaStacks}, see \cite[Section 4.6.3]{PappasRapoportShtukas}. It then follows from \cite[Lemma 2.5.5]{DvHKZIgusaStacks} that there is a canonical isomorphism
\[
  \mathbb{L}_{\mathrm{crys},x}(V) \simeq \mathcal{E}(A_x) \simeq \operatorname{BL}^{\circ}(\mathscr{V}(Y_x)).
\]

\subsubsection{On crystalline tensors} \label{sub:CrystallineTensorsII} If $\mathcal{G}$ is reductive, then $X=\scrshat_{K}\gx$ is a smooth $p$-adic formal scheme, see \cite{KisinModels}, \cite{KimMadapusi}. Recall that the morphism
\begin{align}
    X^{\diamond} \to \shtgmu
\end{align}
is uniquely characterized by the property that $X^{\diamond}_{\eta} \to \shtgmu$ recovers $U_{\mathrm{sht}}$ applied to the de Rham $\mathcal{G}(\zp)$-local system given by $\msh_{K^p}\gx^{\circ, \lozenge} \to X^{\diamond}_{\eta}$, see Section \ref{subsub:FullyfaithfulShtukasPrisms}. Thus it makes sense to ask if $X^{\diamond} \to \shtgmu$ comes (necessarily uniquely) from an object
\begin{align}
    \omega_{\prism} \in \mathcal{G}-\operatorname{Vect}^{\mathrm{an}, \varphi}(X_{\prism}).
\end{align}
We have the following theorem of Imai--Kato--Youcis \cite[Theorem A]{ImaiKatoYoucis} and Madapusi--Youcis \cite[Theorem C]{MadapusiYoucis}.
\begin{Thm} \label{Thm:ImaiKatoYoucis}
    If $\mathcal{G}$ is reductive, then there is an object $\omega_{\prism} \in \mathcal{G}-\operatorname{Vect}^{\mathrm{an}, \varphi}(X_{\prism})$ such that $T_{\mathrm{sht}}(\omega_{\prism})$ gives $\scrs_{K}\gx^{\diamond} \to \shtgmu$.
\end{Thm}

\subsection{Recollections on Igusa stacks} We now review the construction of the Igusa stack of \cite[Section 5.2]{DvHKZIgusaStacks}. 

\subsubsection{}
Suppose we are given a Huber pair $(R, R^+)$ together with a choice of untilts $R^{\sharp_i}$ for $i=1,2$ and a pair of abelian schemes up to prime-to-$p$ isogeny $A_i/\spf(R^{\sharp+}_i)$. We recall the following definition from \cite[Section 2.5.1]{DvHKZIgusaStacks}. 
\begin{Def}\label{defn:formalQIsog}
A \emph{formal quasi-isogeny} $f:A_1 \dashrightarrow A_2$ is a quasi-isogeny
    \[
        f: A_1 \otimes_{R^{\sharp_1+}} R^+/\varpi \dashrightarrow A_2 \otimes_{R^{\sharp_2+}} R^+/\varpi
    \]
for some pseudouniformizer $\varpi \in R^+$ such that there is a natural map $R^{\sharp+}_i \to R^+/\varpi$ for both values of $i$. This definition does not depend on the choice of $\varpi$ by Serre--Tate lifting as explained in loc. cit.
\end{Def}
\subsubsection{} Choose $\Xi$ as in Section \ref{subsub:XiChoice}. Suppose that $A_1, A_2$ arise from points $x_1: \spf(R^{\sharp_{1}+}) \to \scrshat_{K_p}\gx, x_2: \spf(R^{\sharp_2+}) \to \scrshat_{K_p}\gx$. 
\begin{Def} \label{defn:FormalQIsogStructure}
A formal quasi-isogeny $f:A_1 \dashrightarrow A_2$ is $\g$\emph{-structure preserving away from $p$} if the induced map on the prime-to-$p$ adelic Tate modules
    \[
        \mathcal{V}^p(f): \mathcal{V}^p(A_{1} \otimes R^+/\varpi) \to \mathcal{V}^p(A_{2} \otimes R^+/\varpi)
    \]
    fits into a commutative square
    \[
        \begin{tikzcd}
            &\mathcal{V}^p(A_{1} \otimes R^+/\varpi) \arrow{r}{\mathcal{V}^p(f)}\arrow{d}{\eta_{x_1}}& \mathcal{V}^p(A_{2} \otimes R^+/\varpi) \arrow{d}{\eta_{x_2}}\\
            &V \otimes \afp \arrow[r, equals]& V \otimes \afp.
        \end{tikzcd}
    \]
\end{Def}
For $x_i$ as above we write $\mathbb{L}_{x_i}$ for point of $\bun_{G}(R,R^+)$ coming from $x_i:\spd(R,R^+) \to \bun_{G}$ as before, and note that there is a canonical isomorphism $\mathcal{E}(A_{x_i}) \to \mathbb{L}_{x_i} \times^{G} V$, see Section \ref{subsub:CrystallineTensors}.
\begin{Def}\label{defn:CondnAtP}
    We say that a formal quasi-isogeny $f:A_1 \dashrightarrow A_2$ is \emph{$\g$-structure preserving at $p$} if the induced isomorphism (see \cite[Section 2.5.3]{DvHKZIgusaStacks})
      \[
        f \colon \mathcal{E}(A_{x_1}) \to \mathcal{E}(A_{x_2})
      \]
      of vector bundles on $X_S$ is induced by a (necessarily unique) isomorphism of $G$-bundles $\mathbb{L}_{\mathrm{crys},x_{1}} \isom \mathbb{L}_{\mathrm{crys},x_{2}}$. We say that a formal quasi-isogeny $f:A_1 \dashrightarrow A_2$ is \emph{$\g$-structure preserving} if $f$ is $\g$-structure preserving at $p$ and $\g$-structure preserving away from $p$.
\end{Def}
Given this we can recall \cite[Definition 5.2.4]{DvHKZIgusaStacks}.
\begin{Def} \label{Def:IgsPre}
We define $\igspre_{\Xi} \gx$ to be the presheaf of groupoids on $\perf$, whose value on $S=\spa(R,R^+)$ is the following groupoid: 
\begin{itemize}
    \item an object is a pair $(S^\sharp, x)$ where $S^\sharp = \spa(R^\sharp,
      R^{\sharp+})$ is an untilt of $S$ over $\mathcal{O}_E$ and $x$ is a map
      $\spf R^{\sharp+} \to \scrshat_{K_p}\gx$ of formal schemes over
      $\mathcal{O}_E$,
    \item a morphism $f \colon (S^{\sharp_1},x_1) \to (S^{\sharp_2},x_2)$ is a
      formal quasi-isogeny $A_1 \dashrightarrow A_2$ which is $\g$-structure preserving.
\end{itemize}
\end{Def}
This groupoid is $0$-truncated by \cite[Lemma 5.2.6]{DvHKZIgusaStacks} and we let $\igs\gx$ be the v-sheafification of its set isomorphism classes (which does not depend on $\Xi$ by \cite[Proposition 7.3.1]{DvHKZIgusaStacks}). There is a natural map $ \scrs_{K_p}\gx^{\diamond,\pre} \to \igspre_{\Xi}\gx$, and it follows from \cite[Theorem 6.0.1]{DvHKZIgusaStacks} that the following $2$-commutative diagram
\begin{equation} \label{Eq:TheDiagram} \begin{tikzcd} 
    \scrs_{K_p}\gx^{\diamond} \ar[r,"{\pi_\mathrm{crys}}"] \ar[d] & \shtgmuone
    \ar[d,"{\mathrm{BL}^\circ}"] \\ \igs\gx
    \arrow[r,"{\overline{\pi}_\mathrm{HT}}"] & \bun_G
  \end{tikzcd} \end{equation}
of small v-stacks on $\perf$ is 2-Cartesian.

\begin{Rem}
    Note that the definition of the morphisms in the category $\igspre_{\Xi}\gx(R,R^+)$ is quite subtle, in particular the untilt is allowed to change between the source and the target.
\end{Rem}

\subsubsection{} Choose $\Xi=(\mathcal{G},\iota, V_{\zp})$ as before. We now prove some extra properties of the Igusa stack under the assumption that $\mathcal{G}$ is reductive. Let $(R,R^+) \in \perf$ and let $(R^{\sharp_1}, R^{\sharp_1+}), (R^{\sharp_2}, R^{\sharp_2+})$ be two untilts of $(R, R^+)$. Let $x: \spf(R^{\sharp_1+}) \to \scrshat\gx, y: \spf(R^{\sharp_2+}) \to \scrshat\gx$ be two points of $\scrshat_{K_p}\gx$.
\begin{Lem}\label{Lem:IgusaStackIsNotTHATStrange}
Suppose that there is a pseudo-uniformizer $\varpi \in R^+$ such that the induced maps $x_0: \spec(R^+/\varpi) \to \spf(R^{\sharp_1+}) \to \scrshat\gx$, $y_0: \spec(R^+/\varpi) \to \spf(R^{\sharp_2+}) \to \scrshat\gx$ agree. If $\mathcal{G}$ is reductive, then there exists a canonical isomorphism $x \cong y$ as sections of $\igspre_{\Xi}\gx(R,R^+)$. 
\end{Lem}
\begin{Rem}
    At first blush, this may seem a bit of a tautology. The subtlety here is that it is not clear \emph{a priori} that the identity map $A_x \otimes R^+/\varpi \to A_y \otimes R^+/\varpi$ is $\g$-structure preserving at $p$. We check that this is in fact true.
\end{Rem}
\begin{proof}[Proof of Lemma \ref{Lem:IgusaStackIsNotTHATStrange}]
    It suffices to show that the identity map $A_x \otimes R^+/\varpi \to A_y \otimes R^+/\varpi$ is $\g$-structure preserving at $p$, since it is clearly $\g$-structure preserving away from $p$. For this we will use the $F$-isocrystal with $G$-structure over $\scrshg$ constructed as in Section \ref{subsub:IsocrystalConstruction} from the prismatic $F$-crystal with $\mathcal{G}$-structure of Imai--Kato--Youcis (see Theorem \ref{Thm:ImaiKatoYoucis}). \smallskip
    
    Since $x_0=y_0$, the identity map induces an isomorphism of $F$-isocrystals with $G$-structure over $R^+/\varpi$, and thus by Lemma \ref{Lem:PrismaticCompatibilityI}, it induces an isomorphism of $G$-bundles on the Fargues--Fontaine curve $X_{(R,R^+)}$. In other words, the identity map is $\g$-structure preserving at $p$.
\end{proof}
\begin{Lem} \label{Lem:TrivialUntilt}
If $\mathcal{G}$ is reductive, then the natural map $\scrs_{K_p}\gx^{\diamond}_{k_{E}} \to \scrs_{K_p}\gx^{\diamond} \to \igs \gx$ is v-surjective.
\end{Lem}
\begin{proof}
It suffices to prove this for the presheaf versions. We need to show that given $(R,R^+)$ and $x:\spf R^{\sharp+} \to \scrshat_{K_p}\gx$, that $x$ is isomorphic to some $y:\spf R^{+} \to \scrshat_{K_p}\gx$. Since the formal scheme is formally smooth because $\mathcal{G}$ is reductive, we may lift $x \otimes R^+ / \varpi$ to a point $y:\spf R^{+} \to \scrshat_{K_p}\gx$. By Lemma \ref{Lem:IgusaStackIsNotTHATStrange}, we see that $x$ is isomorphic to $y$ in $\igspre_{\Xi}\gx(R,R^+)$. 
\end{proof}

\subsubsection{} Let $(R^{\sharp_1}, R^{\sharp_1+}), (R^{\sharp_2}, R^{\sharp_2+})$ be two untilts of $(R, R^+)$, and let $x: \spf(R^{\sharp_1+}) \to \scrshat\gx, y: \spf(R^{\sharp_2+}) \to \scrshat\gx$ be two points of $\scrshat\gx$. 

\begin{Lem} \label{Lem:IgusaStackIsStrange}
If $x$ and $y$ are equal when restricted to $\spec R^+_{\red}$, then $x$ and $y$ are isomorphic in $\igspre_{\Xi}\gx(R,R^+)$.
\end{Lem}
\begin{proof}
Let $x,y$ be as in the statement of the proposition, and let $x_0,y_0$ denote their restrictions to $\spec R^+_{\red}$. By Lemma \ref{Lem:IgusaStackIsNotTHATStrange}, it suffices to show that $x$ and $y$ agree when restricted to $\spec R^+/\varpi$ for some pseudo-uniformizer $\varpi$. \smallskip 

Choose a sufficiently small compact open subgroup $K^p \subset \gafp$ so that we have a pro-\'etale $K^p$-torsor $\scrs_{K_p}\gx \to \scrs_{K}\gx$. Since $R^+_{\red}=\varinjlim_{J \subset R^{\circ \circ}} R^+/J$, where $J$ runs over finitely generated ideals contained in $R^{\circ \circ}$, it follows from spreading out that the projections of $x$ and $y$ to $\scrs_{K}\gx$ agree when restricted to $\spec R^+_{J}$ for some $J$. \smallskip 

Now choose a pseudo-uniformizer $\varpi$. After replacing $J$ by the ideal generated by $J$ and $\varpi$, we may assume $\varpi \in J$. Since the topology on $R^+$ is the $\varpi$-adic topology, we see that every element of $J$ is nilpotent modulo $\varpi$ implying that $J \subset (\varpi^n)$ for some $n$. This shows that the kernel of $R^+/J \to R^+_{\red}$ is generated by $p$-th power roots of $\varpi$, so that $\spec R^+_{\red} \to \spec R^+/J$ is a universal homeomorphism. Since $\scrs_{K_p}\gx \to \scrs_{K}\gx$ is pro-\'etale, it in fact follows that $x$ and $y$ agree when restricted to $\spec R^+/{J}$. In particular, $x$ and $y$ agree when restricted to $\spec R^+/\varpi^n$ for some $n$, and the lemma follows from Lemma \ref{Lem:IgusaStackIsNotTHATStrange} since $\varpi^n$ is a pseudo-uniformizer.
\end{proof}

\subsubsection{} We define $\igspree_{\Xi}\gx \subset \igspre_{\Xi}\gx$ to be the presheaf of groupoids consisting of points $x:\spf R^{\sharp+} \to \scrshat_{K}\gx$ where $R^{\sharp}=R$. By Lemma \ref{Lem:TrivialUntilt} and its proof, the natural map $\igspree_{\Xi}\gx \to \igspre_{\Xi}\gx$ is an equivalence of presheaves of groupoids if $\mathcal{G}$ is reductive. We let $\igse_{\Xi}\gx$ be the v-sheafification of $\igspree_{\Xi}\gx$; this is per definition a quotient of $\sh_{K^p}\gx^{\diamond}=\scrshat_{K}\gx^{\diamond}_{k_{E}}$. 

\subsection{The Igusa stack and its reduction} \label{Sub:ReductionIgusaStack} The goal of this section is to recall the reduction $\igs\gx^{\red}$ of \cite[Theorem 6.5.1]{DvHKZIgusaStacks} and to prove Theorem \ref{Thm:IgusaDCircII}, which shows that $\igs \gx$ can be reconstructed from $\igs\gx^{\red}$ under some assumptions.

\subsubsection{} Choose $\Xi=(\mathcal{G}, \iota, V_{\zp})$ as before. We will write $\shginf$ for the perfect special fiber of $\scrs_{K_p}\gx$, which is naturally identified with $(\scrs_{K_p}\gx^{\diamond}_{k_E})^{\red}=(\scrs_{K_p}\gx^{\diamond})^{\red}$ (see Lemma \ref{Lem:ReductionOfScheme}). The reduction of the morphism $\pi_{\mathrm{crys}}:\scrs_{K_p}\gx^{\diamond} \to \shtgmu$ gives rise to a morphism $\pi_{\mathrm{crys}}^{\red}:\shginf \to \shtlocgmu$ (see Lemma \ref{Lem:ReductionOfScheme} and Section \ref{subsub:ReductionShtukaBung}). Its
composition with the reduction of $\mathrm{BL}^{\circ}:\shtgmu \to \bun_G$ gives rise to a morphism $\shginf \to \gisoc$, see Section \ref{subsub:ReductionShtukaBung}. 

\subsubsection{} For $x:\spec R\to \shginf$ we will denote $\mathbb{L}_{\mathrm{crys},x}$ the induced $G$-isocrystal over $\spec R$. We similarly let $\mathscr{E}(A_{x})$ be the $\mathrm{GL}_V$-isocrystal corresponding to Dieudonn\'e-module $\mathbb{D}^{\natural}(A_x[p^{\infty}])$ of the pullback along $i:\shginf \to \shgvinf$ of the universal abelian scheme up to prime-to-$p$ isogeny $A_{\Xi}$. It follows from the discussion in  \cite[Section 5.1.10]{DvHKZIgusaStacks} that there is a natural isomorphism
\begin{align}
    \mathbb{L}_{\mathrm{crys},x} \times^{G} \mathrm{GL}_V \simeq \mathscr{E}(A_{x}).
\end{align}
We will use the induced maps $\mathbb{L}_{\mathrm{crys},x} \to \mathscr{E}(A_{x})$ below.
\begin{Def} \label{defn:FormalQIsogStructurePerf}
Given $(x_1,x_2):\spec R \to \shginf \times \shginf$, we say that a quasi-isogeny $f:A_{x_1} \dashrightarrow A_{x_2}$ is $\g$\emph{-structure preserving away from $p$} if the induced map on the prime-to-$p$ adelic Tate modules fits into a commutative square
    \[
        \begin{tikzcd}
            &\mathcal{V}^p(A_{x_{1}}) \arrow{r}{\mathcal{V}^p(f)}\arrow{d}{\eta_{x_1}}& \mathcal{V}^p(A_{x_{2}}) \arrow{d}{\eta_{x_2}}\\
            &V \otimes \afp \arrow[r, equals]& V \otimes \afp.
        \end{tikzcd}
    \]
\end{Def}
\begin{Def}\label{defn:CondnAtPPerf}
Given $(x_1,x_2):\spec R \to \shginf \times \shginf$, we say that a quasi-isogeny $f:A_{x_1} \dashrightarrow A_{x_2}$ is \emph{$\g$-structure preserving at $p$} if the induced isomorphism (see \cite[Section 2.5.3]{DvHKZIgusaStacks})
      \[
        \mathscr{E}(A_{x_{1}}) \to \mathscr{E}(A_{x_{2}})
      \]
      of $\mathrm{GL}_V$-isocrystals on $\spec R$ is induced by a (necessarily unique) isomorphism of $G$-isocrystals fitting in a commutative diagram
 \begin{equation}
     \begin{tikzcd}
     \mathbb{L}_{\mathrm{crys},{x_1}} \arrow{d} \arrow{r} & \mathbb{L}_{\mathrm{crys},{x_2}} \arrow{d} \\
         \mathscr{E}(A_{x_{1}}) \arrow{r} &  \mathscr{E}(A_{x_{2}}).
     \end{tikzcd}
 \end{equation}
We say that $f$ is $\g$-structure preserving if $f$ is $\g$-structure preserving at $p$ and $\g$-structure preserving away from $p$. 
\end{Def}

\subsubsection{} We define $\igspreperf_{\Xi}\gx$ to be the presheaf of groupoids on $\affperfk$ sending $R$ to the following groupoid: Its objects are morphisms $x:\spec R \to \shginf$ and its morphisms $x_1 \to x_2$ are $G$-structure preserving quasi-isogenies $g:A_1 \dashrightarrow A_2$. We let $\igsperf \gx$ be its v-sheafification. By \cite[Theorem 6.5.1]{DvHKZIgusaStacks}, the reduction of the Igusa stack morphism gives a v-quotient
\begin{align}
    \shginf \to \igs \gx ^{\red}
\end{align}
which agrees with the v-quotient given by $\igsperf \gx$. 

\subsubsection{} We now begin our comparison of the Igusa stack with its reduction. \textbf{Assume for the rest of this section that $\mathcal{G}$ is reductive} (although we will restate these assumptions when stating results). We write $\igsprecc_{\Xi}\gx$ for $(\operatorname{Igs}^{\pre, \mathrm{perf}}_{\Xi}\gx)^{\diamond/\circ, \pre}$ and $\igscc_{\Xi}\gx$ for its v-sheafification $(\operatorname{Igs}^{\pre, \mathrm{perf}}_{\Xi}\gx)^{\diamond/\circ}$.

\subsubsection{} We want to define a morphism $\igspree_{\Xi}\gx \to \igsprecc_{\Xi}\gx$, which will require the following lemma.

\begin{Lem} \label{Lem:ReductionGStructure}
Let $(R,R^+) \in \perf$. Let $x_{1,2}:\spf R^+ \to \scrshat_{K_p}\gx$ and let $f:A_{x_1} \dashrightarrow A_{x_2} $ be a formal quasi-isogeny. Assume that $\mathcal{G}$ is reductive. If $f$ is $\g$-structure preserving, then so is its reduction $f_{\red}:A_{x_1} \otimes R^+_{\red} \dashrightarrow A_{x_2} \otimes R^+_{\red}$
\end{Lem}
\begin{proof}
That $f_{\red}$ is $\g$-structure preserving away from $p$ follows from the fact that $\spec R^+_{\red} \to \spec R^+/\varpi$ is a universal homeomorphism, so it remains to check that $f_{\red}$ is $\g$-structure preserving at $p$. There is a closed subset of $\spec R^+_{\red}$ where $f_{\red}$ is $\g$-structure preserving at $p$, see Lemma \ref{Lem:GStructurePreservingClosedPerfect}. Since $\spec(R^\circ_{\red})$ is Zariski dense in $\spec(R^+_{\red})$ (because $R^+_{\red} \to R^\circ_{\red}$ is injective), we may thus assume that $R^{+}=R^{\circ}$. \smallskip 

We once again make use of the prismatic $F$-crystal with $\mathcal{G}$-structure of Theorem \ref{Thm:ImaiKatoYoucis}, using the assumption that $\mathcal{G}$ is reductive. By Lemma \ref{Lem:PrismaticCompatibilityI}, the morphism $\scrshat_{K_p}\gx^{\lozenge} \to \bun_{G}$ is induced from an $F$-isocrystal with $G$-structure on $\scrs_{K_p}\gx_{k_{E}}$. Since we have assumed $R^{+}=R^{\circ}$, the functor sending an $F$-isocrystal on $\spec R^+/\varpi$ to a vector bundle on $X_{(R,R^+)}$ is fully faithful, see \cite[discussion after Corollary 6.3]{FarguesConjecture}. Thus since $f$ is $\g$-structure preserving, it follows that $f$ induces an isomorphism of $F$-crystals with $G$-structure over $\spec R^+/\varpi$. It follows that $f_{\red}$ induces an isomorphism of $F$-crystal with $G$-structure over $\spec R^+_{\red}$. By Lemma \ref{Lem:ObviousHopefully}, this implies that $f_{\red}$ is $\g$-structure preserving at $p$. 
\end{proof}

\begin{Construction} \label{Constr:ReductionMorphism}
Consider the morphism $\rho_{\Xi}:\igspree_{\Xi}\gx \to \igsprecc_{\Xi}\gx$ defined on $(R,R^+)$ points by: On objects we send $x:\spf R^+ \to \scrshat_{K_p}\gx$ to the induced morphism $\spec R^+_{\red} \to \shginf$ by restriction to $\spec R^+_{\red}$, and on morphisms it takes a formal quasi-isogeny $A_x \dashrightarrow A_y$ to its restriction to $\spec R^+_{\red}$; this is well defined by Lemma \ref{Lem:ReductionGStructure}. 
\end{Construction}
\begin{Lem} \label{Lem:ReductionMorphismDiagramCommutative}
    The following diagram is $2$-commutative
    \begin{equation}
        \begin{tikzcd}
            \shginf^{\diamond, \pre} \arrow{d} \arrow{r} & \shginf^{\diamond/\circ, \pre} \arrow{d} \\
            \igspree_{\Xi}\gx \arrow{r}{\rho_{\Xi}} & \igsprecc_{\Xi}\gx.
        \end{tikzcd}
    \end{equation}
\end{Lem}
\begin{proof}
    This is a straightforward consequence of Construction \ref{Constr:ReductionMorphism}, the point being that the map $\shginf^{\diamond, \pre} \to \shginf^{\diamond/\circ, \pre}$ is the map sending a $\spf R^+$-point $x$ of $\scrshat_{K_p}\gx$ to the restriction of $x$ to $\spec R^+_{\red}$. Lemma \ref{Lem:ReductionGStructure} is precisely the statement that this morphism is compatible with the ($0$-truncated) groupoids defining $\igspree_{\Xi}\gx$ and $\igsprecc_{\Xi}\gx$.
\end{proof}

\begin{Thm} \label{Thm:IgusaDCirc}
If $\mathcal{G}$ is reductive, then $\rho_{\Xi}:\igspree_{\Xi}\gx \to \igsprecc_{\Xi}\gx$ is an isomorphism. 
\end{Thm}

\begin{Lem} \label{Lem:ReductionShimuraSurjective}
If $\mathcal{G}$ is reductive, then the morphism $\shginf^{\diamond, \pre} \to \shginf^{\diamond/\circ, \pre}$ is surjective. 
\end{Lem}
\begin{proof}
Let $(R,R^+) \in \perf$ and let $x_0:\spec R^+_{\red} \to \shginf$ be a morphism. By Lemma \ref{Lem:IgusaStackIsStrange}, it suffices to show that $x_0$ can be lifted to a morphism $x:\spf R^+/\varpi \to \scrs_{K_p}\gx$. By spreading out as in the proof of Lemma \ref{Lem:IgusaStackIsNotTHATStrange}, we see that $x_0$ lifts to a morphism $\spf R^+/J \to \scrshat_{K_p}\gx$ for some finitely generated ideal $J \subset (\varpi^n)$ for some $n$. Using the formal smoothness as in the proof of Lemma \ref{Lem:TrivialUntilt}, it moreover lifts to a morphism $\spf R^+ \to \scrshat_{K_p}\gx$.
\end{proof}

\subsubsection{} Let $(R,R^+) \in \perf$ and let $s$ be a section of the morphism of rings $R^+ \to R^+_{\red}$ (in particular we are assuming such a section exists). 
\begin{Lem} \label{Lem:SectionGStructure}
    Let $x_0,y_0:\spec R^+_{\red} \to \shginf$ and let $f_0:A_{x_0} \dashrightarrow A_{y_0}$ be a $\g$-structure preserving quasi-isogeny. Consider $x=x_0 \circ s$ and $y=y_0 \circ s$ as morphism $\spf R^+ \to \scrshat_{K_p}\gx$ and let $\tilde{f}:A_{x} \dashrightarrow A_{y}$ be the quasi-isogeny induced by $f_0$. For a pseudo-uniformizer $\varpi \in R^+$ we write $f$ for the formal quasi-isogeny induced by restricting $\tilde{f}$ to $R^+/\varpi$. If $\mathcal{G}$ is reductive, then $f$ is $\g$-structure preserving. 
\end{Lem}
\begin{proof}
That $f$ is $\g$-structure preserving away from $p$ follows from the fact that $\spec R^+_{\red} \to \spec R^+/\varpi$ is a universal homeomorphism, so it remains to check that $f$ is $\g$-structure preserving at $p$. 

We once again make use of the prismatic $F$-crystal with $\mathcal{G}$-structure of Theorem \ref{Thm:ImaiKatoYoucis}, using the assumption that $\mathcal{G}$ is reductive. By Lemma \ref{Lem:PrismaticCompatibilityI}, the morphism $\scrshat_{K_p}\gx^{\lozenge} \to \bun_{G}$ is induced from an $F$-isocrystal with $G$-structure on $\scrs_{K_p}\gx_{k_{E}}$. Since $f_0$ is $\g$-structure preserving (i.e., induces an isomorphism of $F$-isocrystals with $G$-structure), it follows from Lemma \ref{Lem:ObviousHopefully} that $\tilde{f}$ induces an isomorphism of $F$-isocrystals with $G$-structure over $R^+$. It follows that $f$ induces an isomorphism of $F$-isocrystals with $G$-structure over $R^+/\varpi$, and thus $f$ is $\g$-structure preserving by Lemma \ref{Lem:PrismaticCompatibilityI}. 
\end{proof}
Let $(R,R^+) \in \perf$ and let $s$ be a section of the morphism of rings $R^+ \to R^+_{\red}$. 
\begin{Construction}
Assume that $\mathcal{G}$ is reductive. Then there is a functor $\sigma_{\Xi,s}:\igsprecc_{\Xi}\gx(R,R^+) \to \igspree_{\Xi}\gx(R,R^+)$ given by: On objects sending $x_0:\spec R^+_{\red} \to \shginf$ to $x=x_0 \circ s:\spf R^+ \to \scrshat_{K_p}\gx$, on morphisms sending $f_0:A_{x_0} \dashrightarrow A_{y_0}$ to the restriction of $\tilde{f}$ to $R^+/\varpi$; this is well defined by Lemma \ref{Lem:SectionGStructure}. 
\end{Construction}

\begin{Prop} \label{Prop:IgusaStackIsStrangeII}
Let $(R,R^+) \in \perf$ and assume that $\mathcal{G}$ is reductive. If the morphism of rings $R^+ \to R^+_{\red}$ admits a section $s$, then $\rho_{\Xi}(R,R^+)$ is an equivalence of categories with quasi-inverse $\sigma_{\Xi,s}$.
\end{Prop}
\begin{proof}
Take $(R,R^+) \in \perf$. It suffices to show that $\rho_{\Xi}(R,R+) \circ \sigma_{\Xi,s}$ and $\sigma_{\Xi,s} \circ \rho_{\Xi}(R,R^+)$ are naturally isomorphic to the identity. The former composition is the identity on objects (the restriction of $x_0 \circ s$ to $\spec R^+_{\red}$ equals $x_0$). The latter composition is naturally isomorphic to the identity by Lemma \ref{Lem:IgusaStackIsStrange}, which says that $x$ is uniquely isomorphic to $x_0 \circ s$ in $\igspree_{\Xi}\gx(R,R^+)$.
\end{proof}

\subsubsection{} We now prove fully faithfulness of $\rho_{\Xi}$, for which we start with the following simple observation of Heuer, see \cite{HeuerProEtaleUniformisation}.
\begin{Lem} \label{Lem:SpreadingOut}
Let $(R,R^+)$ be a perfectoid Huber pair, let $\varpi \in R^+$ be a pseudo-uniformizer. For abelian schemes $A,B$ over $\spec R^+/\varpi$ with reductions $A_0$ and $B_0$ over $R^+_{\red}$, the natural map
\begin{align}
    \operatorname{Hom}_{R^+/\varpi}(A,B) \otimes_{\mathbb{Z}} \mathbb{Q} \to \operatorname{Hom}_{R^+_{\red}}(A_0,B_0) \otimes_{\mathbb{Z}} \mathbb{Q}
\end{align}
is a bijection. 
\end{Lem}
\begin{proof}
It is injective (even before tensoring by $\mathbb{Q}$) by \cite[Proposition 6.1]{MumfordGIT}. To prove surjectivity, let $f \in \operatorname{Hom}_{R^+_{\red}}(A_0,B_0) \otimes_{\mathbb{Z}} \mathbb{Q}$. Since $R^+_{\red}=\varinjlim_{J \subset R^{\circ \circ}} R^+/J$, where $J$ runs over finitely generated ideals contained in $R^{\circ \circ}$, we can use spreading out to lift $f$ to a quasi-isogeny over $R^+/J$. After replacing $J$ by the ideal generated by $J$ and $\varpi$, we may assume $\varpi \in J$. Then since $J/\varpi$ is a nilpotent ideal of $R^+/\varpi$, we can lift $f$ to a quasi-isogeny over $R^+/\varpi$ by the rigidity of quasi-isogenies \cite[Lemma~1.1.3]{KatzSerreTate}.
\end{proof}

\begin{Prop} \label{Prop:Fullyfaithful}
If $\mathcal{G}$ is reductive, then the morphism $\rho_{\Xi}:\igspree_{\Xi}\gx \to \igsprecc_{\Xi}\gx$ of presheaves of groupoids is fully faithful.
\end{Prop}
\begin{proof}
The faithfulness follows from Lemma \ref{Lem:SpreadingOut}. For fullness, we take $(R,R^+) \in \perf$, two morphisms $x_1,x_2:\spf R^+ \to \scrshat_{K_p} \gx$ and a $\g$-structure preserving quasi-isogeny $f_0:A_{x_1, R^+_{\red}} \dashrightarrow A_{x_2, R^+_{\red}}$. By Lemma \ref{Lem:SpreadingOut}, we know that $f_0$ lifts uniquely to a quasi-isogeny $f:A_{x_1, R^+/\varpi} \dashrightarrow A_{x_2, R^+/\varpi}$, and it suffices to show that $f$ is $\g$-structure preserving at $p$. By Lemma \ref{Lem:GStructurePreservingClosedPerfectoid}, there is a closed subset $Z \subset \spd(R,R^+)$ over which $f$ is $\g$-structure preserving. To show that $Z$ equals all of $\spd(R,R^+)$, it suffices to show that $Z$ contains all rank one geometric points of $|\spd(R,R^+)|$ (e.g. by \cite[Lemma 11.11]{EtCohDiam} and the fact that closed immersions are partially proper). In other words, we have to show that $f$ is $\g$-structure preserving at $p$ in the special case that $(R,R^+)=(C, \mathcal{O}_C)$. In this case, the morphism $\mathcal{O}_C \to k=\mathcal{O}_C^{\red}$ admits a section by \cite[Lemma 6.2.1]{DvHKZIgusaStacks}. We may thus apply Proposition \ref{Prop:IgusaStackIsStrangeII} to conclude that $\rho_{\Xi}$ is fully faithful, which by the uniqueness of $f$ implies that $f$ must be $\g$-structure preserving at $p$. 
\end{proof}
We can now finish the proof of Theorem \ref{Thm:IgusaDCirc}. 
\begin{proof}[Proof of Theorem \ref{Thm:IgusaDCirc}]
The morphism $\igspree_{\Xi}\gx \to \igsprecc_{\Xi}\gx$ is fully faithful by Proposition \ref{Prop:Fullyfaithful}, thus it suffices to show it is also essentially surjective, which follows from Lemma \ref{Lem:ReductionShimuraSurjective}.
\end{proof}

\subsubsection{} We have the following corollary of Theorem \ref{Thm:IgusaDCirc}.\footnote{If Question \ref{Question:VtopologyLemmaReduction} had a positive answer, it would be an immediate corollary.}
\begin{Thm} \label{Thm:IgusaDCircII}
If $\mathcal{G}$ is reductive, then there is a commutative diagram
\begin{equation}
    \begin{tikzcd}
        \shginf^{\diamond} \arrow{r} \arrow{d} & \shginf^{\diamond/\circ} \arrow{d} \\
        \igse\gx  \arrow[r] & (\igsperf\gx)^{\diamond/\circ},
    \end{tikzcd}
\end{equation}
where the bottom horizontal arrow is an isomorphism.
\end{Thm}
The following proposition is the essential ingredient in the proof.
\begin{Prop} \label{Prop:DiamondCircVSurjectiveSpecialCase}
    The natural map $\shginf^{\diamond/\circ} \to (\igsperf)^{\diamond/\circ}$ is v-surjective.
\end{Prop}
\begin{proof}
Using the Cartesian diagram
\begin{equation}
    \begin{tikzcd}
        \shginf^{\diamond/\circ} \arrow{r} \arrow{d} & (\shtlocgmuone)^{\diamond/\circ} \arrow{d} \\
        (\igsperf\gx)^{\diamond/\circ} \arrow{r} & (\gisocmu)^{\diamond/\circ},
    \end{tikzcd}
\end{equation}
it suffices to show that $(\shtlocgmuone)^{\diamond/\circ} \to (\gisocmu)^{\diamond/\circ}$ is v-surjective. If $\mathcal{G}^{\circ}$ is the identity component of $\mathcal{G}$, then we have a natural map $\operatorname{Sht}^{\mathrm{W}}_{\mathcal{G}^{\circ}, \mu} \to \shtlocgmuone$. We may therefore assume that $\mathcal{G}$ is parahoric, and then the result is Proposition \ref{Prop:VSurjectiveNewtonDiamondCirc}. 
\end{proof}

\begin{proof}[Proof of Theorem \ref{Thm:IgusaDCircII}]
Write $\mathcal{R}^{\mathrm{perf}}=\shginf \times_{\igspreperf \gx} \shginf$. This is in fact a sheaf; this follows from the ind-representability of the sheaf of quasi-isogenies between two abelian schemes.\footnote{Note that the condition that a quasi-isogeny is $\g$-structure preserving is representable in closed immersions, see Lemma \ref{Lem:GStructurePreservingClosedPerfect} for the representability in closed immersions of being $\g$-structure preserving at $p$; the representability in closed immersions of being $\g$-structure preserving away from $p$ follows from the fact that $\ul{\gafp} \to \ul{\g_V(\afp)}$ is a closed immersion.} Therefore, $\mathcal{R}^{\mathrm{perf}}$ is also equal to (since sheafification preserves fiber products)
\begin{align}
    \shginf \times_{\igsperf \gx} \shginf.
\end{align}
Write $\mathcal{R}^{\mathrm{pre}}=\shginf^{\diamond,\pre} \times_{\igspree_{\Xi}\gx} \shginf^{\diamond,\pre}$. Then the following diagram
\begin{equation}
    \begin{tikzcd}
        \mathcal{R}^{\mathrm{pre}} \arrow{r} \arrow{d} & \left(\mathcal{R}^{\mathrm{perf}}\right)^{\dcirc, \pre} \arrow{d} \\
        \shginf^{\diamond,\pre} \arrow{r} \arrow{d} & \shginf^{\diamond/\circ, \pre} \arrow{d} \\
        \igspree_{\Xi}\gx  \arrow{r}{\rho_{\Xi}} & (\igspreperf_{\Xi}\gx)^{\diamond/\circ,\pre} \arrow{r} & \igsperf\gx^{\diamond/\circ, \pre}.
    \end{tikzcd}
\end{equation}
is commutative by Lemma \ref{Lem:ReductionMorphismDiagramCommutative}, and the bottom left horizontal morphism is an isomorphism by Theorem \ref{Thm:IgusaDCirc}. The top diagram is moreover Cartesian; this follows from the bottom left horizontal map being an isomorphism. We moreover observe that the rightmost vertical arrow is v-surjective by definition. 

Applying v-sheafification, we get the following commutative diagram
\begin{equation}
    \begin{tikzcd}
        \mathcal{R} \arrow{r} \arrow{d} & \left(\mathcal{R}^{\mathrm{perf}}\right)^{\dcirc} \arrow{d} \\
        \shginf^{\diamond} \arrow{r} \arrow{d} & \shginf^{\diamond/\circ} \arrow{d} \\
        \igs\gx  \arrow{r}{\rho_{\Xi}} & \igspreperf_{\Xi}\gx^{\diamond/\circ} \arrow{r} & \igsperf\gx^{\diamond/\circ}, 
    \end{tikzcd}
\end{equation}
in which the top square is still Cartesian. The natural map
\begin{align}
    \shginf^{\diamond/\circ} \times_{\igspreperf_{\Xi}\gx^{\diamond/\circ}} \shginf^{\diamond/\circ} \to \shginf^{\diamond/\circ} \times_{\igsperf\gx^{\diamond/\circ}} \shginf^{\diamond/\circ}
\end{align}
is an isomorphism since this is true before applying $\diamond/\circ$ and since sheafification commutes with fiber products. Thus
\begin{align}
    \igspreperf_{\Xi}\gx)^{\diamond/\circ} \to \igsperf\gx^{\diamond/\circ}
\end{align}
is injective, and it remains to prove it is also v-surjective. This follows directly from Proposition \ref{Prop:DiamondCircVSurjectiveSpecialCase} and Lemma \ref{Lem:ReductionShimuraSurjective}.
\end{proof}

\subsubsection{} We collect the following result for later use. Let $Z \subset B(G,-\mu)$ be a closed subset defining a closed subfunctor $\igsperf\gx_{Z} \subset \igsperf\gx$ and an open subfunctor $\igs \gx_{Z} \subset \igs \gx$. Recall that Theorem \ref{Thm:IgusaDCircII} gives us a natural isomorphism $\igsperf\gx^{\dcirc} \to \igs\gx$. 
\begin{Prop} \label{Prop:NewtonStrata}
    The natural map $(\igsperf\gx_{Z})^{\dcirc} \to \igsperf\gx^{\dcirc} \to \igs\gx$ induces an isomorphism between $(\igsperf\gx_{Z})^{\dcirc}$ and $\igs\gx_{Z}$.
\end{Prop}
In other words, applying $\dcirc$ to the schematic (closed union of) Newton strata $\igsperf\gx_{Z}$ gives the analytic (open union of) Newton strata $\igs\gx_{Z}$. 
\begin{proof}[Proof of Proposition \ref{Prop:NewtonStrata}]
It follows from Lemma \ref{Lem:ClosedToOpen} that $$(\igsperf\gx_{Z})^{\dcirc} \to \igsperf\gx^{\dcirc} \isom \igs \gx$$ is representable in open immersions, and the same is true for $\igs \gx_{Z} \to \igs \gx$. So we have two open subfunctors that we are trying to check are equal, and we can do this after pullback via the v-cover $\sh_{K_p}\gx^{\diamond} \to \igs \gx$ 
of Lemma \ref{Lem:TrivialUntilt}. We can moreover pass to an open cover of $\sh_{K_p}\gx^{\diamond}$ given by affine opens, and the result then follows from Lemma \ref{Lem:NewtonStrata} and descent.
\end{proof}

\section{Correspondences} In Section \ref{Sub:Correspondences} we define a subsheaf $\igscorr$ of $\igs \gx \times \igs \gxp$ that is conjecturally the graph of the exotic isomorphism of Conjecture \ref{Conj:IntroMain}. This depends a priori on a large number of choices (captured by the subscript $\Omega$), and we check that it is in fact functorial in Section \ref{sub:NotationFunctoriality}. In Section \ref{sub:PerfectCorrespondence} we define a similar subsheaf $\igscorrperf$ of the product of the perfect Igusa stacks, and in Section \ref{Sub:ReductionCorrespondence}, we verify that $(\igscorr)^{\mathrm{red}}=\igscorrperf$.

\subsection{Correspondences}\label{Sub:Correspondences} Here we construct a correspondence between the Igusa stacks associated to two different Shimura data, see Definition \ref{defn:Corr}. We conjecture that this correspondence is the graph of an isomorphism between two open substacks, see Conjecture \ref{Conj:ConstructionWorks}. 

\subsubsection{} Let $\g$ be a smooth group scheme over an affine scheme $\spec A$, and write $\operatorname{Rep}_{A} \g$ for the category of representations of $\g$ on finite projective $A$-modules. Let $\mathsf{P}$ be an \'etale $\g$-torsor and let $\g'=\operatorname{Aut}_{\g}(\mathsf{P})$ be the corresponding pure inner form. 
\begin{Construction}
Given $V \in \operatorname{Rep}_{A} \g$ we consider $V'=V \times^{\g} \mathsf{P}$, on which $\g'$ naturally acts on the left.\footnote{Here we are conflating the $A$-module $V$ with the sheaf $B \mapsto V \otimes_{A} B$, which is represented by the total space of the vector bundle corresponding to $V$.} The addition on $V$ induces a binary operation on $V'$ with identity element given by $0 \times^{\g} \mathsf{P} \to V'$.\footnote{The map $V \times V \times \mathsf{P} \times \mathsf{P} = V \times V \times \mathsf{P} \times \g \to V \times \mathsf{P}$ given by $(v,w,p,gp) \mapsto (g^{-1} v+w,p)$ is $\g \times \g$-equivariant via projection onto the second factor, and thus induces a binary operation $V' \times V' \to V'$.}
\end{Construction}

\begin{Lem} \label{Lem:TwistingRepresentations}
For $V \in \operatorname{Rep}_{A} \g$ as above, the sheaf $V'$ defined above is representable by a finite projective $A$-module, and the functor $\beta_{\mathsf{P}}:V \mapsto V \times^{\g} \mathsf{P}$ (equipped with its left action of $\g'$) induces an equivalence
    \begin{align}
        \operatorname{Rep}_{A} \g \to \operatorname{Rep}_{A} \g'
    \end{align}
    of $A$-linear tensor categories. 
\end{Lem}
\begin{proof}
The functor becomes isomorphic to the identity functor over an \'etale cover of $A$ trivializing $\mathsf{P}$, and the lemma thus follows from \'etale descent (for projective modules on the stack $[\spec A /\g]$, see \cite[Tag 023S]{stacks-project}).
\end{proof}
\subsubsection{} \label{subsub:TwistSymplectic}
Now let $(V, \psi)$ be a symplectic vector space over $\f$ with similitude group $\g_{V}=\g_{V, \psi}$, let $\mathsf{P}$ be a $\g_{V}$-torsor over $\f$ and let $\g'$ be as above. Then $V'=\beta_{\mathsf{P}}(V)$ is a vector space over $\f$ equipped with a perfect alternating pairing 
\begin{align}
    \psi':V' \otimes V' \to \beta_{\mathsf{P}}(\f(1)),
\end{align}
where $\f(1)$ is the one-dimensional representation of $\g_{V}$ where $\g_{V}$ acts on $\f$ through the similitude character $\operatorname{sim}:\g_{V} \to \mathbb{G}_m$. There is a similitude group $\g_{V'}=\g_{V', \psi'}$ of automorphisms of $V'$ which commute with $\psi'$ up to a scalar. If we choose an isomorphism $\beta_{\mathsf{P}}(\f(1)) \xrightarrow{\sim} \f$ or equivalently an isomorphism $\mathsf{P} \times^{\g_{V}} \mathbb{G}_m \to \mathbb{G}_m$, then $(V', \psi')$ is a symplectic vector space over $\f$ and $\g_{V',\psi'}$ is its group of symplectic similitudes. This allows us to equip $\g_{V'}$ with the standard Siegel Shimura datum (which does not depend on any choices).

\subsubsection{}\label{Subsub:allthedata} In order to define our correspondences there are a number of choices to make. We first fix a prime $p$ and an isomorphism $\mathbb{C} \isom \qpbar$. We fix a Shimura datum $\gx$ of Hodge type and a $\g$-torsor $\mathsf{P}$, write $\g'=\operatorname{Aut}_{\g}(\mathsf{P})$ and let $\x'$ be a Shimura datum for $\g'$. We fix a Hodge embedding $\iota:\gx \to \gvx$, write $\mathsf{P}_V=\mathsf{P} \times^{\g, \iota} \gv$, and we write $\gvxp$ for the Siegel Shimura datum of Section \ref{subsub:TwistSymplectic}, so that there is an induced morphism $\iota':\g' \to \g_{V'}$. \smallskip 

We further choose an isomorphism $\kappa^p:\mathsf{P} \otimes \afp \to \g \otimes \afp$ (in particular assuming such an isomorphism exists), which induces an identification $\chi_{\kappa^p}:\g \otimes \afp \xrightarrow{\sim} \g' \otimes \afp$. Moreover, it induces an isomorphism $\nu_{\kappa^p}: V \otimes \afp \xrightarrow{\sim} V' \otimes \afp$ of representations of $\g \otimes \afp \xrightarrow{\chi_{\kappa^p}} \g' \otimes \afp$. We further extend $\iota$ and $\iota'$ to triples $\Xi = (\mathcal{G}, \iota, V_{\zp})$ and $\Xi'=(\mathcal{G}', \iota', V_{\zp}')$ as above (in particular we assume that this is possible). We denote these data by the notation $\Omega = (\iota, \mathcal{G}, \mathcal{G}', V_{\zp}, V_{\zp}', \kappa^p)$, which we call a \emph{Hodge datum} for $(\g,\x,\mathsf{P})$. Consider the following assumption on $\x'$.
\begin{Assump} \label{Assump:InfinityHodge}
    The morphism $\iota'$ underlies a morphism of Shimura data $\gxp \to \gvxp$.
\end{Assump}
We will call the Hodge datum $\Omega$ \emph{admissible} if Assumption \ref{Assump:InfinityHodge} holds. We will write $K_p=\mathcal{G}(\zp)$ and $K_p'=\mathcal{G}'(\zp)$. We will use the notation $P=\mathsf{P}_{\qp}$ and $P_V=\mathsf{P}_{V,\qp}$.
\begin{Rem}
By \cite[Proposition 6.3.8]{XiaoZhu2}, Assumption \ref{Assump:InfinityHodge} implies Assumption \ref{Assump:Infinity}. 
\end{Rem}

\subsubsection{} Suppose that we are given $(R,R^+) \in \perf$, untilts $R^{\sharp_i}$ for $i=1,2$ and morphisms $x:\spf R^{\sharp_1+} \to \scrshat_{K_p}\gx, y: \spf R^{\sharp_2+} \to \scrshat_{K_p'}\gxp$. We note that there is a canonical isomorphism of $G_{V'}$-bundles
\begin{align} \label{Eq:CanonicalIsomorphismVectorBundlesTwist}
     \beta_{P}(\mathbb{L}_{\crys,x}) \times^{\iota',G'} G_{V'} = \beta_{P_V}\left(\mathbb{L}_{\crys,x} \times^{\iota, G} G_{V}\right).
\end{align}
\begin{Def} \label{Def:GtoGpStructure}
We say that a formal quasi-isogeny $f:A_{x} \dashrightarrow A_{y}$ is \emph{$\Omega$-structure preserving away from $p$} if the induced map on the prime-to-$p$ adelic Tate modules
    \[
        V^p(f): V^p(A_{x} \otimes R^+/\varpi) \to V^p(A_{y} \otimes R^+/\varpi)
    \]
    fits into a commutative square
    \[
        \begin{tikzcd}
            &V^p(A_{x} \otimes R^+/\varpi) \arrow{r}{V^p(f)}\arrow{d}{\eta_{x}}& V^p(A_{y} \otimes R^+/\varpi) \arrow{d}{\eta_{y}}\\
            &V \otimes \afp \arrow{r}{\nu_{\kappa^p}}& V' \otimes \afp.
        \end{tikzcd}
    \]
    We say that $f$ is \emph{$\Omega$-structure preserving at $p$} if there exists a (necessarily unique) isomorphism of $G'$-bundles $\varphi_{x, y}: \beta_{P}(\mathbb{L}_{\crys,x}) \isom \mathbb{L}_{\crys,y}$ whose pushout along $G' \to \operatorname{GL}_{V'}$ recovers the isomorphism of $\operatorname{GL}_{V'}$-bundles
\begin{align}
    \mathcal{E}(A_x) \xrightarrow{\mathcal{E}(f)} \mathcal{E}(A_y),
\end{align}
using the identifications
\begin{align}
    \mathcal{E}(A_x) &= \beta_{P_V}\left(\mathbb{L}_{\crys,x} \times^{\iota, G} G_{V}\right) = \beta_{P}(\mathbb{L}_{\crys,x}) \times^{\iota',G'} G_{V'} \\
    \mathcal{E}(A_y)&=\mathbb{L}_{\crys,y} \times^{\iota',G'} G_{V'}.
\end{align}
We say that $f$ is \emph{$\Omega$-structure preserving} if it is both $\Omega$-structure preserving away from $p$ and $\Omega$-structure preserving at $p$.
\end{Def}
\begin{Rem}
    Given an isomorphism $\kappa_p:P \to G$ of $G$-torsors with induced isomorphism $\kappa_{p}:P_V \to G_V$, we can identify $\beta_{P}$ and $\beta_{P_V}$ with the identity map. Then $f$ is $\Omega$-structure preserving at $p$ precisely when it is $\g$-structure preserving at $p$.
\end{Rem}
\begin{Rem}
There is some asymmetry in primes $\ell \not=p$ and the prime $p$ in Definition \ref{Def:GtoGpStructure} having to do with the fact that the Igusa stack $\igs \gx$ has "infinite level away from $p$". If we instead worked with $\igs \gx / \ul{\gafp}$ then we would not have to fix $\kappa^p$ and would instead ask that $V^p(f)$ preserves certain tensors, see \cite[Section 5.1.9]{DvHKZIgusaStacks}.
\end{Rem}

\begin{defn}\label{defn:Corr}
    We define the space $\igscorrpre$ to be the presheaf in groupoids on $\perf^{op}$ which assigns to a Huber pair $(R, R^+)$ the following category: Its \textbf{objects} are quintuples $(R^{\sharp_1}, R^{\sharp_2}, x,y, f)$ where 
    \begin{align*}
        (R^{\sharp_1}, R^{\sharp_1+}, x) &\in \igspre_{\Xi} \gx\\
        (R^{\sharp_2}, R^{\sharp_2+}, y) &\in \igspre_{\Xi'} \gxp\\
        f: A_{x}  &\dashrightarrow A_{y}
    \end{align*}
    where $f$ is a formal quasi-isogeny that is $\Omega$-structure preserving. The \textbf{morphisms} $(R^{\sharp_1}, R^{\sharp_2}, x,y, f) \to (R^{\sharp_3}, R^{\sharp_4}, x',y', f')$ in $\igscorrpre(R,R^+)$ are commutative diagrams of formal quasi-isogenies 
    \[
        \begin{tikzcd}
            & A_{x}  \arrow[d, dashed, "g_1"] \arrow[r,"f", dashed] & A_{y}  \arrow[d, "g_2", dashed] \\
            & A_{x'}  \arrow[r, "f'", dashed] & A_{y'} 
        \end{tikzcd}
    \]
    such that $g_1$ is $\g$-structure preserving and $g_2$ is $\g'$-structure preserving. We define the space $\igscorr$ to be the $v$-sheafification of the presheaf of isomorphism classes of $\igscorrpre$. 
\end{defn}

We note that by the definition of $\igscorrpre$ it lies in a $2$-commutative diagram
\begin{equation} \label{Eq:GStructureDiagram}
    \begin{tikzcd}
        \igscorrpre \arrow{r} \arrow{d} & \bun_{G} \arrow{d}{(\operatorname{Id}, \beta_{\mathsf{P}})} \\
        \igspre_{\Xi}\gx \times \igspre_{\Xi'}\gxp \arrow{r} & \bun_{G} \times \bun_{G'}.
    \end{tikzcd}
\end{equation}
\begin{Lem} \label{Lem:CorrespondenceNewton}
    The morphisms $\igscorr \to \igs \gx$ and $\igscorr \to \igs \gxp$ are monomorphisms with image contained in $\igs \gx_{-\mu'}$ and $\igs \gxp_{-\mu}$ respectively. 
\end{Lem}
\begin{proof}
The map of presheaves in groupoids $\igscorrpre \to \igspre_{\Xi}\gx$ is a fully faithful functor, because in the definition of morphisms it is clear that $g_1$ uniquely determines $g_2$. By symmetry the same is true for the map $\igscorrpre \to \igspre_{\Xi'}\gxp$. The fact that the image lies in $\igs \gx_{-\mu'}$ follows from the fact that $\igspre\gx$ maps to $\bungmu$ and that $\igspre\gxp$ maps to $\bungpmup$, in combination with \eqref{Eq:GStructureDiagram}.
\end{proof}

\begin{Conj} \label{Conj:ConstructionWorks}
If $\Sha^1(\mathbb{Q},\g)=0$, then the morphisms $\igscorr \to \igs \gx_{-\mu'}$ and $\igscorr \to \igs \gxp_{-\mu}$ are isomorphisms.
\end{Conj}
Note that Conjecture \ref{Conj:ConstructionWorks} implies Conjecture \ref{Conj:Main}. 
\begin{Rem}
    It is not at all clear that $\igscorr$ is nonempty, even when $\g$ is a torus. 
\end{Rem}

\subsubsection{} Recall our identification $\chi_{\kappa^p}:\g \otimes \afp \xrightarrow{\sim} \g' \otimes \afp$ induced by $\kappa^p$. Note that under this identification, for $g \in \gafp$, the following diagram commutes
\begin{equation}
    \begin{tikzcd}
        V \otimes \afp \arrow{r}{\nu_{\kappa^p}} \arrow{d}{g} & V' \otimes \afp \arrow{d}{\chi_{\kappa^p}(g)} \\
         V \otimes \afp \arrow{r}{\nu_{\kappa^p}} & V' \otimes \afp.
    \end{tikzcd}
\end{equation}
It follows from this that $\igscorr \subset \igs \gx \times \igs \gxp$ is stable under the action of $\ul{\gafp}$ via $\ul{\gafp} \xrightarrow{1 \times \chi_{\kappa^p}} \ul{\gafp} \times \ul{\g'(\afp)}$. Therefore, the obvious forgetful morphisms $\igscorr \to \igs \gx$ and $\igscorr \to \igs \gxp$ are $\ul{\gafp}$-equivariant. 

\subsubsection{} We start with the following trivial case of Conjecture \ref{Conj:ConstructionWorks}.
\begin{Lem} \label{Lem:TrivialSiegelCase}
Suppose that $\gx=\gvx$ and that $\iota$ is the identity. Then Conjecture \ref{Conj:ConstructionWorks} holds.
\end{Lem}
\begin{proof}
We first note that if $\Omega'=(\iota, \mathcal{G}, \mathcal{G}', V_{\zp}, V_{\zp}', g \circ \kappa^p )$ for $g \in \gafp$, then there is an induced isomorphism (which does not commute with the projection map to $\igs\gxp$!)
\begin{align}
    \igscorr \xrightarrow{g} \operatorname{IgsCorr}_{\Omega', \mathsf{P}}
\end{align}
given on objects by postcomposing $\eta_{y}$ by $\chi_{\kappa^p}(g)$ and given by the identity on morphisms. Thus we can choose a trivialization $\kappa_{V}:\mathsf{P} \xrightarrow{} \gv$ inducing $\kappa^p$. By weak approximation for $\gv(\mathbb{Q}) \to \gv(\qp)$, and the fact that all self dual lattices in $(V_{\qp}, \psi)$ are $\gv(\qp)$-conjugate, we may moreover change $\kappa_V$ to assume that $V_{\zp}=V'_{\zp}$ under $V =V'$. It is now a direct consequence of the definition that
\begin{align}
    \igscorrpre \to \igspre_{\Xi} \gvx \times \igspre_{\Xi} \gvx
\end{align}
may be identified with
\begin{align}
    \igspre_{\Xi} \times_{\igspre_{\Xi}} \igspre_{\Xi} \to \igspre_{\Xi} \times \igspre_{\Xi},
\end{align}
compatible with the maps to $\igscorrpre \to \bun_{G_V} \times_{\bun_{G_V}} \bun_{G_V}$. Conjecture \ref{Conj:ConstructionWorks} now asserts that the maps
\begin{align}
   \igs \gvx \xrightarrow{\Delta} \igs \gvx \times_{\igs \gvx} \igs \gvx \xrightarrow{p_{1,2}} \igs \gvx
\end{align}
are isomorphisms, which is tautologically true.
\end{proof}
\begin{Cor} \label{Cor:Independence}
Suppose that $\gx=\gvx$ and that $\iota$ is the identity, then $\igscorrpre$ only depends on $\mathsf{P}, \kappa^p$, and not on $V_{\zp}, V_{\zp}'$.
\end{Cor}
\begin{proof}
    This follows from the proof of Lemma \ref{Lem:TrivialSiegelCase} by inspection.
\end{proof}
Let us now write $\igscorrpreviotakappa$ for $\igscorrpre$ when $\gx=\gvx$ and $\iota$ is the identity.

\subsubsection{} Let $\gx,\mathsf{P}, \kappa^p$ be as above, let $\iota:\gx \to \gvx$ be a Hodge embedding and let $\kappa^p_V:\mathsf{P}_V \otimes \afp \to \g_V \otimes \afp$ be induced from $\kappa^p$. Assume that Assumption \ref{Assump:InfinityHodge} holds, and define
\begin{align}
\igscorriotakappa^{\mathrm{naive}}&:= \left(\igs \gx \times \igs \gxp \right) \times_{\igs \gvx \times \igs \gvxp} \igscorrviotakappa \\
\igscorriotakappa&:=\igscorriotakappa^{\mathrm{naive}} \times_{\left(\bung \times_{\bun_{G_{V'}}} \bungp \right)} \left(\bung \times_{\bungp} \bungp \right).
    \end{align}
\begin{Prop} \label{Prop:IndepenceI}
If $\Omega$ is an admissible Hodge datum containing $\iota$ and $\kappa^p$, then $\igscorr$ is isomorphic to $\igscorriotakappa$ compatible with the structure maps to $\igs \gx \times \igs \gxp$ and to $\bung \times_{\bungp} \bungp$
\end{Prop}
\begin{proof}
The fiber product
\begin{align}
     \left(\igscorriotakappa\right)^{\mathrm{naive}}&:= \left(\igspre_{\Xi} \gx \times \igspre_{\Xi'} \gxp \right) \times_{\igspre_{\Xi} \gvx \times \igspre_{\Xi'} \gvxp} \igscorrpreviotakappa
\end{align}
is the presheaf of groupoids on $\perf$ sending $(R,R^+)$ to the groupoid of quintuples $(R^{\sharp_{1}}, R^{\sharp_{2}}, x_1, x_2, f)$ as in Definition \ref{defn:Corr}, except that $f$ is only required to be $\Omega$-structure preserving away from $p$. The further fiber product with 
\begin{align}
    \bung \times_{\bungp} \bungp
\end{align}
over
\begin{align}
    \bung \times_{\bun_{G_{V'}}} \bungp
\end{align}
enforces the condition that $f$ is $\Omega$-structure preserving at $p$. Since stackification preserves fiber products, we are done.
\end{proof}
A particular consequence of Proposition \ref{Prop:IndepenceI} is that $\igscorr$ does not depend on the choice of stabilizer quasi-parahoric integral models. 
\begin{Rem}
    It follows as in the proof of Lemma \ref{Lem:CorrespondenceNewton} that the natural map $\igscorriotakappa \to \igs \gx$ factors through $\igs \gx_{-\mu'}$ and similarly that the map $\igscorriotakappa \to \igs \gxp$ factors through $\igs \gxp_{-\mu}$. The downside of the definition of $\igscorriotakappa$ is that it is not clear that the natural maps $\igscorriotakappa \to \igs \gx, \igs \gxp$ are monomorphisms, unless there is an admissible Hodge datum $\Omega$ containing $\iota$ and $\kappa^p$. 
\end{Rem}

\subsection{Functoriality} \label{sub:NotationFunctoriality} In this section we prove that the construction of Section \ref{Sub:Correspondences} is functorial for morphisms $a:\gxone \to \gxtwo$, compatible with the functoriality of Igusa stacks. 

\subsubsection{} Let $\gxone$ be a Shimura datum of Hodge type, let $\mathsf{P}_1$ be a $\g_1$-torsor over $\spec \mathbb{Q}$ that is trivial over $\afp$. Let $\g_1'=\operatorname{Aut}_{\g_1}(\mathsf{P}_{1})$, and let $\x'_1$ be a Shimura datum for $\g'_1$. Let $\iota_1:\gxone \to \gvxone$ be a Hodge embedding such that $\iota'_1$ satisfies Assumption \ref{Assump:InfinityHodge}, let $\kappa_1^p:\mathsf{P}_1 \otimes \afp \to \g_1 \otimes \afp$ be an isomorphism and choose $\kappa_{1,V}$ as above.

\subsubsection{} Let $\gxtwo$ be a Shimura datum of Hodge type, let $\mathsf{P}_2$ be a $\g_2$-torsor over $\spec \mathbb{Q}$ that is trivial over $\afp$. Let $\g_2'=\operatorname{Aut}_{\g_2}(\mathsf{P}_{2})$, and let $\x'_2$ be a Shimura datum for $\g'_2$. Let $\iota_2:\gxtwo \to \gvxtwo$ be a Hodge embedding such that $\iota'_2$ satisfies Assumption \ref{Assump:InfinityHodge}, let $\kappa_2^p:\mathsf{P}_2 \otimes \afp \to \g_2 \otimes \afp$ be an isomorphism and choose $\kappa_{2,V}$ as above.

\subsubsection{} Let $a:\gxone \to \gxtwo$ be a morphism of Shimura data. Choose an isomorphism $\lambda: \mathsf{P}_1 \times^{\mathsf{G}_1} \mathsf{G}_2 \to \mathsf{P}_2$ of $\mathsf{G}_2$-torsors and let $a_{\lambda}':\gxonep \to \gxtwop$ be the induced morphism of Shimura data. Consider $V_3= V_1 \oplus V_2$ equipped with the direct sum symplectic form. The morphism $(\iota_1, \iota_2 \circ a):\gxone \to \gvxone \times \gvxtwo$ factors through $\gvxone \times_{\mathbb{G}_m} \gvxtwo$ by \cite[Lemma 7.1.1]{DvHKZIgusaStacks}, and we can compose with $\gvxone \times_{\mathbb{G}_m} \gvxtwo \to \gvxthree$ to get a new Hodge embedding $\iota_3:\gxone \to \gvxthree$ by \cite[Lemma 7.1.1]{DvHKZIgusaStacks}. Recall that the morphisms $a$ and $a'_{\lambda}$ induce (unique) morphisms of Igusa stacks
\begin{align}
    a:\igs \gxone &\to \igs \gxtwo \\
    a'_{\lambda}: \igs \gxonep &\to \igs \gxtwop,
\end{align}
by \cite[Theorem I]{DvHKZIgusaStacks} or \cite[Theorem C]{KimFunctoriality}.

\begin{Prop} \label{Prop:IntegralFunctoriality}
For $\lambda$ as above, there is a (necessarily unique) commutative diagram of v-sheaves 
\begin{equation} \label{Eq:IntegralFunctorialityDiagram}
    \begin{tikzcd}  
        \operatorname{IgsCorr}_{(\iota_1, \kappa_1^p)} \gxone \arrow{r} & \igs \gxone \times \igs \gxonep \\ 
        \operatorname{IgsCorr}_{(\iota_3, \kappa_1^p)} \gxone \arrow{r} \arrow[u] \arrow{d}{a} & \igs \gxone \times \igs \gxonep \arrow{d}{(a, a'_{\lambda})} \arrow[u, equals] \\
        \operatorname{IgsCorr}_{(\iota_2, \kappa_2^p)} \gxtwo \arrow{r} & \igs \gxtwo \times \igs \gxtwop. 
    \end{tikzcd}
\end{equation}
Moreover, this diagram fits into a three dimensional $2$-commutative diagram with the diagram
\begin{equation} \label{Eq:DiagramCompatibilityBunG}
    \begin{tikzcd}  
         \bun_{G_1} \arrow{r}{\operatorname{id} \times \beta_{\mathsf{P}_1}}  & \bun_{G_1} \times \bun_{G_1'} \\ 
        \bun_{G_1} \arrow{r}{\operatorname{id} \times \beta_{\mathsf{P}_1}} \arrow[u, equals ] \arrow{d}{a} & \bun_{G_1} \times \bun_{G_1'} \arrow{d}{(a, a'_{\lambda})} \arrow[u, equals] \\
        \bun_{G_2} \arrow{r}{\operatorname{id} \times \beta_{\mathsf{P}_2}} &\bun_{G_2} \times \bun_{G_2'}.
    \end{tikzcd}
\end{equation}
If there are admissible Hodge data $\Omega_i$ containing $\iota_i$ and $\kappa_i^p$ for $i=1,2,3$ (where $\kappa_3^p=\kappa_1^p)$, and Conjecture \ref{Conj:ConstructionWorks} holds for $\igscorrthree\gxone$, then it also holds for $\igscorrone\gxone$, and the map $\igscorrthree\gxone \to \igscorrone \gxone$ is an isomorphism.
\end{Prop}
\begin{proof}
Write $(\g_{V_1,V_2}, \h_{V_1,V_2}):=\gvxone \times_{\mathbb{G}_m} \gvxtwo$ and consider the following commutative diagram of Shimura data
    \begin{equation}
        \begin{tikzcd}
            & \gvxone \\
            \gxone \arrow{r}{(\iota_{1},\iota_2 \circ a)} \arrow[ur, "\iota_1"] \arrow{d}{a} & (\g_{V_1,V_2}, \h_{V_1,V_2}) \arrow{d} \arrow{u} \arrow{r} & \gvxthree \\
            \gxtwo \arrow{r}_{\iota_2} & \gvxtwo,
        \end{tikzcd}
    \end{equation}
    there is an analogous diagram for $\gxonep$ and $\gxtwop$. From this, we immediately get maps
\begin{equation}
    \begin{tikzcd}
        \igscorrvoneiotakappa \times_{\left(\bun_{G_1} \times_{\bun_{G_{V_1}}} \bun_{G'_1}\right)} \left(\bun_{G_1}\times_{\bun_{G'_1}} \bun_{G_1'} \right) \\
        \igscorrvonetwoiotakappa \times_{\left(\bun_{G_1} \times_{\bun_{G_{V_1,V_2}}} \bun_{G'_1}\right)} \left(\bun_{G_1}\times_{\bun_{G'_1}} \bun_{G_1'} \right) \arrow{d} \arrow{u} \\
       \igscorrvtwoiotakappa \times_{\left(\bun_{G_2} \times_{\bun_{G_{V_2}}} \bun_{G_2'}\right)} \left(\bun_{G_2}\times_{\bun_{G'_2}} \bun_{G_2'} \right)
    \end{tikzcd}
\end{equation}
and
\begin{equation} \label{Eq:TwoManyFiberProducts}
\begin{tikzcd}
    &\igscorrvonetwoiotakappa \times_{\left(\bun_{G_1} \times_{\bun_{G_{V_1,V_2}}} \bun_{G'_1}\right)} \left(\bun_{G_1}\times_{\bun_{G'_1}} \bun_{G_1'} \right) \arrow{d} \\
    &\igscorrvthreeiotakappa \times_{\left(\bun_{G_1} \times_{\bun_{G_{V_3}}} \bun_{G'_1}\right)} \left(\bun_{G_1}\times_{\bun_{G'_1}} \bun_{G_1'} \right).
\end{tikzcd}
\end{equation}
The first part of the proposition follows from the definitions if we can show that the bottom map of \eqref{Eq:TwoManyFiberProducts} is an isomorphism. This follows from Lemma \ref{Lem:GStructurePreservingClosedPerfectoid} together with the fact that $\igs (\g_{V_1,V_2}, \h_{V_1,V_2}) \to \igs \gvxthree$ is a monomorphism. \smallskip 

To prove this fact, choose an idempotent endomorphism $\Theta$ of $V_3$ such that $\Theta(V_3)=V_1$. Then $\igs \gvxthree$ is the v-sheafification of the presheaf of groupoids sending $(R,R^+)$ to the groupoid whose objects are triples $(A, \lambda, \eta^p)$ of polarized abelian schemes $(A, \lambda)$ over $\spec R^+/\varpi$ up to prime-to-$p$ quasi-isogeny, together with an isomorphism $\eta^p:V^p A \to V_3 \otimes \afp$ of $\afp$-local systems compatible with the Weil pairing, and whose morphisms are quasi-isogenies compatible with $\lambda$. The sheaf $\igs (\g_{V_1,V_2}, \h_{V_1,V_2})$ is the sheafification of the presheaf of groupoids sending $(R,R^+)$ to the subgroupoid of triples $(A, \lambda, \eta^p)$ such that the idempotent $\Theta_{A}$ of $V^p A$ induced by $\eta^p$ comes (uniquely) from a self prime-to-$p$ quasi-isogeny of $A$.\footnote{Indeed, this follows from \cite[Theorem 1.3, Definition 8.1]{ZhangThesis} after taking the inverse limit over the level subgroups away from $p$, see also \cite[Lemma 5.2.13]{DvHKZIgusaStacks}.}\smallskip 

Next, we prove the final statement of the proposition. Since $\igscorrthree$ and $\igscorrone$ are subsheaves of $\igs \gxone \times \igs \gxonep$, the results proved above show that $\igscorrthree \subset \igscorrone$. By Lemma \ref{Lem:CorrespondenceNewton}, we moreover have an inclusion $\igscorrone \subset \igs \gxone_{-\mu'_1} \times \igs \gxonep_{-\mu_1}$ and the induced maps $\igscorrone \to \igs \gxone_{-\mu'_1},\igs \gxonep_{-\mu_1}$ are monomorphisms. If Conjecture \ref{Conj:ConstructionWorks} holds for $\igscorrthree$, then the induced maps 
\begin{align}
    \igscorrthree \to \igs \gxone_{-\mu'_1}, \igs \gxonep_{-\mu_1}
\end{align}
are isomorphisms. It is clear that the same then holds for the maps 
\begin{align}
    \igscorrone \to \igs \gxone_{-\mu'_1}, \igs \gxonep_{-\mu_1},
\end{align}
and that furthermore the inclusion $\igscorrthree \subset \igscorrone$ is an isomorphism. 
\end{proof}

\subsection{Correspondences between perfect Igusa stacks} \label{sub:PerfectCorrespondence} Let the notation be as in Section \ref{Subsub:allthedata}, in particular we have a fixed $\gx$ and $\mathsf{P}$ and an admissible Hodge datum $\Omega$ for $(\g,\x, \mathsf{P})$. Further recall the notation from Section \ref{Sub:ReductionIgusaStack}.

\subsubsection{} Suppose that we are given $R \in \affperf$ and morphisms $x:\spec R \to \shginf, y: \spec R \to \shginf$. We note that there is a canonical isomorphism of $G_{V'}$-isocrystals
\begin{align} \label{Eq:CanonicalIsomorphismVectorBundlesTwistPerfect}
     \beta_{P}(\mathbb{L}_{x,\mathrm{crys}}) \times^{\iota',G'} G_{V'} = \beta_{P_V}\left(\mathbb{L}_{x,\mathrm{crys}} \times^{\iota, G} G_{V}\right).
\end{align}
\begin{Def} \label{Def:PerfectOmegaStructurePreserving}
We say that a quasi-isogeny $f:A_{x} \dashrightarrow A_{y}$ is \emph{$\Omega$-structure preserving away from $p$} if the induced map on the prime-to-$p$ adelic Tate modules
    \[
        V^p(f): V^p(A_{x}) \to V^p(A_{y})
    \]
    fits into a commutative square
    \[
        \begin{tikzcd}
            &V^p(A_{x}) \arrow{r}{V^p(f)}\arrow{d}{\eta_{x}}& V^p(A_{y}) \arrow{d}{\eta_{y}}\\
            &V \otimes \afp \arrow{r}{\nu_{\kappa^p}}& V' \otimes \afp.
        \end{tikzcd}
    \]
    We say that $f$ is \emph{$\Omega$-structure preserving at $p$} if there exists a (necessarily unique) isomorphism of $G'$-isocrystals $\varphi_{x, y}: \beta_{P}(\mathbb{L}_{\crys,x}) \isom \mathbb{L}_{\crys,y}$ whose pushout along $G' \to \operatorname{GL}_{V'}$ recovers the isomorphism of $\operatorname{GL}_{V'}$-isocrystals
\begin{align}
    \mathscr{E}(A_x) \xrightarrow{\mathscr{E}(f)} \mathscr{E}(A_y),
\end{align}
using the identifications
\begin{align}
    \mathscr{E}(A_x) &= \beta_{P_V}\left(\mathbb{L}_{\crys,x} \times^{\iota, G} G_{V}\right) = \beta_{P}(\mathbb{L}_{\crys,x}) \times^{\iota',G'} G_{V'} \\
    \mathscr{E}(A_y)&=\mathbb{L}_{\crys,y} \times^{\iota',G'} G_{V'}.
\end{align}
We say that $f$ is \emph{$\Omega$-structure preserving} if it is both $\Omega$-structure preserving away from $p$ and $\Omega$-structure preserving at $p$.
\end{Def}
\begin{Rem}
    Given an isomorphism $\kappa_p:P \to G$ of $G$-torsors with induced isomorphism $\kappa_{p}:P_V \to G_V$ and use it to identity $\beta_{P}$ and $\beta_{P_V}$ with the identity map, then $f$ is $\Omega$-structure preserving at $p$ precisely when it is $\g$-structure preserving at $p$.
\end{Rem}

\begin{Def} \label{Def:IgsCorrperfpre}
Define $\igscorrperfpre\gx$ to be the presheaf of groupoids on $\affperf$ sending $R$ to the following groupoid: Its objects are triples $(x,y,f)$ where $(x,y):\spec R \to \shginf \times \shgpinf$ and where $f$ is an $\Omega$-structure preserving quasi-isogeny $A_x \dashrightarrow A_y$. A morphism $(g_x,g_y):(x,y,f) \to (x',y',f')$ consists of a pair of quasi-isogenies $g_x:A_x \dashrightarrow A_{x'}$ and $g_y:A_y \dashrightarrow A_{y'}$ such that the following diagram commutes
\begin{equation} \label{Eq:StructurePreserving}
    \begin{tikzcd}
        A_x \arrow[r, dashed, "f"] \arrow[d, dashed, "g_x"] & A_y \arrow[d, dashed, "g_y"] \\
        A_{x'} \arrow[r, dashed, "f'"] & A_{y'},
    \end{tikzcd}
\end{equation}
and such that $g_x$ is $\g$-structure preserving and $g_y$ is $\g'$-structure preserving. We write $\igscorrperf\gx$ for the v-sheafification of $\igscorrpreperf\gx$.
\end{Def}
There are natural forgetful maps $\igscorrpreperf \to \igspreperf \gx$ and $\igscorrpreperf \to \igspreperf \gxp$. We moreover note that by the definition of $\igscorrpre$ they lie in a $2$-commutative diagram
\begin{equation} \label{Eq:GStructureDiagramPerf}
    \begin{tikzcd}
        \igscorrpreperf \arrow{r} \arrow{d} & \gisoc \arrow{d}{(\operatorname{Id}, \beta_{\mathsf{P}})} \\
        \igspreperf\gx \times \igspreperf \gxp \arrow{r} & \gisoc \times \gpisoc
    \end{tikzcd}
\end{equation}
We let $\igscorrperf$ be the v-sheafification of $\igscorrpreperf$, which comes equipped with maps $\igscorrperf \to \igsperf \gx \times \igsperf \gxp$. 
\begin{Lem} \label{Lem:CorrespondenceNewtonPerf}
    The morphisms $\igscorrperf \to \igsperf \gx$ and $\igscorrperf \to \igsperf \gxp$ are monomorphisms with image contained in $\igsperf \gx_{-\mu'}$ and $\igsperf \gxp_{-\mu}$ respectively. 
\end{Lem}
\begin{proof}
The map of presheaves in groupoids $\igscorrpreperf \to \igspreperf\gx$ is a fully faithful functor, because in the definition of morphisms it is clear that $g_x$ uniquely determines $g_y$. By symmetry the same is true for the map $\igscorrpreperf \to \igspreperf\gxp$. The fact that the image lies in $\igsperf \gx_{-\mu'}$ follows from the fact that $\igspreperf\gx$ maps to $\gisocmu$ and that $\igspreperf\gxp$ maps to $\gpisocmup$, in combination with \eqref{Eq:GStructureDiagramPerf}.
\end{proof}

\begin{Conj} \label{Conj:ConstructionWorksSpecialFiberIgusa}
If $\Sha^1(\mathbb{Q},\g)=0$, then the morphisms $\igscorrperf \to \igsperf \gx_{-\mu'}$ and $\igscorrperf \to \igsperf \gxp_{-\mu}$ are isomorphisms.
\end{Conj}

\begin{Lem} \label{Lem:TrivalSiegelCasePerfect}
    If $\gx=\gvx$ and $\iota$ is the identity, then Conjecture \ref{Conj:ConstructionWorksSpecialFiberIgusa} holds.
\end{Lem}
\begin{proof}
    The proof of Lemma \ref{Lem:TrivialSiegelCase} can be adapted immediately.
\end{proof}
\begin{Cor} \label{Cor:IndependencePerfect}
Suppose that $\gx=\gvx$ and that $\iota$ is the identity, then $\igscorrpre$ only depends on $\mathsf{P}, \kappa^p$, and not on $V_{\zp}$ and $V_{\zp'}$. 
\end{Cor}
Let us now write $\igscorrviotakappa^{\mathrm{pre}}$ for $\igscorrpre$ when $\gx=\gvx$ and $\iota$ is the identity.
\begin{proof}
    This follows from the proof of Lemma \ref{Lem:TrivialSiegelCase} by inspection.
\end{proof}

\subsection{Reduction of correspondences} \label{Sub:ReductionCorrespondence} Let $\gx$ and $\mathsf{P}$ be as above, choose $\iota, \kappa^p$ as before. Define 
\begin{align}
\igscorrperfnaiveiotakappa&:= \igscorrvperfiotakappa \times_{\igsperf \gvx \times \igsperf \gvxp} \left(\igsperf \gx \times \igsperf \gxp\right) \\
\igscorrperfiotakappa &:= \igscorrperfnaiveiotakappa \times_{\left(\gisoc \times_{\gvpisoc} \gpisoc \right)} \left( \gisoc \times_{\gpisoc} \gpisoc\right).
\end{align}
\begin{Prop} \label{Prop:IndepenceII}
If $\Omega$ is an admissible Hodge datum containing $\iota$ and $\kappa^p$, then $\igscorrperf$ is isomorphic to $\igscorrperfiotakappa$ compatible with the structure maps to $\igsperf \gx \times \igsperf \gxp$ and to $\gisoc \times_{\gpisoc} \gpisoc $.
\end{Prop}
\begin{proof}
The fiber product
\begin{align}
   \igscorrvperfpreiotakappa \times_{\igspreperf_{\Xi} \gvx \times \igspreperf_{\Xi'} \gvxp} \left(\igspreperf_{\Xi} \gx \times \igspreperf_{\Xi'} \gxp\right) 
\end{align}
is the presheaf of groupoids on $\affperf$ sending $R$ to the groupoid of triples $(x,y, f)$ as in Definition \ref{Def:IgsCorrperfpre}, except that $f$ is only required to be $\Omega$-preserving away from $p$. The further fiber product with 
\begin{align}
   \gisoc \times_{\gpisoc} \gpisoc 
\end{align}
over
\begin{align}
   \gisoc \times_{\gvpisoc} \gpisoc 
\end{align}
enforces the condition that $f$ is $\Omega$-structure preserving at $p$. Since stackification preserves fiber products, we are done.
\end{proof}

\subsubsection{} We have the following immediate corollary:
\begin{Prop} \label{Prop:ReductionCorrespondences}
There is an equality $\igscorrperfiotakappa = \left(\igscorriotakappa\right)^{\mathrm{red}}$ of subsheaves of $\igs \gx^{\mathrm{red}} \times \igs \gxp^{\mathrm{red}}$. 
\end{Prop}
\begin{proof}
This is true in the Siegel case by Lemmas \ref{Lem:TrivialSiegelCase} and \ref{Lem:TrivalSiegelCasePerfect} because $\mathcal{F} \mapsto \mathcal{F}^{\red}$ commutes with fiber products and thus reduces the diagonal to the diagonal. In general, it is a direct consequence of Proposition \ref{Prop:IndepenceI}, Proposition \ref{Prop:IndepenceII} and the fact that $\mathcal{F} \mapsto \mathcal{F}^{\red}$ commutes with fiber products. 
\end{proof}

\subsubsection{Exotic Hecke correspondences} Let us now define actual exotic Hecke correspondences between $\shginf$ and $\shgpinf$. We define $\shcorrpre \to \shginf \times \shgpinf$ to be the pre-sheaf on $\affperf$ sending $(x,y):\spec R \to \shginf \times \shgpinf$ to the set of $\Omega$-preserving quasi-isogenies $f:A_x \dashrightarrow A_y$. It is clear from the definition that there is a map $\shcorrpre \to \igscorrpreperf$, which induces a commutative diagram of sheaves which is Cartesian by inspection (sheafification preserves fiber products).
\begin{equation} \label{Eq:IgsCorrToShCorr}
    \begin{tikzcd} 
        \shcorr \arrow{r} \arrow{d} & \shginf \times \shgpinf \arrow{d} \\
        \igscorrperf \arrow{r} & \igsperf\gx \times \igsperf\gxp.
    \end{tikzcd}
\end{equation}
\subsubsection{} By construction, there is a $2$-commutative diagram
\begin{equation} \label{Eq:KeyDiagram}
    \begin{tikzcd} 
        \shginf \arrow{d} & \shcorr \arrow{r} \arrow{l} \arrow{d} & \shgpinf \arrow{d} \\
        \shtlocgmu & \shtlocgmumup \arrow{r} \arrow{l} &  \shtlocgpmup.
    \end{tikzcd}
\end{equation}
\begin{Conj} \label{Conj:ConstructionWorksSpecialFiber}
If $\Sha^1(\mathbb{Q},\g)=0$, then both squares in \eqref{Eq:KeyDiagram} are Cartesian.
\end{Conj}
If $\mathcal{G}$ and $\mathcal{G}'$ are reductive and $P=\mathsf{P} \otimes \qp$ is trivial, this conjecture will be proved in forthcoming work of Xiao--Zhu. For now, we will state this case separately.
\begin{Conj} \label{Conj:XiaoZhu}
If $\Sha^1(\mathbb{Q},\g)=0$ and $\mathcal{G}$ and $\mathcal{G}'$ are reductive and $P=\mathsf{P} \otimes \qp$ is trivial, then both squares in \eqref{Eq:KeyDiagram} are Cartesian.
\end{Conj}
\begin{Prop} \label{Prop:PEL}
If $\gx$ is of PEL type AC and $\Sha^1(\mathbb{Q},\g)=0$, then Conjecture \ref{Conj:XiaoZhu} holds.
\end{Prop}
\begin{proof}
This is proved in \cite[Proposition 7.3.9]{XiaoZhu}. The assumption that $\Sha^1(\mathbb{Q},\g)=0$ ensures that the PEL type moduli spaces are equal to Shimura varieties rather than a finite disjoint union of Shimura varieties indexed by $\Sha^1(\mathbb{Q},\g)$.
\end{proof}

\subsubsection{} Now recall that there is a $2$-Cartesian square
\begin{equation}
    \begin{tikzcd}
        \shtlocgmumup \arrow{r} \arrow{d} & \shtlocgmu \arrow{d} \\
        \shtlocgmup \arrow{r} & \gisoc.
    \end{tikzcd}
\end{equation}
The maps $\shtlocgmu \to \gisocmu$ and $\shtlocgpmup \to \gpisocmup$ are v-surjective by \cite[Proposition 3.2.3]{DvHKZIgusaStacks}.  In particular, it follows that $\shtlocgmumup \to \shtlocgmu$ and $\shtlocgmumup \to  \shtlocgmup$ surject onto a closed union of Newton strata. 
\begin{Lem} \label{Lem:ComparisonConjectureSpecialFiber}
    Conjecture \ref{Conj:ConstructionWorksSpecialFiber} holds if and only if Conjecture \ref{Conj:ConstructionWorksSpecialFiberIgusa} holds.
\end{Lem}
\begin{proof}
If Conjecture \ref{Conj:ConstructionWorksSpecialFiber} holds then $\shcorr \to \shginf_{-\mu'}$ and $\shcorr \to \shgpinf_{-\mu}$ are v-surjective by the discussion above. This proves that the maps $\igscorrperf \to \igsperf \gx_{-\mu'}$ and $\igscorrperf \to \igsperf \gxp_{-\mu}$ are v-surjective, and thus must be isomorphisms by Lemma \ref{Lem:CorrespondenceNewtonPerf}. Conversely if Conjecture \ref{Conj:ConstructionWorksSpecialFiberIgusa} holds, then Conjecture \ref{Conj:ConstructionWorksSpecialFiber} holds formally, see Section \ref{subsub:IgusaConjToCorrConj}.
\end{proof}
\begin{Cor} \label{Cor:PerfectCorrespondenceWorks}
If Conjecture \ref{Conj:ConstructionWorksSpecialFiber} holds, then the natural maps $$\igscorrperf \to \igs \gx^{\mathrm{red}}_{\mu'}, \igs \gxp^{\mathrm{red}}_{\mu}$$ are isomorphisms of v-sheaves.
\end{Cor}

\section{A study of the case in which \texorpdfstring{$\mathsf{P}_{\mathbb{A}_f}$}{P(Af)} is trivial}

In this section we specialize to the case that $\mathsf{P}_{\mathbb{A}_f}$ is trivial and that $G$ is an unramified reductive group. In Section \ref{Sub:ReductionToPerfect} We will deduce Conjectures \ref{Conj:MainII} and \ref{Conj:Main} from Conjecture \ref{Conj:MainIII}, a proof of which has been announced by Xiao--Zhu. In Section \ref{Sub:DeligneGroup}, we will verify that our correspondences are stable under the action of (a variant of) Deligne's $\mathcal{A}$-group.

\subsection{From Conjecture \ref{Conj:ConstructionWorksSpecialFiberIgusa} to Conjecture \ref{Conj:ConstructionWorks}} \label{Sub:ReductionToPerfect}Let the notation be as in Section \ref{sub:PerfectCorrespondence} above. Assume from now on that $G$ is an unramified reductive group and that $P$ is trivial. Suppose that $\Sha^1(\mathbb{Q},\g)=0$. We will now show that Conjecture \ref{Conj:ConstructionWorks} follows from Conjecture \ref{Conj:ConstructionWorksSpecialFiberIgusa}. 
\begin{Thm} \label{Thm:XZCorrespondences}
Assume that $G$ is an unramified reductive group, that $P$ is trivial and that $\Sha^1(\mathbb{Q},\g)=0$. If $\mathcal{G}$ and $\mathcal{G}'$ are reductive and if Conjecture \ref{Conj:ConstructionWorksSpecialFiber} holds, then the natural maps
\begin{align}
   \igscorr \to \igs \gx_{-\mu'}, \igscorr \to \igs \gxp_{-\mu}
\end{align}
are isomorphisms.
\end{Thm}
\begin{Lem} \label{Lem:CompatibilityCorrespondencesReductionPre}
If $\mathcal{G}$ and $\mathcal{G}'$ are reductive, then there exists a unique dashed arrow making the following diagram commute
\begin{equation}
    \begin{tikzcd}
        \igscorrpre \arrow[r, dashed] \arrow{d} & (\igscorrperfpre)^{\diamond/\circ,\pre} \arrow{d} \\
        \igspre_{\Xi} \gx \times \igspre_{\Xi'}\gxp \arrow{r} & (\igspreperf_{\Xi}\gx)^{\diamond/\circ, \pre} \times (\igspreperf_{\Xi'}\gxp)^{\diamond/\circ, \pre}.
    \end{tikzcd}
\end{equation}
The dashed arrow is moreover an isomorphism.
\end{Lem}
\begin{proof}
The uniqueness follows from the fact that the bottom arrow is an isomorphism by Theorem \ref{Thm:IgusaDCirc}, and that the vertical arrows are monomorphisms (Lemmas \ref{Lem:CorrespondenceNewton} and \ref{Lem:CorrespondenceNewtonPerf}). The lemma comes down to proving that for $(R,R^+) \in \perf$ with untilts $R^{\sharp_i}$ for $i=1,2$, morphisms $x:\spf R^{\sharp_1+} \to \scrshat_{K_p}\gx, y:\spf R^{\sharp_2+} \to \scrshat_{K_p'}\gxp$ and a formal quasi-isogeny $f:A_x \dashrightarrow A_y$, that $f$ is $\Omega$-structure preserving in the sense of Definition \ref{Def:GtoGpStructure} if and only if the reduction $f_0:A_{x_0} \dashrightarrow A_{y_0}$ is $\Omega$-structure preserving in the sense of Definition \ref{Def:PerfectOmegaStructurePreserving}. But this follows from Lemma \ref{Lem:ReductionGStructure} and Proposition \ref{Prop:Fullyfaithful}, using the fact that $\mathcal{G}$ and $\mathcal{G}'$ are reductive, and making use of Theorem \ref{Thm:ImaiKatoYoucis}. 
\end{proof}

\begin{proof}[Proof of Theorem \ref{Thm:XZCorrespondences}]
We use Lemma \ref{Lem:CompatibilityCorrespondencesReductionPre} together with sheafification to get the commutative diagram
\begin{equation}
    \begin{tikzcd}
        \igscorr \arrow[r] \arrow{d} & (\igscorrperfpre)^{\diamond/\circ} \arrow{d} \arrow{r} & (\igscorrperf)^{\diamond/\circ} \arrow{dl} \\
        \igs \gx \times \igs \gxp \arrow{r} & \igsperf\gx^{\diamond/\circ} \times \igsperf \gxp^{\diamond/\circ}.
    \end{tikzcd}
\end{equation}
By Corollary \ref{Cor:PerfectCorrespondenceWorks}, it suffices to show that the composition of the two horizontal arrows is an isomorphism. The top left horizontal arrow is an isomorphism by Lemma \ref{Lem:CompatibilityCorrespondencesReductionPre}. The right horizontal arrow is a monomorphism because it is a morphism of subsheaves of $\igsperf\gx^{\diamond/\circ} \times \igsperf \gxp^{\diamond/\circ}$, and so it suffices to show that it is an epimorphism. We can do this after passing to the fiber product over $\igsperf\gx^{\diamond/\circ} \times \igsperf \gxp^{\diamond/\circ}$ with $\shginf \times \shgpinf$ (this is a v-cover by Proposition \ref{Prop:VSurjectiveNewtonDiamondCirc}). Since fiber products commute with applying $\diamond/\circ$, we are considering the map $S^{\diamond/\circ} \to T^{\diamond/\circ}$, where $S \to T$ is the natural map
\begin{align}
    &S:=\igscorrperfpre \times_{\igspreperf_{\Xi}\gx \times \igspreperf_{\Xi'}\gxp} \left(\shginf \times \shgpinf\right) \to \\
    &\igscorrperf \times_{\igsperf \gx \times \igsperf \gxp} \left(\shginf \times \shgpinf\right)=:T.
\end{align}
Since sheafification commutes with fiber products, we see that $T$ is the sheafification of $S$. Thus to show that $S \to T$ is an isomorphism it suffices to show that $S$ is already a v-sheaf. This follows as in the proof of Theorem \ref{Thm:IgusaDCircII} from the ind-representability of the sheaf of quasi-isogenies between two abelian schemes.
\end{proof}

\subsection{Correspondences and Deligne group actions} \label{Sub:DeligneGroup} In this section we describe the (extended) Deligne groups of Xiao--Zhu, following forthcoming work of Xiao--Zhu \cite{XiaoZhu2}. We then show that $\igscorr$ is stable under the action of these Deligne groups in a suitable way.

Let $\g$ be an algebraic group over $\mathbb{Q}$ with center $\zg$. We let $\agext$ be the set of isomorphism classes of pairs $(\zeta, \gamma)$ where $\zeta$ is a $\zg$-torsor over $\mathbb{Q}$ such that $\zeta_{\mathbb{R}}$ is trivial, and where $$\gamma:\left(\zeta \times ^{\zg} \g\right) \otimes \af \to \g \otimes \af$$ is an isomorphism. There is an action of $\gaf$ on $\agext$ given by postcomposition. The Baer sum of $\zg$-torsors together with the product map on $\g \otimes \af$ defines a map
\begin{align}
    \agext \times \agext \to \agext
\end{align}
described in \cite{XiaoZhu2}, which turns $\agext$ into a group equipped with an action of $\g(\af)$. There is moreover a group homomorphism $\agext \to \Sha^1(\mathbb{Q},\g)$ sending $(\zeta, \gamma)$ to the isomorphism class of $\zeta \times ^{\zg} \g$, see \cite[Lemma 6.1.14]{XiaoZhu2}. We define $\ag$ to be the kernel of this morphism. It is fairly straightforward to see that this group can be identified with Deligne's group (see \cite[Section 6.1.13]{XiaoZhu2})
\begin{align}
    \left(\frac{\g(\af)}{\zgq} \rtimes \Gad(\mathbb{Q})^1 \right)/ \frac{\g(\mathbb{Q})}{\zgq},
\end{align}
introduced in \cite[2.1.13.1]{DeligneTravaux}. Here $\gad(\mathbb{R})^1$ is the image of $\g(\mathbb{R}) \to \Gad(\mathbb{R})$, and we will use the same superscript for other subgroups of $\Gad(\mathbb{R})$.

\subsubsection{} \label{subsub:TwistAG} If $\g'=\operatorname{Aut}_{\g}(\mathsf{P})$ for some $\g$-torsor $\mathsf{P}$, then there is a canonical morphism $\zg \to \g'$ identifying $\zg \isom \zg'$. Choose an isomorphism $\kappa:\mathsf{P} \otimes \af \to \g \otimes \af$, which induces an isomorphism
\begin{align}
    \rho_{\kappa}:\agext \to \agpext
\end{align}
by sending $(\xi, \gamma) \mapsto (\xi \times^{\zg} \zg', \chi_{\kappa} \circ \gamma)$. In particular, if $\Sha^1(\mathbb{Q},\g)=0$, then it induces an isomorphism $\ag \to \agp$.

\subsubsection{} Now let $\gx$ be a Shimura datum of Hodge type, fix a prime $p$ and $\mathbb{C} \isom \qpbar$. By the functoriality of Igusa stacks, see \cite[Theorem I]{DvHKZIgusaStacks}, the group $\ag$ acts on $\igs \gx$. We topologise $\ag$ following \cite[Section 4.2.4]{DvHKZIgusaStacksII}: We give $\gaf/\zgq$ the quotient topology, we give $\gadqone$ the discrete topology, the product $\gaf/\zgq \times \gadqone$ the product topology, and $\ag$ the quotient topology. We thus get an action of $\ul{\ag}$ on $\igs \gx$, which factors through 
\begin{align}
    \ul{\apg}= \ul{\ag}/\ul{G(\qp)}.
\end{align}

\subsubsection{} \label{subsub:Good} We say that an admissible Hodge datum $\Omega$ is good if there is an isomorphism $\kappa_p:P \to G$ such that the induced map $G \to G'$ satisfies $\kappa_p(\mathcal{G}(\zpbr))=\mathcal{G}'(\zpbr)$. Now let $\Omega$ be a good Hodge datum for $(\g, \x, \mathsf{P})$ and fix $\kappa_p$ as above, and write $\kappa=\kappa^p \times \kappa_p$. Then $\kappa$ defines a $\gaf$-equivariant continuous isomorphism $\ag \isom \agp$ and therefore also an isomorphism $ \ul{\apg} \isom  \ul{\apgp}$. 
\begin{Prop} \label{Prop:AgEquivariance}
Let $\Omega$ be a good Hodge datum for $\gx$ and $\mathsf{P}$ and fix $\kappa_p$ as above. If $\Sha^1(\mathbb{Q},\g)=0$ and $\zg$ is a torus with $\rH^1(\qp, Z_{G})=0$, then $\igscorr \subset \igs \gx \times \igs \gxp$ is stable under the diagonal action of $\ul{\apg}$.
\end{Prop}

\subsubsection{} In the proof of Proposition \ref{Prop:AgEquivariance}, we will work with a variant $\agint$ instead of $\ul{\apg}$. If we are given a reductive model $\mathcal{G}$ of $G$ over $\zp$, then this can be used to define a model $\mathcal{G}_{\zlocp}$ of $\mathsf{G}$ over $\zlocp$, see \cite[Proposition 3.14]{DanielsYoucis}. To define a $p$-integral variant of $\agext$ we will assume that $\zg$ is an unramified torus. 

\subsubsection{} We let $\mathcal{Z}_{\zlocp}$ be the Zariski closure of $\mathsf{Z}_{\g}$ in $\mathcal{G}_{\zlocp}$. This is a smooth group scheme over $\zlocp$ with connected special fiber, see \cite[Proposition 2.4.14]{KisinZhou}. There is then a $p$-integral variant $\agintext$ of $\agext$ given by the set of isomorphism classes of pairs $(\zeta, \gamma^p)$ where $\zeta$ is a $\mathcal{Z}_{\zlocp}$-torsor (in the \'etale topology) over $\zlocp$ such that $\zeta_{\mathbb{R}}$ is trivial, and where $$\gamma^p:\left(\zeta \times ^{\mathcal{Z}_{\zlocp}} \mathcal{G}_{\zlocp} \right) \otimes \afp \to \g \otimes \afp$$ is an isomorphism. Once again there is a short exact sequence 
\begin{align}
    1 \to \agint \to \agintext \to \Sha^1(\mathbb{Q},\g) \to 1.
\end{align}
In the short exact sequence
\begin{align}
    1 \to \mathcal{Z}_{\zlocp} \to \mathcal{G}_{\zlocp} \to \mathcal{G}^{\mathrm{ad}}_{\zlocp} \to 1,
\end{align}
the group $\mathcal{G}^{\mathrm{ad}}_{\zp}$ can be identified with a reductive model of $\gad$ corresponding to $\mathcal{G}$, see \cite[Proposition 2.4.14]{KisinZhou}. Then as before we may identify
\begin{align}
    \agint = \left(\g(\afp) \rtimes \mathcal{G}^{\mathrm{ad}}_{\zlocp}(\zlocp)^1 \right)/ \mathcal{G}_{\zlocp}(\zlocp).
\end{align}
We topologise $\agint$ similarly to how we topologised $\ag$: We give $\gafp / \mathcal{Z}_{\zlocp}(\zlocp)$ the quotient topology, we give $\mathcal{G}^{\mathrm{ad}}_{\zlocp}(\zlocp)^1$ the discrete topology, we give $\gafp / \mathcal{Z}_{\zlocp}(\zlocp) \times \mathcal{G}^{\mathrm{ad}}_{\zlocp}(\zlocp)^1$ the product topology and $\agint$ the quotient topology. Consider the natural map $\ul{\agint} \to \ul{\ag} \to \ul{\apg}$ of v-sheaves of groups. 
\begin{Lem} \label{Lem:AgVsAgInt}
    If $\zg$ is a torus such that $\rH^1(\qp, \zg)=0$, then the natural map $\ul{\agint} \to \ul{\apg}$ is a surjective quotient map.
\end{Lem}
\begin{proof}
    There is a commutative diagram of short exact sequences (we have not underlined the cokernels because they have the discrete topology).
    \begin{equation}
        \begin{tikzcd}
            1 \arrow{r} & \frac{\gafp}{\zglocp} \arrow[d] \arrow{r} & \ul{\agint} \arrow{d} \arrow{r} & \frac{\calgad_{\zlocp}(\zlocp)^1}{\mathcal{G}_{\zlocp}(\zlocp)/\zglocp} \arrow{r} \arrow{d} & 1 \\
            1 \arrow{r} & \ul{\gafp}/\ul{\zgq} \arrow{r} & \ul{\apg} \arrow{r} & \frac{\Gad(\mathbb{Q})^1}{\g(\mathbb{Q})/\zgq} \arrow{r} & 1.
        \end{tikzcd}
    \end{equation}
    To show that the third vertical arrow is surjective, we consider the following commutative diagram of short exact sequences
    \begin{equation}
        \begin{tikzcd}
            \mathcal{G}_{\zlocp}(\zlocp)/\zglocp \arrow{d} \arrow{r} & \calgad_{\zlocp}(\zlocp)^1 \arrow{d} \arrow{r} & \rH^1(\zlocp, \mathcal{Z}_{\zlocp}) \arrow{d} \arrow{r} & \rH^1(\zlocp, \mathcal{G}_{\zlocp}) \arrow{r} \arrow{d} & \cdots \\
        \g(\mathbb{Q})/\zgq \arrow{r} & \Gad(\mathbb{Q}) \arrow{r} & \rH^1(\mathbb{Q}, \zg) \arrow{r} & \rH^1(\mathbb{Q}, \g) \arrow{r} & \cdots
        \end{tikzcd}
    \end{equation}
We see that it suffices to prove that the natural map
\begin{align}
    \operatorname{Ker} \left( \rH^1(\zlocp, \mathcal{Z}_{\zlocp}) \to \rH^1(\zlocp, \mathcal{G}_{\zlocp}) \right) \to \operatorname{Ker} \left(H^1(\mathbb{Q}, \zg) \to \rH^1(\mathbb{Q}, \g)\right)
\end{align}
is surjective. This comes down to asking that any $\zg$-torsor $Q$ over $\spec \mathbb{Q}$ whose induced $\g$-torsor is trivial, admits a reduction to a $\mathcal{Z}_{\zlocp}$-torsor whose induced $\mathcal{G}_{\zlocp}$-torsor is trivial. Since any $\zg$-torsor over $\spec \mathbb{Q}$ becomes trivial over $\qp$ by assumption, we can build such a $\mathcal{Z}_{\zlocp}$-torsor by gluing $Q$ to the trivial $\mathcal{Z}_{\zp}$-torsor along the trivial $Z_{\g,\qp}$-torsor using fpqc descent along the cover $\spec \mathbb{Q} \coprod \spec \zp \to \spec \zlocp$ as in the proof of \cite[Proposition 3.14]{DanielsYoucis}.
\end{proof}
\begin{Lem} \label{Lem:IntegralTorsor}
    If $P$ is trivial, then there is a $\mathcal{G}_{\zlocp}$-torsor $\mathsf{P}_{\zlocp,\kappa_{p}}$ such that $\mathsf{P}_{\zlocp} \times \spec \mathbb{Q}  \simeq\mathsf{P}$.
\end{Lem}
\begin{proof}
   Our fixed isomorphism $\kappa_{p}:\mathsf{P} \otimes \qp \to G$ provides us with a $\mathcal{G}$-torsor $\mathsf{P}_{\zp} \subset \mathsf{P} \otimes \qp$, namely $(\kappa_p)^{-1} \mathcal{G}$. The torsor $\mathsf{P}_{\zlocp}$ can now be constructed using fpqc descent along the cover $\spec \mathbb{Q} \coprod \spec \zp \to \spec \zlocp$ as in the proof of \cite[Proposition 3.14]{DanielsYoucis}.
\end{proof}
\subsubsection{} Write $\mathsf{P}_{\zlocp,\kappa_{p}}$ as in the conclusion of Lemma \ref{Lem:IntegralTorsor}. \textbf{From now on, we will work under the assumption that $\mathcal{G}'=\operatorname{Aut}_{\mathcal{G}}(\mathsf{P}_{\zlocp, \kappa_p})$.} This gives us a natural map
\begin{align}
    \mathcal{Z}_{\zlocp} \to \mathcal{G}'_{\zlocp}
\end{align}
identifying its image with the center of $\mathcal{G}'_{\zlocp}$. If $\Sha^1(\mathbb{Q},\g)=0$, this combines with $\kappa$ to give a canonical isomorphism
\begin{align}
    \rho_{\kappa}:\agint \to \agpint.
\end{align}
\begin{Lem} \label{Lem:AgIntEquivariance}
Assume that $\mathcal{G}'=\operatorname{Aut}_{\mathcal{G}}(\mathsf{P}_{\zlocp, \kappa_p})$. If $\Sha^1(\mathbb{Q},\g)=0$ and $\zg$ is a torus such that $\rH^1(\qp, \zg)=0$, then the subsheaf $\igscorr \subset \igs \gx \times \igs \gxp$ is stable under the action of $\ul{\apg}$.
\end{Lem}
\begin{proof}
Using the surjectivity of $\ul{\agint} \to \ul{\apg}$, see Lemma \ref{Lem:AgVsAgInt}, it suffices to show that the subsheaf is $\ul{\agint}$-stable. It is clearly $\gafp$-stable, and thus it suffices to prove stability under the action of the abstract group $\agint$. Recall from \cite[Lemma 4.5.7]{DanielsYoucis} that $\ul{\agint}$ acts on $\scrs_{K_p}\gx$ and that $\ul{\agpint}$ acts on $\scrs_{K_p'}\gxp$. The action of $\agint$ on $\scrs_{K_p}\gx$ has the following moduli-theoretic description. Take $x:\spec R \to \scrs_{K_p}\gx$ giving rise to $(A_x, \lambda_x, \eta_x)$ and $(\xi, \gamma^p) \in \agint$. Then by \cite[Lemma 4.5.2]{KisinPappas} there is a natural map $\mathcal{Z}_{\zlocp} \to \operatorname{End}(A_x)$ of group schemes over $\zlocp$. We may then form the abelian scheme up to prime-to-$p$ isogeny $A^{\xi}$ of \cite[Lemma 4.5.2]{KisinPappas}, which comes equipped with a weak polarization $\lambda_x^{\xi}$, see \cite[Lemma 4.4.8]{KisinPappas}, and a prime-to-$p$ level structure $\eta_{x}^{\gamma}$, see \cite[Lemma 4.5.4]{KisinPappas}. These constructions describe a map 
\begin{align} \label{Eq:ActionMap}
    \agint \times \scrs_{K_p}\gx &\to \scrs_{M_p}\gvx,
\end{align}
which factors uniquely through $\scrs_{K_p}\gx \to \scrs_{M_p}\gvx$ by \cite[Lemma 4.5.7]{KisinPappas}. In other words, there is a canonical isomorphism $(A_{x}^{\xi}, \lambda_x^{\xi}, \eta_x^{\gamma}) = (A_{(\xi, \gamma)x}, \lambda_{(\xi, \gamma)x}, \eta_{(\xi, \gamma)x})$. The same process describes the action of $\agint=\agpint$ on $\scrs_{K_p'}\gxp$. 

Now given $(x,y,f) \in \igscorrpre(R,R^+)$, the functoriality of the twisting construction in \cite[Section 4.5]{KisinPappas} gives a formal quasi-isogeny
\begin{align}
    f^{\xi}:A_{x}^{\xi} \dashrightarrow A_{y}^{\xi}
\end{align}
compatible with $\eta_{x}^{\gamma}$ and $\eta_{y}^{\gamma}$. To show that $((\xi, \gamma)x, (\xi, \gamma)y, f^{\xi})$ is a point of $\igscorrpre(R,R^+)$, it suffices to show that $f^{\xi}$ is $\g$-structure preserving (or equivalently $\Omega$-structure preserving). For this, we need the following claim
\begin{Claim} \label{Claim:Twistingtorsor}
    There is a canonical isomorphism
    \begin{align}
        \mathbb{L}_{(\xi, \gamma)x,\crys} \isom \mathbb{L}_{x, \crys}^{\xi}
    \end{align}
    where $\mathbb{L}_{x, \crys}^{\xi}$ is the twist of the $G$-bundle $\mathbb{L}_{x, \crys}$ over $X_{\spd(R,R^+)}$ by the $\zg$-torsor $\xi \otimes_{\zlocp} \qp$.
\end{Claim}
\begin{proof}
The statement encodes a $2$-categorical functoriality under the action of $\agint$ of the map
\begin{align}
    \scrs_{K_p}\gx^{\diamond} \to \shtgmu \to \bun_{G}
\end{align}
and in fact it is enough to consider the action of $\calgad_{\zlocp}(\zp)^{1}$. The $2$-categorical functoriality under the action of $\calgad_{\zlocp}(\zp)^{1}$ of the first map is \cite[Corollary 4.2]{vHSemplinerFixedPoints}, and for the second map is \cite[Lemma 2.6.1]{vHSemplinerFixedPoints}. 
\end{proof}
We will now deduce from the claim that $f^{\xi}$ is $\g$-structure preserving at $p$. Choosing a trivialization of the $\zg$-torsor $\xi \otimes_{\zlocp} \qp$, we see using the claim that the condition that $f^{\xi}$ is $\g$-structure preserving is equivalent to the claim that $f$ is $\g$-structure preserving. A similar argument shows that $f^{\xi}$ is $\g$-structure preserving away from $p$. 
\end{proof}
\begin{proof}[Proof of Proposition \ref{Prop:AgEquivariance}]
    The proposition is a direct consequence of Lemma \ref{Lem:AgIntEquivariance} and Lemma \ref{Lem:AgVsAgInt} if we choose $\Omega$ such that $\mathcal{G}'=\operatorname{Aut}_{\mathcal{G}}(\mathsf{P}_{\zlocp, \kappa_p})$, which is possible by Lemma \ref{Lem:IntegralTorsor} and the assumption that $\Omega$ is good.
\end{proof}

\section{Correspondences in some abelian type cases} The goal of this section is to prove Conjecture \ref{Conj:Main} for some Shimura varieties of abelian type under the assumption that the conjecture holds for some auxiliary Hodge type Shimura variety with good properties, see Theorem \ref{Thm:CorrespondenceAdjoint}. Such auxiliary Hodge type Shimura varieties will be shown to exist in Section \ref{Sec:Examples}.

\subsection{Correspondences in some abelian type cases} \label{sub:AdjointCorrespondencesI} To push forward exotic Hecke correspondence from Shimura varieties of Hodge type to Shimura varieties of abelian type, we have to use our understanding of the $\ag$-action on these correspondences proved in Proposition \ref{Prop:AgEquivariance}. 

\subsubsection{} We first establish some definitions to simplify the notation. 
\begin{Def} \label{Def:CorrespondenceDatum}
    A \emph{correspondence datum} is a quadruple $\Theta_{2}=(\g_2,\x_2,\mathsf{P}_2,\x_{2}')$ where $\gxtwo$ is a Shimura datum of abelian type such that $\Sha^1(\mathbb{Q},\g_2)=0$ and such that $G_2=\g_{2} \otimes \qp$ splits over an unramified extension, where $\mathsf{P}_2$ is a $\g_2$-torsor trivial over $\afp$, and where $\x_2'$ is a Shimura datum for $\g_{2}'=\operatorname{Aut}_{\g_2}(\mathsf{P}_2)$ such that Assumption \ref{Assump:Infinity} holds. A morphism $\Theta_{2}=(\g_2, \x_2, \mathsf{P}_2, \x_{2}') \to (\g_3, \x_3, \mathsf{P}_3,\x_3')=\Theta_{3}$ is a pair $(f,\lambda)$ where $f:\gxtwo \to \gxthree$ is a morphism of Shimura data, and where $\lambda$ is an isomorphism $\mathsf{P}_2 \times^{\g_2} \g_3 \xrightarrow{\sim} \mathsf{P}_3$.
\end{Def}
\begin{Def} \label{Def:HodgeTypeLifting}
A Hodge--type lifting of a correspondence datum $\Theta_{2}=(\g_{2}, \x_{2}, \mathsf{P}_2, \x_2')$ is a tuple $\Sigma=(\g, \x, \sigma, \mathsf{P}, \omega,\x')$, where $\gx$ is a Shimura datum of Hodge type such that $G$ splits over an unramified extension, where $\sigma:\gx \to \gxtwo$ is a morphism of Shimura data that is an ad-isomorphism, where $\mathsf{P}$ is a $\g$-torsor that is trivial over $\afp$, and where $\omega:\mathsf{P} \times^{\g} \g_2 \to \mathsf{P}_2$ is an isomorphism of $\g_2$-torsors, and where $\x'$ is a Shimura datum for $\g'=\operatorname{Aut}_{\g}(\mathsf{P})$ such that the induced map $\sigma:\g' \to \g'_2$ underlies a morphism of Shimura data $\gxp \to \gxtwop$.

Given a morphism $(f,\lambda):\Theta_{2}=(\g_2, \x_2, \mathsf{P}_2, \x_{2}') \to (\g_3, \x_3, \mathsf{P}_3,\x_3')=\Theta_3$ of correspondence data and Hodge--type liftings $\Sigma=(\g, \x, \sigma, \mathsf{P}, \omega, \x')$ of $(\g_2, \x_2, \mathsf{P}_2, \x_{2}')$ and $\Sigma_1=(\g_1, \x_1, \sigma_1, \mathsf{P}_1, \omega_1, \x_1')$ of $(\g_3, \x_3, \mathsf{P}_3,\x_3')$, a morphism $\Sigma \to \Sigma_{1}$ over $(f,\lambda)$ is a pair $(g, \lambda_1)$, where $g:\gx \to \gxone$ is a morphism of Shimura data fitting into a commutative diagram  
\begin{equation}
    \begin{tikzcd}
        \gx \arrow{d}{\sigma} \arrow{r}{g} & \gxone \arrow{d}{\sigma_1} \\
        \gxtwo \arrow{r}{f} & \gxthree, 
    \end{tikzcd}
\end{equation}
such that $\gx \to \gxone$ maps $\zg$ to $\mathsf{Z}(\g_{1})$, and where $\lambda_1:\mathsf{P} \times^{\g} \g_1 \to \mathsf{P}_1$ is an isomorphism such that the following diagram commutes
\begin{equation}
    \begin{tikzcd}
        \mathsf{P} \arrow{r}{\lambda_1} \arrow{d}{\omega} & \mathsf{P}_1 \arrow{d}{\omega_1} \\
        \mathsf{P}_2 \arrow{r}{\lambda} & \mathsf{P}_3,
    \end{tikzcd}
\end{equation}
and such that $g':\g' \to \g_1'$ underlies a morphism of Shimura data $\gxp \to \gxonep$.
\end{Def}
\begin{Def} \label{Def:ExcellentHodgeTypeLifting}
Let $(\g_2, \x_2, \mathsf{P}_2, \x_{2}')$ be a correspondence datum such that $\g_2$ is quasi-split and such that $\mathsf{P}_2$ is trivial over $\qp$. A Hodge type lifting $\Sigma=(\g, \x, \sigma, \mathsf{P}, \omega,\x')$ of $(\g_2, \x_2, \mathsf{P}_2, \x_{2}')$ is called \emph{excellent} if the following conditions hold: 
\begin{enumerate}
    \item The group $\Sha^1(\Q,\g)$ is trivial.

    \item The torsor $\mathsf{P}$ is trivial over $\qp$.

    \item The center $\mathsf{Z}_{\g}$ is connected, and $H^1(\qp, \mathsf{Z}_{\g})=0$.

    \item The set $B(G,-\mu) \cap B(G,-\mu')$ is nonempty.

    \item There exists a Hodge embedding $\iota:\gx \to \gvx$ such that $\iota':\g' \to \g_{V'}$ satisfies Assumption \ref{Assump:InfinityHodge}.
\end{enumerate}
\end{Def}

\subsubsection{} Let $(\g_2, \x_2, \mathsf{P}_2, \x_{2}')$ be a correspondence datum such that $G_2$ is quasi-split and such that $\mathsf{P}_2$ is trivial over $\qp$, let $\gxtwop$ be as before. Choose an isomorphism $\mathbb{C} \isom \qpbar$.\footnote{Throughout this section we will implicitly use this isomorphism to define places above $p$ of reflex fields of Shimura data, and hence Igusa stacks.} Let $\Sigma$ be an excellent Hodge type lifting of $\Theta_2$, and choose a good Hodge datum $\Omega$ for $\gx$ and $\mathsf{P}$.\footnote{This is always possible, because if we choose $\iota:\gx \to \gvx$ such that $\iota':\g' \to \g_{V'}$ satisfies Assumption \ref{Assump:InfinityHodge}, then we can simply take $\mathcal{G}'$ corresponding to $\mathcal{G}$ under a choice of $\kappa_p$.}

\subsubsection{} Recall the correspondence $\igscorr \subset \igs \gx \times \igs \gxp$ from Definition \ref{defn:Corr}. It follows from Proposition \ref{Prop:AgEquivariance}, whose assumptions hold because $\Sigma$ is an excellent Hodge type lifting and $\Omega$ is good, that $\igscorr \subset \igs \gx \times \igs \gxp$ is stable under the diagonal action of $\ul{\ag}$ (under the natural identification of $\ag$ with $\agp$ induced by $\kappa$, see Section \ref{subsub:TwistAG}). We have the following lemma.
\begin{Lem} \label{Lem:AdjointIgusaTorsor}
The natural map $\igs \gx \times \ul{\agtwop} \to \igs \gxtwo$ is a torsor for $\ul{\apg}$ under the diagonal action.
\end{Lem}
\begin{proof}
This is explained in \cite[Remark 4.3.11]{DvHKZIgusaStacksII}. It is assumed there that $\gx$ and $\gxtwo$ have the same local reflex fields, but the arguments work in general. 
\end{proof}
Note that this action factors through $\underline{\apg}$, and so it makes sense to define
\begin{align}
    \igscorrad:=\igscorr \times^{\underline{\apg}} \underline{\apgtwo},
\end{align}
which comes equipped with a natural map to $\igs \gxtwo \times \igs \gxtwop$. 
\begin{Lem} \label{Lem:TwoCategoryTheoryCube}
In the diagram
\begin{equation}
    \begin{tikzcd}[column sep=0.5]
        & \igs \gx \times \igs \gxp \arrow{rr} \arrow{dd} & & \bun_{G} \times \bun_{G'} \arrow{dd} \\
        \igscorr \arrow{rr} \arrow{dd} \arrow{ur}  && \bun_{G} \times_{\bun_{G'}} \bun_{G} \arrow{dd} \arrow{ur} \\
&\igs \gxtwo \times \igs \gxtwop \arrow{rr} & & \bun_{G_2} \times \bun_{G_2'}  \\
        \igscorrad \arrow{ur} \arrow[rr, dashed] & & \bun_{G_2} \times_{\bun_{G_2'}} \bun_{G_2'} \arrow{ur}, 
    \end{tikzcd}
\end{equation}
the dashed arrow exists making the diagram commute.
\end{Lem}
\begin{proof}
There is a natural map $\ag \to \gad(\qp)$ such that $\pi_{\mathrm{HT}}$ is equivariant for the induced $\ul{\ag}$ action on the flag variety, see \cite[Section 4.3.6]{DvHKZIgusaStacksII}. It follows from the $\ul{G(\qp)}$-equivariance of the Beauville--Laszlo map that $\igs \gx \to \bun_{G}$ is $\apg$-equivariant via a natural map $\apg \to \frac{\gad(\qp)}{\operatorname{Im}(G(\qp) \to \g(\qp)}$. Since $H^1(\qp, \zg)$ is assumed to be trivial, the natural map $G(\qp) \to \gad(\qp)$ is surjective, and in fact $\agp$ acts trivially on $\bun_G$ in the $2$-categorical sense. \smallskip 

It follows from Claim \ref{Claim:Twistingtorsor} that $\igscorr \to \bun_{G} \times_{\bun_{G'}} \bun_{G}$ is $\ul{\apg}$-equivariant for the trivial action of $\apg$ on the target, and so the same is true for the composition with $\bun_{G} \times_{\bun_{G'}} \bun_{G} \to \bun_{G_2} \times_{\bun_{G_2'}} \bun_{G_2'}$. The lemma follows immediately from this.  
\end{proof}
It follows from Lemma \ref{Lem:TwoCategoryTheoryCube} that the natural maps $\igscorrad \to \igs \gxtwo, \igs \gxtwop$ factor through $\igs \gxtwo_{-\mu'_2}, \igs \gxtwop_{-\mu_2}$ respectively. 
\begin{Cor} \label{Cor:AdjointCorrespondence}
Let $\Theta_2=(\g_2, \x_2, \mathsf{P}_2, \x_{2}')$ be as above. Let $\Sigma=(\g, \x, \sigma, \mathsf{P}, \omega,\x')$ be an excellent Hodge type lifting of $\Theta_2$ together with a good Hodge datum $\Omega$ for $\gx$ and $\mathsf{P}$ such that $\mathcal{G}$ is reductive and chosen to match $\mathcal{G}'$ under $\kappa_p$. If Conjecture \ref{Conj:XiaoZhu} holds for $\gx,\Omega$, then the natural maps $\igscorrad \gxtwo \to \igs \gxtwo_{-\mu'_2}, \igs \gxtwop_{-\mu_2}$ are isomorphisms.
\end{Cor}
\begin{proof}[Proof of Corollary \ref{Cor:AdjointCorrespondence}]
By Theorem \ref{Thm:XZCorrespondences}, the result holds for $\igscorr$ under our assumptions. To show that the result holds for $\igscorrad \gxtwo$, we need to show that the natural maps
\begin{align}
    \igs \gx_{-\mu'} \times \apgtwo \to \igs \gxtwo_{-\mu_2'} \\
    \igs \gxp_{-\mu} \times  \apgtwo \to \igs \gxtwop_{-\mu_2}
\end{align}
are torsors for $\ul{\apg}$. This follows from Lemma \ref{Lem:AdjointIgusaTorsor} and Lemma \ref{Lem:Intersection}.
\end{proof}
\begin{Lem} \label{Lem:AdjointCorrespondenceFunctorialityI}
Let $\Theta_2$ be as above and let $\Sigma_4 \to \Sigma$ be a morphism of excellent Hodge type liftings of $\Theta_2$. There is an inclusion $\igscorradfour \gxfour \subset \igscorrad \gxtwo$ of subsheaves of $\igs \gxtwo \times \igs\gxtwop$. 
\end{Lem}
\begin{proof}
This is a direct consequence of Proposition \ref{Prop:IntegralFunctoriality}.
\end{proof}

\subsubsection{Functoriality of abelian type correspondences} Let $\Theta_2$ and $\Theta_3$ as in Section \ref{sub:AdjointCorrespondencesI}. Suppose that we are given a morphism $(f,\lambda):\Theta_2 \to \Theta_3$. We would like to prove a version of Proposition \ref{Prop:IntegralFunctoriality} for the correspondences defined in Section \ref{sub:AdjointCorrespondencesI}. We let $\Sigma$ be an excellent Hodge type lifting of $\Theta_2$ and $\Sigma_1$ be an excellent Hodge type lifting of $\Theta_3$. 
\begin{Lem} \label{Lem:AdjointCorrespondenceFunctorialityII}
If there is a morphism $(g, \lambda_1):\Sigma \to \Sigma_1$ over $(f, \lambda)$, then there exists a (necessarily unique) commutative diagram of v-sheaves
\begin{equation}
    \begin{tikzcd}
        \igscorrad \arrow{d}{f_{\lambda}} \arrow{r} & \igs\gxtwo \times \igs\gxtwop \arrow{d}{f \times f_{\lambda}} \\
        \igscorradone \arrow{r} & \igs \gxthree \times \igs \gxthreep
    \end{tikzcd}
\end{equation}
Moreover, this diagram fits into a three-dimensional $2$-commutative diagram with the diagram \eqref{Eq:DiagramCompatibilityBunG}. 
\end{Lem}
\begin{proof}
Our commutative diagram of Shimura data induces a commutative diagram of Deligne groups since it maps centers to centers. By Proposition \ref{Prop:IntegralFunctoriality}, there are induced morphisms of correspondence spaces for $\gx \to \gxone$ which are all $\ag \to \agone$-equivariant by Proposition \ref{Prop:AgEquivariance} (the maps of Igusa stacks are $\ag \to \agone$-equivariant by the functoriality of Igusa stacks \cite[Theorem B]{KimFunctoriality}, and the correspondences are stable under this action; the induced map is thus also equivariant). The lemma now follows from the definition of the correspondences for $\gxtwo$ and $\gxthree$. 
\end{proof}

\subsection{Fixed points of some abelian type Igusa stacks} 
In this section we show an analogue of \cite[Theorem 6.2.1]{vHSemplinerFixedPoints} for Igusa stacks of abelian type. We fix a Shimura datum $\gxtwo$ of abelian type and for a totally real field $\f$, we let $\h_2=\operatorname{Res}_{\f/\Q} \g_{2,\f}$ and let $\hytwo$ be as in \cite[First paragraph of Section 3]{vHSemplinerFixedPoints}. If $\f$ is Galois over $\mathbb{Q}$ with Galois group $\Gamma$, then $\Gamma$ acts on $\hytwo$ by morphisms of Shimura data, and thus acts on $\igs \hytwo$ by \cite[Theorem B]{KimFunctoriality}. 
\begin{Thm} \label{Thm:FixedPointsAdjointIgusaStacks}
If the center $\mathsf{Z}_{\g_{2}}$ has $\mathbb{R}$-split rank zero and $\Sha^1(\mathbb{Q},\g_2) \to \Sha^1(\mathbb{Q},\h_2)$ is injective, then the natural map
    \begin{align}
         \igs \gxtwo \to  \igs \hytwo^{h \Gamma} \times_{\bun_{H_2}^{h \Gamma}} \bun_{G_2, -\mu_2}
    \end{align}
    is an isomorphism, where the superscript $h \Gamma$ denotes taking homotopy fixed points as in \cite[Definition A.1.8]{vHSemplinerFixedPoints}.
\end{Thm}
\begin{proof}
    We largely follow the proof of \cite[Theorem 6.2.1]{vHSemplinerFixedPoints}. It suffices to show that the natural map is an isomorphism after basechanging via the v-cover
    \begin{align}
        \operatorname{Gr}_{G_2, -\mu_2} \to \bun_{G_2, -\mu_2}.
    \end{align}
    By the fiber product diagram and \cite[Lemma A.1.11]{vHSemplinerFixedPoints}, we may identify this with the map
    \begin{align} \label{Eq:NaturalMapToFixedPointsBaseChange}
        \mathbf{Sh}\gxtwo^{\circ, \lozenge} \to \left(\mathbf{Sh}\hytwo^{\circ, \lozenge}\right)^{\Gamma} \times_{\operatorname{Gr}_{H_2, -\mu_2}^{\Gamma}} \operatorname{Gr}_{G_2, -\mu_2}.
    \end{align}
    Since $\operatorname{Gr}_{G_{2}, -\mu_2} \to \operatorname{Gr}_{H_2, -\mu_2}^{\Gamma}$ is an isomorphism by \cite[Proposition 2.4.7]{vHSemplinerFixedPoints}, this is just the natural map $\mathbf{Sh}\gxtwo^{\circ, \lozenge}\to \left(\mathbf{Sh}\hytwo^{\circ, \lozenge}\right)^{\Gamma}$. The natural map $\mathbf{Sh}\gxtwo \to \mathbf{Sh}\hytwo^{\Gamma}$ is a closed immersion of reduced Jacobson $\mathbb{C}$-schemes (see the proof of \cite[Lemma 3.6.7]{vHSemplinerFixedPoints}), which is a bijection on $\mathbb{C}$-points by \cite[Corollary 3.5.2]{vHSemplinerFixedPoints} and thus an isomorphism by the proof of \cite[Lemma 3.6.7]{vHSemplinerFixedPoints}. This implies that the natural map $\mathbf{Sh}\gxtwo^{\lozenge} \to \left(\mathbf{Sh}\hytwo^{\lozenge}\right)^{\Gamma}$ is an isomorphism, and using \cite[Lemma 6.2.3]{vHSemplinerFixedPoints}, we moreover deduce that 
    \begin{align}
        \mathbf{Sh}\gxtwo^{\circ, \lozenge} \to \left(\mathbf{Sh}\hytwo^{\circ, \lozenge}\right)^{\Gamma}
    \end{align}
    is an isomorphism. In particular, it follows that the map in \eqref{Eq:NaturalMapToFixedPointsBaseChange} is an isomorphism. By v-descent, this implies that the natural map of the theorem is an isomorphism.
\end{proof}

\subsection{Correspondences in some abelian type cases II} Let $\Theta_{2}=(\g_2,\x_2,\mathsf{P}_2, \x_2')$ be a correspondence datum. Given a totally real field $\f$, there is a new correspondence datum $\Theta_{2,\f}=(\h_{2}, \y_{2}, \mathsf{P}_{2,\f}, \y_{2}')$, where $\h_{2}=\operatorname{Res}_{\f/\mathbb{Q}} \g_{2,\f}$ and $\mathsf{P}_{2,\f}$ is the pushout of $\mathsf{P}_{2}$ along $\g_{2} \to \h_{2}$, and where $\y_{2}, \y_{2}'$ are the induced Shimura datum along $\g_2 \to \h_2$ and $\g_2' \to \h_2'$. To be precise, it is not automatic that $\Sha^1(\mathbb{Q},\h_{2})=0$, so we will work under this assumption. 

\subsubsection{} If $\Sigma$ is a Hodge type lifting of $\Theta_{2}$, then there is a Hodge type lifting $\Sigma_{\f}$ of $\Theta_{2,\f}$ given by $\Sigma_{\f}=(\h_{1}, \y_{1}, \sigma_{\f}, \mathsf{P}_{\f}, \omega_{\f}, \y_{1}')$, where $\h_{1} \subset \operatorname{Res}_{\f/\mathbb{Q}} \g_{f}$ and $\hyo$ are as in Section \ref{Subsub:PSConstruction}, where
\begin{align}
    \sigma_{\f}=\restr{\operatorname{Res}_{\f/\mathbb{Q}} \sigma_{\f}}{\h_{1}} \\
    \mathsf{P}_{\f} = \mathsf{P} \times^{\g} \h_{1}
\end{align}
and $\omega_{\f}$ is the induced isomorphism. 

\begin{Thm} \label{Thm:CorrespondenceAdjoint}
Let $\Theta_{2}=(\g_2, \x_2,\mathsf{P}_2,\x_2')$ be a correspondence datum such $\mathsf{Z}_{\g_{2}}$ has $\mathbb{R}$-split rank zero. If there is a Hodge type lifting $\Sigma$ of $\Theta_{2}$ and a Galois totally real field $\f$ in which $p$ is unramified such that $\Sigma_{\f}$ is excellent, and such that Conjecture \ref{Conj:ConstructionWorks} holds for $\Sigma_{\f}$ and a good Hodge datum $\Omega$ for $\Sigma_{\f}$, then there is an $\underline{\gafp}$-equivariant isomorphism
\begin{align}
    \igs \gxtwo_{-\mu_2'} \to \igs \gxtwop_{-\mu_2}
\end{align}
over $\bun_{G_2'}$.
\end{Thm}
\begin{proof}
If $\mathsf{Z}_{\g_{2}}$ has $\mathbb{R}$-split rank zero, then by Corollary \ref{Cor:AdjointCorrespondence}, the natural maps $\operatorname{IgsCorr}_{\Theta_{2,\f}, \Sigma_{\f}} \to \igs \hytwo_{-\mu'_{2,H}}, \igs \hytwop_{-\mu_{2,H}}$ are isomorphisms. Recall that $\Gamma$ denotes $\gal(\f/\mathbb{Q})$.
\begin{Claim} \label{Claim:GammaStable}
The subsheaf
\begin{align}
    \operatorname{IgsCorr}_{\Theta_{2,\f}, \Sigma_{\f}} \subset \hytwo_{-\mu'_{2,H}} \times \hytwop_{-\mu_{2,H}}
\end{align}
is $\Gamma$-stable.
\end{Claim}
\begin{proof}
This is a direct consequence of Lemma \ref{Lem:AdjointCorrespondenceFunctorialityII}, because there is an action of $\Gamma$ on the tuple $\Sigma_{\f}=(\h_1, \y_1, \sigma_{\f}, \mathsf{P}_{\f}, \omega_{\f}, \y')$ since it is induced from the tuple $\Sigma$. 
\end{proof}
It follows from Claim \ref{Claim:GammaStable} that the natural maps above are $\Gamma$-equivariant. It now follows from Theorem \ref{Thm:FixedPointsAdjointIgusaStacks} that after taking $\Gamma$-homotopy fixed points and basechanging there is an induced isomorphism
\begin{align}
    \igs \gxtwo_{-\mu'} \to \igs \gxtwop_{-\mu}.
\end{align}
\end{proof}

\subsubsection{} Finally, we prove a Hodge type variant of Theorem \ref{Thm:CorrespondenceAdjoint}. Let $\Theta_2=(\g_2, \x_2, \mathsf{P}_2,\x_2')$ be a correspondence datum with $\gxtwo$ of Hodge type, and choose the trivial Hodge-type lifting $\Sigma=(\g_2, \x_2, \operatorname{Id}, \mathsf{P}_2, \operatorname{Id}, \x_2')$. 
\begin{Thm} \label{Thm:HodgeTypeCorrespondences}
If there is a Galois totally real field $\f$ in which $p$ is unramified such that $\Sigma_{\f}$ is excellent, and such that Conjecture \ref{Conj:ConstructionWorks} holds for $\Sigma_{\f}$ and a good Hodge datum $\Omega$ for $\Sigma_{\f}$, then there is an $\underline{\gafp}$-equivariant isomorphism
\begin{align}
    \igs \gxtwo_{-\mu_2'} \to \igs \gxtwop_{-\mu_2}
\end{align}
over $\bun_{G_2'}$.
\end{Thm}
\begin{proof}
The proof of Theorem \ref{Thm:CorrespondenceAdjoint} can be adapted  with Theorem \ref{Thm:FixedPointsAdjointIgusaStacks} replaced by \cite[Theorem 6.2.1]{vHSemplinerFixedPoints}. 
\end{proof}

\begin{Rem}
    In Section \ref{Sec:Examples}, we will show that there are many examples in which the hypotheses of Theorems \ref{Thm:CorrespondenceAdjoint} and \ref{Thm:HodgeTypeCorrespondences} are satisfied. For example, let $\gxtwo$ be as in Section \ref{Sub:ExampleBDR} or \ref{Sub:ExampleC} and not of type $D$. Note that $\Sha^1(\f, \g_2)=0$ for any number field $\f$, and that the center of $\g_2$ is finite, and thus of $\mathbb{R}$-split rank zero. Let $\Theta_{2}=(\g_2, \x_2, \mathsf{P}_2, \x_{2}')$ be a correspondence datum. It follows from Corollary \ref{Cor:ExistenceHodgeTypeBCVeryGood} that we can find a Hodge type lifting $\Sigma$ and a Galois totally real field $\f$ satisfying the assumptions of Theorem \ref{Thm:CorrespondenceAdjoint}, except perhaps the validity of Conjecture \ref{Conj:ConstructionWorks} for a choice of good Hodge datum $\Omega$ for $\Sigma_{\f}$. If we choose a good Hodge datum $\Omega$ for $\Sigma_{\f}$ with reductive parahorics (which is always possible), then Conjecture \ref{Conj:ConstructionWorks} follows from Conjecture \ref{Conj:XiaoZhu}, whose proof has been announced by Xiao--Zhu. 
\end{Rem}

\subsection{On Igusa varieties for some abelian type Shimura varieties} The goal of this section is to analyze Igusa varieties for Shimura varieties of abelian type. Given $b \in G_2(\qpbr)$, we define the v-sheaf Igusa variety as
\begin{equation}
    \begin{tikzcd}
        \ig_{b, \mathrm{v}}\gxtwo \arrow{r}{x} \arrow{d} & \igs \gxtwo \arrow{d} \\
        \spd \fpbar \arrow{r}{b} & \bun_{G_{2}},
    \end{tikzcd}
\end{equation}
and the Igusa variety over $\spec \fpbar$ by
\begin{equation}
        \begin{tikzcd}
        \operatorname{Ig}_{b}\gxtwo \arrow{r}{x} \arrow{d} & \igs \gxtwo_{}^{\mathrm{red}} \arrow{d} \\
        \spec \fpbar \arrow{r}{b} & \gtwoisoc. 
    \end{tikzcd} 
\end{equation}
Note that $\operatorname{Ig}_{b}\gxtwo$ is the reduction of $\ig_{b,\mathrm{v}}\gxtwo$, that $\tilde{G}_b$ acts on $\ig_{b,\mathrm{v}}\gxtwo$ and that $\ul{G_b(\qp)}$ acts on $\operatorname{Ig}_{b}\gxtwo$ in a compatible way. The main result of this section is the following theorem. Let $K^p_{2} \subset \g_2(\afp)$ be a neat compact open subgroup.
\begin{Thm} \label{Thm:IgusaVarietiesAbelianType}
Suppose that $\gx$ satisfies Milne's axiom SV5. The sheaf $\operatorname{Ig}_{b}\gxtwo$ is representable by a perfect scheme over $\fpbar$ and for all sufficiently small compact open subgroups $J_2 \subset G_{2,b}(\qp)$, the quotient
\begin{align}
    \operatorname{Ig}_{b}\gxtwo / \ul{K^p_{2} \times J_2}
\end{align}
is representable by a perfectly of finite type perfect scheme of pure dimension $\langle 2 \rho, \nu_{b}\rangle$. Moreover, the natural map
\begin{align}
    \operatorname{Ig}_{b}\gxtwo^{\diamond} \to \ig_{b,\mathrm{v}}\gxtwo
\end{align}
is an open immersion which induces a bijection on canonical compactifications. 
\end{Thm}

\subsubsection{} Let $\gx \to \gxtwo$ be an ad-isomorphism of Shimura data of abelian type. Fix a place $v$ of the reflex field $\mathsf{E}$ of $\gx$ inducing a place $v_2$ of $\mathsf{E}_2 \subset \mathsf{E}$, let $E$ and $E_2$ be their completions. Let $C$ be the completion of an algebraic closure of $E$ and let $x \in \grgmu(C)$. Consider the fiber product
\begin{equation}
    \begin{tikzcd}
        \ig_x\gx \arrow{r} \arrow{d} & \spd C \arrow{d}{x} \\
        \mathbf{Sh}\gx^{\circ, \lozenge} \arrow{r}{\pi_{\mathrm{HT}}} & \grgmu
    \end{tikzcd}
\end{equation}
and define $\ig_x\gxtwo$ similarly. Our first goal is to understand the map $\ig_x \gx \to \ig_x \gxtwo$. 

\subsubsection{} \label{subsub:CentralLeavesGeneric} Let $K=K^pK_p \subset \gaf$ and $K_2=K_{2}^p K_{2,p} \subset \g_2(\af)$ be neat compact open subgroups with $K$ mapping to $K_2$. We now assume that $\g \to \g_{2}$ is surjective and that $\gx$ and $\gxtwo$ satisfy SV5. We consider the stabilizers $K_{p,x} \subset K_{p}$ and $K_{p,2,x} \subset K_{p,2}$.
\begin{Lem} \label{Lem:FiniteEtaleIgusaIII}
The natural
\begin{align}
    \ig_x\gx / \ul{K^p \times K_{p,x}} \to \ig_x\gxtwo / \ul{K_2^p \times K_{p,2,x}}
\end{align}
is finite \'etale.
\end{Lem}
\begin{proof}
Since $G \to G_2$ is surjective, it follows that $G(\qp) \to G_2(\qp)$ is open. Therefore, after replacing $K_{p,2}$ with a finite index subgroup, we may assume that $K_{p}$ surjects onto $K_{p,2}$. It follows from the Igusa stack diagram that we have a $2$-Cartesian diagram
\begin{equation}
    \begin{tikzcd}
        \ig_x\gx / \ul{K^p} \arrow{r} \arrow{d} & \spd C \arrow{d}\\
        \mathbf{Sh}_K\gx^{\circ, \lozenge}_{[x]} \arrow{r} & \left[ \spd C / \ul{K_{p,x}}\right],
    \end{tikzcd}
\end{equation}
where $\mathbf{Sh}_K\gx^{\circ, \lozenge}_{[x]}$ is the inverse image in $\mathbf{Sh}_K\gx^{\circ, \lozenge}_{C}$ of $$\left[O(x) / \ul{K_p} \right] \subset \left[\grgmu/\ul{K_{p}}\right],$$ with $O(x)$ the image of the orbit map through $x$ of $\ul{K_p}$. There is a similar diagram for $\ig_x\gx / \ul{K_2^p}$ and the natural map $\mathbf{Sh}_K\gx^{\circ, \lozenge}_{[x]} \to \mathbf{Sh}_{K_2}\gxtwo^{\circ, \lozenge}_{[x]}$ is finite \'etale because $\mathbf{Sh}_K\gx^{\circ, \lozenge} \to \mathbf{Sh}_{K_2}\gxtwo^{\circ, \lozenge}$ is finite \'etale; this proves the lemma.
\end{proof}

\subsubsection{} Let $Z$ be the kernel of $G \to G_2$. 
\begin{Cor} \label{Cor:FiniteEtaleIgusaIV}
For $U \subset Z(\qp)$ a compact open subgroup, the natural map 
\begin{align}
    \ig_x\gx / \ul{K^p \times U} \to \ig_x\gxtwo / \ul{K_2^p}
\end{align}
is finite \'etale. 
\end{Cor}
\begin{proof}
After replacing $U$ by $K_{p,x} \cap Z(\qp)$ and replacing $K_{p,2,x}$ by the image of $K_{p,2}$, both of which induce finite \'etale maps, it follows from Lemma \ref{Lem:FiniteEtaleIgusaIII} and the Cartesian diagram
\begin{equation}
    \begin{tikzcd}
    \ig_x\gx / \ul{K^p \times U} \arrow{r} \arrow{d} &  \ig_x\gxtwo / \ul{K_2^p} \arrow{d} \\
        \ig_x\gx / \ul{K^p \times K_{p,x}} \arrow{r} &  \ig_x\gxtwo / \ul{K_2^p \times K_{p,2,x}}.
    \end{tikzcd}
\end{equation}
\end{proof}
\begin{Prop} \label{Prop:FiniteEtaleIgusaVariety}
    For $J \subset G_b(\qp)$ compact open mapping to $J_2 \subset G_{2,b}(\qp)$ compact open, the natural map
    \begin{align}
         \ig_x\gx / \ul{K^p \times J} \to \ig_x\gxtwo / \ul{K_2^p \times J_2}
    \end{align}
    is finite \'etale. 
\end{Prop}
\begin{proof}
After replacing $J_2$ by the image of $J$ and setting $U= J_2 \cap Z(\qp)$, it follows from Corollary \ref{Cor:FiniteEtaleIgusaIV} and the Cartesian diagram   
\begin{equation}
    \begin{tikzcd}
    \ig_x\gx / \ul{K^p \times U} \arrow{r} \arrow{d} &  \ig_x\gxtwo / \ul{K_2^p} \arrow{d} \\
        \ig_x\gx / \ul{K^p \times J} \arrow{r} &  \ig_x\gxtwo / \ul{K_2^p \times J_{2}}.
    \end{tikzcd}
\end{equation}
\end{proof}

\subsubsection{} The v-sheaf $\ig_x \gx$ is a locally spatial diamond, and so it has a locally profinite set of connected components $\pi_0(\ig_x \gx)$, see \cite[Proposition 11.19.(i)]{EtCohDiam}. 
\begin{Lem} \label{Lem:FiniteEtaleIgusaV}
Suppose that $\g_2(\afp)$ acts transitively on $\pi_0(\mathbf{Sh}_{K_{p,2}}\gxtwo)$ for all $K_{p,2}$. Then $\ig_x\gxtwo / \ul{K_2^p \times K_{p,2,x}}$ is a finite union of $\g_2(\afp)$-translates of the (open and closed) image of $\ig_x\gx / \ul{K^p \times K_{p,x}}$. Moreover, if the locally profinite set $\pi_0(\ig_x\gx / \ul{K^p \times K_{p,x}})$ is quasi-compact, then $\pi_0(\ig_x\gxtwo / \ul{K_2^p \times K_{p,2,x}})$ is quasi-compact. 
\end{Lem}
\begin{proof}
Note that $\mathbf{Sh}_{K_2}\gxtwo$ is a finite union of $\g_2(\afp)$ translates of the image of $\mathbf{Sh}_{K}\gx \to \mathbf{Sh}_{K_2}\gxtwo$, the image being an open and closed union of connected components. Using \cite[Lemma 4.1.3]{DvHKZIgusaStacksII}, the same result holds for $\mathbf{Sh}_{K_2}\gxtwo^{\circ, \lozenge}$. This immediately implies that $\ig_x\gxtwo / \ul{K_2^p \times K_{p,2,x}}$ is a finite union of $\g_2(\afp)$-translates of the image of $\ig_x\gx / \ul{K^p \times K_{p,x}}$. These all have quasicompact $\pi_0$ since they are finite \'etale quotients of things with quasicompact $\pi_0$. Thus we deduce that $\ig_x\gxtwo / \ul{K_2^p \times K_{p,2,x}}$ has quasicompact $\pi_0$. 
\end{proof}

\begin{Lem} \label{Lem:FiniteEtaleIgusaVI}
Suppose that $\g_2(\afp)$ acts transitively on $\pi_0(\mathbf{Sh}_{K_{p,2}}\gxtwo)$ for all $K_{p,2}$.  If the v-sheaf $\ig_x\gx / \ul{K^p \times J}$ has finitely many connected components, then so does $\ig_x\gxtwo / \ul{K^p_{2} \times J_2}$. 
\end{Lem}
\begin{proof}
Since $\ig_x\gx / \ul{K^p \times J}$ has finitely many connected components, it follows that $\pi_0(\ig_x \gx)$ is quasi-compact, and thus that $\pi_0(\ig_x\gx / \ul{K^p \times K_{p,x}})$ is quasi-compact. It follows from Lemma \ref{Lem:FiniteEtaleIgusaV} that $\pi_0(\ig_x\gxtwo / \ul{K_2^p \times K_{p,2,x}})$ is quasi-compact and thus that $\pi_0(\ig_x\gxtwo)$ is quasi-compact. \smallskip 

Since the natural map $\ig_x\gx / \ul{K^p \times J} \to \ig_x\gxtwo / \ul{K^p_{2} \times J_2}$ is finite \'etale by Proposition \ref{Prop:FiniteEtaleIgusaVariety}, we see that the image of 
\begin{align}
    \pi_0(\ig_x\gx / \ul{K^p_{2} \times J}) \to \pi_0(\ig_x\gxtwo / \ul{K^p_{2} \times J_2})
\end{align}
is open and closed and finite. By Lemma \ref{Lem:FiniteEtaleIgusaV}, the rest of $\pi_0(\ig_x\gxtwo / \ul{K^p_{2} \times J_2})$ is covered by the $\g_2(\afp)$-translates of this image. Since $\pi_0(\ig_x\gx / \ul{K^p \times J})$ is quasi-compact, we only need finitely many translates, and thus $\pi_0(\ig_x\gxtwo / \ul{K^p_{2} \times J_2})$ is finite. 
\end{proof}

\subsubsection{} Our choice of $b$ defines a v-sheaf Igusa variety over $\spd \fpbar$ by
\begin{equation}
    \begin{tikzcd}
        \ig_{b,\mathrm{v}}\gx \arrow{r}{x} \arrow{d} & \igs \gx \arrow{d} \\
        \spd \fpbar \arrow{r}{b} & \bun_{G},
    \end{tikzcd}
\end{equation}
and an Igusa variety over $\spec \fpbar$ by
\begin{equation}
        \begin{tikzcd}
        \ig_{b}\gx \arrow{r}{x} \arrow{d} & \igs \gx^{\mathrm{red}} \arrow{d} \\
        \spec \fpbar \arrow{r}{b} & \gisoc. 
    \end{tikzcd} 
\end{equation}
Since $\igs \gx \to \bun_{G}$ is $0$-truncated, and the second $2$-Cartesian diagram is the reduction of the first, we find that $\ig_{b,\mathrm{v}}\gx$ and $\ig_{b}\gx$ are both sheaves. There is moreover a $\tilde{G}_b \times \ul{\gafp}$-equivariant isomorphism $\ig_{b,\mathrm{v}}\gx \times_{\spd \fpbar} \spd C \xrightarrow{\sim} \ig_x\gx$.
\begin{Lem} \label{Lem:ReductionFiniteEtale}
    If $X \to Y$ is a finite \'etale surjective morphism of v-sheaves on $\perf$, then $X^{\mathrm{red}} \to Y^{\mathrm{red}}$ is a finite \'etale surjective morphism of v-sheaves on $\affperf$.
\end{Lem}
\begin{proof}
By \cite[Theorem 1.5]{BhattScholze} and fpqc descent for finite \'etale morphisms, it suffices to show that for all $w$-contractible $A \in \affperf$ and all maps $f:\spec A \to Y^{\mathrm{red}}$, the natural map $\spec A \times_{Y^{\mathrm{red}}} X^{\mathrm{red}} \to \spec A$ is finite \'etale surjective. The fiber product $\spd A \times_{Y} X$ is finite \'etale over $\spd A$ and thus of the form $\spd B \to \spd A$ for a finite \'etale $A$-algebra $B$ by \cite[Theorem 1.4]{kim2026descendingfiniteprojectivemodules}. Since reduction commutes with fiber products, we see that $\spec A \times_{Y^{\mathrm{red}}} X^{\mathrm{red}} \simeq \spec B$ is finite \'etale over $\spec A$, showing that $X^{\mathrm{red}} \to Y^{\mathrm{red}}$ is finite \'etale. Since $A$ is strictly $w$-local, the map $\spec B \to \spec A$ has a section, showing that $\spec A \to Y^{\mathrm{red}}$ lifts to $\spec A \to X^{\mathrm{red}}$, proving surjectivity.
\end{proof}
\begin{Lem} \label{Lem:ReductionPRoFiniteEtale}
If $X \to Y$ is a torsor on $\perf$ for a profinite group $H$, then $X^{\mathrm{red}} \to Y^{\mathrm{red}}$ is a torsor on $\affperf$ for $H$. 
\end{Lem}
\begin{proof}
The reduction is clearly a quasi-torsor for $\ul{H}=\ul{H}^{\mathrm{red}}$. To prove surjectivity, we can argue as in the proof of Lemma \ref{Lem:ReductionFiniteEtale} by writing $H= \varprojlim_n H_n$ as an inverse limit of finite groups, which we can use to write $H$-torsors over $\spd A$ as inverse limits of finite \'etale covers of $\spd A$. 
\end{proof}
\begin{Cor} \label{Cor:FiniteEtaleIgusaVarietyScheme}
For $J \subset G_b(\qp)$ compact open mapping to $J_2 \subset G_{2,b}(\qp)$ compact open, the natural map (of sheaves on $\affperf$)
    \begin{align}
         \operatorname{Ig}_{b} \gx / \ul{K^p \times J} \to \operatorname{Ig}_b\gxtwo / \ul{K_2^p \times J_2}
    \end{align}
    is finite \'etale. 
\end{Cor}
\begin{proof}
By Lemma \ref{Lem:ReductionPRoFiniteEtale}, the map of the corollary is the reduction of the map of v-sheaves on $\perf$, see \cite[Corollary 9.11]{EtCohDiam}, given by
\begin{align}
         \operatorname{Ig}_{b, \mathrm{v}} \gx / \ul{K^p \times J} \to \operatorname{Ig}_{b,\mathrm{v}}\gxtwo / \ul{K_2^p \times J_2}
\end{align}
By Proposition \ref{Prop:FiniteEtaleIgusaVariety} and Lemma \ref{Lem:ReductionFiniteEtale}, together with v-descent for finite \'etale morphisms, this map is finite \'etale. 
\end{proof}

\begin{proof}[Proof of Theorem \ref{Thm:IgusaVarietiesAbelianType}]
Let $\gx$ be an auxiliary Shimura datum of Hodge type. Then there is a diagram of surjective morphisms of Shimura data
\begin{equation}
    \begin{tikzcd}
        \gxtwo \arrow{dr} && \gx \arrow{dl} \\
        & \gxad.
    \end{tikzcd}
\end{equation}
The theorem holds for $\gx$ by \cite[Theorem VIII, Lemma 4.4.2, Proposition 8.7.1]{DvHKZIgusaStacks} and \cite[Proposition 5.14.(4)]{MaoCompact}. Using Corollary \ref{Cor:FiniteEtaleIgusaVarietyScheme}, we see that the natural map
\begin{align} \label{eq:NaturalMapIgusaDevissage}
         \operatorname{Ig}_{b} \gx / \ul{K^p \times J} \to \operatorname{Ig}_{b}\gxad / \ul{K^{p, \mathrm{ad}} \times J_2^{\mathrm{ad}}}
\end{align}
is finite \'etale. Since finite \'etale quotients of (perfections of) quasiprojective varieties of pure dimension $\langle 2 \rho, \nu_b \rangle$ are again quasiprojective varieties of pure dimension $\langle 2 \rho, \nu_b \rangle$, we deduce that the theorem holds for the (open and closed) image of the natural map in \eqref{eq:NaturalMapIgusaDevissage}. The rest of $\operatorname{Ig}_{b}\gxad / \ul{K^{p, \mathrm{ad}} \times J^{\mathrm{ad}}}$ is covered by the open and closed $\gad(\afp)$-translates of this image, see Lemma \ref{Lem:FiniteEtaleIgusaV}.\footnote{Note that $\gx(\afp)$ acts transitively on $\pi_0(\mathbf{Sh}\gxad)$ because of weak approximation for adjoint groups, see \cite[Theorem 7.8 on p. 415]{PR}, and the description of $\pi_0(\mathbf{Sh}\gxad)$ of \cite[Section 2.1.3]{DeligneVarietes}.} Thus it suffices to show that $\operatorname{Ig}_{b}\gxtwo / \ul{K_2^p \times J_2}$ has finitely many connected components. This follows from Lemma \ref{Lem:FiniteEtaleIgusaVI} together with the following three facts: Taking $\pi_0$ is unaffected by taking diamonds, partial compactification, and base change from $\spd \fpbar$ to $\spd C$. The latter follows e.g. from \cite[Theorem 1.13.(ii)]{EtCohDiam}. The first fact holds because taking $\pi_0$ agrees with hom into the perfect scheme $\spec \fpbar \coprod \spec \fpbar$, and that the diamond functor is fully faithful. The second fact follows from the partial properness of $\spd \fpbar \coprod \spd \fpbar$. \smallskip 

The theorem for $\gxtwo$ now follows directly from Corollary \ref{Cor:FiniteEtaleIgusaVarietyScheme}, since finite \'etale covers of perfectly finite type schemes that are equidimensional of dimension $\langle 2 \rho, \nu_b \rangle$ are again perfectly finite type schemes that are equidimensional of dimension $\langle 2 \rho, \nu_b \rangle$. We moreover note that the finite level Igusa varieties for $\gxad$ are perfections of quasi-projective varieties, showing that the same holds for the finite level Igusa varieties for $\gxtwo$. 

\subsubsection{Basic Igusa varieties} Let $\gxtwo$ be an abelian type Shimura datum as before with a place $v$ above $p$ of the reflex field. Let $\g'_2$ be the inner form of $\g_2$ of \cite[Proposition 3.1]{HansenMiddle}, and choose isomorphisms $\g_2' \otimes \afp \xrightarrow{\sim} \g_2 \otimes \afp$ and $\g_2' \otimes \qp \xrightarrow{\sim} \g_{2,b}$, where $b \in G(\qpbr)$ is a representative of the basic element of $\bgmu$. Let $\mathcal{M}_{G_2,b,\mu,\infty}$ be the infinite level Rapoport--Zink space associated to $(G_2,b,\mu)$ with its action of $\tilde{G}_b \times \ul{G(\qp)}$. Consider the following version of \cite[Definition 3.2]{HansenMiddle}.\footnote{It is expected that $\mathbf{Sh}\gxtwo^{\circ,[b]}=\mathbf{Sh}\gxtwo^{[b]}$, i.e., that the basic locus does not meet the boundary, and thus that our definition is equivalent to \cite[Definition 3.2]{HansenMiddle}. This is known in all cases that basic uniformization is known.}
\begin{Def} \label{Def:BasicUniformization}
We say that $\gxtwo$ satisfies \emph{$\circ$ basic uniformization} at $v$ if there is a $\ul{\g_2(\af)}$-equivariant isomorphism
\begin{align}
    \mathbf{Sh}\gxtwo^{\circ,[b]} \xrightarrow{\sim} \varprojlim_{K_{2}^p} \g_2'(\mathbb{Q}) \backslash \left(\mathcal{M}_{G_2,b,\mu,\infty} \times \g_2(\afp) / K_{2}^p \right),
\end{align}
where $\g_2'(\mathbb{Q})$ acts on $\mathcal{M}_{G_2,b,\mu,\infty}$ via its inclusion into $\ul{G_b(\qp)}$, such that: Under this isomorphism, the Hodge--Tate period map $\mathbf{Sh}\gxtwo^{\circ,[b]} \to \grgmu$ is identified with the projection 
\begin{align}
    \varprojlim_{K_{2}^p} \g_2'(\mathbb{Q}) \backslash \left(\mathcal{M}_{G_2,b,\mu,\infty} \times \g_2(\afp) / K_{2}^p \right) \to \mathcal{M}_{G_2,b,\mu,\infty} / \ul{G_b(\qp)} 
\end{align}
followed by the local Hodge--Tate period map.
\end{Def}
As explained in the proof of \cite[Theorem 3.3]{HansenMiddle}, it follows from the work of Xiao--Zhu, Kim, Howard--Pappas, Shen \cite{XiaoZhu}, \cite{KimFunctoriality}, \cite{HowardPappas}, \cite{ShenUniformization} that: If $G_2$ is unramified and $p>2$, that then $\gxtwo$ satisfies $\circ$ basic uniformization at $v$.

\subsubsection{} Choose $\tilde{x} \in \mathcal{M}_{G_2,b,\mu,\infty}(C)$ with image $x \in \grgmu(C)$ under the local Hodge--Tate period map.
\begin{Cor}
Suppose that $\gxtwo$ satisfies $\circ$ basic uniformization at $v$, then there is a $\ul{\g_2'(\af)}$-equivariant isomorphism
\begin{align}
    \ig_x\gxtwo \xrightarrow{\sim} \varprojlim_{K_{2}^p} \g_2'(\mathbb{Q}) \backslash \left(G_b \times \g_2(\afp) / K_{2}^p \right),
\end{align}
where $\g_2'(\mathbb{Q})$ acts on $G_b$ via its inclusion into $G_b(\qp)$.
\end{Cor}
\begin{proof}
It follows from the definition that $\ig_x\gxtwo$ is $\ul{\g_2(\afp)}$-equivariantly identified with
\begin{align}
    \ig_x\gxtwo \xrightarrow{\sim} \varprojlim_{K_{2}^p} \g_2'(\mathbb{Q}) \backslash \left(\ul{G_{2,b}(\qp)} \cdot \tilde{x} \times \g_2(\afp) / K_{2}^p \right).
\end{align}
We now show $G_b(\qp)$-equivariance. From uniformization we have a $G_2(\qp)$-equivariant isomorphism 
\[
    \overline{}{\Psi}: \ig_x\gxtwo \times^{G_b(\qp)} \mathcal{M}_{G_2,b,\mu,\infty} \xrightarrow{\sim}  \varprojlim_{K_{2}^p} \g_2'(\mathbb{Q}) \backslash \ul{\g_2'(\af)}/\ul{K^p} \times^{\ul{G_b(\qp)}} \mathcal{M}_{G_2,b,\mu,\infty}
\]
pulling back to the product we obtain an isomorphism
\[
    \Psi: \ig_x\gxtwo \times \mathcal{M}_{G_2,b,\mu,\infty} \xrightarrow{\sim}  \varprojlim_{K_{2}^p} \g_2'(\mathbb{Q}) \backslash \ul{\g_2'(\af)}/\ul{K^p} \times \mathcal{M}_{G_2,b,\mu,\infty}
\]
which induces a family of isomorphisms indexed by points $\widetilde{x} \in \mathcal{M}_{G_2,b,\mu,\infty}$ denoted by $\psi_{\widetilde{x}}: \ig_x\gxtwo \to \varprojlim_{K_{2}^p} \g_2'(\Q)\backslash\ul{\g_2'(\af)}/K_2^p$. Since the map $\Psi$ is per definition $\underline{G_2(\qp)}$ and $\underline{G_b(\qp)}$-equivariant the maps $\psi_{-}(*)$ satisfy $\psi_{g-}(*) = \psi_{-}(*)$ for all $g \in \ul{G_2(\qp)}$ and $\psi_{h-}(*h^{-1}) = \psi_{-}(*)h^{-1}$ for all $h \in \ul{G_{b}(\qp)}$. On the other hand because the map $\Psi$ is continuous, for fixed $x \in \ig_x\gxtwo$ the map $\widetilde{x} \mapsto \psi_{\widetilde{x}}(x)$ is continuous, and by the $G_2$-equivariance above it factors through $\mathcal{M}_{G_2,b,\mu,\infty}/\ul{G_2(\qp)}$. By the main result of \cite{Connectedness} this quotient is connected, but $\varprojlim_{K_{2}^p}\g_2'(\Q)\backslash\ul{\g_2'(\af)}/\ul{K_2^p}$ is totally disconnected, thus the map is constant. So we deduce that in fact $\psi_{\widetilde{x}}(x)h^{-1} = \psi_{h\widetilde{x}}(xh^{-1}) = \psi_{\widetilde{x}}(xh^{-1})$ for any $h \in \ul{G_{b}(\qp)}$ as desired.
\end{proof}

\begin{Cor} \label{Cor:BasicIgusaVariety}
    Suppose that $\gxtwo$ satisfies \emph{$\circ$ basic uniformization} at $v$, then there is a $\ul{\g_2(\afp)}$-equivariant isomorphism of stacks over $\mathrm{Bun}_{G_2, \fpbar}$
    \[
        \igs\gxtwo_{\fpbar} \xrightarrow{\sim} [\g_2'(\Q)\backslash\ul{\g_2'(\af)}/\ul{G_2(\qp)}].
    \]
\end{Cor}

\smallskip 
\end{proof}

\section{Examples} \label{Sec:Examples}
The goal of this section is to describe a large number of examples where the hypotheses of Theorem \ref{Thm:CorrespondenceAdjoint} are satisfied. Using the classification of Shimura data of abelian type, one can write down all examples of triples $(\g_2, \x_2, \mathsf{P}_2)$ with $\gxtwo$ of adjoint abelian type where Conjecture \ref{Conj:Main} should apply. However, we do not know how to construct Hodge type lifts for all of these, and we are agnostic about the existence of Hodge type lifts in general. Instead, we give examples of Shimura data of types $A,B,C,D^{\mathbb{R}}$ (in the sense of the classification of abelian type Shimura varieties of \cite[Appendix B]{Milne}) where we can construct Hodge type lifts. 
\begin{itemize}
    \item In type $B,C$ we will focus on the adjoint setting and deduce that: There are not many $\mathsf{P}_2$ for which $\g'_2$ admits a Shimura datum, but they all admit good Hodge type liftings satisfying the conditions of Theorem \ref{Thm:CorrespondenceAdjoint}.

    \item In type $D^{\mathbb{R}}$ we will focus on special orthogonal groups rather than their adjoint quotients.\footnote{This is because there are cohomology classes for the adjoint group that we do not know how to lift to Hodge type Shimura varieties.} Here we will similarly find that there are not many $\mathsf{P}_2$ for which $\g'_2$ admits a Shimura datum, but they all admit good Hodge type liftings. Unfortunately these Hodge type liftings do not satisfy the assumptions of Theorem \ref{Thm:CorrespondenceAdjoint} because they do not have connected center. Nevertheless, we have decided to include them for completeness.  

    \item In type $A$, we focus on Shimura data of PEL type. Here there are many examples of $\gx$ and $\mathsf{P}$ for which $\g'$ admits a Shimura datum (of PEL type). Unfortunately, these often fail to satisfy the Hasse principle. Moreover, the induced adjoint Shimura datum admits examples of $\mathsf{P}$ which do not lift to $\gx$. 

\end{itemize}

\subsection{Type \texorpdfstring{$B,C,D^{\mathbb{R}}$}{B,C,DR}} Let $\mL^{+}$ be a totally real field with set of real places $\Sigma_{\infty}$. Let $\g_0, \g_2$ over $\mlp$ be as in Section \ref{Sub:ExampleBDR}, Section \ref{Sub:ExampleC} or Section \ref{Sub:ExampleDR}. We will also use the decomposition $\Sigma_{\infty}=\Sigma_{\infty,c} \coprod \Sigma_{\infty,nc}$ from the appendix. 

\subsubsection{} \label{subsub:PadicTorsorChoiceBC} Now let $p$ be a prime number. Recall that (where the product runs over primes $\mathfrak{p}$ of $\mlp$ over $p$).
\begin{align}
    \rH^1(\qp, G_2) \simeq \pi_1(G_2)_{\gal_{\qp}} \simeq \prod_{\mathfrak{p}} \mathbb{Z}/2 \mathbb{Z}.
\end{align}
Under this isomorphism, the natural map to $\pi_1(\g_2)_{\gal_{\mathbb{Q}}}=\mathbb{Z}/2 \mathbb{Z}$ is given by the direct sum of the maps $\alpha_{\mathfrak{p}}$, where $\alpha_{\mathfrak{p}}$ is the identity map $\mathbb{Z}/2 \mathbb{Z} \to \mathbb{Z}/2\mathbb{Z}$.

\subsubsection{} \label{subsub:RealTorsorChoiceBC} Over the real numbers $\R$ we have an isomorphism
\begin{align}
    \g_{2,\R} \simeq \prod_{\tau \in \Sigma_{\infty}} \g_{2,0} \otimes_{\mlp, \tau} \mathbb{R}
\end{align}
and $\pi_1(\g_{2,\R})=\prod_{\tau \in \Sigma_{\infty}} \mathbb{Z}/2 \mathbb{Z}$. By Lemmas \ref{Lem:KottwitzInvariantCompactForm} and \ref{Lem:KottwitzInvariantShimuraForm}, for every subset $S_{\infty} \subset \Sigma_{\infty,\mathrm{nc}}$ there is a class
\begin{align}
    [P_{S_{\infty}}] \in H^1(\mathbb{R}, \g_{2})
\end{align}
with the following properties:
\begin{itemize}
    \item For $\tau \in S_{\infty}$, the image of $[P_{S_{\infty}}]$ in $\pi_1(\g_{2,0} \otimes_{\mlp, \tau} \mathbb{R})$ is nontrivial. 

    \item For $\tau \not \in S_{\infty}$, the class of $[P_{S_{\infty}}]$ in $H^1(\mathbb{R}, \g_{2,0} \otimes_{\mlp, \tau} \mathbb{R})$ is trivial. 
\end{itemize}

\subsubsection{} \label{subsub:TypeBCohomologyClasses} We now combine the claim with the results of the previous section and the short exact\footnote{Here we are using that $\Sha^1(\mlp,\g_{2,0})=0$. In type $B$ and $C$ this follows because $\g_{2,0}$ is adjoint, see \cite[Theorem 6.22]{PR}). In type $D$ this follows because the map $\g_{0,2}^{\mathrm{sc}} \to \g_{0,2}$ has kernel $\mu_{2}$, see \cite[Remark on page 337]{PR}.} sequence of \cite[Theorem 5.16]{Borovoi}
\begin{align}
    1 \to \rH^1(\mathbb{Q},\g_{2}) \to \bigoplus \rH^1(\mathbb{Q}_v, \g_{2}) \to \pi_1(\g_{2})_{\gal(\qbar)/\Q} \to 1.
\end{align}
Let $S_{\infty} \subset \Sigma_{\infty}$ be as before and let $S_p$ be a subset of $\Sigma_{p}$, the set of primes $\mathfrak{p}$ of $\mlp$ above $p$. We will use $\widetilde{\Sigma_{p}}$ to denote the set of embeddings $\mlp \to \qpbar$, which comes with a forgetful map $\widetilde{\Sigma_{p}} \to \Sigma_{p}$. Finally, we write $S=S_{\infty} \cup S_{p}$.

\begin{Lem} \label{Lem:ExistenceCohomologyTypeBCD}
If $S$ has even cardinality, then there is a unique cohomology class $[\mathsf{P}_S]$ in $H^1(\mathbb{Q},\g_{2})$ such that: It is nontrivial precisely at places $v \in S$ and for $v \in S_{\infty}$ it agrees with the class of Section \ref{subsub:RealTorsorChoiceBC}. 
\end{Lem}
\begin{proof}
    This follows from the computations in Sections \ref{subsub:RealTorsorChoiceBC} and \ref{subsub:PadicTorsorChoiceBC}, together with the Hasse principle and the fact that $\pi_1(\g_2)_{\gal(\qbar)/\Q}=\mathbb{Z}/2 \mathbb{Z}$.
\end{proof}
Let $S$ be as in the statement of Lemma \ref{Lem:ExistenceCohomologyTypeBCD}, choose a $\g_{2}$-torsor $\mathsf{P}_2 \in [\mathsf{P}_S]$ and let $\g_{2}'$ be the corresponding inner form. We have the following important lemma.
\begin{Lem} \label{Lem:ExistenceShimuraDatumTypeB}
    There is a weak Shimura datum $\mathsf{X}_{2}'$ for $\g_{2}'$. 
\end{Lem}
\begin{proof}
We take the weak real Shimura datum $\mathsf{X}'$ for $\g'_{2,\mathbb{R}}$ which equals $\mathsf{X}_{\tau}$ at $\tau \not \in S_{\infty}$ and is trivial at $\tau \in S_{\infty}$. Since $g'_{2,\tau,\R}$ is compact at places $\tau \in S_{\infty}$ by construction, this is indeed a weak real Shimura datum.
\end{proof}

\subsubsection{} Let $S$ and $\mathsf{P}_S$ be as before and consider the Shimura datum $\gxtwop$ with reflex field $\mathsf{E}'_2$. 
Choose an isomorphism $\mathbb{C} \isom \qpbar$ determining $p$-adic places $v,v'$ of the reflex fields $\mathsf{E}_2, \mathsf{E}'_2$ respectively. We can now ask when the set $\bgmumup$ is nonempty. For this, we note that our isomorphism $\mathbb{C} \isom \qpbar$ induces a bijection $\Sigma_{\infty} \to \widetilde{\Sigma_{p}}$. For a prime $\mathfrak{p}$ of $\mlp$ above $p$ we will write $\widetilde{\Sigma}_{\mathfrak{p}}$ for the preimage of $\mathfrak{p}$ in $\widetilde{\Sigma_{p}}$.
\begin{Lem} \label{Lem:NonemptyB}
    The set $B(G_2,P_S,-\mu, -\mu')$ is nonempty if and only if the following conditions hold.
    \begin{itemize}
    
        \item For $\mathfrak{p} \in S_p$, the preimage of $\widetilde{\Sigma}_{\mathfrak{p}}$ in $\Sigma_{\infty}$ intersects $S_{\infty}$ in a set of odd cardinality 

        \item For $\mathfrak{p} \not \in S_p$, the preimage of $\widetilde{\Sigma}_{\mathfrak{p}}$ in $\Sigma_{\infty}$ intersects $S_{\infty}$ in a set of even cardinality.
    \end{itemize}
\end{Lem}
\begin{proof}
It is clear that the condition can be checked for one $\mathfrak{p}$ at a time, so fix one. Then we need to check that 
\begin{align}
    \kappa(\mu) + \kappa([P])  = \kappa(\mu')
\end{align}
in $\pi_1(G_{2})_{\operatorname{Gal}_{\mlp_{\mathfrak{p}}}}=\mathbb{Z} / 2 \mathbb{Z}$. The element $\kappa(\mu) - \kappa(\mu')$ counts the cardinality (modulo $2$) of the intersection of the preimage of $\widetilde{\Sigma}_{\mathfrak{p}}$ in $\Sigma_{\infty}$ with $S_{\infty}$, and $ \kappa([P])$ is $1 \in \mathbb{Z}/2 \mathbb{Z}$ when $\mathfrak{p} \in S_{p}$ and zero else (by construction). The lemma can be read off directly from this. 
\end{proof}

\subsubsection{Hodge type} \label{subsub:HodgetypeB}
Let $\gxtwo$, $S$, $\mathsf{P}_S$,$\gxtwop$ and $\mathbb{C} \isom \qpbar$ be as before. Let $\mL$ be a quadratic totally imaginary extension of $\mlp$ such that all primes above $p$ in $\mlp$ stay inert in $\mL$, and write $\ml=\mlp(\sqrt{\epsilon})$ for a totally negative element $\epsilon \in \mlp$. Let $\Upsilon=(\g_0, \x_0, \iota_0, V, \psi, \epsilon)$, where $\iota_0, V, \psi$ are as in Sections \ref{Sub:ExampleBDR}, \ref{Sub:ExampleC}, \ref{Sub:ExampleDR}. Let $\gx \to \gtwo$ be the Hodge type Shimura datum constructed from $\Upsilon$ in Section \ref{Sub:AbstractHodge}, together with its Hodge embedding $\iota:\gx \to \gvlx$. 
\begin{Prop} \label{Prop:HodgeCoverBC}
There is a $\g$-torsor $\mathsf{P}$ over $\spec \mathbb{Q}$ together with an isomorphism $\mathsf{P} \times^{\g} \g_2 \to \mathsf{P}_{S}$ such that: The group $\g'=\operatorname{Aut}_{\g}(\mathsf{P})$ admits a Shimura datum $\x'$ such that $\g' \to \g_2'$ underlies a morphism of Shimura data $\gxp \to \gxtwop$, such that the map $\iota':\g' \to \gvlxp$ satisfies Assumption \ref{Assump:InfinityHodge}, and such that
\begin{align}
    B(G,P,-\mu, -\mu')
\end{align}
is nonempty. 
\end{Prop}
\begin{proof}
We show that $\mathsf{P}_S$ lifts to a torsor for $\operatorname{Res}_{\mlp/\Q} \g_{0,++} \subset \g$. We use Shapiro's lemma to reduce it to a question over $\mlp$. By Corollary \ref{Cor:AlgebraicFundamentalGroupHodgeAdjoint}, the natural map
\begin{align}
     \pi_1(\g_{0,++})_{\gal_{\mlp}} \to \pi_1(\g_{2})_{\gal_{\mlp}}
\end{align}
is an isomorphism. Since $\g_{0,++}$ and $\g_{2,0}$ satisfy the Hasse principle, see Lemma \ref{Lem:HassePrinciple}, it suffices to construct lifts of $[\mathsf{P}_S]$ locally at all places. At places above $p$, the (unique) existence of the local lift is guaranteed by Lemma \ref{Lem:CohomologyComputationHodgetypeII}, and at infinite places, we choose the lift constructed to the cocycle $k$ described in Section \ref{subsub:AppendixInnerTwisting}. \smallskip 

Next, we observe that $B(G,P,-\mu, -\mu')$ is nonempty. This is because the equality
\begin{align}
    \kappa_{G_2}(\mu_2) + \kappa_{G_2}(\mathsf{P}_S) = \kappa_{G_2}(\mu_2')
\end{align}
in $\pi_1(G_2)_{\gal_{\qp}}$ implies the equality
\begin{align}
     \kappa_{G}(\mu) + \kappa_{G}(\mathsf{P}) = \kappa_{G}(\mu')
\end{align}
in $\pi_1(G)_{\gal_{\qp}}=\pi_1(G_2)_{\gal_{\qp}} \oplus \mathbb{Z}$ (see Lemma \ref{Lem:CohomologyComputationHodgetypeII}). Indeed, this comes down to checking that $\mu,\mu'$ have the same composition (namely, $1$) with the similitude map, which is well known. \smallskip 

Finally, note that Assumption \ref{Assump:InfinityHodge} holds by Proposition \ref{Prop:TwistIsHodgeEmbedding}.
\end{proof}
\begin{Rem}
When choosing $\ml$, we only really need those places $\mathfrak{p}$ such that $\mathfrak{p} \in S_{p}$ to be inert in $\ml$ over $\mlp$. In the type $B$ case, this is precisely saying that the quadratic spaces $V$ and $V'$ become isomorphic over $\ml$. This should be compared with the choice of auxiliary CM extension of the totally real field in the work of Tian--Xiao \cite{TianXiao}.
\end{Rem}

\subsubsection{} Write $\Theta_{2}=(\g_2, \x_2, \mathsf{P}_S, \x_{2}')$, write $\sigma$ for the natural map $\gx \to \gxtwo$ and let $\omega:\mathsf{P} \times^{\g} \g_2 \to \mathsf{P}_S$ be an isomorphism (see Proposition \ref{Prop:HodgeCoverBC}), so that $\Sigma=(\g,\x,\sigma,\mathsf{P}, \omega)$ is a Hodge type lifting of $\Theta_{2}$ in the sense of Definition \ref{Def:HodgeTypeLifting}.
\begin{Cor} \label{Cor:ExistenceHodgeTypeBCVeryGood}
If $\g^{\mathrm{ad}}_{\mathbb{C}}$ is not of type $D$, then there is a Galois totally real field $\f$ such that $\Sigma_{\f}$ is an excellent Hodge type cover of $\Theta_{2,\f}$.
\end{Cor}
\begin{proof}
We first construct a Galois totally real field $\f$ such that:
\begin{itemize}
    \item The prime $p$ is unramified in $\f$.

    \item The field $\mlp$ is disjoint from $\f$.

    \item In the extension $\mL \cdot \f$ of $\mlp \cdot \f$, all primes above $p$ of $\mlp \cdot \f$ split.

    \item The algebraic group $\h_{2}$ is quasi-split at $p$.

    \item The class of $\mathsf{P}$ maps to zero in $H^1(\qp,\h_{2})$.
\end{itemize}
To construct it, choose an unramified extension $K$ of $\qp$ containing $\mL_{\mathfrak{p}}$ for all primes $\mathfrak{p}$ above $p$, such that $\g_2$ is quasi-split over $K$, and such that $[\mathsf{P}]$ is trivial over $K$. This is possible because $\mL$ is unramified over $\qp$, because groups become quasi-split over unramified extensions, and because cohomology classes become trivial over unramified extensions. Using weak approximation, we can choose a totally real number field $\f_1$ disjoint from (the Galois closure of) $\mL$ in which $p$ is inert and such that $\f_1 \otimes \qp$ is $K$. Then we let $\f$ be the Galois closure of $\f_1$, which has all the properties above by construction. \smallskip

By Lemma \ref{Lem:PSCompatibility} and the second property, the group $\hyo$ is constructed from $\Upsilon_{\f}$. It follows from Lemma \ref{Lem:HassePrinciple} that $\hyo$ satisfies the Hasse principle. It moreover follows from the last property and Lemma \ref{Lem:CohomologyComputationHodgetypeII} that $\mathsf{P}$ maps to zero in $H^1(\qp,\ho)$, and from the fourth property that $\hyo$ is quasi-split at $p$. It follows from the third property and Lemmas \ref{Lem:CocentersHodgeType} and \ref{Lem:ConnectedCenter} that $H^1(\qp, Z_{H_{1}})=0$ and that $\mathsf{Z}_{\h_{1}}$ is a torus. It furthermore follows from Proposition \ref{Prop:HodgeCoverBC} that $B(H_1, -\mu_1) \cap B(H_1,-\mu_1')$ is nonempty, because it contains $B(G,-\mu, -\mu')$. Assumption \ref{Assump:InfinityHodge} holds because it holds for $\gx, \iota$ and $\mathsf{P}$, or alternatively, by another application of Proposition \ref{Prop:TwistIsHodgeEmbedding}.
\end{proof}

\subsection{Type \texorpdfstring{$A$}{A}} Let $\mL$ be a CM field with totally real subfield $\mL^{+}$. Let $\mb$ be a division algebra with center $\mL$ equipped with an involution of the second kind $\ast$ such that $\mL^{\ast}=\mL^{+}$. Let $V$ (resp. $V'$) be an $n$-dimensional Hermitian space over $(\mb, \ast)$ and for $\tau:\mL^{+} \to \mathbb{R}$ let $(r_{\tau}, s_{\tau})$ (resp. $(r_{\tau}', s_{\tau}')$) the signatures of $V \otimes_{\mlp, \tau} \mathbb{R}$ in the sense of \cite[Definition 1.2.5.2]{Lan}. Consider the algebraic groups 
\begin{align}
    \g_{0}&:=\operatorname{U}_{\mb}(V) \\
    \g_{0}&:=\operatorname{U}_{\mb}(V') \\
    \g_{+}&:=\operatorname{Res}_{\mlp/\mathbb{Q}} \operatorname{U}_{\mb}(V) \subset \operatorname{GU}_{\mb}(V):=\g \\
    \g_{+}'&:=\operatorname{Res}_{\mlp/\mathbb{Q}} \operatorname{U}_{\mb}(V') \subset \operatorname{GU}_{\mb}(V'):=\g'
\end{align}
associated to $V$ and $V'$. Write $\x$ and $\x'$ for the Shimura data associated to $\g$ and $\g'$ by \cite[Proposition 8.14]{Milne}, and $\mathsf{E}$ and $\mathsf{E}'$ for their reflex fields. 

\subsubsection{} Let $\mathsf{P}_0$ be the $(\operatorname{U}_{\mb}(V),\operatorname{U}_{\mb}(V'))$ bitorsor over $\mlp$ of $\mb$-linear isomorphisms of Hermitian spaces $V \to V'$, and let $\mathsf{P}$ be the induced $\g$-torsor. Let $\iota:\gx \to \gvx$ be the tautological Hodge embedding, and note that the tautological Hodge embedding $\iota':\gxp \to \gvxp$ is naturally identified with the twist by $\mathsf{P}$ of $\iota$. Thus Assumption \ref{Assump:InfinityHodge} is tautologically satisfied. 

\subsubsection{} Fix a prime $p$ such that $\mb$ splits over all primes $\mathfrak{p}$ of $\mlp$ above $p$. Assume moreover that $\mathsf{P}_0$ is trivial over $\mathbb{A}_{\mlp,f}^{p}$ and fix an isomorphism $\mathbb{C} \isom \qpbar$, inducing places $v$ and $v'$ of $\mathsf{E}$ and $\mathsf{E}'$ with completions $E$ and $E'$, respectively. We now want to compute when $B(G, -\mu, -\mu')$ is nonempty in our situation. This is slightly easier to do for $B(G^{\mathrm{ad}}, -\mu, -\mu')$, as we will now do.

\subsubsection{} Note that $G^{\mathrm{ad}}=\prod_{\mathfrak{p}} \g_{0, \mathfrak{p}}^{\mathrm{ad}}$, where the product runs over primes $\mathfrak{p}$ of $\mlp$ above $p$. Furthermore
\begin{align}
\pi_1(\g_{0,\mathfrak{p}}^{\mathrm{ad}})_{\gal_{\mlp_{\mathfrak{p}}}} = \threepartdef{\mathbb{Z}/n\mathbb{Z}}{\mathfrak{p} \text{ splits in } \ml}{\mathbb{Z}/2\mathbb{Z}}{\mathfrak{p} \text{ does not split in } \ml \text{ and $n$ is even}}{0}{\mathfrak{p} \text{ does not split in } \ml \text{ and $n$ is odd.}}
\end{align}
Note that there is an identification $\Sigma_{\infty}= \widetilde{\Sigma_p}$ from the embeddings $\tau:\mlp \to \R$ to the set of embeddings $\tau:\mlp \to \qpbar$, and that there is a natural forgetful map $u:\widetilde{\Sigma_p} \to \Sigma_p$. 

\begin{Lem} \label{Lem:NonemptynessBGMUTypeAAdjoint}
    The set $B(G^{\mathrm{ad}},-\mu,-\mu')$ is nonempty if and only if the following two conditions hold:
    \begin{itemize}
        \item For all primes $\mathfrak{p}$ split in $\ml$ we have
        \begin{align}
          \sum_{\tau \in u^{-1}(\mathfrak{p})} r_{\tau} = \sum_{\tau \in u^{-1}(\mathfrak{p})} r_{\tau}' \mod n
        \end{align}

        \item If $n$ is even, then for all primes $\mathfrak{p}$ nonsplit in $\ml$ we have 
        \begin{align}
            \kappa_{G^{\mathrm{ad}}}(\mathsf{P}_{\mathfrak{p}})+\sum_{\tau \in u^{-1}(\mathfrak{p})} r_{\tau} = \sum_{\tau \in u^{-1}(\mathfrak{p})} r_{\tau}' \mod 2
        \end{align}
    \end{itemize}
\end{Lem}
\begin{proof}
We only need to check when $\kappa_{G^{\mathrm{ad}}}(\mu) + \kappa_{G^{\mathrm{ad}}}(P) = \kappa_{G^{'\mathrm{ad}}}(\mu')$. We may check this equality one prime $\mathfrak{p}$ above $p$ at a time. The lemma now comes down to the claim that the image of $-\mu$ (resp. $-\mu'$) in $\pi_1(\g_{0,\mathfrak{p}}^{\mathrm{ad}})_{\gal_{\mlp_{\mathfrak{p}}}}$ is given by $r_{\tau}$ (resp. $r_{\tau}'$). This follows from the description of a standard $h \in \x$ and $h' \in \x'$ as in \cite[Example 2.1.8.(ii)]{XiaoZhu}. 
\end{proof}

\begin{Rem}
We note that the existence of $V'$ forces the equality
\begin{align}
    \sum_{\mathfrak{p}} \kappa_{G^{\mathrm{ad}}}(\mathsf{P}_{\mathfrak{p}}) + \sum_{\tau} r_{\tau} = \sum_{\tau} r_{\tau}' \mod 2,
\end{align}
see \cite[Example 2.1.8.(ii)]{XiaoZhu}. 
\end{Rem}

\newcommand{\mlpi}{\mathsf{L}^{+}_{\mathfrak{p}_i}}
\newcommand{\mli}{\mathsf{L}_{\mathfrak{p}_i}}

The following Lemma recovers \cite[Remark 7.3.10]{XiaoZhu} as a special case. More precisely, \cite[Remark 7.3.10]{XiaoZhu} deals with the case that $\mathfrak{p}$ is inert in $\ml$ and $P$ is trivial. 
\begin{Lem} \label{Lem:NonemptynessBGMUTypeA}
    The set $B(G,-\mu,-\mu')$ is nonempty if and only if the following two conditions hold:
    \begin{itemize}
        \item For all primes $\mathfrak{p}$ split in $\ml$ we have
        \begin{align}
          \sum_{\tau \in u^{-1}(\mathfrak{p})} r_{\tau} = \sum_{\tau \in u^{-1}(\mathfrak{p})} r_{\tau}'
        \end{align}

        \item If $n$ is even, then for all primes $\mathfrak{p}$ nonsplit in $\ml$ we have 
        \begin{align}
            \kappa_{G}(\mathsf{P}_{\mathfrak{p}})+\sum_{\tau \in u^{-1}(\mathfrak{p})} r_{\tau} = \sum_{\tau \in u^{-1}(\mathfrak{p})} r_{\tau}' \mod 2.
        \end{align}

        \item If $n$ is odd, then either for all primes $\mathfrak{p}$ nonsplit in $\ml$ we have 
        \begin{align}
            \kappa_{G}(\mathsf{P}_{\mathfrak{p}})+\sum_{\tau \in u^{-1}(\mathfrak{p})} r_{\tau} = \sum_{\tau \in u^{-1}(\mathfrak{p})} r_{\tau}' \mod 2,
        \end{align}
        or for all primes $\mathfrak{p}$ nonsplit in $\ml$ we have 
        \begin{align}
            \kappa_{G}(\mathsf{P}_{\mathfrak{p}})+\sum_{\tau \in u^{-1}(\mathfrak{p})} r_{\tau} \neq \sum_{\tau \in u^{-1}(\mathfrak{p})} r_{\tau}' \mod 2.
        \end{align}
    \end{itemize}
    
\end{Lem}
\begin{proof}
    For $G$ as above we have an isomorphism
    \[
        G_{\qp} \simeq H \subset \left(\prod_{\mathfrak{p}|p} \mathrm{GU}(L_\mathfrak{p}/L^+_\mathfrak{p})\right) 
    \]
    where $H$ is the subgroup of elements $(g_v, \lambda_v)$ such that for all $v$ there is some $\lambda \in \mathbb{G}_m$ such that $\lambda_v = \lambda$. Its cocenter is thus the group
    \[
        T(R) \coloneqq \{((t_v), \lambda) \, | \, t_v \in (L_v \otimes R)^\times, \lambda \in R^\times \text{ and for all v, }\mathrm{Nm}^{L_v}_{L^+_v}(t_v) = \lambda^n \}.
    \]
    So we calculate that
    \[
        X_*(T) \subset \mathbb{Z} \times \mathbb{Z}^{\Sigma_p(L)}
    \]
    is precisely the subset of $(x, (y_w))$ such that $nx = y_w + y_{w^c}$ for all $w$, and the Galois group $\gal_{\qp}$ acts in the obvious way, here $\Sigma_p(L)$ denotes the set of places of $L$ above $p$ with its natural Galois action. Let $I$ denote the set of places of $L_p^+$ inert in $L_p$, and let $S$ denote the set of places of $L_p^+$ which are split in $L_p$. Since there is a short exact sequence
    \[
        0 \to X_*(T) \to \mathbb{Z} \times \mathbb{Z}^{\Sigma_p(L)} \xrightarrow{\sigma} \mathbb{Z}^{\Sigma_p^+(L)} \to 0
    \]
    taking coinvariants under $\gal_\qp$, we obtain
    \[
        0 \to Y \to X_*(T)_\Gamma \to (\mathbb{Z} \times \mathbb{Z}^{I/\gal_{\qp}} \times \mathbb{Z}^{(S/\gal_{\qp})^2}) \to \mathbb{Z}^{I/\gal_{\qp}} \times \mathbb{Z}^{S/\gal_{\qp}} \to 0
    \]
    where $Y = \mathrm{Coker}\left(\mathrm{H}_1(\gal_{\qp}, \mathbb{Z}^{\Sigma_p} \times \mathbb{Z}) \to \mathrm{H}_1(\gal_{\qp}, \mathbb{Z}^{\Sigma_p^+}) \right)$. Using Shapiro's lemma and transfer maps to calculate $Y$ we get
    \[
         0 \to (\mathbb{Z}/2\mathbb{Z})^{I/\gal_{\qp}}/(n \cdot \sum_{i \in I/\gal_{\qp}} e_i) \to X_*(T)_\Gamma \to \mathbb{Z} \times \mathbb{Z}^{I/\gal_{\qp}} \times \mathbb{Z}^{S/\gal_{\qp}} \to \mathbb{Z}^{I/\gal_{\qp}} \times \mathbb{Z}^{S/\gal_{\qp}} \to 0,
    \]
    where the $e_i$ are basis vectors for $(\mathbb{Z}/2\mathbb{Z})^{I/\gal_{\qp}}$. Since the kernel of the rightmost map is the free module $\mathbb{Z}^{S/\gal_{\qp}} \times \mathbb{Z}$, we find that 
    \[
        X_*(T)_\Gamma \cong (\mathbb{Z}/2\mathbb{Z})^{I/\gal_{\qp}}/(n \cdot \sum_{i \in I/\gal_{\qp}} e_i) \oplus \mathbb{Z}^{S/\gal_{\qp}} \times \mathbb{Z}.
    \]
    From this the result follows by a simple calculation.
\end{proof}

\subsubsection{} For $\f$ a totally real field disjoint from $\mlp$ we write $\mb_{\f}=\mb \otimes_{\mathbb{Q}} \f$, $\ml_{\f}= \ml \otimes_{\mathbb{Q}} \f$ and $V_{\f}=V \otimes_{\mathbb{Q}} \f$ and similarly define $V'_{\f}$. This induces $\h_1=\operatorname{GU}_{\mb_{\f}}(V_{\f})$ and $\h_1'=\operatorname{GU}_{\mb_{\f}}(V_{\f})$.
\begin{Lem} \label{Lem:PELCovering}
There is a Galois totally real field $\f$ disjoint from $\mlp$ such that: The field $\f$ is unramified at $p$, the $\mb_{\f} \otimes \qp$ Hermitian spaces $V \otimes \qp$ and $V' \otimes \qp$ are isomorphic, and $Z_{H_1}$ is a product of induced tori.
\end{Lem}
\begin{proof}
To construct it, choose an unramified extension $K$ of $\qp$ containing $\mL_{\mathfrak{p}}$ for all primes $\mathfrak{p}$ above $p$, such that $G$ is quasi-split over $K$, and such that $[P]$ is trivial over $K$. This is possible because $\mL$ is unramified over $\qp$, because groups become quasi-split over unramified extensions, and because cohomology classes become trivial over unramified extensions. Using weak approximation, we can choose a totally real number field $\f_1$ disjoint from $\mL$ in which $p$ is inert and such that $\f_1 \otimes \qp$ is $K$. Then we let $\f$ be the Galois closure of $\f_1$, which has all the properties above by construction.
\end{proof}

\begin{proof}[Proof of Theorem \ref{Thm:IntroMainIgusaStack}]
In types $B,C$, the theorem follows from Theorem \ref{Thm:CorrespondenceAdjoint} together with Corollary \ref{Cor:ExistenceHodgeTypeBCVeryGood}. In type $A$, it follows from Theorem \ref{Thm:HodgeTypeCorrespondences} together with Lemma \ref{Lem:PELCovering} and Proposition \ref{Prop:PEL}. 
\end{proof}

\section{Cohomological applications} \label{Sec:Cohomology} In this section we give cohomological applications of our main results. In Section \ref{sub:gluing} we glue together the Igusa sheaves for the two different Shimura data. In Section \ref{sub:HeckeOperators}, we show that Hecke operators applied to this glued sheaf compute the cohomology of both Shimura varieties. In Section \ref{Sec:SpecAct} we prove some results about the localization of the spectral action on affine open subschemes and formal completions of the stack of $L$-parameters. In Sections \ref{sub:ExamplesSpectral} and \ref{Sub:DifferentAtP} we use these results to prove Theorem \ref{Thm:IntroCohomology}.

\subsection{Gluing} \label{sub:gluing} Let $\gx$ be a Shimura datum of abelian type, and fix a prime $p$ together with an isomorphism $\mathbb{C} \simeq \qpbar$. Let $\mathsf{P}$ be a $\g$-torsor over $\spec \mathbb{Q}$ trivial over $\afp$, let $\g'=\operatorname{Aut}_{\g}(\mathsf{P})$ and let $\gxp$ be a Shimura datum satisfying Assumption \ref{Assump:Infinity}. Choose a trivialization $\kappa^p:\mathsf{P} \otimes \afp \to \g \otimes \afp$ inducing an identification $\g \otimes \afp \to \g' \otimes \afp$. We let $v$ and $v'$ be the places of $\mathsf{E}$ and $\mathsf{E}'$ determined by our identification $\mathbb{C} \simeq \qpbar$. Assume that Conjecture \ref{Conj:Main} holds, and consider the commutative diagram produced by it
\begin{equation}
    \begin{tikzcd}
        \igs \gx_{-\mu'} \arrow[r, "\sim"] \arrow{d} & \igs \gxp_{-\mu} \arrow{d} \\
        \bun_G \arrow{r}{\beta_{\mathsf{P}}} & \bun_{G'}.
    \end{tikzcd}
\end{equation}
Notice that $\igs \gx_{-\mu'} \subset \igs \gx$ and $\igs \gxp_{-\mu} \subset \igs \gxp$ are open substacks. Thus we may use the isomorphism above to glue $\igs \gx$ to $\igs \gxp$ along their common open substack, we will denote this gluing by $\igs \gxkappa$. The $\underline{\gafp}$-equivariance of the isomorphism gives us an action of $\underline{\gafp}$ on $\igs \gxkappa$. If we use $\beta_{\mathsf{P}}^{-1}$ to give $\igs \gxp$ a structure map $\tildepi$ to $\bun_G$, then $\igs \gxkappa$ lives over $\bun_{G}$.

\subsubsection{} Fix an algebraic closure $k$ of $\fp$ and let $\ell \not=p$. We let $\bungmucmup \subset \bung$ be the union of the quasicompact open substacks $\bungmu$ and $\bungmup$. Fix a neat compact open subgroup $K^p \subset \gafp$ and consider (where the level $K^p$ objects are as in \cite[Section 4.5]{DvHKZIgusaStacksII})
\begin{align}
    \pi:\igs_{K^p} \gx  &\to \bungmucmup \\
    \pi':\igs_{K^p}\gxp  &\to \bungmucmup \\
    \tilde{\pi}:\igs_{K^p} \gxkappa  &\to \bungmucmup.
\end{align}
\subsubsection{} For $\Lambda$ a $\zl$-algebra where $\ell$ is nilpotent containing a fixed square root of $p$, and for $W$ a standard $\Lambda$ local system as in \cite[Section 5.2.1]{DvHKZIgusaStacksII}, we define objects in $D_{\text{\'et}}(\bungmucmupk, \Lambda)$ by
\begin{align}
    \mathcal{F}=\mathcal{F}_{\mathbb{W}}&:=R \pi_{k, \ast} \mathbb{W}, \qquad 
     \mathcal{F}'=\mathcal{F}_{\mathbb{W}}':=R \pi'_{k, \ast} \mathbb{W}, \qquad
    \tildef=\tildef_{\mathbb{W}}:=R \tildepi_{k, \ast} \mathbb{W}. 
\end{align}
Let $\mathbb{T}_{K^p}$ denote the Hecke algebra for $\gafp$ with level $K^p$ and coefficients in $\Lambda$; note that $\mathbb{T}_{K^p}$ acts on $\mathcal{F}, \mathcal{F}', \tildef$. 
\begin{Lem} \label{Lem:RestrictionTildeF}
    There are natural $\mathbb{T}_{K^p}$-equivariant isomorphisms 
    \begin{align}
    \restr{\tildef}{\bungmu} &\isom \mathcal{F} \\
    \restr{\tildef}{\bungmup} &\isom \mathcal{F}'.
\end{align}
\end{Lem}
\begin{proof}
This is a direct consequence of smooth base change along the open immersions $\bungmu \to \bungmucmup $ and $\bungmup \to \bungmucmup $, see \cite[Theorem 1.10.(iii)]{EtCohDiam}. The Hecke-equivariance follows from the $\gafp$-equivariance of the isomorphism $\igs \gx_{-\mu'} \isom \igs \gxp_{-\mu}$.
\end{proof}
\begin{Cor} \label{Cor:ULA}
The sheaf $\tildef$ is ULA with respect to $\bungmucmupk \to \spd k$ in the following three situations.
\begin{itemize}
    \item The Shimura varieties for both $\gx$ and $\gxp$ are proper. 

    \item Both $\gx$ and $\gxp$ are of Hodge type.

    \item The Shimura data $\gx$ and $\gxp$ are of abelian type and satisfy Milne's axiom SV5.
\end{itemize}
\end{Cor}
\begin{proof}
By \cite[Theorem V.7.1]{FarguesScholze}, it suffices to prove that both $\mathcal{F}$ and $\mathcal{F}'$ are ULA. In the first situation, this is \cite[Proposition 5.2.6]{DvHKZIgusaStacksII}, and in the second situation this is \cite[Corollary 8.5.4]{DvHKZIgusaStacks}. In the third situation, this follows from Theorem \ref{Thm:IgusaVarietiesAbelianType} and \cite[Theorem V.7.1]{FarguesScholze}, as in the proof of \cite[Corollary 8.5.4]{DvHKZIgusaStacks}.
\end{proof}

\subsection{Hecke operators} \label{sub:HeckeOperators} Recall the global and local Hecke stacks and the Hecke correspondences of Fargues--Scholze, see \cite[Section 8.4]{FarguesScholze}. We have the following lemma.
\begin{Lem} \label{Lem:HeckeCorrespondencesInnerForm}
    The torsor $\mathsf{P}$ induces an isomorphism $\beta_{\mathsf{P}}:\mathrm{Hck}_{G} \to \mathrm{Hck}_{G'}$ of global Hecke stacks fitting a $2$-commutative diagram
    \begin{equation}
        \begin{tikzcd}
            & \mathrm{Hck}_{G} \arrow{dd}{\beta_{\mathsf{P}}} \arrow{dr}{h_{1,G}} \arrow[dl,"h_{2,G}", swap] \\
            \bun_{G} \arrow{dd}{\beta_{\mathsf{P}}} & & \bun_{G} \times \operatorname{Div}^1 \arrow{dd}{\beta_{\mathsf{P}} \times \operatorname{Id}} \\
            & \mathrm{Hck}_{G'} \arrow{dr}{h_{1,G'}} \arrow[dl,"h_{2,G'}", swap] \\
            \bun_{G'} & & \bun_{G'} \times \operatorname{Div}^1. 
        \end{tikzcd}
    \end{equation}
    It moreover induces an isomorphism $\beta_{\mathsf{P}}:\mathcal{H}ck_{G} \to \mathcal{H}ck_{G'}$ of local Hecke stacks fitting in a $2$-commutative diagram
    \begin{equation}
        \begin{tikzcd}
            \mathrm{Hck}_{G} \arrow{r} \arrow{d}{\beta_{\mathsf{P}}} & \mathcal{H}ck_{G} \arrow{d}{\beta_{\mathsf{P}}} \\
            \mathrm{Hck}_{G'} \arrow{r} & \mathcal{H}ck_{G'}.
        \end{tikzcd}
    \end{equation}
\end{Lem}
\begin{proof}
We first recall the construction of the isomorphism $\beta_{P}:\bung \to \bungp$, see \cite[Corollary III.4.3]{FarguesScholze}. First of all, the $G$-torsor $P$ over $\spec \qp$ defines for $(R,R^+) \in \perf$ a $G$-torsor $P_{(R,R^+)}$ on $X_{(R,R^+)}$ by pullback. Given $(R,R^+) \in \perf$ and a $G$-bundle $\mathcal{E}$ on $X_{(R,R^+)}$, then $\beta_{P}(\mathcal{E})=\operatorname{Isom}_{G}(\mathcal{E}, P_{(R,R^+)})$, considered as a $G'$-bundle via its natural action of $G'_{X_{(R,R^+)}}=\operatorname{Aut}_G(P_{(R,R^+)})$. \smallskip 

The global Hecke stack is the moduli stack sending $S=\spa(R,R^+)$ to the set of quadruples $(\mathcal{E}_1, \mathcal{E}_2, D,f)$, where $\mathcal{E}_1,\mathcal{E}_2 \in \bung(R,R^+)$, where $D \subset X_{(R,R^+)}$ is a degree $1$ Cartier divisor and where 
\begin{align}
    f:\restr{\mathcal{E}_1}{X_{(R,R^+)} \setminus D} \to \restr{\mathcal{E}_2}{X_{(R,R^+)} \setminus D}
\end{align}
is an isomorphism of $G$-bundles. It is clear that $f$ induces an isomorphism
\begin{align}
    \beta_P(f):\restr{\beta_P(\mathcal{E}_1)}{X_{(R,R^+)} \setminus D} \to \restr{\beta_P(\mathcal{E}_2)}{X_{(R,R^+)} \setminus D}
\end{align}
and so we get a quadruple $(\beta_P(\mathcal{E}_1),\beta_P(\mathcal{E}_2), D, \beta_P(f)) \in \mathrm{Hck}_{G'}(R,R^+)$. This defines a morphism $\beta_P:\mathrm{Hck}_{G} \to \mathrm{Hck}_{G'}$ making the first diagram of the lemma commute. \smallskip 

A similar argument shows the existence of the isomorphism $\beta_P:\mathcal{H}ck_{G} \to \mathcal{H}ck_{G'}$ of local Hecke stacks sitting inside the second commutative diagram of the lemma.
\end{proof}
We can also consider $\mathrm{Hck}_{G, \le \mu} \to \bun_{G} \times \Div$ and $\mathrm{Hck}_{G, \le \mu'} \to \bun_{G} \times \Divp$ and their basechanges to $k$.

\subsubsection{} We note that Lemma \ref{Lem:RestrictionTildeF} provides us with natural maps $\tildef \to \mathcal{F}$ and $\tildef \to \mathcal{F}'$. Let us denote by $i_1$ the inclusion of $\bun_{G,k}^{[1]} \to \bun_{G,k}$ and let us denote by $i_{1'}$ the inclusion of $\bun_{G,k}^{[P]} \to \bun_{G',k}$. Under the morphism $\beta_{P}$, the map $i_{1'}$ gets identified with the inclusion of $\bun_{G',k}^{[1]} \to \bun_{G',k}$. 
\begin{Lem} \label{Lem:Boundedness}
The natural maps
\begin{align}
    i_1^{\ast} T_{\mu} j_{k,!} \tildef & \to i_1^{\ast} T_{\mu} j_{k,!} \mathcal{F} \\
    i_{1'}^{\ast} T_{\mu'} j_{k,!} \tildef & \to i_{1'}^{\ast} T_{\mu'} j_{k,!} \mathcal{F}'
\end{align}
are isomorphisms.
\end{Lem}
\begin{proof}
We only prove the first statement, the second statement can be proved in a similar fashion using Lemma \ref{Lem:HeckeCorrespondencesInnerForm}. We consider the following commutative diagram
\begin{equation}
\begin{tikzcd}
    \bungmu \arrow[d] & \operatorname{Hck}_{G, \mu, k}^{[1]} \arrow{d}{\tilde{i}_1} \arrow{r}{\tilde{h}_2} \arrow[l, densely dotted] & \bun_{G,k}^{[1]} \times \Div \arrow{d}{i_1} \\
    \bun_{G} &\arrow{l}{h_1} \operatorname{Hck}_{G, \mu, k} \arrow{r}{h_2} & \bun_{G,k} \times \Div,
\end{tikzcd}
\end{equation}
where we note that the dotted arrow exists by \cite[Lemma 8.4.4]{DvHKZIgusaStacks}. The Hecke operator $T_{\mu}$ is given by $h_{2,\ast} h_1^{\ast}[-d](-\tfrac{d}{2})$, see the proof of \cite[Theorem 8.4.8]{DvHKZIgusaStacks}. Using smooth base change for the open immersion $\bun_{G,k}^{[1]} \to \bun_{G,k}$, we see that 
\begin{align}
    i_1^{\ast} T_{\mu} &\simeq i_1^{\ast} h_{2,\ast} h_1^{\ast}[-d](-\tfrac{d}{2}) \\
    & \simeq \tilde{h}_{2,\ast} \tilde{i}_1^{\ast} h_1^{\ast}[-d](-\tfrac{d}{2}).
 \end{align}
The lemma now follows from the fact that the natural map $\tilde{i}_1^{\ast} h_1^{\ast} j_{k,!} \tildef \to \tilde{i}_1^{\ast} h_1^{\ast} j_{k,!} \mathcal{F}$ is an isomorphism, which is a consequence of the above commutative diagram and Lemma \ref{Lem:RestrictionTildeF}.
\end{proof}
We have the following important result. Let $j:\bungmucmup \to \bun_{G}$ be the inclusion and use $\mathbb{W}$ to denote the automorphic local system on $\mathbf{Sh}_{K^p}\gx$ and $\mathbf{Sh}_{K^p}\gxp$ induced by $W$. 
\begin{Prop} \label{Prop:WeilCohShimVar}
There is a $G(\qp) \times W_E \times \mathbb{T}_{K^p}$-equivariant isomorphism
\begin{align}
    i_1^{\ast} T_{\mu} j_{k,!} \tildef_{\mathbb{W}} &\to R \Gamma(\mathbf{Sh}_{K^p}\gx_{\ebar}, \mathbb{W})[d](d/2)
\end{align}
and a $G'(\qp) \times W_{E'} \times \mathbb{T}_{K^p}$-equivariant isomorphism
\begin{align}
      i_{1'}^{\ast} T_{\mu'} j_{k,!} \tildef_{\mathbb{W}} &\to R \Gamma(\mathbf{Sh}_{K^p}\gxp_{\ebar}, \mathbb{W})[d'](d'/2)
\end{align}
\end{Prop}
\begin{proof}
This can be proved by combining \cite[proof of Theorem 5.3.11]{DvHKZIgusaStacksII} with \cite[Proposition 5.21]{wu2025arithmeticcompactificationsintegralmodels}, as explained in \cite[Remark 5.3.12]{DvHKZIgusaStacksII}.
\end{proof}

\subsubsection{} \label{subsub:Rational} Now suppose that $L$ is a finite extension of $\ql$ containing a fixed square root of $p$. Let $W$ be a standard $\mathcal{O}_L$ local system as in \cite[Section 5.2.1]{DvHKZIgusaStacksII}.
\begin{Hyp} \label{Hyp:ULA}
    For all $n$ the sheaf $\tilde{\mathcal{F}}_{\mathbb{W}/\ell^n \mathbb{W}}$ is ULA.
\end{Hyp}
Hypothesis \ref{Hyp:ULA} is often satisfied by Corollary \ref{Cor:ULA}. Consider the pro-system
\begin{align}
    \{\tilde{\mathcal{F}}_{\mathbb{W}/\ell^n \mathbb{W}}\}_{n \in \mathbb{N}}.
\end{align}
Assuming Hypothesis \ref{Hyp:ULA}, this defines an object $\tilde{F}_{\mathbb{W}} \in D_{\mathrm{lis}}(\bungmucmupk, \mathcal{O}_L)$ via the equivalence of categories of \cite[Lemma 7.4]{CaraianiHamannZhang}. We can moreover base change this object to consider it as an object $\tilde{F}_{\mathbb{W},\Lambda'}$ in $D_{\mathrm{lis}}(\bungmucmupk, \Lambda)$ for $\Lambda=L, \zlbar, \qlbar$, see \cite[Lemma 7.5]{CaraianiHamannZhang}. For such $\Lambda$ we also consider
\begin{align}
    R \Gamma(\mathbf{Sh}_{K^p}\gx_{\ebar}, \mathbb{W}_{\Lambda})[d](d/2):=\varinjlim_{K_p} R \Gamma(\mathbf{Sh}_{K^pK_p}\gx_{\ebar}, \mathbb{W}_{\Lambda})[d](d/2),
\end{align}
and its variant for $\gxp$.
\begin{Prop} \label{Prop:WeilCohShimVarRat}
Let $\Lambda \in \{\mathcal{O}_L,L,\zlbar,\qlbar\}$ and assume Hypothesis \ref{Hyp:ULA}. There is a $G(\qp) \times W_E \times \mathbb{T}_{K^p}$-equivariant isomorphism
\begin{align}
    i_1^{\ast} T_{\mu} j_{k,!} \tildef_{\mathbb{W}, \Lambda} &\to R \Gamma(\mathbf{Sh}_{K^p}\gx_{\ebar}, \mathbb{W}_{\Lambda})[d](d/2)
\end{align}
and a $G'(\qp) \times W_{E'} \times \mathbb{T}_{K^p}$-equivariant isomorphism
\begin{align}
      i_{1'}^{\ast} T_{\mu} j_{k,!} \tildef_{\mathbb{W}} &\to R \Gamma(\mathbf{Sh}_{K^p}\gxp_{\ebar}, \mathbb{W})[d'](d'/2)
\end{align}
\end{Prop}
\begin{proof}
    This immediately follows from \cite[Lemma 7.4.(3)]{CaraianiHamannZhang} and Proposition \ref{Prop:WeilCohShimVar}.
\end{proof}

\subsection{Applying the spectral action}\label{Sec:SpecAct} Let $\widehat{G}$ be the dual group of $G$ over $\zl[\sqrt{p}]$ equipped with its action of $W_{\qp}$. We define the $L$-group ${}^LG$ as the semi-direct product $\widehat{G}\rtimes W_\qp$, see \cite[Section~2.1]{BuzzardGee}. We consider the $\zl[\sqrt{p}]$-representation $V_\mu$ (resp. $V_{\mu'}$) of $\widehat{G}$ as in \cite[Section 8.4.1]{DvHKZIgusaStacks}. It extends canonically to a representation $r_\mu$ of ${}^LG_E\coloneqq \widehat{G}\rtimes W_E$ (resp. ${}^LG_{E'}$) as in \cite[Section 8.4.6]{DvHKZIgusaStacks}. We denote by $\loc$ the stack of $L$-parameters over $\zl[\sqrt{p}]$ as in \cite{DHKMModuli}, \cite{ZhuCoherent}, and \cite{FarguesScholze}. This is the stack quotient of the moduli space $\cocycle_\Lambda$ of 1-cocycles (in $\Lambda/\zl[\sqrt{p}]$-algebras) by the conjugation action of $\dualG$. 

\subsubsection{} From now on we will work in the following setup. 
\begin{Setup} \label{Setup:StandardSetup}
Let $L$ be a finite extension of $\ql[\sqrt{p}]$. We let $\Lambda$ and $W$ be one of the following.
\begin{itemize}
    \item $\Lambda$ is an $\mathcal{O}_L$-algebra in which $\ell$ is nilpotent, and $W$ is any standard $\Lambda$-local system.

    \item $W$ is a standard $\mathcal{O}_L$-local system satisfying Hypothesis \ref{Hyp:ULA} and $\Lambda \in \{\mathcal{O}_L,L,\zlbar,\qlbar\}$.
\end{itemize}
\end{Setup}

\subsubsection{} Assume that $P=\mathsf{P} \otimes \qp$ is trivial and fix a trivialization $\kappa_p$ inducing an identification $G=G'$. The following result is a consequence of the spectral action, see \cite[Theorem X.0.2]{FarguesScholze} and \cite[Theorem IX.0.1]{FarguesScholze}.
\begin{Cor}\label{Cor:SpectralActnExists}
If either $\ell$ is coprime to the order of $\pi_0(\zg)$ or $\ell$ is invertible in $\Lambda$, then there is a natural map
    \begin{equation} \label{Eq:SpectralActionMapZhuConjecture}
    \begin{tikzcd}
        \operatorname{RHom}_{\loc}(V_{\mu}, V_{\mu'}) \arrow{d} \\
        \operatorname{RHom}_{G(\qp) \times  \mathbb{T}_{K^p}}\left(R \Gamma(\mathbf{Sh}_{K^p}\gx_{\ebar},\mathbb{W})[d], R \Gamma(\mathbf{Sh}_{K^p}\gxp_{\ebar},\mathbb{W})[d'] \right)
        \end{tikzcd}
    \end{equation}
\end{Cor}
Note that we are taking homomorphisms of $G(\qp) \times  \mathbb{T}_{K^p}$-modules only here; these morphisms are not required to be Galois equivariant in any way.
\begin{Rem}
    If $P$ is not trivial, then left hand side of equation \eqref{Eq:SpectralActionMapZhuConjecture} is zero, see Section \ref{Sub:DifferentAtP}. 
\end{Rem}

\begin{proof}[Proof of Corollary \ref{Cor:SpectralActnExists}]
By \cite[Theorem X.0.2]{FarguesScholze} and \cite[Theorem IX.0.1]{FarguesScholze}, there is a natural morphism
\begin{align}
    \operatorname{RHom}_{\loc}(V_{\mu}, V_{\mu'}) \to \operatorname{RHom}_{D_{\mathrm{lis}}(\bun_{G,k}, \Lambda)}(T_{\mu} \tildef_{\mathbb{W}}, T_{\mu'} \tildef_{\mathbb{W}}). 
\end{align}
It follows from the functoriality of the spectral action that it lands in the subspace of 
$\operatorname{RHom}_{D_{\mathrm{lis}}(\bun_{G,k}, \Lambda)}(T_{\mu} \tildef_{\mathbb{W}}, T_{\mu'} \tildef_{\mathbb{W}})$ consisting of morphisms commuting the action of $\mathbb{T}_{K^p}$ on source and target. The corollary follows by composing with
\begin{align}
    \operatorname{RHom}_{D_{\mathrm{lis}}(\bun_{G,k}, \Lambda)}(T_{\mu} j_{k,!}\tildef_{\mathbb{W}}, T_{\mu'} j_{k,!}\tildef_{\mathbb{W}}) \to \operatorname{RHom}_{D_{\mathrm{lis}}(\bun_{G,k}^{[1]}, \Lambda)}(i_{1}^{\ast} T_{\mu} j_{k,!}\tildef_{\mathbb{W}}, i_{1}^{\ast} T_{\mu'} \tildef_{\mathbb{W}}),
\end{align}
and identifying the right hand side with 
\begin{align}
    \operatorname{RHom}_{G(\qp)}\left(R \Gamma(\mathbf{Sh}_{K^p}\gx_{\ebar},\mathbb{W})[d], R \Gamma(\mathbf{Sh}_{K^p}\gxp_{\ebar},\mathbb{W})[d'] \right)
\end{align}
using Proposition \ref{Prop:WeilCohShimVar}. 
\end{proof}

\begin{Rem}
It is not clear that the map of Corollary \ref{Cor:SpectralActnExists} is nonzero. In the rest of this section we will prove refinements of Corollary \ref{Cor:SpectralActnExists} which allows us to "localize" over the GIT quotient 
\begin{align}
    \cocycle_\Lambda \sslash \hat{G}=:\xspec.
\end{align}
\end{Rem}

\subsubsection{} Fix a field $\Lambda \to \kappa$ and a closed point $\phi:\spec \kappa \to \xspec$ with inverse image $q^{-1}(\phi)$ in $\lock$ along the map $q:\lock \to \xspec$. Let us write $\lochat$ for the formal completion of $\lock$ in the closed subset $q^{-1}(\phi)$. By \cite[Lemma 3.13]{ZouCategoricalLS}, the stable infinity category $\indperf(\lochat)$ can be identified with the full subcategory of $\indperf(\loc)$ of those complexes that are set-theoretically supported on $q^{-1}(\phi)$. 

\subsubsection{} Zou defines in \cite[Definition 3.9]{ZouCategoricalLS} the category (using the Lurie tensor product)
\begin{align}
    \mathcal{D}_{\mathrm{lis}}(\bun_{G},\kappa)_{\phi}^{\wedge}:=\mathcal{D}(\bun_{G}, \kappa) \otimes_{\indperf (\loc)} \indperf(\lochat),
\end{align}
where the morphism $\indperf(\loc) \to \indperf(\lochat)$ is given by pullback. By construction, it admits an action of $\indperf(\lochat)$. If $k$ is algebraically closed, then by \cite[Corollary 3.18]{ZouCategoricalLS}, it can be identified with the full subcategory of $\mathcal{D}(\bun_{G}, \kappa)$ consisting of objects $\mathcal{F}$ such that all irreducible subquotients of all cohomology sheaves of $\mathcal{F}$ have Fargues--Scholze $L$-parameter $\phi$.

\subsubsection{} We recall from \cite[Definition A.1]{HamannLee} the full subcategory
\begin{align}
    \mathcal{D}_{\mathrm{lis}}(\bun_{G},\kappa)_{\phi} \subset \mathcal{D}(\bun_{G}, \kappa)
\end{align}
consisting of those objects where the action of $\mathcal{Z}^{\mathrm{spec}}(G,\kappa)$ factors through the local ring of $\xspec \otimes_{\Lambda} k_{\phi}$ at the closed point $\phi$. 
\begin{Lem}
We have an inclusion $\mathcal{D}_{\mathrm{lis}}(\bun_{G},\kappa)_{\phi}^{\wedge} \subset \mathcal{D}_{\mathrm{lis}}(\bun_{G},\kappa)_{\phi}$.
\end{Lem}
\begin{proof}
By \cite[Lemma 3.17]{ZouCategoricalLS}, we know that $\mathcal{D}_{\mathrm{lis}}(\bun_{G},\kappa)_{\phi}^{\wedge}$ is compactly generated by objects which are compact and ULA in $\mathcal{D}(\bun_{G}, \kappa)$. Since $\mathcal{D}_{\mathrm{lis}}(\bun_{G},\kappa)_{\phi}$ is closed under colimits, it suffices to show that it contains the compact generators of $\mathcal{D}_{\mathrm{lis}}(\bun_{G},\kappa)_{\phi}^{\wedge}$. But this follows directly from \cite[Corollary 3.18]{ZouCategoricalLS}. 
\end{proof}
We thank David Hansen for helpful suggestions about the proof of the next lemma.
\begin{Lem} \label{Lem:CompletingLocalization}
If $\mathcal{F} \in \mathcal{D}_{\mathrm{lis}}(\bun_{G},\kappa)_{\phi}$ is an object whose corresponding object in $\mathcal{D}(\bun_{G}, \kappa)$ is ULA and supported at finitely many Newton strata, then $\mathcal{F} \in \mathcal{D}_{\mathrm{lis}}(\bun_{G},\kappa)_{\phi}^{\wedge}$. 
\end{Lem}
\begin{proof}
By \cite[Corollary 3.18]{ZouCategoricalLS}, it suffices to show that all irreducible subquotients of all cohomology sheaves of $\mathcal{F}$ have Fargues--Scholze $L$-parameter $\phi$. The category $D_{\mathrm{lis}}(\operatorname{Bun}_{G},\kappa)$ has a semi-orthogonal decomposition into $D_{\mathrm{lis}}(\bun^b_G,\kappa)$'s via excision triangles \cite[Theorem I.5.1]{FarguesScholze}; it follows that $\mathcal{F}$ has a filtration whose graded pieces are $i_{b!}i_b^\ast\mathcal{F}$. Thus it suffices to prove that the irreducible subquotients of the cohomology sheaves of $i_{b!}i_b^\ast\mathcal{F}$ have Fargues--Scholze $L$-parameter $\phi$. Since $i_{b!}$ is fully faithful, this comes down to showing that the cohomology sheaves of $i_b^\ast\mathcal{F}$ have Fargues--Scholze $L$-parameter $\phi$. 

Let $\mathcal{G}$ be one of these cohomology sheaves and let $f:\mathcal{G}_0 \to \mathcal{H}$ be an irreducible quotient of $\mathcal{G}_0 \subset \mathcal{G}$. Then $\mathcal{H}$ is an irreducible smooth representation, hence admissible (ULA). The map $f$ has a direct sum decomposition (source and target are ULA, now use \cite[Proposition A.5]{HamannLee})
\begin{align}
    \bigoplus_{\phi'} \mathcal{G}_{0,\phi'} \to \bigoplus_{\phi'} \mathcal{H}_{\phi'},
\end{align}
where the direct sums run over semisimple $L$-parameters over $\kappa$. Since $\mathcal{G}$ and hence $\mathcal{G}_0$ is $\phi$-local, it follows that $\mathcal{G}_{0,\phi'}$ is zero unless $\phi'=\phi$. Since $f$ is surjective, we see that $\mathcal{H}=\mathcal{H}_{\phi}$. Since $\mathcal{H}$ is irreducible, it follows from the definition that its $L$-parameter must equal $\phi$.
\end{proof}

\begin{Cor} \label{Cor:LocalizationCompletion}
Let $\phi$ be as above and suppose that either $\ell$ is coprime to the order of $\pi_0(\zg)$ or that $\ell$ is invertible in $\kappa$. If $\tildef_{\mathbb{W}}$ is ULA, then there is a morphism 
\begin{equation} \label{Eq:SpectralActionMapZhuConjectureCompletion}
    \begin{tikzcd}
        \operatorname{RHom}_{\lochat}(V_{\mu}, V_{\mu'}) \arrow{d} \\
        \operatorname{RHom}_{G(\qp) \times  \mathbb{T}_{K^p}}\left(R \Gamma(\mathbf{Sh}_{K^p}\gx_{\ebar},\mathbb{W})_{\phi}[d], R \Gamma(\mathbf{Sh}_{K^p}\gxp_{\ebar},\mathbb{W})_{\phi}[d'] \right).
        \end{tikzcd}
\end{equation}
\end{Cor}
\begin{proof}
Since $\tildef_{\mathbb{W}}$ is ULA, so is $\tildef_{\mathbb{W},\phi}$. Therefore by Lemma \ref{Lem:CompletingLocalization}, we see that $\tildef_{\mathbb{W},\phi} \in  \mathcal{D}_{\mathrm{lis}}(\bun_{G},\kappa)_{\phi}^{\wedge}$. The corollary follows.
\end{proof}

\newcommand{\colim}{\operatorname{colim}}

\subsubsection{} Recall that $\xspec$ is a disjoint union of affine schemes (\cite[Theorem VIII.1.3]{FarguesScholze}). Recall moreover from \cite[discussion before Theorem IX.5.2]{FarguesScholze} that there is a direct product decomposition
\begin{align}
    \mathcal{D}(\bung)=\prod_{c \in \pi_0(\mathrm{Par}_{\hat{G}})} \mathcal{D}(\bung)_{c}
\end{align}
according to the action of $\mathcal{O}(\mathrm{Par}_{\hat{G}})$ on $\mathcal{D}(\bung)$; this is compatible with the spectral action. 

\subsubsection{} From now until the end of this section we work in the following somewhat complicated setting. Fix a sufficiently large finite set $\mathscr{W} \subset \pi_0(\mathrm{Par}_{\hat{G}})$ and consider the induced spectral action of $\operatorname{Perf}(\mathrm{Par}_{\hat{G}},\mathscr{W})$ on $\mathcal{D}(\bung, \Lambda)_{\mathscr{W}}:=\textstyle \prod_{c \in \mathscr{W}} \mathcal{D}(\bung)_{c}$. Let $S \subset \mathcal{O}(\mathrm{Par}_{\hat{G}, \mathscr{W}})$ be a multiplicative subset and recall that $\mathcal{O}(\mathrm{Par}_{\hat{G}, \mathscr{W}})$ acts functorially on all objects of $\mathcal{D}(\bung)_{\mathscr{W}}$. The following definition is closely modeled on \cite[Definition A1]{HamannLee}.
\begin{Def}
    Write $\mathcal{D}(\bung)_{S}$ for the full subcategory of $\mathcal{D}(\bung)_{\mathscr{W}}$ consisting of objects where all $f \in S$ act via isomorphisms. 
\end{Def}
The following proposition is the direct analogue of \cite[Proposition A.2]{HamannLee}. 
\begin{Prop}\label{Prop:HansenApp}
If $\Lambda$ is a field, then the following hold.
\begin{enumerate}
        \item The subcategory $\mathcal{D}(\bung)_{S}$ is stable under the spectral action.

        \item The inclusion $\mathcal{D}(\bung)_{S} \subset \mathcal{D}(\bung)_{\mathscr{W}}$ has a left adjoint $\mathcal{L}_{S}$ which commutes with the spectral action. 
        
        \item The natural map $\mathrm{Hom}(B, A)[S^{-1}] \to \mathrm{Hom}(B, \mathcal{L}_S(A))$ is an isomorphism for all $A$ and for any compact object $B$. 
        \item The support of $A$ contains that of $\mathcal{L}_S(A)=A_{S}$.
    \end{enumerate}
\end{Prop}
We write $\mathcal{L}_S(A)=:A_{S}$.
\begin{proof}
    The proof is the same as \cite[A.1-A.3]{HamannLee}.
\end{proof}

\subsubsection{} Now assume that $S$ is finitely generated. Let $j_0: U_0 \to X_{\widehat{G}, \mathscr{W}}^{\mathrm{spec}}$ be the affine open immersion defined by localization at the set $S$, with closed complement $i_0: Z_0 \to X_{\widehat{G}, \mathscr{W}}^{\mathrm{spec}}$ and we let $R_{U_0} =\mathcal{O}_{X^{\mathrm{spec}}_{\hat{G}, \mathscr{W}}}(U_0)$. Finally we let $U$ be the fiber product
\begin{align}
    \mathrm{Par}_{\hat{G}, \mathscr{W}} \times_{X_{\hat{G}, \mathscr{W}}^{\mathrm{spec}}} \spec R_{U_0}
\end{align}
and let $j: U \to \mathrm{Par}_{\hat{G}}$ be the corresponding open immersion, with $i: Z \to \mathrm{Par}_{\hat{G}}$ the closed complement. 

\begin{Prop} \label{Prop:LocalizationOpen}
    Let $U$ be as above. If $\Lambda$ is a field, then the spectral action induces an action of the category $\operatorname{IndPerf}(U)$ on the category $\mathcal{D}(\bung, \Lambda)_S$ defined above. 
\end{Prop}
\begin{proof}
    Recall from \cite[Theorem VIII.5.2]{FarguesScholze} that the $\infty$-category $\operatorname{IndPerf}(\mathrm{Par}_{\hat{G}})_{\mathscr{W}}$ is generated under cones and retracts by the pullbacks of elements of $\mathrm{Perf}(\hat{G})$. One way to verify this is to use \cite[Theorem VIII.5.8 of v2]{FarguesScholze} applied to the scheme $Z^1(W_{\mathbb{Q}_p}, \widehat{G})_\Lambda$ parameterizing local Langlands parameters. We note that conditions (1), (2), (3) of this theorem are stable under passage to the open $\tilde{U} \subset Z^1(W_{\mathbb{Q}_p}, \widehat{G})_\Lambda$ that is the inverse image of $U$. Indeed, condition (1) clearly passes to any open whatsoever, whereas conditions (2), (3) are only checked on closed $G$-orbits in $Z^1(W_{\mathbb{Q}_p}, \widehat{G})_\Lambda$, and since $\widetilde{U}$ is obtained from base change from the open $U_0$, we see that closed orbits in $\widetilde{U}$ are in bijection with closed orbits in $Z^1(W_{\mathbb{Q}_p}, \widehat{G})_\Lambda$ which lie in $\widetilde{U}$. Thus applying the higher Bar--Beck--Lurie theorem \cite[Theorem 4.7.3.5]{HA}, and using the commutative diagram
    \[
    \begin{tikzcd}
        &U \arrow{r}{j}\arrow{dr}& \mathrm{Par}_{\hat{G}} \arrow{d}\\
        && \mathrm{B}\hat{G}
    \end{tikzcd}
    \]
    we obtain 
\begin{equation}
    \begin{tikzcd}
        \mathrm{Mod}_{\mathcal{O}(Z^1(W_{\mathbb{Q}_p}, \widehat{G})_\Lambda)}(\mathrm{IndPerf}(\mathrm{B}\widehat{G})) \arrow[d, "\sim"] \arrow{r}{L} & \mathrm{Mod}_{\mathcal{O}(\widetilde{U})}(\mathrm{IndPerf}(\mathrm{B}\widehat{G}))  \\
        \mathrm{IndPerf}(\mathrm{Par}_{\hat{G}}) \arrow{r}{j^{\ast}} & \mathrm{IndPerf}(U). \arrow{u}{\sim} 
    \end{tikzcd}
\end{equation}
Here the vertical equivalences are induced by pushforward, and the top horizontal arrow (defined such that the diagram commutes) is the functor $L(M) = M \otimes_{\mathcal{O}(Z^1(W_{\mathbb{Q}_p}, \widehat{G})_\Lambda)} \mathcal{O}(\widetilde{U})$. By abuse of notation we denote by $S$ the subset of $Z(\mathrm{Mod}_{\mathcal{O}(Z^1(W_{\mathbb{Q}_p}, \widehat{G})_\Lambda)}(\mathrm{IndPerf}(\mathrm{B}\widehat{G})))$ induced by the multiplicative set $S$, then the functor $L$ induces an equivalence
    \[
        \overline{L}: \mathrm{Mod}_{\mathcal{O}(Z^1(W_{\mathbb{Q}_p}, \widehat{G})_\Lambda)}(\mathrm{IndPerf}(\mathrm{B}\widehat{G}))[S^{-1}] \xrightarrow{\sim} \mathrm{Mod}_{\mathcal{O}(\widetilde{U})}(\mathrm{IndPerf}(\mathrm{B}\widehat{G})).
    \]

    Now we relate this back to the spectral action. Because the spectral action of $\mathrm{IndPerf}(\mathrm{Par}_{\widehat{G}})$ on $\mathcal{D}(\bung)_S$ sends the elements of $S$ to invertible morphisms (by part (1) of proposition \ref{Prop:HansenApp}), by the universal mapping property of localization there exists an induced action of $\mathrm{IndPerf}(U)$ on $\mathcal{D}(\bung)_S$. 
\end{proof}

We have the following upgrade of Corollary \ref{Cor:SpectralActnExists}.

\begin{Cor}\label{Cor:SpectralActnLocalizes}
Suppose that either $\ell$ is coprime to the order of $\pi_0(\zg)$ or that $\ell$ is invertible in $\Lambda$, and let $U \subset \loc$ be an open substack cut out by a finitely generated multiplicative set $S \subset \mathcal{O}_{\rZ^1(W_{\qp}, \hat{G})//\hat{G}}$. Then there exists a natural map
\begin{equation} \label{Eq:SpectralActionMapZhuConjectureOpen}
    \begin{tikzcd}
        \operatorname{RHom}_{U}(V_{\mu}, V_{\mu'}) \arrow{d} \\
        \operatorname{RHom}_{G(\qp) \times  \mathbb{T}_{K^p}}\left(R \Gamma(\mathbf{Sh}_{K^p}\gx_{\ebar},\mathbb{W})_{\phi}[d]_S, R \Gamma(\mathbf{Sh}_{K^p}\gxp_{\ebar},\mathbb{W})_{\phi}[d']_S\right).
        \end{tikzcd}
\end{equation}
\end{Cor}

\subsection{Examples} \label{sub:ExamplesSpectral} The goal of this section is to give example applications of Corollaries \ref{Cor:SpectralActnExists}, \ref{Cor:LocalizationCompletion} and \ref{Cor:SpectralActnLocalizes}. We let $\Lambda$ and $W$ be as in Setup \ref{Setup:StandardSetup} and continue to assume that $P$ is trivial. Fix a field $\Lambda \to \kappa$ and a closed point $\phi:\spec \kappa \to \xspec$ corresponding to a semisimple $L$-parameter $\phi$. 

\subsubsection{} Let $S_{\phi} \subset \dualG$ be the centralizer of $\phi$. Recall that $\phi$ is called \emph{generous}, see \cite[Definition 2.1.5]{hansen2024beijingnotescategoricallocal}, if there is an isomorphism
\begin{align}
    \loc \times^{\mathbb{L}}_{\xspec} \{\phi\} \simeq \left[\spec \kappa / S_{\phi} \right],
\end{align}
where $\mathbb{L}$ denotes the derived fiber product. When $\kappa$ has characteristic $\ell$, generous parameters are also called parameters of Langlands--Shahidi type, see \cite[Remark 6.5]{HamannLee}. We have the following consequence of an amazing theorem of Alper--Hall--Rydh \cite{AlperHallRydh}.
\begin{Thm} \label{Thm:LinearlyReductiveCompletion}
    If $\phi$ is generous and $S_{\phi}$ is linearly reductive over $\kappa$, then there is an isomorphism 
    \begin{align}
        \lochat \simeq \left[ \spf A / S_{\phi} \right],
    \end{align}
    where $\spf A$ is the completion of a finite type $\kappa$-algebra in a finitely generated ideal of definition, together with a commutative diagram
    \begin{equation}
        \begin{tikzcd}
            \lochat \arrow{r} \arrow{d} & \loc \arrow{dr} \\
            \left[ \spf A / S_{\phi} \right] \arrow{r} & \left[\spec \kappa / S_{\phi} \right] \arrow{r} & \left[\spec \kappa / \dualG \right].
        \end{tikzcd}
    \end{equation}
\end{Thm}
\begin{proof}
This is a direct consequence of \cite[Theorem 1.1]{AlperHallRydh} applied to $\mathfrak{X}=\loc$, $x=\phi$ and $H=G_x=S_{\phi}$.
\end{proof}

\begin{Rem}
If $\kappa$ has characteristic zero, then $S_{\phi}$ is linearly reductive if and only if the identity component of $S_{\phi}$ is reductive over $\kappa$. If $\kappa$ does not have characteristic zero, then $S_{\phi}$ is linearly reductive if the identity component of $S_{\phi}$ is a torus and $\pi_0(S_{\phi})$ is of order prime to $\ell$.
\end{Rem}

\begin{Cor} 
Suppose that $\tildef_{\mathbb{W}}$ is ULA. If $\phi$ is generous and $S_{\phi}$ is linearly reductive over $\kappa$, then there is a morphism 
\begin{equation} \label{Eq:SpectralActionMapZhuConjectureCompletionII}
    \begin{tikzcd}
        \operatorname{RHom}_{S_{\phi}}(V_{\mu}, V_{\mu'}) \arrow{d} \\
        \operatorname{RHom}_{G(\qp) \times  \mathbb{T}_{K^p}}\left(R \Gamma(\mathbf{Sh}_{K^p}\gx_{\ebar},\mathbb{W})_{\phi}[d], R \Gamma(\mathbf{Sh}_{K^p}\gxp_{\ebar},\mathbb{W})_{\phi}[d'] \right).
        \end{tikzcd}
\end{equation}
\end{Cor}
\begin{proof}
This is a direct consequence of Theorem \ref{Thm:LinearlyReductiveCompletion} together with Corollary \ref{Cor:LocalizationCompletion}.
\end{proof}
\begin{Cor} \label{Cor:DirectSummandI}
Suppose that $\tildef_{\mathbb{W}}$ is ULA. If $\phi$ is generous and $S_{\phi}$ is linearly reductive over $\kappa$ and $V_{\mu}$ is a direct summand of $V_{\mu'}$ as a representation of $S_{\phi}$, then $R \Gamma(\mathbf{Sh}_{K^p}\gx_{\ebar},\mathbb{W})_{\phi}[d]$ is a $G(\qp) \times  \mathbb{T}_{K^p}$-equivariant direct summand of 
$R \Gamma(\mathbf{Sh}_{K^p}\gxp_{\ebar},\mathbb{W})_{\phi}[d']$.
\end{Cor}
\begin{proof}
This is a direct consequence of Theorem \ref{Thm:LinearlyReductiveCompletion} together with Lemma \ref{Lem:CompletingLocalization} applied to $\tildef_{\mathbb{W}, \phi}$, and the additivity of the spectral action of $\operatorname{Perf}(\lochat)$ on $\mathcal{D}_{\mathrm{lis}}(\bun_{G},\kappa)_{\phi}^{\wedge}$.
\end{proof}
\begin{eg}
Suppose that $S_{\phi}=Z(\dualG)^{\gal_{\qp}}$. Then $V_{\mu}$ is a direct summand of $V_{\mu'}$ as a representation of $S_{\phi}$ (or vice versa) if and only if $\kappa(\mu)=\kappa(\mu')$ if and only if $\bgmumup$ is nonempty.
\end{eg}

\subsection{What if the groups are different at \texorpdfstring{$p$}{p}} \label{Sub:DifferentAtP} Let us start again from Proposition \ref{Prop:WeilCohShimVar}, but now in the situation that $P$ is not trivial so that $[1] \not=[1']$. In this case, it follows as in the proof of Corollary \ref{Cor:SpectralActnExists} that there is a morphism
\begin{align}
    \operatorname{RHom}_{\loc}(V_{\mu}, V_{\mu'}) \to \operatorname{RHom}_{D_{\mathrm{lis}}(\bun_{G,k}, \Lambda)}(T_{\mu} \tildef, T_{\mu'} \tildef).
\end{align}
The left hand side is zero, because there are no homomorphisms between $V_{\mu}$ and $V_{\mu'}$ considered as representations of $Z(\dualG)^{\gal_{\qp}}$. Indeed, both are direct sums of a single character $\kappa(\mu), \kappa(\mu') \in \pi_1(G)_{\gal_{\qp}}$, which differ by $\kappa([P]) \not=0$. The right hand side is also zero, because the sheaves $T_{\mu} \tildef, T_{\mu'} \tildef$ are supported on different connected components of $\bun_{G,k}$ (this follows from the compatibility of the spectral action with the central grading, see \cite[Lemma 5.3.2]{zou2024categoricalformfarguesconjecture}). Therefore, it is clear that a new idea is needed to relate the cohomology groups of the different Shimura varieties (see \cite[Conjecture 4.7.16]{ZhuCoherent} for a conjectural comparison at parahoric level).

\subsubsection{} Let $\Lambda, W$ be as in Setup \ref{Setup:StandardSetup}. Suppose that $\Lambda$ is a field, choose a semisimple $L$-parameter $\phi:W_{\qp} \to {}^{L} G(\flbar)$ and a morphism $\mathbb{T}_{K^p} \to \Lambda$ with kernel $\mathfrak{m}$. Recall that $\phi$ is called \emph{supercuspidal} if $Z(\dualG)^{\gal_{\qp}} \subset S_{\phi}$ has finite index. The following result and its proof were suggested to us by David Hansen.
\begin{Prop} \label{Prop:NonzeroIFF}
Assume that $\kappa(\mu)+\kappa([P])=\kappa(\mu')$. Suppose that $\phi$ is supercuspidal and that $\dim V_{\mu}, \dim V_{\mu'}$ are invertible in $\Lambda$. Then $$R \Gamma(\mathbf{Sh}_{K^p}\gx_{\ebar},\Lambda)_{\phi, \mathfrak{m}} \not=0$$ if and only if $$R \Gamma(\mathbf{Sh}_{K^p}\gxp_{\ebar},\Lambda)_{\phi, \mathfrak{m}} \not=0.$$
\end{Prop}
\begin{proof}
Under our assumption, the inclusion $\Lambda \to V_{\mu} \otimes V_{\mu}^{\vee}=\operatorname{End}(V_{\mu})$ of $\hat{G}$-representations has a section given by $f \mapsto \tfrac{\operatorname{Tr} f}{\dim V_{\mu}}$, turning it into a direct summand, and the same works with $\Lambda \to V_{\mu'} \otimes V_{\mu'}^{\vee}$. \smallskip 

It follows from the above observation and the functoriality of the spectral action that $\tildef$ is functorially a direct summand of $T_{V_{\mu}^{\vee}} \circ T_{\mu}$ and also functorially a direct summand of $T_{V_{\mu'}^{\vee}} \circ T_{\mu'}$; this implies that the same holds for $\tildef_{\phi, \mathfrak{m}}$. In particular, we find that
\begin{align}
    T_{\mu} \tildef_{\phi, \mathfrak{m}} \not =0 \iff  \tildef_{\phi, \mathfrak{m}} \not =0 \iff T_{\mu'} \tildef_{\phi, \mathfrak{m}} \not =0.
\end{align}
Our assumption that $\kappa(\mu)+\kappa([P])=\kappa(\mu')$ tells us that $\tildef$ is supported on a single component of $\bung$, namely the component associated to $\kappa(\mu)$. It follows that $T_{\mu} \tildef$ and $T_{\mu'} \tildef$ are also supported on a single component, namely the components containing $[1]$ and $[1']$, respectively. Since 
$\phi$ is supercuspidal, it follows from this that (by applying \cite[Theorem IX.7.2]{FarguesScholze})
\begin{align}
    T_{\mu} \tildef_{\phi, \mathfrak{m}} \not =0 &\iff i_1^{\ast} T_{\mu} \tildef_{\phi, \mathfrak{m}} \not =0 \\
    T_{\mu'} \tildef_{\phi, \mathfrak{m}} \not =0 &\iff i_{1'}^{\ast} T_{\mu'} \tildef_{\phi, \mathfrak{m}} \not =0.
\end{align}
The proposition now follows from Propositions \ref{Prop:WeilCohShimVar} and \ref{Prop:WeilCohShimVarRat}.
\end{proof}

\subsubsection{} Now suppose that $S_{\phi}=Z(\dualG)^{\gal_{\qp}}$. Recall that characters of $Z(\dualG)^{\gal_{\qp}}$ correspond to $\pi_1(G)_{\gal_{\qp}}$. Using Theorem \ref{Thm:LinearlyReductiveCompletion} we can define line bundles $\mathcal{L}(\lambda)$ on $\lochat$ for $\lambda \in \pi_1(G)_{\gal_{\qp}}$. Suppose that $\kappa(\mu)+\kappa([P])=\kappa(\mu')$, and define the line bundle $\mathcal{L}=\mathcal{L}(-\kappa([P]))$ on $\lochat$ as above. There is a natural equivalence (see \cite[discussion before Lemma 5.3.2]{zou2024categoricalformfarguesconjecture})
\begin{align}
    D_{\mathrm{lis}}(\bun_{G,k}, \Lambda)_{\phi} = \prod_{\lambda \in \pi_1(G)_{\gal_{\qp}}} D_{\mathrm{lis}}(\bun_{G,k}, \Lambda)_{\phi, \lambda}. 
\end{align}
The Hecke operator $T_{\mathcal{L}}$ defines an equivalence (because $\mathcal{L}$ is a line bundle)
\begin{align}
    D_{\mathrm{lis}}(\bun_{G,k}, \Lambda)^{\wedge}_{\phi} \isom D_{\mathrm{lis}}(\bun_{G,k}, \Lambda)^{\wedge}_{\phi},
\end{align}
which takes the summand indexed by $\lambda$ to the summand indexed by $\lambda-\kappa([P])$ (by \cite[Lemma 5.3.2]{zou2024categoricalformfarguesconjecture}). In particular, since $\phi$ is supercuspidal, it induces an equivalence
\begin{align}
    D_{\mathrm{lis}}(\bun_{G,k}^{[1']}, \Lambda)_{\phi}^{\wedge} \to D_{\mathrm{lis}}(\bun_{G,k}^{[1]}, \Lambda)_{\phi}^{\wedge}
\end{align}
or alternatively
\begin{align}
    T_{\mathcal{L}}:D_{\mathrm{lis}}(G'(\qp), \Lambda)_{\phi}^{\wedge} \to D_{\mathrm{lis}}(G(\qp), \Lambda)_{\phi}^{\wedge}.
\end{align}
This is some incarnation of a local Jacquet--Langlands correspondence.
\begin{Thm} \label{Thm:arii2j3asd}
Suppose that $S_{\phi}=Z(\dualG)^{\gal_{\qp}}$ and that $\kappa(\mu)+\kappa([P])=\kappa(\mu')$. For $N=i_{1}^{\ast} T_{\mathcal{L}(\kappa(\mu))}\tildef_{\mathbb{W},\phi}$, there is a $G(\qp) \times \mathbb{T}_{K^p}$-equivariant isomorphism 
    \begin{align}
        N^{\oplus \dim V_{\mu}} \isom  R \Gamma(\mathbf{Sh}_{K^p}\gx_{\ebar}, \mathbb{W})_{\phi}[d]
    \end{align}
    and a $G(\qp) \times \mathbb{T}_{K^p}$-equivariant isomorphism
    \begin{align}
        N^{\oplus \dim V_{\mu'}} \isom T_{\mathcal{L}} R \Gamma(\mathbf{Sh}_{K^p}\gxp_{\ebar}, \mathbb{W})_{\phi}[d'].
    \end{align}
\end{Thm}
\begin{proof}
The first statement follows directly from the definitions. For the second statement, we note that
\begin{align}
     T_{\mathcal{L}(\kappa(\mu'))} j_{k,!} \tildef_{\mathbb{W},\phi} =  T_{\kappa(\mathcal{P})} \circ T_{\mathcal{L}(\kappa(\mu))} j_{k,!} \tildef_{\mathbb{W},\phi}.
\end{align}
Recall moreover that
\begin{align}
    i_{1'}^{\ast} T_{\mathcal{L}(\kappa(\mu'))}^{\oplus \operatorname{Dim} V_{\mu'}} j_{k,!} \tildef_{\mathbb{W},\phi} \isom R \Gamma(\mathbf{Sh}_{K^p}\gxp_{\ebar}, \mathbb{W})_{\phi}[d'].
\end{align}
Since $\phi$ is supercuspidal, the result now follows as in the proof of Proposition \ref{Prop:NonzeroIFF}.
\end{proof}

\appendix

\section{Standard Hodge type Shimura data} \label{Appendix:HodgeEmbeddings}
The goal of this section is to give an abstract presentation of the construction of certain standard Hodge type Shimura data of type $B,C$ and $D^{\mathbb{R}}$, and to analyze their properties. This is used in Section \ref{Sec:Examples} and has been moved here for future reference and to de-clutter that section. Our presentation is inspired by \cite[Section 6]{DeligneTravaux}, \cite{Shih}, \cite[Section 12 of v1]{KretShinSymplecticArxiv}. 

\subsection{Abstract setup} \label{Sub:AbstractHodge}

\subsubsection{} Let $\mL^{+}$ be a totally real field and let $\g_0$ be a connected reductive group over $\mL^{+}$ with $\g_0^{\mathrm{ad}}$ simple. Let $\Sigma_{\infty}$ be the set of morphisms $\tau:\mlp \to \mathbb{R}$ and write $\Sigma_{\infty,\mathrm{c}} \subset \Sigma_{\infty}$ be the set of infinite places $\tau$ where $\g_{0,\tau,\mathbb{R}}$ has compact modulo center $\mathbb{R}$-points, write $\Sigma_{\infty,\mathrm{nc}}$ for its complement. Let $\{\x_{\tau}\}_{\tau \in \Sigma_{\infty,\mathrm{nc}}}$ be weak real Shimura data for $\g_{0,\tau,\mathbb{R}}:=\g_0 \otimes_{\mlp, \tau} \mathbb{R}$. We use this to equip $\g_4=\operatorname{Res}_{\mlp/\mathbb{Q}} \g_0$ with a weak Shimura datum $\x_4$. It is a Shimura datum (i.e., it satisfies Milne's axiom SV3) if $\Sigma_{\infty,\mathrm{nc}} \not=\emptyset$. \smallskip

Let $(V_0,\psi_{0})$ be a symplectic space over $\mlp$ with similitude group $\g_{V_{0}}$ (over $\mlp$) and let $\iota_0:\g_0 \to \g_{V_{0}}$ be a closed immersion of connected reductive groups over $\mlp$ such that the image of $\iota_0$ contains the scalars. We \textbf{assume} that for $\tau \in \Sigma_{\infty, \mathrm{nc}}$, the map $\iota_{0,\tau,\mathbb{R}}:=\iota_0 \otimes_{\mlp, \tau} \mathbb{R}$ is a morphism of weak real Shimura data and \textbf{assume} that the induced map $X_{\tau} \to H_{V_{0},\tau, \mathbb{R}}$ is surjective on $\pi_0$. Write $\g_{0}^0  $ for the kernel of $\g_0 \to \g_{V_{0}} \xrightarrow{\operatorname{sim}} \mathbb{G}_m$, and \textbf{assume} that it is connected. Note that the induced map $\g_4 \to \operatorname{Res}_{\mlp/\mathbb{Q}} \g_{V_{0}}$ is not a morphism of Shimura data unless $\Sigma_{\infty, \mathrm{nc}}=\Sigma_{\infty}$. \smallskip 

We are going to construct a Shimura datum of Hodge type $\gx$ together with a distinguished Hodge embedding $\iota$, from $\g_0, \x_0=\{\x_{\tau}\}_{\tau \in \Sigma_{\infty,\mathrm{nc}}}$ and $\iota_0$, depending on a choice of totally negative element $\epsilon \in \mlp$. This Shimura datum will come equipped with a natural ad-isomorphism $\g \to \g_4^{\mad}=:\g_2$ compatible with Shimura data. 

\subsubsection{} Let $\mL$ be a quadratic totally imaginary extension of $\mlp$, let $\mathsf{T}_{0,\mL}=\operatorname{Res}_{\mL/\mlp} \mathbb{G}_{m}$ and write $V_{0,\ml}=V_{0} \otimes_{\mlp} \mL$. Write $\mL=\mlp(\sqrt{\epsilon})$ for some totally negative element $\epsilon \in \mlp$. Let $\psi'_{\epsilon}$ be the bilinear form on $\ml$ given by $\psi'_{\epsilon}(y_1, y_2)=\operatorname{Tr}_{\ml/\mlp}(\epsilon y_1 y_2^c)$. Finally we consider the $\mlp$-bilinear alternating form $\psi_{\mL,\epsilon} = \psi_0 \otimes \psi'_{\epsilon}$ on $V_{0,\mL}$.

\subsubsection{} Our choice of $\epsilon$ induces for each $\tau \in \Sigma_{\infty}$ a complex place $\tilde{\tau}$ of $\mL$ extending $\tau$. This induces an isomorphism
\begin{align}
    \mathsf{T}_{0,\mL} \otimes_{\mlp, \tau} \mathbb{R} \xrightarrow{\sim} \mathbb{S},
\end{align}
and we use this to equip $\mathsf{T}_{\mL}=\operatorname{Res}_{\mlp/\Q} \mathsf{T}_{0,\mL} = \operatorname{Res}_{\ml/\Q} \mathbb{G}_m$ with the weak real Shimura datum which is trivial for $\tau \in \Sigma_{\infty,\mathrm{nc}}$ and the identity map for $\tau \in \Sigma_{\infty,\mathrm{c}}$. We take the product of $\mathsf{T}_{\mL} \times \g_4$ and equip it with the product weak Shimura datum. 

\subsubsection{} The natural map
\begin{align}
    \mathsf{T}_{0,\mL} \times \g_0 \to \g_{V_{0,\mL}}
\end{align}
which lets $(t,g)$ act on $V_{0,\ml}$ by $t \otimes g$, factors through the quotient $\g_{0,+}$ of $\mathsf{T}_{0,\mL} \times \g_0$ by the anti-diagonally embedded
\begin{align}
    \mathbb{G}_{m} \subset \mathsf{T}_{0,\mL} \times \g_0.
\end{align}
Let us write $\iota_{0,+}$ for the induced map from $\g_{0,+} \to \g_{V_{0,\mL}}$. We moreover note that the induced map
\begin{align}
    \iota_3:\g_{3}:=\operatorname{Res}_{\mlp/\Q} \g_{0,+}\to \operatorname{Res}_{\mlp/\mathbb{Q}} \g_{V_{0,\mL}}
\end{align}
underlies a morphism of weak Shimura data, see e.g. Section \ref{subsub:ComputationHodgeEmbedding}. Let us write $\gxthree$ for the weak Shimura datum with underlying group $\g_3$ constructed above.

\subsubsection{} Now \textbf{assume} that
\begin{align}
    \g_{0} \xrightarrow{\iota_0} \g_{V_{0}} \xrightarrow{\operatorname{sim}} \gm
\end{align}
identifies $\gm$ with the maximal abelian quotient of $\g_0$; this will be satisfied in all our examples. Then the maximal abelian quotient of $\g_0 \times \mathsf{T}_{0, \mL}$ is given by (this follows from a straightforward dimension argument)
\begin{align}
    \gm \times \mathsf{T}_{0, \mL}
\end{align}
via the map $(g,t) \mapsto (\operatorname{sim}(g),t)$. Let $t \mapsto t^c$ be the nontrivial element of $\operatorname{Gal}(\mL/\mlp)$, and let $\mathsf{T}^1_{0, \mL} \subset \mathsf{T}_{0, \mL}$ be the kernel of $t \mapsto tt^c$ (the norm one torus). Consider the natural map
\begin{align}
    \operatorname{sim}':\g_0 \times \mathsf{T}_{0,\ml} &\to \gm \times \mathsf{T}^1_{0, \mL} \\
    (g,t) &\mapsto (\operatorname{sim}(g) t t^c, t/t^c).
\end{align}
\begin{Lem} \label{Lem:CocentersHodgeType}
If $\operatorname{sim} \circ \iota_0$ identifies $\gm$ with the maximal abelian quotient of $\g_0$, then the map $\operatorname{sim}'$ factors through $\g_0 \to \g_{0,+}$ and the induced map
    \begin{align}
    \operatorname{sim}'':\g_{0,+} \to \gm \times \mathsf{T}^1_{0, \mL}
    \end{align}
    identifies the target with the maximal abelian quotient of $\g_{0,+}$.
\end{Lem}
\begin{proof}
The map $\operatorname{sim}'$ is clearly surjective, and the kernel of the induced map $\g_0^{\mathrm{ab}} \times \mathsf{T}_{0,\ml} \to \gm \times \mathsf{T}^1_{0, \mL}$ is given by the anti-diagonally embedded $\gm \subset \gm \times \mathsf{T}_{0,\ml}$. Thus $\operatorname{sim}''$ is surjective, and identifies the target with the maximal abelian quotient of $\g_{0,+}$ since $\mathsf{G}_{0,+}$ has maximal abelian quotient of rank two.
\end{proof}

\subsubsection{} As a corollary of Lemma \ref{Lem:CocentersHodgeType}, we see that $\g_{0,+}$ is connected because it sits in a short exact sequence
\begin{align}
    1 \to \g_{0}^0 \to \g_{0,+} \to \gm \times \mathsf{T}^1_{0, \mL} \to 1
\end{align}
and $\g_{0}^0$ is connected by assumption. Let us write $\g_{0,++} \subset \g_{0,+}$ for the inverse image of $\{1\} \times \mathsf{T}^1_{0, \mL}$, it is connected because it is an extension of $\mathsf{T}^1_{0, \mL}$ by $\g_{0}^0$.

\subsubsection{} We can consider the symplectic space $V_{0,\ml}$ over $\mathbb{Q}$ equipped with the symplectic form $\psi_{\mathbb{Q}}=\operatorname{Tr}_{\mlp/\Q} \psi_{\ml,\epsilon}$; we will denote it by $V_{\ml}$. Consider the similitude map 
\begin{align}
    \operatorname{Res}_{\mlp/\mathbb{Q}} \g_{V_{0,\mL}} \xrightarrow{\operatorname{sim}} \operatorname{Res}_{\mlp/\mathbb{Q}} \mathbb{G}_{m,\mlp},
\end{align}
and $\operatorname{Res}_{\mlp/\mathbb{Q}}' \g_{V_{0,\mL}}:=\operatorname{sim}^{-1}(\mathbb{G}_m)$ equipped with the weak Shimura datum
\begin{align} \label{Eq:SignedShimuraDatum}
\left(\prod_{\tau} H_{V_{\ml},0,\tau}^+\right) \coprod \left(\prod_{\tau} H_{V_{\ml},0,\tau}^{-}\right) \subset \prod_{\tau} H_{V_{\ml},0,\tau}.
\end{align}
There is a natural map $\operatorname{Res}_{\mlp/\mathbb{Q}}' \g_{V_{0,\mL}} \to \g_{V_{\ml}}$, which is a morphism of weak Shimura data. We now define $\g \subset \g_{3}$ to be the inverse image of $\gm \subset \operatorname{Res}_{\mlp/\mathbb{Q}} \mathbb{G}_{m,\mlp}$ under $\operatorname{sim} \circ \iota_3$ (which is connected since $\g_{0,++}$ is). We equip it with the weak Shimura datum given by the intersection of $\x_3$ and \eqref{Eq:SignedShimuraDatum}.\footnote{This weak Shimura datum is nonempty because we have assumed that for all $\tau \in \Sigma_{\infty, \mathrm{nc}}$, the map $X_{\tau} \to H_{V_{0},\tau,\mathbb{R}}$ induced by $\iota_0$ is surjective on $\pi_0$.} We get a commutative diagram
\begin{equation}
    \begin{tikzcd}
        &\g \arrow{dl}{\iota} \arrow[r, hook] \arrow{d}{j} & \g_3 \arrow{d}{\iota_3} \\
        \g_{V_{\ml}} & \operatorname{Res}_{\mlp/\mathbb{Q}}' \g_{V_{0,\mL}} \arrow[r, hook] \arrow{d} \arrow{l}& \operatorname{Res}_{\mlp/\mathbb{Q}} \g_{V_{0,\mL}} \arrow{d}{\operatorname{sim}} \\
        & \mathbb{G}_m \arrow[r, hook] & \operatorname{Res}_{\mlp/\mathbb{Q}} \mathbb{G}_{m,\mlp}.
    \end{tikzcd}
\end{equation}
We will sometimes adorn $\g, \x, \iota, \g_{0,+}$ and $\g_{0,++}$ with the subscript $\Upsilon$, where $\Upsilon=(\g_0, \x_0, \iota_0, V_0, \psi_0, \epsilon)$, to emphasize their dependence on the choices made above. 

\subsubsection{Compatibility with inner twisting} \label{subsub:AppendixInnerTwisting} Write $\g_3 \otimes \mathbb{R}$ as 
\begin{align}
    \prod_{\tau} \g_{0,+,\tau, \mathbb{R}} = \prod_{\tau} \frac{\g_{0,\tau,\mathbb{R}} \times \mathbb{S}}{\mathbb{G}_{m}}.
\end{align}
Let $S_{\infty} \subset \Sigma_{\infty,\mathrm{nc}}$ be a subset of places and let $h=(h_{\tau})_{\tau} \in \prod_{\tau} x_{\tau}$ be an element such that $\iota_{0,\tau} \circ h_{\tau} \in \mathsf{H}_{V_{0,\tau}}^{\pm}$. Following the notation in Section \ref{subsub:ComputationHodgeEmbedding}, consider the element $k=(k_{\tau})_{\tau} \in \g_3(\mathbb{R})$ where
\begin{align}
    k_{\tau} = \twopartdef{(h_{\tau}(i),-i)}{\tau \in S_{\infty}}{1}{}.
\end{align}
Then $k$ lies in the kernel of the similitude map, and thus in particular in $\g(\mathbb{R})$. Similarly consider the morphism $h':\mathbb{S} \to \g_3$ given by $h'=(h'_{\tau})_{\tau}$, where $h'_{\tau}=h_{\tau}$ when $\tau \not \in S_{\infty}$ and given by $z \mapsto h_{\tau}(z) \cdot [(1,z)]$ if $\tau \in S_{\infty}$. Let $\mathsf{P}$ be a $\g$-torsor and consider $\g'=\operatorname{Aut}_{\g}(\mathsf{P})$ together with $\iota':\g' \to \g_{V_{L}'}$. 
\begin{Prop} \label{Prop:TwistIsHodgeEmbedding}
If $\mathsf{P} \otimes \mathbb{R}$ is isomorphic to the torsor determined by the cocycle $k$, then $h'$ naturally factors through $\g'$ and $\iota' \circ h' \in \mathsf{H}_{V_{\mathsf{L}}'}^{\pm}$.
\end{Prop}
\begin{proof}
Write $\mathsf{P}_3$ for the pushout of $\mathsf{P}$ to $\g_3$, and write $\g_3'$ for the induced pure inner twist with map $\iota_3'$. It follows directly from Lemma \ref{Lem:HodgeEmbeddingComputation} that $\iota_3'$ is a morphism of weak Shimura data. We will show that $j'$ is also a morphism of weak Shimura data, implying that $\iota'$ is a morphism of weak Shimura data. For this, we have to check that if $h$ lies in \eqref{Eq:SignedShimuraDatum}, then so does $h'$ (with $V$ replaced by $V'$). But this follows from the statement of Lemma \ref{Lem:HodgeEmbeddingComputation}, since it shows that $h'_{3,\tau}$ and $h_{3,\tau}$ land in the same connected component of $H_{V_{L},0,\tau}$.
\end{proof}

\subsubsection{Compatibility with the Piatetski--Shapiro construction} \label{Subsub:PSConstruction} Let $\mlp$ and $\Upsilon$ be as above, let $\f$ denote a totally real field which is disjoint from $\mlp$ and let $\mkp$ be the composite of $\f$ and $\mlp$. Consider $\gx$ and $\iota$ defined as above using $\Upsilon$. Let $\h_{1} \subset \operatorname{Res}_{\f/\mathbb{Q}} \g_{\f}$ be the inverse image of $\mathbb{G}_m \subset \operatorname{Res}_{\f/\mathbb{Q}} \gm$ under the similitude map. The natural map $\g \to \operatorname{Res}_{\f/\mathbb{Q}} \g_{\f}$ factors through $\h_{1}$, and in fact we get morphisms of weak Shimura data $\gx \to \hyo$. It comes equipped with a Hodge embedding (noting that $V_{\mkp} = V_{\mlp} \otimes_{\mathbb{Q}} \f$, since $\f \otimes_{\mathbb{Q}} \mlp=\mkp$)
\begin{align}
    \hyo \to (\g_{V_{\mkp}}, \mathsf{H}_{V_{\mkp}}).
\end{align}
\subsubsection{} On the other hand, define 
\begin{align}
    \h_0=\h_{0,\f}=\operatorname{Res}_{\mkp/\mlp} \g_{0,\mkp},
\end{align}
together with $\iota_{0,\f}:\h_0 \to \operatorname{Res}_{\mkp/\mlp} \g_{V_{0},\mkp} \to \g_{V_{\mkp},0}$, where $V_{\mkp}$ is equipped with the $\mlp$-linear symplectic form $\operatorname{Tr}_{\mkp/\mlp} \psi_{\mkp}$. We equip $\h_0$ with the weak Shimura datum $\mathsf{Y}_0$ via the natural map $\g_0 \to \h_0$. We now write $\Upsilon_{\f}=(\h_0,\mathsf{Y}_0, \iota_{0,\f}, V_{\mkp}, \operatorname{Tr}_{\mkp/\mlp} \psi_{\mkp}, \epsilon)$ and write $\mathsf{K}=\mkp(\sqrt{\epsilon})$. Write $\hyop=\gx_{\Upsilon_{\f}}$, together with the Hodge embedding
\begin{align}
    \hyop \to (\g_{V_{\mkp}}, \mathsf{H}_{V_{\mkp}}).
\end{align}
\begin{Lem} \label{Lem:PSCompatibility}
There is an equality $\hyop=\hyo$ as sub-(weak Shimura data) of $(\g_{V_{\mkp}}, \mathsf{H}_{V_{\mkp}})$.
\end{Lem}
\begin{proof}
This is a straightforward consequence of the definitions, using the exactness of Weil restriction. 
\end{proof}

\subsection{Type \texorpdfstring{$B, D^{\mathbb{R}}$} {B, DR}} 
\label{Sub:ExampleBDR} Let $\mL^{+}$ be a totally real field and let $(W_0,q)$ be a quadratic space over $\mL^{+}$ of dimension $n$. We will assume that for all $\tau \in \Sigma_{\infty}$ the signature $(r_{\tau}, s_{\tau})$ of $W_0 \otimes_{\mlp,\tau} \R$ has the property that $r_{\tau}, s_{\tau} \in \{0,2,n-2,n\}$. Write $\Sigma_{\infty,\mathrm{c}}$ for the set of embeddings $\tau:\mL^{+} \to \mathbb{R}$ where the signature $(r_{\tau}, s_{\tau})$ of $W_0 \otimes_{\mlp, \tau} \R$ is given by $(0,n)$ or $(n,0)$, and write $\Sigma_{\infty,\mathrm{nc}}$ for the set of embeddings so that the signature is $(2,n-2)$ or $(n-2,2)$. Let $\g_0=\operatorname{GSpin}_{W_{0}}$ be the group of spinor similitudes of $(W_0,q)$, see \cite[Section 3]{KretShinOrthogonal}. For $\tau \in \Sigma_{\infty,\mathrm{nc}}$ we define a weak real Shimura datum $X_{\tau}$ for $\g_{0,\tau}$ as in \cite[Section 3.1]{MadapusiSpin} or \cite[Section 2.2]{AndreattaGOrenHowardMadapusi}: If the signature $(r_{\tau}, s_{\tau})$ is equal to $(n-2,2)$, then $X_{\tau}$ is the space of oriented negative definite planes in $W_0 \otimes_{\mlp,\tau} \mathbb{R}$, and when the signature is $(2,n-2)$, we take the space of oriented negative definite planes in $W_0^{\vee} \otimes_{\mlp,\tau} \mathbb{R}$. \smallskip 

We let $C=\operatorname{CL}(W_0)$ denote the Clifford algebra of $W_0$ over $\mL^{+}$, which comes equipped with a canonical anti-involution $\ast$ (see \cite[Section 3.1]{KretShinOrthogonal}). By \cite[Section 1.6, Lemma 1.7]{MadapusiSpin}, a choice of an element $\delta \in C^{\times}$ satisfying $\delta^{\ast}=-\delta$ defines an $\mL^{+}$-linear symplectic form $\Psi_{\delta}$ on $C$. This defines a closed immersion
\begin{align}
    \iota_{0,\delta}:\g_0 \to \mathsf{G}_{0,C,\psi_{\delta}}
\end{align}
such that the similitude character of $\mathsf{G}_{W_{0}}$ restricts to the spinor norm $\mathcal{N}:\operatorname{GSpin}_{W_{0}} \to \gm$.
\begin{Lem} \label{Lem:ShimuraDataExistTypeB}
There is a $\delta$ with the property that for all $\tau \in \Sigma_{\infty,\mathrm{nc}}$ the map $\iota_{0,\delta} \otimes_{\mlp, \tau} \mathbb{R}$ is a morphism of weak real Shimura data and the induced map $X_{\tau} \to H_{V_{0,\tau,\mathbb{R}}}$ is surjective on $\pi_0$.
\end{Lem}
\begin{proof}
It is explained in \cite[Discussion before Lemma 2.2.1]{AndreattaGOrenHowardMadapusi} that if $\delta \in C^{\times}_{\mathbb{R}}$ has a certain positivity property for $\tau \in \Sigma_{\infty,\mathrm{nc}}$, then $\iota_{\delta, \tau, \mathbb{R}}$ is compatible with weak Shimura data. To show that we can choose a $\delta \in C^{\times}$ satisfying such a positive property for all $\tau$ at once, we note that the subspace $\{\delta^{\ast}=-\delta\} \subset C$ is a $\mathbb{Q}$-linear subspace and thus satisfies weak approximation. We find that $\{\delta^{\ast}=-\delta\} \subset C^{\times}$ is dense in $\{\delta^{\ast}=-\delta\} \subset C^{\times}_{\mathbb{R}}$, and we can thus find a $\delta$ which satisfies the right positive condition at all $\tau$ simultaneously (since the positivity condition is an open condition). The final statement about connected components follows the discussion in \cite[Section 7.1.2-7.1.3]{HowardPappas}.
\end{proof}
\begin{Lem} \label{Lem:CenterGSpinTypeBD}
The scalars $\gm \subset \g_{W_0}$ are contained in the image of $\iota_{0,\delta}$. If $n$ is odd, then $\gm \subset \mathsf{Z}_{\g_0}$ is an equality. 
\end{Lem}
\begin{proof}
    The first claim follows from \cite[equation (3.2)]{KretShinOrthogonal}, and the second claim is well known (it can be deduced from \cite[Lemma 3.1.(v)]{KretShinOrthogonal}.
\end{proof}
Finally, we define $\g_{2,0}=\operatorname{SO}(W_0)$ and $\g_{2}=\operatorname{Res}_{\mlp/\mathbb{Q}} \g_{2,0}$. We let $\gxtwo$ be the weak Shimura datum induced from the weak Shimura datum $\x_3$. If $W_0$ is odd dimensional, then the group $\g_{2,0}$ is an adjoint group. If $W_0$ is even dimensional, then the group $\g_{2,0}$ has center $\mu_2$.

\subsection{Type \texorpdfstring{$C$}{C}} \label{Sub:ExampleC} Now we closely follow \cite[Section 6.2]{DeligneTravaux}. Let $\mlp$ be a totally real field and let $\mb$ be a quaternion algebra over $\mlp$ with canonical involution $b \mapsto \ul{b}$. Let $V_{0}$ be a free $\mb$-module of finite rank $n$ equipped with a symmetric $\mlp$-bilinear nondegenerate pairing $\Phi$ satisfying for all $b \in \mb, x,y \in V_{0}$
\begin{align}
    \Phi(bx, y) = \Phi(x, \ul{b}y). 
\end{align}
Let $\g_{0}$ be the group of $\mb$-linear similitudes of $(V_{0},\Phi)$ over $\mlp$, with similitude character $\operatorname{sim}:\g_{0} \to \mathbb{G}_m$. Let $\iota_0:\g_{0} \to \g_{V_{0}}$ be the forgetful map to the group of $\mlp$-linear symplectic similitudes of $V_{0}$.

\subsubsection{} Let $\Sigma_{\infty,c}$ be the set of places of $\mlp$ such that $\mb \otimes_{\mlp, \tau} \mathbb{R}$ is a nonsplit (or definite) quaternion algebra, and let $\Sigma_{\infty,\mathrm{nc}}$ be its complement. We assume that $\Phi$ is chosen such that for $\tau \in \Sigma_{\infty,c}$, the group $\g_{0,\tau,\mathbb{R}}$ has $\mathbb{R}$-points that are compact modulo center; this is possible by \cite[the discussion surrounding Lemma 12.1 of v1]{KretShinSymplecticArxiv}. 
\begin{Lem} \label{Lem:TypeCComputation}
For each $\tau \in \Sigma_{\infty,\mathrm{nc}}$ there is a weak real Shimura datum $\x_{\tau}$ on $\g_{0,\tau,\mathbb{R}}$  with the property that $\iota_0 \otimes_{\mlp, \tau} \mathbb{R}$ is a morphism of weak real Shimura data. 
\end{Lem}
\begin{proof}
It follows from \cite[Lemma 12.1 of v1]{KretShinSymplecticArxiv} that for $\tau \in \Sigma_{\infty,\mathrm{nc}}$ we can identify 
\begin{align}
    \iota_0 \otimes_{\mlp, \tau} \mathbb{R}
\end{align}
with the natural map $\operatorname{GSp}_{2n, \mathbb{R}} \to \operatorname{GSp}_{4n, \mathbb{R}}$; thus we can use the standard Siegel Shimura datum for $\operatorname{GSp}_{2n, \mathbb{R}}$.
\end{proof}
As before, we get a Shimura datum $\x_3$ for $\g_3=\operatorname{Res}_{\mlp/\Q} \g_0$. In this case, it is clear that the image of $\iota_0$ contains the scalars $\gm$, and that this is equal to the center of $\g_{0}$ (for example by checking it over $\mathbb{R}$ using the arguments of Lemma \ref{Lem:TypeCComputation}).

\subsubsection{} Finally, we define $\g_{2,0}=\g_{0}^{\mathrm{ad}}$, set $\g_2=\g_3^{\mathrm{ad}}$ and let $\gxtwo$ be the adjoint Shimura datum associated to $\gxthree$.

\subsection{Type \texorpdfstring{$D^{\mathbb{R}}$}{DR}} \label{Sub:ExampleDR} We now give a generalization of the type $D^{\mathbb{R}}$ cases treated before. 

\subsubsection{} Let $\mlp$ be a totally real field and let $\mb$ be a totally indefinite quaternion algebra over $\mlp$ with canonical involution $b \mapsto \ul{b}$. Let $W$ be a free $\mb$-module of finite rank $n \ge 4$ equipped with a nondegenerate pairing $\Phi: W \times W \to \mb$ satisfying for all $b,c \in \mb, x,y \in W$
\begin{align}   
    \Phi(xb, yc) &= \ul{b} \Phi(x,y) c \\
    \Phi(x, y) &= \ul{-\Phi(x, y)}. 
\end{align}
Let $\g_0$ be the special Clifford group of $(W,\Phi)$ over $\mlp$ as in \cite[Section 2, Case B]{Shih}. 
\begin{Lem}
    For $\tau \in \Sigma_{\infty}$, there is an isomorphism
    \begin{align}
        \g_{0,\tau,\mathbb{R}} \xrightarrow{\sim} \operatorname{GSpin}(r_{\tau},s_{\tau}),
    \end{align}
    with $r_{\tau}+s_{\tau}=2n$. 
\end{Lem}
\begin{proof}
This follows from the discussion in \cite[Section 3, Case B]{Shih}. Note that op. cit. (and \cite[Appendix B]{Milne}) writes $\operatorname{Gpin}$ for the special Clifford group which is denoted by $\operatorname{GSpin}$ in the more recent literature that we have been citing (e.g. \cite{MadapusiSpin}, \cite{KretShinOrthogonal}, \cite{AndreattaGOrenHowardMadapusi}). 
\end{proof}
We assume that $r_{\tau},s_{\tau} \in \{0,2,2n,2n-2\}$ for all $\tau$, and write $\Sigma_{\infty,\mathrm{c}}$ for those $\tau$ where either $r_{\tau}$ or $s_{\tau}$ is zero, and $\Sigma_{\infty,\mathrm{nc}}$ for its complement. We can moreover choose Shimura data $\x_{\tau}$ for $\tau \in \Sigma_{\infty,\mathrm{nc}}$ as in the proof of Lemma \ref{Lem:ShimuraDataExistTypeB}. This induces a Shimura datum $\x_4$ on $\g_4=\operatorname{Res}_{\mlp/\Q} \g_0$.

\subsubsection{} Following \cite[Section 5]{Shih}, we choose a totally real quadratic extension $\mlp_{1}$ of $\mlp$ over which $\mb$ splits; this is possible because $\mb$ is totally indefinite. We then observe as in loc. cit. that 
\begin{align}
    \operatorname{Res}_{\mlp_{1}/\mlp} \g_{0,\mlp_{1}}
\end{align}
is isomorphic to $\operatorname{GSpin}(\tilde{W})$ for some quadratic space $(\tilde{W},\tilde{q})$ over $\mlp_{1}$ of dimension $2n$. The Shimura data $\x_{\tau}$ for $\tau \in \Sigma_{\infty, \mathrm{nc}}$ induce Shimura data $\x_{\tau_1}$ for $\tau \in \Sigma_{1,\infty, \mathrm{nc}}$, where $\Sigma_{1,\infty, \mathrm{nc}}$ denotes the set of maps $\mlp_1 \to \mathbb{R}$ whose restriction to $\mlp$ lies in $\Sigma_{\infty, \mathrm{nc}}$. \smallskip 

If we let $C$ be the Clifford algebra of $(\tilde{W},\tilde{q})$, then by Lemma \ref{Lem:ShimuraDataExistTypeB} there is a choice of $\delta$ together with a closed embedding
\begin{align}
    \iota_{0,\delta}:\operatorname{Res}_{\mlp_{1}/\mlp} \g_0 \to \operatorname{GSp}(C, \psi_{\delta})
\end{align}
which is compatible with Shimura data after tensoring up to $\mathbb{R}$ along $\tau \in \Sigma_{1,\infty, \mathrm{nc}}$. We let $\iota_0$ be the composition of $\iota_{0,\delta}$ with the natural map $\g_0 \to \operatorname{Res}_{\mlp_{1}/\mlp} \g_{0,\mlp_{1}}$, so that $\iota_0$ is compatible with Shimura data after tensoring up to $\mathbb{R}$ along $\tau \in \Sigma_{\infty, \mathrm{nc}}$.

\subsubsection{} Finally, we consider $\g_{2,0}=\operatorname{SO}_{\mb}(W)$ and $\g_{2}=\operatorname{Res}_{\mlp/\mathbb{Q}} \g_{2,0}$ with Shimura datum $\x_2$ induced from the Shimura datum $\x_3$. The group $\g_{2,0}$ has center $\mu_{2}$.

\subsection{Computing the center and the cocenter} \textbf{We now assume that we are in one of the cases of Sections \ref{Sub:ExampleBDR}, \ref{Sub:ExampleC} or \ref{Sub:ExampleDR}.} 
Then the image of $\iota_0:\g_0 \to \g_{V_{0}}$ contains the scalars $\mathbb{G}_m$. It is moreover true that $\mathsf{Z}_{\g_0}$ is equal to $\gm \times \mu_2$ in type $D^{\mathbb{R}}$ and equal to $\mathbb{G}_m$ otherwise; this follows from Lemma \ref{Lem:CenterGSpinTypeBD} and its proof. Moreover, the natural map $\nu:\mathsf{Z}_{\g_0} \to \mathbb{G}_m$ (where the target is the maximal abelian quotient of $\g_0$) agrees with $z \mapsto z^2$ on the scalars and is injective on the $\mu_2$ factor (if it exists). \smallskip

Consider the map (whose source we think of as the center of $\g_0 \times \mathsf{T}_{0,\ml}$)
\begin{align}
    \mathsf{Z}_{\g_0} \times \mathsf{T}_{0,\mL} &\to \mathbb{G}_m \times \mathsf{T}_{0,\mL}^1 \\
    (z,t) &\mapsto (\nu(z) t t^c, t/t^c).
\end{align}
This visibly factors through the quotient by the anti-diagonally embedded $\mathbb{G}_m$ via a map
\begin{align}
    \nu':\mathsf{Z}_{\g_{0,+}} \to \mathbb{G}_m \times \mathsf{T}_{0,\mL}^1=\g_{0,+}^{\ab}. 
\end{align}
\begin{Lem} \label{Lem:CenterComputationHodgeType}
The natural map $\{1\} \times \mathsf{T}_{0,\mL}^1 \to \mathsf{Z}_{\g_{0,+}}$ maps isomorphically onto $\mathsf{Z}_{\g_{0,++}}^{\circ}$
\end{Lem}
\begin{proof}
We first note that $\{1\} \times \mathsf{T}_{0,\mL}^1$ maps to $\{1\} \times \mathsf{T}^1_{\mL}$ under $\nu$, so that it indeed maps (injectively) to $\mathsf{Z}_{\g_{0,++}} \subset \mathsf{Z}_{\g_{0,+}}$. The lemma now follows from a straightforward dimension argument. 
\end{proof}
\begin{Cor} \label{Cor:CenterCoCenterComputationHodge}
The connected center and the maximal abelian quotient of $\g_{0,++}$ are both isomorphic to $\mathsf{T}_{\mL}^1$. Under these identifications, the natural map is given by the multiplication by $2$ map $[2]$.
\end{Cor}
\begin{proof}
This is a direct consequence of Lemma \ref{Lem:CenterComputationHodgeType} and its proof.
\end{proof}

\begin{Lem} \label{Lem:ConnectedCenter}
If $G^{\mathrm{ad}}_{\mathbb{C}}$ is not of type $D$, then $\mathsf{Z}_{\g_{0,++}}$ is connected.   
\end{Lem}
\begin{proof}
By Lemma \ref{Lem:CenterGSpinTypeBD}, the natural map $\mathbb{G}_m \to \mathsf{Z}_{\mathsf{G}_{0}}$ is an isomorphism under our assumption. It follows that the center of $\g_{0,+}$ is isomorphic to
\begin{align}
    \frac{\mathsf{T}_{0,\ml} \times \mathbb{G}_m}{\mathbb{G}_m}. 
\end{align}
Under this identification, the similitude map to $\mathbb{G}_m \times \mathsf{T}^1_{0,\ml}$ is given by
\begin{align}
    (t,z) \mapsto (z^2 tt^c, t/t^c). 
\end{align}
Writing
\begin{align}
    \mathsf{T}_{u}=\{(t,z) \in \mathsf{T}_{0,\ml} \times \mathbb{G}_m \; | \; (t^2=t t^c) \},
\end{align}
we see that $\mathsf{Z}_{\g_{0,++}}$ is given by $T_u / \mathbb{G}_m$; it thus suffices to show that $T_u$ is connected. This torus lies in a Cartesian diagram
\begin{equation}
    \begin{tikzcd}
        T_u \arrow{r} \arrow{d} & \mathsf{T}_{0,\ml} \arrow{d}{t \mapsto t t^c} \\
        \mathbb{G}_m \arrow{r}{[2]} & \mathbb{G}_m.
    \end{tikzcd}
\end{equation}
Since the right vertical map is surjective with connected kernel, the same is true for the left vertical map, showing that $T_u$ is connected.
\end{proof}

\subsection{Some cohomology computations} Let the notation be as in the previous section, in particular we will continue to \textbf{assume} that we are in one of the examples of Sections \ref{Sub:ExampleBDR}, \ref{Sub:ExampleC}, \ref{Sub:ExampleDR}. We will now deduce some cohomological consequences of Corollary \ref{Cor:CenterCoCenterComputationHodge} which will be used in Section \ref{Sec:Examples}.

\begin{Lem} \label{Lem:HassePrinciple}
The group $\g_{0,++}$ satisfies $\Sha^1(\mlp,\g_{0,++})=0$ and the group $\g$ satisfies $\Sha^1(\Q,\g)=0$.
\end{Lem}
\begin{proof}
Since $\g_0^{\der}=\g_{0,++}^{\der}$ is simply connected, we have that
\begin{align}
    \Sha^1(\mlp, \g_{0,++}) = \Sha^1(\mlp, \g_{0,++}^{\ab})
\end{align}
and $\g_{0,++}^{\ab}=\mathsf{T}_{0,\ml}^{1}$ by Corollary \ref{Cor:CenterCoCenterComputationHodge}. The result now follows from the Hasse principle for quadratic extensions. Similarly since $\g^{\der}$ is simply connected, it follows that 
\begin{align}
    \Sha^1(\Q, \g) = \Sha^1(\Q, \g^{\ab})
\end{align}
and $\g^{\ab}=\gm \times \mathsf{T}_{\ml}^{1}$; the result follows as before (using Shapiro's lemma).
\end{proof}

\begin{Lem} \label{Lem:AlgebraicFundamentalGroupComputationHodge}
    The group $\pi_1(\g_{0,++})$ is isomorphic to $\mathbb{Z}$, with the Galois group $\operatorname{Gal}_{\mL^{+}}$ acting through $\operatorname{Gal}(\mL/\mlp)$ such that the nontrivial element acts by multiplication by $-1$.
\end{Lem}
\begin{proof}
    Since $\g_0^{\der}=\g_{0,++}^{\der}$ is simply connected, we have that
    \begin{align}
        \pi_1(\g_{0,++}) \xrightarrow{\sim} X_{\ast}(\g_{0,++}^{\ab}),
    \end{align}
    which is isomorphic to $\mathbb{Z}$ with the desired Galois action by Lemma \ref{Lem:CocentersHodgeType}.
\end{proof}

\begin{Lem} \label{Lem:FundamentalGroupComputationAdjoint}
There is an isomorphism (where the Galois action is trivial on the right hand side)
\begin{align}
    \pi_1(\g_{2,0}) \simeq \mathbb{Z}/2\mathbb{Z}
\end{align}
\end{Lem}
\begin{proof}
    This is well known. The Galois action is trivial because $\g_{2,0}$ is an inner form of a group that splits over a quadratic extension of $\mlp$, and groups of order $2$ cannot act nontrivially on $\mathbb{Z}/2\mathbb{Z}$.
\end{proof}

\begin{Cor} \label{Cor:AlgebraicFundamentalGroupHodgeAdjoint}
The natural map
\begin{align}
    \pi_1(\g_{0,++})_{\gal_{\mlp}} \to \pi_1(\g_{2,0})_{\gal_{\mlp}}
\end{align}
is an isomorphism
\end{Cor}
\begin{proof}
Consider the long exact sequence in cohomology associated to the short exact sequence
    \begin{align}
        0 \to X_{\ast}(\mathsf{Z}_{\g_{0,++}}^{\circ}) \to \pi_1(\g_{0,++}) \to \pi_1(\g_{2,0}) \to 0.
    \end{align}
The map 
\begin{align}
    X_{\ast}(\mathsf{Z}_{\g_{0,++}}^{\circ}) \to \pi_1(\g_{0,++})
\end{align}
can be identified (by Lemma \ref{Lem:CenterComputationHodgeType}) with $[2]:\mathbb{Z} \to \mathbb{Z}$, and thus induces the zero map after taking coinvariants for the Galois action. Thus
\begin{align}
     \pi_1(\g_{0,++})_{\gal_{\mlp}} \to \pi_1(\g_{0})_{\gal_{\mlp}}
\end{align}
is injective. Both the left hand side and the right hand side are isomorphic to $\mathbb{Z}/2\mathbb{Z}$, and we conclude.
\end{proof}

\begin{Lem} \label{Lem:CohomologyComputationHodgetype}
Let $p$ be a prime number and let $\mathfrak{p}$ be a prime above $p$ in $\mlp$. The natural map (of abelian groups)
\begin{align}
    H^1(\mlp_{\mathfrak{p}}, \mathsf{Z}_{\g_{0,++}}^{\circ}) \to H^1(\mlp_{\mathfrak{p}}, \g_{0,++})
\end{align}
is zero.
\end{Lem}
\begin{proof}
Since $\g_0^{\der}$ is simply connected, we can identify the map of the lemma with
\begin{align}
    H^1(\mlp_{\mathfrak{p}}, \mathsf{Z}_{\g_{0,++}}^{\circ}) \to H^1(\mlp_{\mathfrak{p}}, \g_{0,++}^{\ab}).
\end{align}
Now Corollary \ref{Cor:CenterCoCenterComputationHodge} tells us that $\mathsf{Z}_{\g_{0,++}}^{\circ} \to \g_{0,++}^{\ab}$ may be identified with $[2]:\mathsf{T}^1_{\ml} \to \mathsf{T}^1_{\ml}$. If $\mathfrak{p}$ splits in $\ml$, then $\mathsf{T}^1_{\ml} \otimes \mlp_{\mathfrak{p}}$ has no cohomology and we are done. If $\mathfrak{p}$ does not split in $\ml$, then $H^1(\mlp_{\mathfrak{p}},\mathsf{T}^1_{\ml}) \simeq \mathbb{Z}/2\mathbb{Z}$, so that multiplication by $2$ is equal to the zero map, proving the lemma. 
\end{proof}

\begin{Lem} \label{Lem:CohomologyComputationHodgetypeII}
Let $p$ be a prime number and let $\mathfrak{p}$ be a prime above $p$ in $\mlp$. The natural map (of abelian groups)
\begin{align}
    H^1(\mlp_{\mathfrak{p}}, \g_{0,++}) \to H^1(\mlp_{\mathfrak{p}}, \g_{2,0}) = \pi_1(\g_{2,0})_{\gal_{\mlp_{\mathfrak{p}}}}
\end{align}
is injective, and an isomorphism if $\mathfrak{p}$ is nonsplit in $\ml$.
\end{Lem}
\begin{proof}
The long exact sequence in cohomology associated to $1 \to \mathsf{Z}_{\g_{0,++}}^{\circ} \to \g_{0,++} \to \g_{2,0} \to 1$ together with the proof of Lemma \ref{Lem:CohomologyComputationHodgetype} tells us the map is injective. If $\mathfrak{p}$ is split in $\ml$, then it follows from Corollary \ref{Cor:CenterCoCenterComputationHodge} that $H^1(\mlp_{\mathfrak{p}}, \g_{0,++}) \simeq 0$. If $\mathfrak{p}$ is nonsplit in $\ml$, then it follows from Corollary \ref{Cor:CenterCoCenterComputationHodge} that $H^1(\mlp_{\mathfrak{p}}, \g_{0,++}) \simeq \mathbb{Z}/2 \mathbb{Z}$, and it now follows as in the proof of Corollary \ref{Cor:AlgebraicFundamentalGroupHodgeAdjoint} that it is also surjective.
\end{proof}
\begin{Lem} \label{Lem:FundamentalGroupCoinvariantsR}
Let $\tau \in \Sigma_{\infty}$. The natural map
\begin{align}
    \pi_{1}(\g_{0,++} \otimes_{\mlp, \tau} \mathbb{R})_{\gal_{\mathbb{R}}} \to \pi_{1}(\g_{2,0} \otimes_{\mlp, \tau} \mathbb{R})
\end{align}
is an isomorphism.
\end{Lem}
\begin{proof}
It follows as in the proof of Lemma \ref{Lem:CohomologyComputationHodgetype} that the map to $\pi_{1}(\g_{2,0} \otimes_{\mlp, \tau} \mathbb{R})$ is injective. It moreover follows as in the proof of Corollary \ref{Cor:AlgebraicFundamentalGroupHodgeAdjoint} that the image gets identified with $\pi_1(\g_{2,0})$
\end{proof}
The following lemma is well known. 
\begin{Lem} \label{Lem:KottwitzInvariantCompactForm}
For $\tau \in \Sigma_{\infty,\mathrm{nc}}$, let $[Q] \in H^1(\mathbb{R}, \g_{2,0}^{\mathrm{ad}} \otimes_{\mlp, \tau} \mathbb{R})$ be the class of the compact inner form. Then the image of $[Q]$ in $\pi_1(\g_{2,0}^{\mathrm{ad}})$ is nontrivial. 
\end{Lem}
\begin{proof}
By \cite[Lemma 2.1.4]{XiaoZhu}, the image of $[Q]$ agrees with the image of $\mu_{\tau}$, where $\mu_{\tau}$ is the Hodge cocharacter associated to $\x_{2,\tau}$. This Hodge cocharacter lifts to a cocharacter $\mu_{0,\tau}$ of $\g_{0} \otimes_{\mlp, \tau} \mathbb{R}$ because of the existence of $\x_{\tau}$. It suffices to show that $\mu_{0,\tau}$ maps to $1 \in \mathbb{Z} = \pi_1(\g_{0,\tau})$. Now $\iota_0$ identifies $\pi_1(\g_0) \xrightarrow{\sim} \pi_1(G_{V_{0}})$, and it is well known that in the Siegel case $\mu_0$ maps to $1$ under the similitude map. 
\end{proof}

\begin{Lem} \label{Lem:KottwitzInvariantShimuraForm}
For $\tau \in \Sigma_{\infty,\mathrm{c}}$, let $[Q] \in H^1(\mathbb{R}, \g_{2,0} \otimes_{\mlp, \tau} \mathbb{R})$ be the class of an inner form admitting a Shimura datum. Then the image of $[Q]$ in $\pi_1(\g_{2,0})$ is nontrivial.
\end{Lem}
\begin{proof}
Write $H=\g_{2,0} \otimes_{\mlp, \tau} \mathbb{R}$, choose $Q \in [Q]$ and let $H'=\operatorname{Aut}_H(Q)$ be the inner form admitting a Shimura datum. Then $Q$ induces a canonical Galois equivariant identification $\pi_1(H)=\pi_1(H')$ and we have the following commutative diagram (see Lemma \ref{Lem:RealSquare})
    \begin{equation}
        \begin{tikzcd}
            H^1(\R, H) \arrow{r}{\beta_Q} \arrow{d} & H^1(\R, H') \arrow{d} \\
            \pi_1(H)_{\gal_{\R}} \arrow{r}{+[Q]} & \pi_1(H')_{\gal_{\R}},
        \end{tikzcd}
    \end{equation}
    where the top map sends $[Q]$ to $[1]$. Now we note that $Q$ is a $(H,H')$-bitorsor, and if we write $P$ for the underlying $H'$-torsor, then $\beta_P$ is the inverse of $\beta_{Q}$. So we see that under the identification $\pi_1(H)_{\gal_{\R}}=\pi_1(H')_{\gal_{\R}}$, the image of $[Q]$ is given by minus the image of $[P]$. Now Lemma \ref{Lem:KottwitzInvariantCompactForm} applies to show that the image of $[P]$ is nontrivial, so we are done. 
\end{proof}

\DeclareRobustCommand{\VAN}[3]{#3}
\bibliographystyle{amsalpha}
\bibliography{references}

\end{document}